\documentclass[leqno]{article} 

\usepackage[frenchb,english]{babel}
\usepackage[T1]{fontenc}

\usepackage{amsmath}
\usepackage{amssymb}
\usepackage{amsfonts}
\usepackage{enumerate}
\usepackage{vmargin}
\usepackage[all]{xy}
\usepackage{mathrsfs}
\usepackage{mathtools}
\usepackage{lmodern}
\usepackage[colorlinks=true,linkcolor=blue,pagebackref=false]{hyperref}
\usepackage{ulem}
\usepackage[all]{xy}
\usepackage{comment}
\usepackage{centernot}
\setmarginsrb{3cm}{3cm}{3.5cm}{3cm}{0cm}{0cm}{1.5cm}{3cm}

\newcommand{\B}{\mathrm{B}}
\newcommand{\C}{\ensuremath{\mathbb{C}}}
\newcommand{\E}{\ensuremath{\mathbb{E}}}
\newcommand{\F}{\ensuremath{\mathbb{F}}}
\newcommand{\Fc}{\ensuremath{\mathcal{F}}}
\newcommand{\G}{\ensuremath{\mathbb{G}}}

\newcommand{\I}{\mathrm{I}}
\let\L\relax
\newcommand{\L}{\mathrm{L}}
\newcommand{\M}{\mathrm{M}}

\newcommand{\N}{\ensuremath{\mathbb{N}}}
\newcommand{\Q}{\ensuremath{\mathbb{Q}}}
\newcommand{\R}{\ensuremath{\mathbb{R}}}
\newcommand{\T}{\ensuremath{\mathbb{T}}}
\newcommand{\Z}{\ensuremath{\mathbb{Z}}}
\newcommand{\X}{\ensuremath{\mathrm{X}}}
\newcommand{\trans}[1]{\prescript{t}{}{#1}} 

\newcommand{\tr}{\ensuremath{\mathop{\rm Tr\,}\nolimits}}

\renewcommand{\leq}{\ensuremath{\leqslant}}
\renewcommand{\geq}{\ensuremath{\geqslant}}
\newcommand{\qed}{\hfill \vrule height6pt  width6pt depth0pt}

\newcommand{\bnorm}[1]{ \big\| #1  \big\|}
\newcommand{\Bnorm}[1]{ \Big\| #1  \Big\|}
\newcommand{\bgnorm}[1]{ \bigg\| #1  \bigg\|}
\newcommand{\Bgnorm}[1]{ \Bigg\| #1  \Bigg\|}
\newcommand{\norm}[1]{\left\Vert#1\right\Vert}
\newcommand{\xra}{\xrightarrow}
\newcommand{\otp}{\widehat{\ot}}

\newcommand{\ot}{\otimes}
\newcommand{\epsi}{\varepsilon}
\newcommand{\ovl}{\overline}
\newcommand{\otvn}{\ovl\ot}

\newcommand{\dsp}{\displaystyle}
\newcommand{\co}{\colon}
\renewcommand{\d}{\mathop{}\mathopen{}\mathrm{d}} 
\let\i\relax 
\newcommand{\i}{\mathrm{i}}
\newcommand{\w}{\mathrm{w}}
\newcommand{\exc}{\mathrm{exc}} 
\newcommand{\ov}{\overset}
\newcommand{\Rad}{\mathrm{Rad}}

\newcommand{\dec}{\mathrm{dec}}
\newcommand{\Dec}{\mathrm{Dec}}
\newcommand{\reg}{\mathrm{reg}}
\newcommand{\Aut}{\mathrm{Aut}}
\newcommand{\Reg}{\mathrm{Reg}}

\newcommand{\QWEP}{\mathrm{QWEP}}
\newcommand{\sign}{\mathrm{sign}}

\newcommand{\disc}{\mathrm{disc}}
\newcommand{\dist}{\mathrm{dist}}
\newcommand{\Id}{\mathrm{Id}}
\newcommand{\VN}{\mathrm{VN}}
\newcommand{\CB}{\mathrm{CB}}
\newcommand{\loc}{\mathrm{loc}}
\newcommand{\op}{\mathrm{op}} 
\newcommand{\e}{\mathrm{e}} 
\let\ker\relax 
\DeclareMathOperator{\ker}{Ker} 
\DeclareMathOperator{\Ran}{Ran} 

\DeclareMathOperator{\supp}{supp} 
\DeclareMathOperator{\card}{card} 
\DeclareMathOperator{\Span}{span} 
\newcommand{\cb}{\mathrm{cb}} 
\newcommand{\cp}{\mathrm{cp}} 
\newcommand{\CV}{\mathrm{CV}} 
\newcommand{\HS}{\mathrm{HS}} 
\newcommand{\tree}{\mathcal{T}} 
\newcommand{\St}{\mathrm{St}} 
\newcommand{\ALSS}{\mathrm{ALSS}}
\newcommand{\ALS}{\mathrm{ALS}}
\newcommand{\ADS}{\mathrm{ADS}}
\newcommand{\sing}{\mathrm{sing}} 
\newtheorem{thm}{Theorem}[section]
\newtheorem{defi}[thm]{Definition}

\newtheorem{prop}[thm]{Proposition}

\newtheorem{cor}[thm]{Corollary}
\newtheorem{lemma}[thm]{Lemma}

\newtheorem{remark}[thm]{Remark}

\newenvironment{proof}[1][]{\noindent {\it Proof #1} : }{\hbox{~}\qed
\smallskip
}

\numberwithin{equation}{section}
\usepackage[nottoc,notlot,notlof]{tocbibind}

\let\OLDthebibliography\thebibliography
\renewcommand\thebibliography[1]{
  \OLDthebibliography{#1}
  \setlength{\parskip}{0pt}
  \setlength{\itemsep}{0pt plus 0.3ex}
}

\begin{document}
\selectlanguage{english}
\title{\bfseries{Projections, multipliers and decomposable maps on noncommutative $\L^p$-spaces}}
\date{}

\author{\bfseries{C\'edric Arhancet - Christoph Kriegler}}

\maketitle

\begin{abstract}
We introduce a noncommutative analogue of the absolute value of a regular operator acting on a noncommutative $\L^p$-space. We equally prove that two classical operator norms, the regular norm and the decomposable norm are identical. We also describe precisely the regular norm of several classes of regular multipliers. This includes Schur multipliers and Fourier multipliers on some unimodular locally compact groups which can be approximated by discrete groups in various senses. A main ingredient is to show the existence of a bounded projection from the space of completely bounded $\L^p$ operators onto the subspace of Schur or Fourier multipliers, preserving complete positivity. On the other hand, we show the existence of bounded Fourier multipliers which cannot be approximated by regular operators, on large classes of locally compact groups, including all infinite abelian locally compact groups. We finish by introducing a general procedure for proving positive results on selfadjoint contractively decomposable Fourier multipliers, beyond the amenable case.
\end{abstract}

\makeatletter
 \renewcommand{\@makefntext}[1]{#1}
 \makeatother

\tableofcontents

\newpage

\section{Introduction}
\label{sec:Introduction}

The absolute value $|T|$ and the regular norm $\norm{T}_\reg$ of a regular operator $T$ already appear in the seminal work of Kantorovich \cite{Kan} on operators on linear ordered spaces. These constructions essentially rely on the structure of (Dedekind complete) Banach lattices. These notions are of central importance in the theory of linear operators between Banach lattices, including classical $\L^p$-spaces, since the absolute value is a positive operator. Indeed it is well-known that positive contractions are well-behaved operators. Actually, contractively regular operators on $\L^p$-spaces share in general the same nice properties as contractions on Hilbert spaces. We refer to the books \cite{AbA}, \cite{MeN} and \cite{Sch1} and to the papers \cite{Pis11} and \cite{Pis14} for more information.

Due to the lack of local unconditional structure, on a Schatten space and more generally on a noncommutative $\L^p$-space, the canonical order on the space of selfadjoint elements does not induce a structure of a Banach lattice, see \cite[Chapter 17]{DJT} and \cite[page 1478]{PiX}. Nevertheless, there exists a purely Banach space characterization of regular operators on classical $\L^p$-spaces \cite[Theorem 2.7.2]{{HvNVW}} which says that a linear operator $T \co \L^p(\Omega) \to \L^p(\Omega')$ is regular if and only if for any Banach space $X$ the map $T \ot \Id_X$ induces a bounded operator between the Bochner spaces $\L^p(\Omega,X)$ and $\L^p(\Omega',X)$. In this case, the regular norm is given by
\begin{equation}
\label{Norm-reg-c}
\norm{T}_{\reg,\L^p(\Omega) \to \L^p(\Omega')}
\ov{\mathrm{def}}{=} \sup_X \norm{T \ot \Id_X}_{\L^p(\Omega,X)\to \L^p(\Omega',X)},
\end{equation}
where the supremum runs over all Banach spaces $X$.
Using this property, a natural extension of this notion for noncommutative $\L^p$-spaces is introduced in \cite{Pis4}. A linear map $T \co \L^p(M) \to \L^p(N)$ between noncommutative $\L^p$-spaces, associated with approximately finite-dimensional von Neumann algebras $M$ and $N$, is called regular if for any noncommutative Banach space $E$ (that is, an operator space), the map $T \ot \Id_E$ induces a bounded operator between the vector-valued noncommutative $\L^p$-spaces $\L^p(M,E)$ and $\L^p(N,E)$. As in the commutative case, the regular norm is defined by
\begin{equation}
\label{Norm-reg-nc}
\norm{T}_{\reg,\L^p(M) \to \L^p(N)}
\ov{\mathrm{def}}{=} \sup_E \norm{T \ot \Id_E}_{\L^p(M,E) \to \L^p(N,E)},
\end{equation}
where the supremum runs over all operator spaces $E$. For classical $\L^p$-spaces, this norm coincides with \eqref{Norm-reg-c}. 
Nevertheless, the paper \cite{Pis4} does not give a definition of the absolute value of a regular operator and the definition of the latter is only usable for \textit{approximately finite-dimensional} von Neumann algebras. 

In this paper, we define a noncommutative analogue of the absolute value of a regular operator acting on an arbitrary noncommutative $\L^p$-space for any $1 \leq p \leq \infty$. For that, recall that a linear map $T \co \L^p(M) \to \L^p(N)$ is decomposable \cite{Haa,JuR}  if there exist linear maps $v_1,v_2 \co \L^p(M) \to \L^p(N)$ such that the linear map
\begin{equation}
\label{Matrice-2-2-Phi}
\Phi \ov{\mathrm{def}}{=} \begin{bmatrix}
   v_1  &  T \\
   T^\circ  &  v_2  \\
\end{bmatrix}
\co S^p_2(\L^p(M)) \to S^p_2(\L^p(N)), \quad \begin{bmatrix}
   a  &  b \\
   c &  d  \\
\end{bmatrix}\mapsto 
\begin{bmatrix}
   v_1(a)  &  T(b) \\
   T^\circ(c)  &  v_2(d)  \\
\end{bmatrix}
\end{equation}
is completely positive (a stronger condition than positivity of operators) where $T^\circ(c)\overset{\textrm{def}}{=}T(c^*)^*$ and where $S^p_2(\L^p(M))$ and $S^p_2(\L^p(N))$ are vector-valued Schatten spaces. In this case, $v_1$ and $v_2$ are completely positive and the decomposable norm of $T$ is defined by
\begin{equation}
\label{Norm-dec}
\norm{T}_{\dec,\L^p(M) \to \L^p(N)}
\ov{\mathrm{def}}{=} \inf\big\{\max\{\norm{v_1},\norm{v_2}\}\big\},
\end{equation}
where the infimum is taken over all maps $v_1$ and $v_2$. See the books \cite{BLM}, \cite{ER} and \cite{Pis7} for more information on this classical notion in the case $p=\infty$. If $1<p<\infty$ and if $M$ and $N$ are approximately finite-dimensional, it is alluded in the introduction of \cite{JuR} that these maps coincide with the regular maps. First, we greatly strengthen this statement by showing that the regular norm $\norm{T}_{\reg,\L^p(M) \to \L^p(N)}$ and the decomposable norm $\norm{T}_{\dec,\L^p(M) \to \L^p(N)}$ are identical for a regular map $T$ (see Theorem \ref{thm-dec=reg-hyperfinite}). 
Hence, the decomposable norm is an extension of the regular norm for noncommutative $\L^p$-spaces associated to arbitrary von Neumann algebras. Moreover, we prove that if $T \co \L^p(\Omega) \to \L^p(\Omega')$ is a regular operator between classical $\L^p$-spaces then the map 
$\begin{bmatrix} 
|T| & T \\ 
T^\circ & |T| 
\end{bmatrix} 
\co S^p_2(\L^p(\Omega)) \to S^p_2(\L^p(\Omega'))$ is completely positive (Theorem \ref{prop-reg=decomp}) where $|T| \co \L^p(\Omega) \to \L^p(\Omega')$ denotes the absolute value of $T$. In addition, we show that the infimum \eqref{Norm-dec} is actually a minimum (Proposition \ref{Prop-dec-inf-atteint}). Consequently, the map \eqref{Matrice-2-2-Phi} with some $v_1,v_2$ which realize the infimum \eqref{Norm-dec} can be seen as a natural noncommutative analogue of the absolute value $|T|$ although we have no uniqueness results for $v_1$ and $v_2$.

The ingredients of the identification of the decomposable norm and the regular norm involve a reduction of the problem on noncommutative $\L^p$-spaces to the case of finite-dimensional Schatten spaces $S^p_n$ by approximation. Moreover, a 2x2-matrix trick gives a second reduction to adjoint preserving maps between these spaces. Finally, the case of adjoint preserving maps acting on finite-dimensional Schatten spaces is treated in Theorem \ref{prop-dec-reg-matrix-case}. To conclude, note that the ideas of the manuscript \cite{Jun2} (which seems definitely postponed) could be used to define a notion of regular operator between vector-valued noncommutative $\L^p$-spaces associated with QWEP von Neumann algebras. Of course, it is likely that the identification of the decomposable norm and the regular norm is true in this generalized context. Finally, we refer to the preprint \cite{Arh8} for a generalization of the notion of decomposable map and for applications to contractively complemented subspaces of noncommutative $\L^p$-spaces.

The next task is devoted to identify precisely decomposable Fourier multipliers on noncommutative $\L^p$-spaces $\L^p(\VN(G))$ of a group von Neumann algebra $\VN(G)$ associated to a unimodular locally compact group $G$. Recall that if $G$ is a locally compact group then $\VN(G)$ is the von Neumann algebra, whose elements act on the Hilbert space $\L^2(G)$, generated by the left translation unitaries $\lambda_s \co f \mapsto f(s^{-1} \cdot)$, $s \in G$. If $G$ is abelian, then $\VN(G)$ is $*$-isomorphic to the algebra $\L^\infty(\hat{G})$ of essentially bounded functions on the dual group $\hat{G}$ of $G$. As basic models of quantum groups, they play a fundamental role in operator algebras and this task can be seen as an effort to develop $\L^p$-Fourier analysis of non-abelian locally compact groups, see the contributions \cite{CPPR}, \cite{JMP1}, \cite{JMP2}, \cite{JR1}, \cite{LaS} and \cite{MR} in this line of research and references therein. If $G$ is discrete, a Fourier multiplier $M_\varphi \co \L^p(\VN(G)) \to \L^p(\VN(G))$ is an operator which maps $\lambda_s$ to $\varphi(s)\lambda_s$, where $\varphi \co G \to \C$ is the symbol function (see Definition \ref{defi-Lp-multiplier} for the general case of unimodular locally compact groups). 

We connect this problem with several notions of approximation by discrete groups of the underlying locally compact group $G$. We are able to show that a symbol $\varphi \co G \to \C$ inducing a decomposable Fourier multiplier $M_\varphi \co \L^p(\VN(G)) \to \L^p(\VN(G))$ already induces a decomposable Fourier multiplier $M_\varphi \co \VN(G) \to \VN(G)$ at the level $p=\infty$ for some classes of locally compact groups. 
We also give a comparison between the decomposable norm at the level $p$ and the operator norm at the level $\infty$ in some cases (see Theorem \ref{prop-non-amenable-discrete-Fourier-multiplier-dec-infty}, Theorem \ref{prop-amenable-discrete-Fourier-multiplier-dec-infty}, Theorem \ref{prop-continuous-Fourier-multiplier-dec-infty}, Theorem \ref{prop-continuous-Fourier-multiplier-dec-infty2} and Theorem \ref{thm-convolutor-description-dec-infty}). Our method for this last point relies on some constructions of compatible bounded projections at the level $p=1$ and $p=\infty$ from the spaces of (weak* continuous if $p=\infty)$  completely bounded operators on $\L^p(\VN(G))$ onto the spaces $\mathfrak{M}^{p,\cb}(G)$ of completely bounded Fourier multipliers combined with an argument of interpolation. We highlight that the nature of the group $G$ seems to play a central role in this problem. Indeed, mysteriously, our results are better for a pro-discrete group $G$ than for a non-abelian nilpotent Lie group $G$. More precisely, let us consider the following definition\footnote{\thefootnote. The subscript w* means ``weak* continuous'' and ``$\CB$'' means completely bounded. The compatibility is taken in the sense of interpolation theory \cite{BeL,Tri}.}.

\begin{defi}
\label{Defi-complementation-G}
Let $G$ be a (unimodular) locally compact group. We say that $G$ has property $(\kappa)$ if there exist compatible bounded projections $P_G^\infty \co \CB_{\w^*}(\VN(G)) \to \CB_{\w^*}(\VN(G))$ and $P_G^1 \co \CB(\L^1(\VN(G))) \to \CB(\L^1(\VN(G)))$ onto $\mathfrak{M}^{\infty,\cb}(G)$ and $\mathfrak{M}^{1,\cb}(G)$ preserving complete positivity. In this case, we introduce the constant 
$$
\kappa(G)
\ov{\mathrm{def}}{=} \inf \max\Big\{\norm{P_G^\infty}_{\CB_{\w^*}(\VN(G)) \to \CB_{\w^*}(\VN(G))},\norm{P_G^1}_{\CB(\L^1(\VN(G))) \to \CB(\L^1(\VN(G)))} \Big\}
$$ 
where the infimum is taken over all admissible couples $(P_G^\infty,P_G^1)$ of compatible bounded projections and we let $\kappa(G)=\infty$ if $G$ does not have $(\kappa)$.
\end{defi}

Haagerup has essentially proved that $\kappa(G)=1$ if $G$ is a discrete group by a well-known average argument using the unimodularity and the compactness of the quantum group $\VN(G)$. The key novelty in our approach is the use of approximating methods by discrete groups in various senses to construct bounded projections for non-discrete groups beyond the case of a dual of a unimodular compact quantum group. 
 If $G$ is a second countable pro-discrete locally compact group, we are able to show that $\kappa(G)=1$ (see Theorem \ref{thm-complementation-quotient}). Another main result of the paper gives $\kappa(G) < \infty$ for a certain class of locally compact groups $G$ approximable by lattice subgroups, see Corollary \ref{cor-ADS-complementation-without-amenability}. Note that a straightforward duality argument combined with some results of Derighetti \cite[Theorem 5]{Der3}, Arendt and Voigt \cite[Theorem 1.1]{ArV} says that if $G$ is an abelian locally compact group then $\kappa(G)=1$ (see Proposition \ref{prop-ma-complementation-Fourier-multipliers-abelian}). Furthermore, in most cases, we will show the existence of compatible projections $P_G^p \co \CB(\L^p(\VN(G))) \to \CB(\L^p(\VN(G)))$ onto $\mathfrak{M}^{p,\cb}(G)$ for \textit{all} $1 \leq p \leq \infty$\footnote{\thefootnote. If $p=\infty$, replace $\CB(\L^p(\VN(G)))$ by $\CB_{\w^*}(\VN(G))$}. So we have a strengthening $(\kappa')$ of property $(\kappa)$ for some groups. It is an open question whether $(\kappa')$ is really different from $(\kappa)$. Finally, in a paper \cite{ArK1}, examples of locally compact groups without $(\kappa)$ will be described and important complementary results will be given.

Using classical results from approximation properties of discrete groups, it is not difficult to see that there exist completely bounded Fourier multipliers $M_\varphi \co \L^p(\VN(G)) \to \L^p(\VN(G))$ on some class of discrete groups which are not decomposable (Proposition \ref{Prop-cb-mais-pas-dec-Fourier-mult}). In Chapter \ref{sec:Existence-of-strongly}, we focus on a more difficult task. We examine the problem to construct completely bounded operators $T \co \L^p(M) \to \L^p(M)$ which cannot be approximated by decomposable operators, in the sense that $T$ does not belong to the closure $\ovl{\Dec(\L^p(M))}$ of the space $\Dec(\L^p(M))$ of decomposable operators on $\L^p(M)$ with respect to the operator norm $\norm{\cdot}_{\L^p(M) \to \L^p(M)}$ (or the completely bounded norm $\norm{\cdot}_{\cb,\L^p(M) \to \L^p(M)}$). 

We particularly investigate different types of multipliers.  We show the existence of such completely bounded Fourier multipliers, on large classes of locally compact groups, including all infinite abelian locally compact groups (see Theorem \ref{thm-existence-CB-strongly-non-decomposable-abelian-groups}). Note that it is impossible to find such bad multipliers on finite groups by an argument of finite dimensionality. Our strategy relies on the use of transference theorems which we prove and structure theorems on groups. It consists in dealing with all possible cases. In the abelian situation, the construction of our examples in the critical cases (e.g. if the dual group $\hat{G}$ is an infinite totally disconnected group or an infinite torsion discrete group) is proved by a Littlewood-Paley decomposition argument on the Bochner space $\L^p(G,X)$ where $X$ is a UMD Banach space, which allows us to obtain in addition the complete boundedness of multipliers. We also examine the case of Schur multipliers. In particular, we prove that the discrete noncommutative Hilbert transform $\mathcal{H} \co S^p \to S^p$ on the Schatten space $S^p$ is not approximable by decomposable operators (Corollary \ref{Cor-strongly-non-reg-truncation}). We equally deal with convolutors (Section \ref{sec:Existence-of-strongly-non-regular-convolutors}) and operators on arbitrary noncommutative $\L^p$-spaces associated with infinite dimensional approximately finite-dimensional von Neumann algebras (Theorem \ref{prop-hyperfinite-von-Neumann-strongly-non-decomposable}).

In the case of an amenable group $G$, transference methods \cite{BoF,CDS,NR} between Schur multipliers and Fourier multipliers can sometimes be used for proving theorems on selfadjoint completely bounded Fourier multipliers on $\VN(G)$, see e.g. \cite[Corollary 4.5]{Arh1} and \cite{Arh6}. We finish the paper by introducing a general procedure for proving positive results on selfadjoint contractively decomposable Fourier multipliers on \textit{non-amenable} discrete groups relying on the new characterization of Proposition \ref{prop-prop-P-amenable-groups-multipliers}. This result should allow with reasonable effort to generalize properties which are true for unital completely positive selfadjoint Fourier multipliers by using unital completely positive selfadjoint 2x2 block matrices of Fourier multipliers. Section \ref{Application-Matsaev} illustrates this method by describing Fourier multipliers which satisfy the noncommutative Matsaev inequality (Theorem \ref{Application-NC-Matsaev}), using the new result of factorizability of such 2x2 block matrices of Fourier multipliers (Theorem \ref{Th-Factorizable-matricial mutipliers}).  

The paper is organized as follows. Chapter \ref{sec:Preliminaries} gives background and preliminary results. Some relations between matricial orderings and norms in Section \ref{sec-matricial-orderings} are fundamental to reduce the problem of the comparison of the regular norm and the decomposable norm to the adjoint preserving case. Moreover, in passing, we identify completely positive maps on classical $\L^p$-spaces (Proposition \ref{prop-positive-imply-cp} and Proposition \ref{prop-positive-imply-cp-if-depart-commutative}). 

In Chapter \ref{sec:Regular vs decomposable}, we will investigate the notions of decomposable maps and regular maps on noncommutative $\L^p$-spaces.
We will see in Theorem \ref{thm-dec=reg-hyperfinite} that on approximately finite-dimensional semifinite von Neumann algebras, the notions of decomposable and regular operators coincide isometrically. The proof of this result requires several reduction intermediate steps, such as self-adjoint maps in place of general maps (Section \ref{subsec-Reduction-to-the-selfadjoint-case}) and Schatten spaces in place of general noncommutative $\L^p$-spaces (Theorem \ref{prop-dec-reg-matrix-case} in Section \ref{subsec-decomposable-vs-regular-on-Schatten-spaces}). Moreover, we investigate in this section the relation of the (completely) bounded norm on noncommutative $\L^p$-spaces with the decomposable norm. We will see in Theorem \ref{thm-norm=normcb-hyperfinite} that for completely positive maps on $\L^p$-spaces over approximately finite-dimensional algebras, the bounded norm and the completely bounded norm coincide. If the von Neumann algebra has $\QWEP$, then we will see in Proposition \ref{Prop-cb-leq-dec} that the completely bounded norm is dominated by the decomposable norm, so in case of completely positive maps, the complete bounded norm, the bounded norm and the decomposable norm all coincide (Proposition \ref{quest-cp-versus-dec}). However, we will exhibit a class of concrete examples where the decomposable norm is larger than the complete bounded norm
(Theorem \ref{thm-comparaison-cb-dec-free-group}). Finally, this section contains information on the infimum of the decomposable norm (Section \ref{subsec-infimum-decomposable-norm}), the absolute value $|T|$ and decomposability of an operator $T$ acting on a commutative $\L^p$-space (Section \ref{Modulus}) and examples of completely bounded but non decomposable Fourier multipliers on group von Neumann algebras (Proposition \ref{Prop-cb-mais-pas-dec-Fourier-mult}).
We also give explicit examples of computations of the decomposable norm, see Theorem \ref{thm-computation-norm-finite-factor}. 

In the following Chapter \ref{sec:Decomposable-Schur-multipliers-and-Fourier-multipliers}, we give a generalization of the average argument of Haagerup. 
We will show the existence of contractive projections from some spaces of completely bounded operators onto the spaces of Fourier multipliers, Schur multipliers or even a mix of both (Theorem \ref{Prop-complementation-Schur-Fourier} and Section \ref{Complementation-for-discrete-Schur-multipliers-and-Fourier-multipliers}).
This concerns \textit{discrete} groups, possibly deformed by a 2-cocycle and we will also show the independence of the completely bounded norm and the complete positivity with respect to that 2-cocycle, for a Fourier/Schur-multiplier.
So the natural framework will be the one of twisted (discrete) group von Neumann algebras, explained in Section \ref{section-twisted-von-Neumann}. In particular, this covers the case of noncommutative tori when the group equals $\Z^d$. As an application, we will describe the decomposable norm of such Fourier and Schur multipliers on the $\L^p$ level and see that in the framework of this section, this norm equals the (completely) bounded norm on the $\L^\infty$ level (Section \ref{subsec-description-decomposable-norm}).

In Chapter \ref{sec:approximation-discrete-groups}, we introduce and explore some approximation properties of locally compact groups. We connect these to some notions of approximation introduced by different authors. We clarify these properties in the large setting of second countable compactly generated locally compact groups, see Theorem \ref{Th-compactly-tout-equivalent}.

Hereafter, Chapter \ref{sec:Locally-compact-groups} contains an in-depth study of decomposability of Fourier multipliers on non-discrete locally compact groups.
After having introduced these Fourier multipliers and their basic properties in Section \ref{sec:Generalities-Fourier-multipliers}, we will show in Section \ref{sec-The completely bounded homomorphism theorem} how their completely bounded norm is changed under a continuous homomorphism between two locally compact groups. In Section \ref{sec-Extension-of-multipliers}, we describe an extension property of Fourier multipliers which passes from a lattice subgroup to the locally compact full group. In Section \ref{sec-Groups-approximable-by-lattice-subgroups}, we prove Theorem \ref{thm-complementation-Fourier-ADS-amenable} which gives a complementation for second countable unimodular locally compact groups which satisfy the \textit{approximation by lattice subgroups by shrinking} (ALSS) property of Definition \ref{defi-sequentially-ALS} together with a crucial density condition \eqref{equ-Haar-measure-convergence}. Then in Section \ref{subsec-computations-density}, we describe some concrete groups in which Theorem \ref{thm-complementation-Fourier-ADS-amenable} applies. These examples contain direct and semidirect products of groups, groups acting on trees, a large class of locally compact abelian groups and the semi-discrete Heisenberg group. In Section \ref{subsec-pro-discrete-groups}, we show the complementation result for pro-discrete groups by a similar method as in Theorem \ref{thm-complementation-Fourier-ADS-amenable}, but it turns out that there is no need of a density condition in this case.

There is another notion of generalization of Fourier multipliers on non-abelian groups $G$, but acting on classical $\L^p$-spaces $\L^p(G)$ instead of noncommutative $\L^p$-spaces $\L^p(\VN(G))$. These are the convolutors, that is, the bounded operators commuting with left translations. In Section \ref{sec:complementation-convolutor}, we show a complementation result for them on locally compact amenable groups.
Then in Section \ref{subsec-description-decomposable-norm-Fourier} we apply our complementation to describe the decomposable norm of multipliers.

In Chapter \ref{sec:Existence-of-strongly}, we construct completely bounded operators $T \co \L^p(M) \to \L^p(M)$ which cannot be approximated by decomposable operators.
In Proposition \ref{Prop-cb-mais-pas-dec-Fourier-mult}, we shall see that in general, the class of completely bounded operators on a noncommutative $\L^p$-space is larger than the class of decomposable operators. In Chapter \ref{sec:Existence-of-strongly}, we deepen this fact and show that in many situations of $\L^p$-spaces and classes of operators on them, there are (completely) bounded operators such that in a small (norm or $\CB$-norm) neighborhood of the operator, there is no decomposable map.
This notion of ($\CB$-)strongly non decomposable operator is defined in Section \ref{subsec:Existence-of-strongly-Definitions}.
Our first class of objects are the Fourier multipliers on abelian locally compact groups. We show in Theorem \ref{thm-existence-strongly-non-decomposable-abelian-groups} that on all infinite locally compact abelian groups, there always exists a ($\CB$-)strongly non decomposable Fourier multiplier on $\L^p(G)$. By a transference procedure, this theorem extends to convolutors acting on several non-abelian locally compact groups containing infinite locally compact abelian groups (Section \ref{sec:Existence-of-strongly-non-regular-convolutors}). Then our next goal are Schur multipliers. In Section \ref{sec:Existence-of-strongly-non-Schur-multipliers} (see Corollary \ref{Cor-strongly-non-reg-truncation}) we will show that the very classical discrete noncommutative Hilbert transform and the triangular truncation $\mathcal{T} \co S^p \to S^p$ are $\CB$-strongly non decomposable. Then we study $\CB$-strongly non decomposable Fourier multipliers on discrete non-abelian groups.
We establish some general results and apply them to Riesz transforms associated with cocycles and to free Hilbert transforms (Section \ref{sec:Existence-of-strongly-non-Fourier-multipliers}).
Finally, we enlarge the class of spaces and consider $\L^p$-spaces over general approximately finite-dimensional von Neumann algebras (Section \ref{sec:CB-strongly-non-decomposable-operators-on-afd-algebras}). Namely, in Theorem \ref{prop-hyperfinite-von-Neumann-strongly-non-decomposable}, we show that for $1 < p < \infty$, $p \neq 2$ and for any infinite dimensional approximately finite-dimensional von Neumann algebra $M$, there always exists a $\CB$-strongly non decomposable operator on $\L^p(M)$.

In Chapter \ref{sec:Property-P}, we study a certain property for operators on noncommutative $\L^p$-spaces which is a combination of contractively decomposable and selfadjointness on $\L^2(M)$. In general, this notion is more restrictive than being separately contractively decomposable and selfadjoint. However, in Proposition \ref{prop-prop-P-amenable-groups-multipliers}, we will see that for Fourier multipliers acting on twisted von Neumann algebras over discrete groups and a $\T$-valued $2$-cocycle, this difference disappears. As a consequence, we show in the last two Section \ref{sec:factorizability-of-some-matrix-block-multipliers} and Section \ref{Application-Matsaev} that for contractively decomposable and selfadjoint Fourier multipliers on twisted von Neumann algebras, the noncommutative Matsaev inequality holds.

\section{Preliminaries}
\label{sec:Preliminaries}
%

%

\subsection{Noncommutative $\L^p$-spaces and operator spaces}
\label{Section-noncomLp}

Let $M$ be a von Neumann algebra equipped with a semifinite normal faithful weight $\tau$. We denote by $\mathfrak{m}_\tau^+$ the set of all positive $x \in M$ such that $\tau(x)<\infty$ and $\mathfrak{m}_\tau$ its complex linear span which is a weak* dense $*$-subalgebra of $M$. If $\mathfrak{n}_\tau$ is the left ideal of all $x \in M$ such that $\tau(x^*x)<\infty$ then we have
\begin{equation}
	\label{Formule-mtau}
	\mathfrak{m}_\tau
	=\mathrm{span}\big\{y^*z : y,z \in \mathfrak{n}_\tau\big\}.
\end{equation}

Suppose $1 \leq p<\infty$. If $\tau$ is in addition a trace then for any $x \in \mathfrak{m}_\tau$, the operator $|x|^p$ belongs to $\mathfrak{m}_\tau^+$ and we set $\norm{x}_{\L^p(M)} \ov{\mathrm{def}}{=} \tau\big(|x|^p\big)^{\frac{1}{p}}$. The noncommutative $\L^p$-space $\L^p(M)$ is the completion of $\mathfrak{m}_\tau$ with respect to the norm $\norm{\cdot}_{\L^p(M)}$. One sets $\L^{\infty}(M) \ov{\mathrm{def}}{=} M$. We refer to \cite{PiX}, and the references therein, for more information on these spaces. The subspace $M \cap \L^p(M)$ is dense in $\L^p(M)$.  
The positive cone $\L^p(M)_+$ of $\L^p(M)$ is given by 
\begin{equation}
\label{Cone-LpM}
\L^p(M)_+
\ov{\mathrm{def}}{=}\big\{y^*y\ : \ y \in \L^{2p}(M) \big\}.
\end{equation}
We also have the following dual description.

\begin{prop}Let $M$ be a semifinite von Neumann algebra equipped with a normal semifinite faithful trace. Suppose $1 \leq p < \infty$.  We have
\begin{equation}
\label{equa-polar-Lp}
\L^p(M)_+
=\big\{x \in \L^p(M) \ : \ \langle x,y\rangle_{\L^p(M),\L^{p^*}(M)} \geq 0 \text{ for any }y \in \L^{p^*}(M)_+\big\}.
\end{equation}
\end{prop}

\begin{proof}
Let $x \in \L^p(M)$ such that $\langle x,y\rangle_{\L^p(M),\L^{p^*}(M)} \geq 0$ for any $y \in \L^{p^*}(M)_+$. We can write $x=x_1+\i x_2$ where $x_1,x_2$ are selfadjoint elements of $\L^p(M)$. On the one hand, for any $y \in \L^{p^*}(M)_+$, we have
$$
\langle x_1,y\rangle_{\L^p,\L^{p^*}}+\i \langle x_2,y\rangle_{\L^p,\L^{p^*}}
=\langle x_1+\i x_2,y\rangle_{\L^p,\L^{p^*}}
=\langle x,y\rangle_{\L^p,\L^{p^*}} 
\geq 0.
$$
On the other hand $\langle x_1,y\rangle_{\L^p,\L^{p^*}}$ and $\langle x_2,y\rangle_{\L^p,\L^{p^*}}$ are real numbers. We deduce the equality $\langle x_2,y\rangle_{\L^p,\L^{p^*}}=0$ for any $y \in \L^{p^*}(M)_+$. By duality, we infer that $x_2=0$. We conclude that $x$ is selfadjoint. 

Now, consider a decomposition $x=x_1-x_2$ with $x_1,x_2 \in \L^p(M)_+$ such that there exist\footnote{\thefootnote. If $x=w|x|$ is the polar decomposition of a selfadjoint element $x$ then it is known that $w^*=w$ and $w|x|=|x|w$. We can write $w=e-f$ where $e$ and $f$ are two projections such that $ef=0$. We have $e|x|=|x|e$ and $f|x|=|x|f$. We can take $x_1=e|x|$ and $x_2=f|x|$. See \cite[pages 138-139]{Schm} for useful information.} projections $e,f \in M$ such that $ef = 0$, $x_1=ex_1=x_1e$ and $x_2=x_2f=fx_2$. Suppose $x_2 \not=0$. There exists\footnote{\thefootnote. Any positive element of $\L^p(M)$ admits a positive norming functional.}  a positive element $z \in \L^{p^*}(M)$ such that $\langle x_2,z\rangle_{\L^p,\L^{p^*}} >0$. Then
\begin{align*}
\MoveEqLeft
\big\langle x,fzf \big\rangle_{\L^p,\L^{p^*}}
=\big\langle x_1-x_2,fzf \big\rangle_{\L^p,\L^{p^*}} 
=-\langle x_2,z \rangle_{\L^p,\L^{p^*}}
<0.
\end{align*}
That is impossible since $f z f$ is a positive element of $\L^{p^*}(M)$.
\end{proof}


At several times, we will use the following elementary\footnote{\thefootnote. Let $x$ be a positive element of $\L^p(M)$. We can write $x=y^*y$ for some $y \in \L^{2p}(M)$. Since $M \cap \L^{2p}(M)$ is dense in $\L^{2p}(M)$, there exists a sequence $(y_n)_{}$ of elements of $M \cap \L^{2p}(M)$ which approximate $y$ in $\L^{2p}(M)$. Then we have
$$
\norm{x-y_n^{*}y_n}_{\L^p(M)}
=\norm{y^*y-y_n^{*}y_n}_{\L^p(M)} 
\leq \norm{y^*(y-y_n)}_{\L^p(M)}+\norm{(y^*-y_n^{*})y_n}_{\L^p(M)}  
\xra[n \to+\infty]{} 0.
$$} result.

\begin{lemma}
\label{lemma-M+-dense-dans Lp+}
Let $M$ be a semifinite von Neumann algebra equipped with a normal semifinite faithful trace. Suppose $1 \leq p < \infty$. Then $M_+ \cap \L^p(M)$ is dense in $\L^p(M)_+$ for the topology of $\L^p(M)$. 
\end{lemma}

The readers are referred to \cite{ER}, \cite{Pau} and \cite{Pis7} for details on operator spaces and completely bounded maps. If $T \co E \to F$ is a completely bounded map between two operators spaces $E$ and $F$, we denote by $\norm{T}_{\cb, E \to F}$ its completely bounded norm. If $E \otp F$ is the operator space projective tensor product of $E$ and $F$, we have a canonical complete isometry $(E \otp F)^*=\CB(E,F^*)$, see \cite[Chapter 7]{ER}. We will use the notations $E^\op$ and $\ovl{E}$ for the opposite operator space and the complex conjugate of an operator space $E$. 

The theory of vector-valued noncommutative $\L^p$-spaces was initiated by Pisier \cite{Pis5} for the case where the underlying von Neumann algebra is \textit{hyperfinite} and equipped with a normal semifinite faithful trace (see \cite{Jun2} for the case where the von Neumann algebra is $\QWEP$). Under these assumptions, according to \cite[page 37-38]{Pis5}, for any operator space $E$, the spaces $M \ot_{\min} E$ and $\L^1(M^{\op}) \widehat{\ot} E$ can be embedded by an injective continuous map into a common topological vector space, respecting hereby $\left(M \cap \L^1(M^{\op}) \right) \ot E$. This compatibility in the sense of interpolation theory, explained in \cite[page 37]{Pis5} and \cite[page 139]{Pis7} and based on results of Effros and Ruan \cite{ER2,ER3}, relies heavily on the fact that the von Neumann algebra is \textit{hyperfinite} (i.e. approximately finite-dimensional). Suppose $1 \leq p \leq \infty$. Then we can define by complex interpolation
\begin{equation}
\label{Def-vector-valued-Lp-non-com}
\L^p(M,E)
\overset{\mathrm{def}}=\big(M \ot_{\min} E, \L^1(M^{\op}) \widehat{\ot} E\big)_{\frac{1}{p}}
\end{equation}
where $\ot_{\min}$ and $\widehat{\ot}$ denote the injective and the projective tensor product of operator spaces. When $E=\C$, we get the noncommutative $\L^p$-space $\L^p(M)$.

If $\Omega$ is a measure space then we denote by $\B(\L^2(\Omega))$ the von Neumann algebra of bounded operators on the Hilbert space $\L^2(\Omega)$. Using its canonical trace, we obtain the vector-valued Schatten space $S^p_\Omega(E) \ov{\mathrm{def}}{=} \L^p(\B(\L^2(\Omega)),E)$. With $\Omega=\N$ or $\Omega=\{1,\ldots,n\}$ equipped with the counting measure and $E=\C$ we recover the classical Schatten spaces $S^p$ and $S^p_n$.  

Recall the following classical characterization of completely bounded maps, which is essentially \cite[Lemma 1.4]{Pis5}.

\begin{prop}
\label{lemma-charact-Lp-cb}
Let $E$ and $F$ be operator spaces. Suppose $1 \leq p \leq \infty$. A linear map $T \co E \to F$ is completely bounded if and only if $\Id_{S^p} \ot T$ extends to a bounded operator $\Id_{S^p} \ot T \co S^p(E) \to S^p(F)$. In this case, the completely bounded norm $\norm{T}_{\cb,E \to F}$ is given by
\begin{equation}
\label{defnormecb}
\norm{T}_{\cb,E \to F}
=\norm{\Id_{S^p} \ot T}_{S^p(E) \to S^p(F)}.
\end{equation}
\end{prop}

We will  use the following result \cite[page 984]{Jun}, \cite{Jun2} (see \cite[Appendix]{Arh7} for a proof for approximately finite-dimensional von Neumann algebras).

\begin{thm}
\label{Th-tensor-Lp-non-com}
Let $M_1, M_2, N_1,N_2$ be $\QWEP$ von Neumann algebras. Suppose $1 \leq p \leq \infty$. Let $T_1 \co \L^p(M_1) \to \L^p(N_1)$ et $T_2 \co \L^p(M_2) \to \L^p(N_2)$ be completely bounded maps. Then the map $T_1 \ot T_2 \co \L^p(M_1 \ot N_2) \to \L^p(N_1 \ot N_2)$ is completely bounded and we have
\begin{align}
\label{multiplicativity-T}
\norm{T_1 \ot T_2}_{\cb,\L^p(\L^p) \to \L^p(\L^p)} 
\leq \norm{T_1}_{\cb,\L^p \to \L^p} \norm{T_2}_{\cb,\L^p \to \L^p}.
\end{align}
\end{thm}

A measure space $(\Omega,\mu)$ (also denoted $\Omega)$ is called localizable if its measure algebra\footnote{\thefootnote. The measure algebra \cite[Definition 321I]{Fre3} of a measure space is defined as the quotient of the ring of measurable sets by the ideal of null sets, with the measure of any residue class defined to be the measure of any representative of the class.} is semifinite and Dedekind complete, see \cite[Lemma 2.6]{OkR1}, \cite[Theorem 322B]{Fre3} and \cite[Corollary 3.2.1]{Seg}. By \cite[Theorem 243G]{Fre2}, this is equivalent to the bijectivity of the canonical map $\L^\infty(\Omega) \to \L^1(\Omega)^*$ (in which case it is an isometry). Recall that a $\sigma$-finite measure space \cite[Theorem 211L]{Fre2}, \cite[Corollary 3.2.1]{Seg} and a locally compact group equipped with a left Haar measure \cite[Corollary 5.2]{Seg}, \cite[443A (a)]{Fre4} are localizable. We warn that there are several notions of localizable measure spaces, see \cite{OkR1} and the recent paper \cite{BGL1} for more information. 

The importance of these measure spaces comes from \cite[Theorem 5.1]{Seg} which says that for a measure space $\Omega$, the algebra $\L^\infty(\Omega)$ is a von Neumann algebra if and only if $\Omega$ is a localizable measure space. Note that in this case, the integral defines a semifinite (normal, faithful) trace on the von Neumann algebra $\L^\infty(\Omega)$, and thus, $\L^p(\Omega)$ carries, as any other noncommutative $\L^p$ space, an operator space structure.
Thus, $S^p(\L^p(\Omega))$ is well-defined. Then, if $\Omega$ is a (localizable) measure space, the Banach space $S^p(\L^p(\Omega))$ is isometric to the Bochner space $\L^p(\Omega,S^p)$ of $S^p$-valued functions. Thus, in particular, if $\Omega'$ is another (localizable) measure space then a linear map $T \co \L^p(\Omega) \to \L^p(\Omega')$ is completely bounded if and only if $T \ot \Id_{S^p}$ extends to a bounded operator $T \ot \Id_{S^p} \co \L^p(\Omega,S^p) \to \L^p(\Omega',S^p)$. In this case, we have
\begin{equation}
\label{normecbcommutatif}
\norm{T}_{\cb,\L^p(\Omega) \to \L^p(\Omega')}
=\bnorm{T \ot \Id_{S^p}}_{\L^p(\Omega,S^p) \to \L^p(\Omega',S^p)}.
\end{equation}

If $E$ and $F$ are operator spaces and if $T \co E \to F$ is a linear map, we will use the map $T^\op \co E^\op \to F^\op$, $x \mapsto T(x)$. Of course, since the underlying Banach spaces of $E$ and $E^\op$ and of $F$ and $F^\op$ are identical, the map $T$ is bounded if and only if the map $T^\op$ is bounded. The following lemma shows that the situation is similar for the complete boundedness. Furthermore, this result is useful when we use duality since in the category of operator spaces we have $\L^p(M)^*=\L^{p^*}(M)^\op$ if $1 \leq p < \infty$. In passing, recall that $\L^p(M)^\op=\L^p(M^\op)$.

\begin{lemma}
\label{Lemma-op-mapping-1}
Let $T \co E \to F$ be a linear map between operators spaces. Then $T$ is completely bounded if and only if the map $T^\op \co E^\op \to F^\op$ is completely bounded. Moreover, in this case we have $\norm{T}_{\cb,E \to F}=\norm{T^\op}_{\cb,E^\op \to F^\op}$.
\end{lemma}

\begin{proof}
Assume that $T$ is completely bounded and let $[x_{ij}] \in \M_n(E^\op)$. Then
\begin{align*}
\bnorm{[T(x_{ij})]}_{\M_n(F^\op)} 
=\bnorm{[T(x_{ji})]}_{\M_n(F)} 
\leq \norm{T}_{\cb,E \to F} \bnorm{[x_{ji}]}_{\M_n(E)} 
=\norm{T}_{\cb,E \to F} \bnorm{[x_{ij}]}_{\M_n(E^\op)}.
\end{align*}
We infer that $\norm{T^\op}_{\cb,E^\op \to F^\op} \leq \norm{T}_{\cb,E \to F}$. Since $(E^\op)^\op=E$ completely isometrically, the reverse inequality follows by symmetry.
\end{proof}

\subsection{Matrix ordered operator spaces}


A complex vector space $V$ is matrix ordered \cite[page~173]{ChE} if
\begin{enumerate}
	\item $V$ is a $*$-vector space (hence so is $\M_n(V)$ for any $n \geq 1$),
	\item each $\M_n(V)$, $n \geq 1$, is partially ordered by a cone $\M_n(V)_+ \subset \M_n(V)_{\textrm{sa}}$, and
	\item if $\alpha=[\alpha_{ij}] \in \M_{n,m}$, then $\alpha^*\M_n(V)_+\alpha \subset \M_m(V)_+$.
\end{enumerate}
Now let $V$ and $W$ be matrix ordered vector spaces and let $T \co V \to W$ be
a linear map. If $n \geq 1$, we say that $T$ is $n$-positive if $\Id_{\M_n} \ot T\co \M_n(V) \to \M_n(W)$ is positive. We say that $T$ is completely positive if $T$ is $n$-positive for each $n \geq $. We denote the set of completely positive maps from $V$ to $W$ by $\mathrm{CP}(V,W)$.

An operator space $E$ is called a matrix ordered operator space \cite[page~143]{Schr1} if it is a matrix ordered vector space and if in addition 
\begin{enumerate}
	\item the $*$-operation is an isometry on $\M_n(E)$ for any integer $n \geq 1$ and 
	\item the cones $\M_n(E)_+$ are closed in the norm topology.
\end{enumerate}
For a matrix ordered operator space $E$ and its dual operator space $E^*$, we can define an involution on $E^*$ by $\varphi^*(v)=\ovl{\varphi(v^*)}$  for any $\varphi \in E^*$ and a cone on $\M_n(E^*)$ for each $n \geq 1$ by $\M_n(E^*)^+ = \CB(E,\M_n) \cap \mathrm{CP}(E,\M_n)$. Note that we have an isometric identification $\M_n(E^*)=\CB(E,\M_n)$. A lemma of Itoh \cite{Ito} (see \cite[Lemma 2.3.8]{Schr2} for a complete proof) says that if $E$ is a matrix ordered operator space, we have
\begin{equation}
\label{Dual-cone}
\M_n(E^*)_+
=\bigg\{[y_{ij}] \in \M_n(E^*) : \sum_{i,j=1}^{n} y_{ij}(x_{ij}) \geq 0 \text{ for any }[x_{ij}] \in \M_n(E)_+ \bigg\}.
\end{equation}

\begin{lemma}
\label{Lemma-Bipolar}
Let $E$ be a matrix ordered operator space. We have
$$
\M_n(E)_+
=\bigg\{ x \in \M_n(E) : \sum_{i,j=1}^{n} y_{ij}(x_{ij}) \geq 0 \text{ for any }[y_{ij}] \in \M_n(E^*)_+ \bigg\}.
$$
\end{lemma}

\begin{proof}
Note that the dual cone $S^1_n(E^*)_+$ of $\M_n(E)_+$ is defined by $
S^1_n(E^*)_+=\big\{[y_{ij}] \in S^1_n(E^*) : \sum_{i,j=1}^{n} y_{ij}(x_{ij}) \geq 0 \text{ for any }[x_{ij}] \in \M_n(E)_+ \big\}$ and identifies to $\M_n(E^*)_+$ by \eqref{Dual-cone}. Since $\M_n(E)_+$ is closed in the norm topology, hence weakly closed, we conclude by the bipolar theorem.
\end{proof}

By \cite[Corollary 3.2]{Schr1}, the operator space dual $E^*$ with this positive cone is a matrix ordered operator space. The category of matrix ordered operator spaces contains the class of C$^*$-algebras.

Let $M$ be a von Neumann algebra equipped with a faithful normal semifinite trace. If $1 \leq p \leq \infty$, the noncommutative $\L^p$-space $\L^p(M)$ is canonically equipped with an isometric involution and we can define a cone on $\M_n(\L^p(M))$ by letting 
\begin{equation}
\label{Def-matricial-cones-Lp}
\M_n(\L^p(M))_+
\ov{\mathrm{def}}{=} \L^p(\M_n(M))_+\ (=S^p_n(\L^p(M))_+).
\end{equation}
Note the following easy\footnote{\thefootnote. Consider $x \in \M_n(\L^p(M))_+$, i.e. $x \in S^p_n(\L^p(M))_+$. There exists $y \in S^{2p}_n(\L^{2p}(M))$ such that $y^*y=x$. We can write $y=\sum_{i,j=1}^{n} e_{ij} \ot y_{ij}$ for some $y_{ij} \in \L^{2p}(M)$. For any matrix $\alpha \in \M_{n,m}$, we have
\begin{align*}
\MoveEqLeft
  \alpha^* \cdot x \cdot \alpha
	=\alpha^* \cdot y^*y \cdot \alpha
	=\alpha^* \cdot\bigg(\sum_{i,j=1}^{n} e_{ij}\ot y_{ij}\bigg)^*\bigg(\sum_{k,l=1}^{n} e_{kl}\ot y_{kl}\bigg) \cdot \alpha \\
		&=\alpha^* \cdot\bigg(\sum_{i,j=1}^{n} e_{ji}\ot y_{ij}^*\bigg)\bigg(\sum_{k,l=1}^{n} e_{kl}\ot y_{kl}\bigg) \cdot \alpha 
		= \sum_{i,j,k,l=1}^{n} \alpha^*e_{ji}e_{kl}\alpha \ot y_{ij}^*y_{kl} \\
		&=\bigg(\sum_{i,j=1}^{n} \alpha^*e_{ji} \ot y_{ij}^*\bigg)\bigg(\sum_{k,l=1}^{n}  e_{kl}\alpha \ot y_{kl}\bigg)
		=\bigg(\sum_{i,j=1}^{n}  e_{ij}\alpha \ot y_{ij}\bigg)^*\bigg(\sum_{k,l=1}^{n}  e_{kl}\alpha \ot y_{kl}\bigg).
\end{align*}
We conclude that $\alpha^* \cdot x \cdot \alpha$ is a positive element of $\M_n(\L^p(M))=S^{p}_n(\L^{p}(M))$. We conclude that $\L^p(M)$ is matrix ordered. Moreover, for any $x \in \M_n(\L^p(M))$, using \cite[Lemma 1.7]{Pis4} twice and the isometric involution, we see that 
\begin{align*}
\MoveEqLeft
  \norm{x^*}_{\M_n(\L^p(M))}  
		=\sup\big\{\norm{\alpha \cdot x^* \cdot \beta}_{S^p_n(\L^p(M))} \ :\ \norm{\alpha}_{S^{2p}_{n}} \leq 1, \ \norm{\beta}_{S^{2p}_{n}} \leq 1  \big\} \\
		&=\sup\big\{\norm{\beta^* \cdot x \cdot \alpha^*}_{S^p_n(\L^p(M))} \ :\ \norm{\alpha}_{S^{2p}_{n}} \leq 1, \ \norm{\beta}_{S^{2p}_{n}} \leq 1  \big\} \\
		&=\sup\big\{\norm{\beta \cdot x \cdot \alpha}_{S^p_n(\L^p(M))} \ :\ \norm{\alpha}_{S^{2p}_{n}} \leq 1, \ \norm{\beta}_{S^{2p}_{n}} \leq 1  \big\} 
		=\norm{x}_{\M_n(\L^p(M))}.
\end{align*}.} observation.

\begin{prop}
\label{prop-LpM-is-matrix-ordered}
Let $M$ be a von Neumann algebra equipped with a faithful normal semifinite trace. Suppose $1\leq p \leq\infty$. Then the noncommutative $\L^p$-space $\L^p(M)$ is a matrix ordered operator space.
\end{prop}

If $N$ is another von Neumann algebra equipped with a faithful normal semifinite trace then it is easy to see that a map $T \co \L^p(M) \to \L^p(N)$ is completely positive if the map $\Id_{S^p} \ot T$ induces a (completely) positive map $\Id_{S^p} \ot T \co S^p(\L^p(M)) \to S^p(\L^p(N))$. Moreover, for any matrix $\alpha \in \M_{n,m}$, the map  
\begin{equation}
	\label{conj-cp}
	\L^p(\M_n(M)) \to \L^p(\M_m(M)), \quad x \mapsto \alpha^* x \alpha
\end{equation}
is completely positive. 

\begin{lemma}
\label{equ-1-proof-op-mappings}
Let $M$ be a von Neumann algebra equipped with a faithful normal semifinite trace. Suppose $1 \leq p \leq \infty$. If $b \in \M_n(\L^p(M))$ and if $b^t$ is the transpose of $b$ we have $b \in \M_n(\L^p(M^\op))_+$ if and only if $b^t \in \M_n(\L^p(M))_+$.
\end{lemma} 

\begin{proof}
We start with the case $p=\infty$. We can identify $M^\op$ with $M$ equipped with the opposed product. We will use the notation $\circ$ for some products where the subscript indicates the space. Let $b \in \M_n(M^\op)_+$. Then we can write $b = c^*\circ_{\M_n(M^\op)} c$ for some $c \in \M_n(M)$. For any $1 \leq i,j \leq n$, we have
\begin{align*}
\MoveEqLeft
  b_{ij}
	=\sum_{k=1}^{n} (c^*)_{ik} \circ_{M^\op}c_{kj} 
	=\sum_{k=1}^{n} c_{kj} (c^*)_{ik} 
	=\sum_{k=1}^{n} (c^t)_{jk}(c^{t*})_{ki}
	=\big(c^t \circ_{\M_n(M)} c^{t*}\big)^t.
\end{align*}  
Hence $b^t=c^t \circ_{\M_n(M)} c^{t*}$ belongs to $\M_n(M)_+$. The reverse implication follows by symmetry. Suppose that $b \in \M_n(\L^p(M^\op))_+$, i.e. $b \in S^p_n(\L^p(M^\op))_+$ by \eqref{Def-matricial-cones-Lp}. By Lemma \ref{lemma-M+-dense-dans Lp+}, there exists a sequence $(b_k)$ in $\M_n(M^\op)_+ \cap S^p_n(\L^p(M^\op))$ converging to $b$ for the topology of $S^p_n(\L^p(M))$. By the first part of the proof, each $(b_k)^t$ belongs to $\M_n(M)_+$ and of course to $S^p_n(\L^p(M))$. In particular, $(b_k)^t$ belongs to $\M_n(\L^p(M))_+$. Passing to the limit as $k$ approaches infinity yields $b^t \in \M_n(\L^p(M))_+$. Again, a symmetry argument completes the proof.
\end{proof}

We will often use the following observation.


\begin{lemma}
\label{Lemma-adjoint-cp}
Let $E$ and $F$ be matrix ordered operator spaces. A bounded map $T \co E \to F$ is (completely) positive if and only if the adjoint map $T^* \co F^* \to E^*$ is (completely) positive.
\end{lemma}

\begin{proof}
By Lemma \ref{Lemma-Bipolar}, a map $T \co E \to F$ is positive if and only if $\langle T(x),y \rangle_{F,F^*} \geq 0$ for any $x \in E_+$ and any $y \in F^*_+$ if and only if $\langle x, T^*(y) \rangle_{E,E^*} \geq 0$ for all such $x,y$ if and only if $T^* \co F^* \to E^*$ is positive again by \eqref{Dual-cone}. The completely positive case is similar.
\end{proof}

For further use in Lemma \ref{lem-decomposable-weak-limit}, we record the following.

\begin{lemma}
\label{lem-completely-positive-weak-limit}
Let $E$ and $F$ be matrix ordered operator spaces.
\begin{enumerate}
	\item Let $(T_\alpha)$ be a net of positive (resp. $n$-positive or completely positive) mappings from $E$ into $F$. Suppose that $\lim_\alpha T_\alpha = T$ in the weak operator topology. Then $T$ is also positive (resp. $n$-positive or completely positive).
	
	\item Let $(T_\alpha)$ be a net of positive (resp. $n$-positive or completely positive) mappings from $E$ into $F^*$. Suppose that $\lim_\alpha T_\alpha = T$ in the point weak* topology\footnote{\thefootnote. If $X$ is a Banach space and $Y$ is a dual Banach space, a net $(T_\alpha)$ in $\B(X,Y)$ converges to an operator $T \in \B(X,Y)$ in the point weak* topology if and only if for any $x \in X$ and any $y_* \in Y_*$ we have $\langle T_\alpha(x),y_* \rangle_{Y,Y_*} \xra[\alpha]{} \langle T(x),y_* \rangle_{Y,Y_*}$.} of $\B(E,F^*)$. Then $T$ is also positive (resp. $n$-positive or completely positive).
\end{enumerate}
\end{lemma}

\begin{proof}
1. Suppose that each $T_\alpha \co E \to F$ is a positive map. By Lemma \ref{Lemma-Bipolar}, the map $T \co E \to F$ is positive if and only if $\langle T(x),y \rangle_{F,F^*} \geq 0$ for any $x \in E_+$ and any $y \in F^*_+$. Using again Lemma \ref{Lemma-Bipolar}, we infer that $\langle T(x),y \rangle_{F,F^*}
= \lim_\alpha \langle T_\alpha(x),y \rangle_{F,F^*} 
\geq 0$. Thus we conclude that $T$ is positive.

Suppose that each $T_\alpha$ is completely positive. By Lemma \ref{Lemma-Bipolar}, the map $\Id_{M_n} \ot T \co \M_n(E) \to \M_n(F)$ is positive if and only if $\sum_{i,j=1}^{n} \langle T(x_{ij}) , y_{ij}\rangle_{F,F^*}$ for any $[x_{ij}] \in \M_n(E)_+$ and any $[y_{ij}] \in \M_n(F^*)_+$. Using again Lemma \ref{Lemma-Bipolar}, we infer that 
\[ \sum_{i,j=1}^{n} \langle T(x_{ij}), y_{ij} \rangle_{F,F^*}
=\lim_\alpha  \sum_{i,j=1}^{n} \langle T_\alpha(x_{ij}), y_{ij} \rangle_{F,F^*} \geq 0 .\]
Letting $n$ run over all integers, we conclude that $T$ is completely positive. The argument is the same for the $n$-positive case.

2. Suppose that each $T_\alpha \co E \to F^*$ is a positive map. By \eqref{Dual-cone}, the map $T \co E \to F^*$ is positive if and only if $\langle T(x),y \rangle_{F^*,F} \geq 0$ for any $x \in E_+$ and any $y \in F_+$. Using again \eqref{Dual-cone}, we infer that $\langle T(x),y \rangle_{F^*,F}
= \lim_\alpha \langle T_\alpha(x),y \rangle_{F^*,F} \geq 0$. Thus we conclude that $T$ is positive.

Suppose that each $T_\alpha$ is completely positive. By \eqref{Dual-cone}, the map $\Id_{\M_n} \ot T \co \M_n(E) \to \M_n(F^*)$ is positive if and only if $\sum_{i,j=1}^{n}\langle T(x_{ij}), y_{ij} \rangle_{F^*,F}$ for any $[x_{ij}] \in \M_n(E)_+$ and any $[y_{ij}] \in \M_n(F)_+$. Using again \eqref{Dual-cone}, we infer that $\langle T(x_{ij}), y_{ij} \rangle_{F^*,F}
=\lim_\alpha  \sum_{i,j=1}^{n} \langle T_\alpha(x_{ij}), y_{ij} \rangle_{F^*,F} \geq 0$. Letting $n$ run over all integers, we conclude that $T$ is completely positive. The argument is the same for the $n$-positive case.
\end{proof}

If $E$ is a matrix ordered operator space, by \cite[page~80]{Schr2}, the vector-valued Schatten space $S^p_n(E)=\mathrm{R}_n(1-\frac{1}{p}) \ot_h E \ot_h \mathrm{R}_n(\frac{1}{p})$ admits a structure of a matrix ordered operator space. The cones are defined by the closures
\begin{align*}
\MoveEqLeft
  \M_k(S^p_n(E))_+
	=\ovl{\big\{x^* \odot y \odot x \in \M_k(S^p_n(E)) \,:\, x \in \M_{l,k}(\mathrm{R}_n(\tfrac{1}{p})),\, y \in \M_l(E)_+,\, l \in \N\big\}}.  
\end{align*}

\begin{lemma}
\label{Lemma-cp-S1E}
Suppose $1 \leq p \leq \infty$. Let $E$ and $F$ be matrix ordered operator spaces and let $T \co E \to F$ be a bounded completely positive map. Then for any integer $n$, the map $\Id_{S^p_n} \ot T \co S^p_n(E) \to S^p_n(F)$ is completely positive. 
\end{lemma}

\begin{proof}
For any $n \in \N$, any $x \in \M_{l,k}(\mathrm{R}_n(\tfrac{1}{p}))$ and any $y \in \M_l(E)_+$, the element $(\Id_{S^p_n} \ot T)(x^* \odot y \odot x)=x^* \odot T(y) \odot x$ belongs to $\M_k(S^p_n(E))_+$. An argument of continuity gives the result.
\end{proof}


\subsection{Relations between matricial orderings and norms}
\label{sec-matricial-orderings}

For any $x \in S_n^p(E)$ and any $a,b \in \M_n$, the result \cite[Lemma 1.6 (i)]{Pis5} says that
\begin{equation}
	\label{Inequality-SpnE}
	\norm{axb}_{S_n^p(E)} 
	\leq \norm{a}_{S^\infty_n}\norm{x}_{S^p_n(E)}\norm{b}_{S^\infty_n}.
\end{equation}
Moreover, for any diagonal matrix $x=\mathrm{diag}(x_1,\ldots,x_n) \in S^p_n(E)$, \cite[Corollary 1.3]{Pis5} gives
\begin{equation}
	\label{Block-diagonal}
	\norm{x}_{S_n^p(E)}
	=\bigg(\sum_{k=1}^{n} \norm{x_k}_{E}^p\bigg)^{\frac{1}{p}}.
\end{equation}
\begin{lemma}
\label{lem-regulier-et-selfadjoint1}
Let $E$ be an operator space. Suppose $1 \leq p < \infty$. Then for any $b,c \in E$, we have
$
\left\| 
\begin{bmatrix} 
0 & b \\ 
c & 0 
\end{bmatrix} 
\right\|_{S^p_2(E)} 
=\big(\norm{b}_{E}^p + \norm{c}_{E}^p \big)^{\frac1p}$ and $\left\| 
\begin{bmatrix} 
0 & b \\ 
c & 0 
\end{bmatrix} 
\right\|_{S^\infty_2(E)} 
=\max\big\{\norm{b}_{E},\norm{c}_{E} \big\}$.
\end{lemma}

\begin{proof}
Using the inequality \eqref{Inequality-SpnE}, we see that
\begin{align*}
\MoveEqLeft
 \left\| 
\begin{bmatrix} 
0 & b \\ 
c & 0 
\end{bmatrix} 
\right\|_{S^p_2(E)}    
		=\left\| 
\begin{bmatrix} 
b & 0 \\ 
0 & c 
\end{bmatrix} 
\begin{bmatrix} 
0 & 1 \\ 
1 & 0 
\end{bmatrix}\right\|_{S^p_2(E)}
\leq \left\| 
\begin{bmatrix} 
b & 0 \\ 
0 & c 
\end{bmatrix} 
\right\|_{S^p_2(E)} \left\| 
\begin{bmatrix} 
0 & 1 \\ 
1 & 0 
\end{bmatrix} \right\|_{S^\infty_2} 
=\left\| 
\begin{bmatrix} 
b & 0 \\ 
0 & c 
\end{bmatrix} 
\right\|_{S^p_2(E)}.
\end{align*}
By symmetry, we conclude that $\left\| 
\begin{bmatrix} 
0 & b \\ 
c & 0 
\end{bmatrix} 
\right\|_{S^p_2(E)}   =\left\| 
\begin{bmatrix} 
b & 0 \\ 
0 & c 
\end{bmatrix} 
\right\|_{S^p_2(E)}$. On the other hand, the equality \eqref{Block-diagonal} yields $\left\| 
\begin{bmatrix} 
b & 0 \\ 
0 & c 
\end{bmatrix} 
\right\|_{S^p_2(E)}= \big(\norm{b}_E^p + \norm{c}_E^p\big)^{\frac1p}$. The case $p=\infty$ is similar, so the lemma is proven. 
\end{proof}

\begin{lemma}
\label{Lemma-Matricial-inequality}
Let $M$ be a von Neumann algebra equipped with a faithful normal semifinite trace. Suppose $1 \leq p \leq \infty$. Let $a,b$ and $c$ be elements of $\L^p(M)$ such that the element
$
\begin{bmatrix}
a & b\\
b^* & c
\end{bmatrix}
$
of $S^p_2(\L^p(M))$ is positive. Then we have $\norm{b}_{\L^p(M)} \leq \sqrt{\norm{a}_{\L^p(M)} \norm{c}_{\L^p(M)}}$. So in particular $
\norm{b}_{\L^p(M)} \leq \frac{1}{2^{\frac{1}{p}}}\big(\norm{a}_{\L^p(M)}^p + \norm{c}_{\L^p(M)}^p \big)^{\frac1p}.
$
%
\end{lemma}

\begin{proof}
By Lemma \ref{lemma-M+-dense-dans Lp+}, there exists a sequence $\bigg(\begin{bmatrix}
a_n & b_n\\

b_n^* & c_n
\end{bmatrix}\bigg)$ of elements in $\M_2(M)_+ \cap \L^p(\M_2(M))$ converging to the positive element $\begin{bmatrix}
a & b\\
b^* & c
\end{bmatrix}$ for the topology of $\L^p(\M_2(M))$. By adapting a classical argument \cite[Proposition 1.3.2]{Bha}, \cite[Lemma 1.21]{Zha}, for each integer $n$ there exists $x_n \in M$ with $\norm{x_n}_M \leq 1$ such that $b_n = a_n^{\frac12} x_n c_n^{\frac12}$. Thus $\norm{b_n}_{p}=\bnorm{a_n^{\frac12} x_n c_n^{\frac12}}_{p} \leq \bnorm{a_n^{\frac12}}_{2p} \bnorm{c_n^{\frac12}}_{2p} 
= \sqrt{\norm{a_n}_p \norm{c_n}_p}$. Passing to the limit as $n$ approaches infinity, we obtain the inequality. 

The last sentence of the statement follows from the inequality $\sqrt{xy} \leq 2^{-\frac{1}{p}} (x^p + y^p)^{\frac1p}$ for any reals $x,y \geq 0$.
\end{proof}

The following result is folklore. Unable to locate a proof in the literature, we give a very short proof based on Lemma \ref{Lemma-Matricial-inequality}. 

\begin{prop}
\label{Lemma-Matricial-inequality3}
Let $M$ be a von Neumann algebra equipped with a faithful normal semifinite trace. Suppose $1 \leq p \leq \infty$. Let $b$ be an element of $S^p_n(\L^p(M))$. Then $\norm{b}_{S^p_n(\L^p(M))} \leq 1$ if and only if there are $a, c \in S^p_n(\L^p(M))_+$ with $\norm{a}_{S^p_n(\L^p(M))} \leq 1$ and $\norm{c}_{S^p_n(\L^p(M))}\leq 1$ such that the element
$
\begin{bmatrix}
a   & b\\
b^* & c
\end{bmatrix}
$
of $S^p_{2n}(\L^p(M))$ is positive.
\end{prop}

\begin{proof}
The implication $\Leftarrow$ is Lemma \ref{Lemma-Matricial-inequality}. For the implication $\Rightarrow$, we only need the case $n=1$. Consider $b \in \L^p(M)$ with $\norm{b}_{\L^p(M)} \leq 1$. By Lemma \ref{lemma-M+-dense-dans Lp+}, there exists a sequence $(b_n)$ in $M_+ \cap \L^p(M)$ converging to $b$ for the topology of $\L^p(M)$. By \cite[Exercise 8.8 (vi)]{Pau}, the matrix $\begin{bmatrix} 
   |b_n^*|  &  b_n \\
   b_n^*  &  |b_n| \\
\end{bmatrix}$ is a positive element. Using the continuity of the modulus and passing to the limit as $n$ approaches infinity yields $\begin{bmatrix} 
   |b^*|  &  b \\
   b^*  &  |b| \\
\end{bmatrix} \geq 0$. Moreover, we have $\bnorm{|b|}_{\L^p(M)}=\bnorm{|b^*|}_{\L^p(M)}=\norm{b}_{\L^p(M)} \leq 1$.
\end{proof}

\begin{lemma}
\label{Lemma-Matricial-2}
Suppose $1 \leq p \leq \infty$. Let $M$ be a von Neumann algebra equipped with a faithful normal semifinite trace. Let $a$ and $b$ be selfadjoint elements of $\L^p(M)$ satisfying $-a \leq b \leq a$. Then, in $S^p_2(\L^p(M))$, we have 
$
\begin{bmatrix} 
   a  & b  \\
   b  & a  \\
 \end{bmatrix} \geq 0
$.
\end{lemma}

\begin{proof}
The case $p=\infty$ is well-known, see \cite[Proposition 1.3.5]{ER}. 
Let us turn to the case $1 \leq p < \infty$.
By Lemma \ref{lemma-M+-dense-dans Lp+}, there exists a sequence $(y_n)$ in $M_+ \cap \L^p(M)$ converging to the positive element $b-a$ for the topology of $\L^p(M)$ and a sequence $(z_n)$ of elements of $M_+ \cap \L^p(M)$ converging to the positive element $a+b$. Note that $a_n \ov{\mathrm{def}}{=} \frac{z_n-y_n}{2}$ converges to $a$ and that $b_n \ov{\mathrm{def}}{=} \frac{y_n+z_n}{2}$ converges to $b$. Moreover, we have $-b_n \leq a_n \leq b_n$. According to the case $p=\infty$, we have 
$\begin{bmatrix} 
a_n & b_n \\ 
b_n & a_n 
\end{bmatrix} \geq 0$. Finally passing to the limit as $n$ approaches infinity yields $\begin{bmatrix} 
a & b \\ 
b & a
\end{bmatrix} \geq 0$.
\end{proof}

\begin{lemma}
\label{Lemma-Matricial-3}
Let $M$ be a von Neumann algebra equipped with a faithful normal semifinite trace. Suppose $1 \leq p \leq \infty$.  Let $a$, $b$ and $c$ be elements of $\L^p(M)$ satisfying $
\begin{bmatrix} 
   a  & b  \\
   b  & c  \\
\end{bmatrix}  \geq 0$ in $S^p_2(\L^p(M))$. Then we have $
-\frac{1}{2}(a+c) \leq b \leq \frac{1}{2}(a+c)$.
\end{lemma}

\begin{proof}
Let $A = \begin{bmatrix} 
   a  & b  \\
   b  & c  \\
\end{bmatrix}$.
Since $A \geq 0$, according to \eqref{conj-cp}, we have $uAu^* \geq 0$ for $u = \begin{bmatrix} 1 & 1 \end{bmatrix}$ and for $u = \begin{bmatrix} 1 & - 1 \end{bmatrix}$.
The first choice of $u$ then yields $a+2b+c \geq 0$, so that $b \geq - \frac12 (a+c)$.
The second choice of $u$ yields $a-2b+c \geq 0$, so that $b \leq \frac12 (a+c)$.
\end{proof}

\subsection{Positive and completely positive maps on noncommutative $\L^p$-spaces}
\label{sec:Positive maps noncommutative-Lp-spaces}

\begin{lemma}
\label{lem-op-mappings}
Let $M$ and $N$ be von Neumann algebras equipped with semifinite faithful normal traces. Suppose $1 \leq p \leq \infty$. Then a map $T \co \L^p(M) \to \L^p(N)$ is  completely positive if and only if $T^\op \co \L^p(M)^\op \to \L^p(N)^\op$ is completely positive.
\end{lemma}

\begin{proof}
Assume that $T \co \L^p(M) \to \L^p(M)$ is completely positive. Let $b \in (\M_n(\L^p(M)^\op))_+$. Then applying Lemma \ref{equ-1-proof-op-mappings} twice, we deduce that $(\Id_{\M_n} \ot T^\op)(b) =[T(b_{ij})]=[T((b^t)_{ij})]^t= ((\Id_{\M_n} \ot T)(b^t))^t$ belongs to $(\M_n(\L^p(M)^\op))_+$. We infer that $T^\op \co \L^p(M)^\op \to \L^p(M)^\op$ is completely positive. The reverse statement is obtained by symmetry.
\end{proof}

The boundedness assumption of \cite[Theorem 0.1 and Lemma 2.3]{Pis4} is unnecessary since we have the following elementary result.

\begin{prop} 
\label{prop-cp-imply-bounded}
Let $M$ be a von Neumann algebra equipped with a faithful normal semifinite trace. Suppose $1 \leq p \leq \infty$. Any positive map $T\co \L^p(M) \to \L^p(M)$ is bounded.
\end{prop}

\begin{proof} 
We first show that there exists a constant $K \geq 0$ satisfying for any $x \in \L^p(M)_+$ with $\norm{x}_{\L^p(M)} \leq 1$ the inequality $\norm{T(x)}_{\L^p(M)} \leq K$. Suppose that it is not the case then there exists a sequence $(x_n)$ of positive elements of $\L^p(M)$ with $\norm{x_n}_{\L^p(M)} \leq 1$ and $\norm{T(x_n)}_{\L^p(M)} \geq 4^n$. We have $\sum_{n=1}^{\infty} \bnorm{\frac{1}{2^n} x_n}_{\L^p(M)} \leq \sum_{n=1}^{\infty} \frac{1}{2^n}< \infty$. Hence the series $\sum_{n=1}^{\infty} \frac{1}{2^n}x_n$ is convergent and defines a positive element $x$ of $\L^p(M)$. Now, for any integer $n \geq 1$, we have $0 \leq \frac{1}{2^n}x_n \leq x$. We deduce that $0 \leq \frac{1}{2^n}T(x_n) \leq T(x)$. Hence we obtain $\frac{1}{2^n}\bnorm{T(x_n)}_{\L^p(M)} \leq \bnorm{T(x)}_{\L^p(M)}$ and finally $2^n \leq \bnorm{T(x)}_{\L^p(M)}$. Impossible. 

Now, if $x \in \L^p(M)$ we have a decomposition $x=x_1-x_2+\i(x_3-x_4)$ with $x_1,x_2,x_3,x_4 \in \L^p(M)_+$ and $\norm{x_1}_{\L^p(M)},\norm{x_2}_{\L^p(M)},\norm{x_3}_{\L^p(M)},\norm{x_4}_{\L^p(M)}$ less or equal to $\norm{x}_{\L^p(M)}$. Hence
\begin{align*}
\MoveEqLeft
  \norm{T(x)}_{\L^p(M)}
	=\norm{T(x_1)-T(x_2)+\i \big(T(x_3)-T(x_4)\big)}_{\L^p(M)}\\
	&\leq \norm{T(x_1)}_{\L^p(M)}+\norm{T(x_2)}_{\L^p(M)}+\norm{T(x_3)}_{\L^p(M)}+\norm{T(x_4)}_{\L^p(M)}\\
	&\leq K\big(\norm{x_1}_{\L^p(M)}+\norm{x_2}_{\L^p(M)}+\norm{x_3}_{\L^p(M)}+\norm{x_4}_{\L^p(M)}\big)
	\leq 4K\norm{x}_{\L^p(M)}.
\end{align*}
\end{proof}

This result will imply in particular that a decomposable map is bounded.

The following result is proved in \cite[Proposition 2.2 and Lemma 2.3]{Pis4} for $S^p$. It has been long announced in \cite[page 2]{Jun} for QWEP von Neumann algebras (but seems definitely postponed). We will give a proof for hyperfinite von Neumann algebras, see Theorem \ref{thm-norm=normcb-hyperfinite}. Only Proposition \ref{prop-decomposable-Banach-space}, Proposition \ref{Prop-cb-leq-dec} and Proposition \ref{quest-cp-versus-dec} depend on this result.

\begin{thm}
\label{quest-cp-versus-cb}
Suppose $1< p \leq \infty$. Let $M,N$ be $\QWEP$ von Neumann algebras equipped with faithful semifinite normal traces. Let $T \co \L^p(M) \to \L^p(N)$ be a completely positive map. Then $T$ is completely bounded and $\norm{T}_{\L^p(M) \to \L^p(N)}=\norm{T}_{\cb,\L^p(M) \to \L^p(N)}$.
\end{thm}

The next lemmas are important for the proof of Theorem \ref{thm-dec=reg-hyperfinite}.

\begin{lemma}
\label{Lemma-passage1}
Let $M$ and $N$ be von Neumann algebras equipped with faithful normal semifinite traces. Suppose $1 \leq p \leq \infty$. Let $T, S \co \L^p(M) \to \L^p(N)$ be adjoint preserving maps\footnote{\thefootnote. This means that $T(x^*)=T(x)^*$ and $S(x^*)=S(x)^*$.} maps such that $-S \leq_{\cp} T \leq_{\cp} S$. Then the map $
  \begin{bmatrix}
   S  & T  \\
   T  & S  \\
  \end{bmatrix}
\co \L^p(M) \to S^p_2(\L^p(N))$ is completely positive.
\end{lemma}

\begin{proof}
Suppose $x \in S^p_n(\L^p(M))_+$. Then $-(\Id_{S^p_n} \ot S)(x) \leq (\Id_{S^p_n} \ot T)(x) \leq (\Id_{S^p_n} \ot S)(x)$. By Lemma \ref{Lemma-Matricial-2}, we deduce that 
$
\left(\Id_{S^p_n} \ot 
\begin{bmatrix}
   S  & T  \\
   T  & S  \\
\end{bmatrix} \right)(x)
=
\begin{bmatrix}
   (\Id_{S^p_n} \ot S)(x)  & (\Id_{S^p_n} \ot T)(x)  \\
   (\Id_{S^p_n} \ot T)(x)  & (\Id_{S^p_n} \ot S)(x)  \\
\end{bmatrix} \geq 0$.
\end{proof}

\begin{lemma}
\label{Lemma-passage2}
Let $M$ and $N$ be von Neumann algebras equipped with faithful normal semifinite traces. Suppose $1 \leq p \leq \infty$. Let $T,S_1, S_2 \co \L^p(M) \to \L^p(N)$ be adjoint preserving maps. If the map $
\begin{bmatrix}
   S_1  & T  \\
    T   & S_2 \\
\end{bmatrix}
\co \L^p(M) \to S^p_2(\L^p(N))$ is completely positive then $
-\frac{1}{2}(S_1+S_2) \leq_{\cp} T \leq_{\cp} \frac{1}{2}(S_1+S_2)$.
\end{lemma}

\begin{proof}
Suppose $x \in S^p_n(\L^p(M))_+$. We have
$$
\begin{bmatrix}
   (\Id_{S^p_n} \ot S_1)(x)  & (\Id_{S^p_n} \ot T)(x)  \\
   (\Id_{S^p_n} \ot T)(x)  & (\Id_{S^p_n} \ot S_2)(x)  \\
\end{bmatrix} 
=\left(\Id_{S^p_n} \ot 
\begin{bmatrix}
   S_1  & T  \\
   T  & S_2  \\
\end{bmatrix} \right)(x)
\geq 0.
$$
By Lemma \ref{Lemma-Matricial-3}, we deduce that 
$$
-\frac{1}{2}\big((\Id_{S^p_n} \ot S_1)(x)+(\Id_{S^p_n} \ot S_2)(x)\big) 
\leq (\Id_{S^p_n} \ot T)(x) 
\leq \frac{1}{2}\big((\Id_{S^p_n} \ot S_1)(x)+(\Id_{S^p_n} \ot S_2)(x)\big).
$$
Hence we obtain
$$
-\frac{1}{2}\big((\Id_{S^p_n} \ot (S_1 +S_2))(x)\big) 
\leq (\Id_{S^p_n} \ot T)(x) 
\leq \frac{1}{2}\big((\Id_{S^p_n} \ot (S_1+ S_2))(x)\big).
$$
We conclude that $-\frac{1}{2}(S_1+S_2) \leq_{\cp} T \leq_{\cp} \frac{1}{2}(S_1+S_2)$.
\end{proof}

\subsection{Completely positive maps on commutative $\L^p$-spaces}

We start with a characterization of the positive cone of $S^p_n(\L^p(\Omega))$ where $\Omega$ is a measure space.

\begin{lemma}
\label{lemma-description-con-commutative-Lp}
Let $\Omega$ be a (localizable) measure space. Suppose $1 \leq p <\infty$. Then an element $[f_{ij}]$ of $S^p_n(\L^p(\Omega))$ is positive if and only if $[f_{ij}(\omega)]$ is a positive element of $\M_n$ for almost every $\omega \in \Omega$.
\end{lemma}

\begin{proof}
We have $S^p_n(\L^p(\Omega))=\L^p(\Omega,S^p_n)$ isometrically. Consider $f \in \L^p(\Omega,S^p_n)_{+}$. Using \eqref{Cone-LpM}, there exists $h \in \L^{2p}(\Omega,S^{2p}_n)$ such that $h^*h=f$. Hence, for almost any $\omega \in \Omega$, we have $h(\omega)^*h(\omega)= f(\omega)$ in the space $S^p_n$. Consequently, for almost any $\omega \in \Omega$, we have $f(\omega)\in (S^p_{n})_+$.

For the converse, consider an element $f$ of $\L^p(\Omega,S^p_n)$ such that for almost any $\omega \in \Omega$ we have $f(\omega) \in (S^p_{n})_+$. Let $g \in \L^{p^*}(\Omega,S^{p^*}_n)_+$. By the first part of the proof, for almost any $\omega \in \Omega$, we have  $g(\omega) \in (S^{p^*}_{n})_+$. Using \eqref{equa-polar-Lp}, we deduce that for almost any $\omega \in \Omega$ we have $\tr( f(\omega) g(\omega)) \geq 0$. We infer that
$\dsp
\bigg(\int_{\Omega} \ot \tr \bigg)(fg)=\int_{\Omega} \tr( f(\omega) g(\omega))\d \omega\geq 0.
$ 
Using again \eqref{equa-polar-Lp}, we conclude that $f \in \L^p(\Omega,S^p_n)_+$.
\end{proof}

\begin{prop} 
\label{prop-positive-imply-cp}
Let $\Omega$ be a (localizable) measure space and let $M$ be a von Neumann algebra equipped with a faithful normal semifinite trace.  Suppose $1 \leq p \leq \infty$. A positive map $T \co \L^p(M) \to \L^p(\Omega)$ into a commutative $\L^p$-space is completely positive.
\end{prop}

\begin{proof}
The case $p = \infty$ is a particular case of \cite[Theorem 5.1.4]{ER}, so we can suppose $1 \leq p < \infty$. Let $x=[x_{ij}]$ be a positive element of $S^p_n(L^{p}(M))$. Note that in $S^p_n$, for almost any $\omega \in \Omega$, we have
$$
\big((\Id_{S^p_n} \ot T)([x_{ij}])\big)(\omega)=\big([T(x_{ij})]\big)(\omega)=\big[T(x_{ij})(\omega)\big].
$$
By Proposition \ref{prop-LpM-is-matrix-ordered}, for any matrix $u \in \M_{n,1}$, the element $u^*[x_{ij}]u$ of $\L^p(M)$ is positive. By the positivity of $T$, we see that $T\big(u^*[x_{ij}]u\big)$ is a positive element of $\L^p(\Omega)$. Using Lemma \ref{lemma-description-con-commutative-Lp}, we deduce that for almost every $\omega \in \Omega$
\begin{align*}
\MoveEqLeft
u^*\big[T(x_{ij})(\omega)\big]u
=\sum_{i,j=1}^{n}\ovl{u_i}T(x_{ij})(\omega)u_j
=T\bigg(\sum_{i,j=1}^{n}\ovl{u_i}x_{ij}u_j\bigg)(\omega)
=T\big(u^*[x_{ij}]u\big)(\omega) \geq 0.
\end{align*}
We infer that for almost every $\omega \in \Omega$, the matrix $\big[T(x_{ij})(\omega)\big]$ is a positive element of $\M_n$. By Lemma \ref{lemma-description-con-commutative-Lp}, we conclude that $\big[T(x_{ij})\big]$ is a positive element of $S^p_n(\L^p(\Omega))$.
\end{proof}

Using duality, we also have the following variant.

\begin{prop}
\label{prop-positive-imply-cp-if-depart-commutative}
Let $\Omega$ be a (localizable) measure space and let $M$ be a von Neumann algebra equipped with a faithful normal semifinite trace. Suppose $1 \leq p \leq \infty$. A positive mapping $T \co \L^p(\Omega) \to \L^p(M)$ defined on a commutative $\L^p$-space is completely positive.
\end{prop}

\begin{proof}
The case $p = \infty$ follows from \cite[Theorem 5.1.5]{ER}, so we can suppose $1 \leq p < \infty$. According to Lemma \ref{Lemma-adjoint-cp}, the map $T \co \L^p(\Omega) \to \L^p(M)$ is positive if and only if $T^* \co \L^{p^*}(M) \to \L^{p^*}(\Omega)$ is positive. Thus, by Proposition \ref{prop-positive-imply-cp}, the map $T^*$ is completely positive. Using again Lemma \ref{Lemma-adjoint-cp}, we conclude that $T$ is completely positive.
\end{proof}

\begin{remark} \normalfont
Note that the situation is different for the complete boundedness between commutative $\L^p$-spaces. Indeed, there exists some example of a measure space $\Omega$ and a bounded operator $T \co \L^p(\Omega) \to \L^p(\Omega)$ which is not completely bounded, see\footnote{\thefootnote. We warn the reader that the proof of \cite{DMT} is false. Indeed, the main argument of the paper which begins page 7 with ``therefore we can get a $\L^p(H)$ multiplier'' is really problematic since $H$ can be a finite subgroup (for example, consider the case $G=\Z$).} \cite[Proposition 8.1.3]{Pis5} and \cite{Arh5}.
\end{remark}

\subsection{Markov maps and selfadjoint maps}
\label{sec:Markov-maps}

%

Let $M$ and $N$ be von Neumann algebras equipped with faithful normal semifinite traces $\tau_M$ and $\tau_N$. We say that a linear map $T \co M \to N$ is a $(\tau_M,\tau_N)$-Markov map if $T$ is a normal unital completely positive map which is trace preserving, i.e. for any $x \in \mathfrak{m}_{\tau_M}^+$ we have $\tau_N(T(x))=\tau_M(x)$. When $(M,\tau_M)=(N,\tau_N)$, we say that $T$ is a $\tau_M$-Markov map. It is not difficult to check that a $(\tau_M,\tau_N)$-Markov map $T$ induces a completely positive and completely contractive map $T_p \co \L^p(M) \to \L^p(N)$ on the associated noncommutative $\L^p$-spaces $\L^p(M)$ and $\L^p(N)$ for any $1 \leq p \leq \infty$. 
Moreover, it is easy to prove that there exists a unique normal map $T^* \co N \to M$ such that
\begin{equation}
\label{Markov-dual}
\tau_N\big(T(x)y\big)=\tau_M\big(xT^*(y)\big), \quad x\in M \cap \L^1(M), y \in N \cap \L^1(N). 
\end{equation}
It is easy to show that $T^*$ is a $(\tau_N,\tau_M)$-Markov map. In this case, by density, we have 
\begin{equation}
\label{Markov-dual-Lp}
\tau_N\big(T_p(x)y\big)=\tau_M\big(x(T^*)_{p^*}(y)\big), \quad x \in \L^p(M), y \in \L^{p^*}(N). 
\end{equation}

Let $M$ be a von Neumann algebra equipped with a normal semifinite faithful trace $\tau$. Let $T \co M \to M$ be a normal contraction. We say that $T$ is selfadjoint
if
\begin{equation}
\label{9self}
\tau(T(x)y^*)=\tau(xT(y)^*), \quad x,y \in M \cap \L^1(M).
\end{equation}
In this case, for any $x,y$ in $M \cap \L^1(M)$, we have
$$
\bigl\vert \tau\big(T(x)y\big)\bigr\vert 
= \bigl\vert \tau\big(xT(y^*)^*\big)\bigr\vert
\leq \norm{x}_{\L^1(M)} \norm{T(y^*)^*}_{M}
\leq  \norm{x}_{\L^1(M)} \norm{y}_{M}.
$$
Hence the restriction of $T$ to $M \cap \L^1(M)$ extends to a contraction $T_1 \co \L^1(M)\to \L^1(M)$. It also extends by interpolation to a contraction $T_p \co \L^p(M) \to \L^p(M)$ for any $1 \leq p \leq \infty$. Moreover, for any $1 \leq p <\infty$, we have $(T_p)^*=(T_{p^*})^\circ$. Furthermore, the operator $T_2 \co \L^2(M) \to \L^2(M)$ is selfadjoint. If $T$ is positive then each $T_p$ is positive and hence $(T_p)^\circ=T_p$. Thus in this case, for any $1 \leq p <\infty$, we have $(T_p)^*=T_{p^*}$. Finally, if $T \co M \to M$ is a normal complete contraction, then each $T_p$ is completely contractive.

Finally, it is easy to check that a $\tau_M$-Markov map $T \co M \to M$ is selfadjoint if and only if $T^*=T^\circ$.

\section{Decomposable maps and regular maps}
\label{sec:Regular vs decomposable}

In this section, we start by analyzing decomposable maps on noncommutative $\L^p$-spaces. In particular, in Section \ref{subsec-infimum-decomposable-norm}, we prove that the infimum of the decomposable norm is actually a minimum.
In Section \ref{Decomposable-approx}, we state our first main result, Theorem \ref{thm-dec=reg-hyperfinite}, and give the end of the proof of this result. In passing, we prove that completely positive maps on noncommutative $\L^p$-spaces of approximately finite-dimensional algebras are necessarily completely bounded. In Section \ref{section-Decomposable-cb}, we compare the space of completely bounded operators and the space of decomposable operators. We show that these are different in general. We also give explicit examples of computations of the decomposable norm, see Theorem \ref{thm-computation-norm-finite-factor}. 

\subsection{Preliminary results}
\label{Sec-Preliminary-results}


We need some background on second dual algebras and we refer to \cite{Bin}, \cite{BLM}, \cite{KaR}, \cite{Tak1} and \cite{Tom1} for more information. Let $M$ be a von Neumann algebra of predual $M_*$. We can see $M^{**}$ as a von Neumann algebra. Since we have a canonical inclusion $M_* \subset M^{*}$, we can consider the annihilator 
$$
(M_*)^\perp 
\ov{\mathrm{def}}{=} \big\{\nu \in M^{**} : \langle \varphi,\nu \rangle_{M^{*},M^{**}}=0 \text{ for any } \varphi \in M_* \big\}
$$ 
of $M_*$ in $M^{**}$. It is well-known \cite[Proposition 4.2.3]{Bin} that there exists a unique central projection $e$ of $M^{**}$ such that $(M_*)^\perp=(1-e)M^{**}$. Using the notation $(R_{x}\varphi)(y) \ov{\mathrm{def}}{=} \varphi(yx)$ for any $x,y \in M$ and any $\varphi \in M^*$, we have $M_*=R_e(M^*)$ and\footnote{\thefootnote. That means that preduals of von Neumann algebras are L-summands in their biduals.} $M^*=M_*\oplus_1 R_{1-e}(M^*)$. The non-zero elements of $R_{1-e}M^*$ are the singular functionals. 

A bounded map $T \co M \to N$ is called singular \cite[p. 128]{Tak1} \cite{Tom1} if $T^*(N_*) \subset R_{1-e}M^*$. By \cite[Theorem 1]{Tom1}, for any bounded map $T \co M \to N$ there exists a unique couple $(T_{\w^*} \co M \to N,T_{\sing} \co M \to N)$ of bounded maps with $T_{\w^*}$ weak* continuous, $T_{\sing}$ singular and such that 
$$
T
=T_{\w^*}+T_{\sing}.
$$ 
Consider the completely contractive and completely positive map $\Phi_M \co M^{**} \to M^{**}$, $\eta \mapsto \eta e=e\eta e$ and the completely isometric canonical map $i_{N_*} \co N_* \to N^*$. By the proof of \cite[Theorem 1]{Tom1}, the map $T_{\w^*}$ is given by
\begin{equation*}
T_{\w^*}
\ov{\mathrm{def}}{=} \tilde{T} \circ \Phi_M \circ i_M
\end{equation*}
where $i_M \co M \to M^{**}$, $\tilde{T} \ov{\mathrm{def}}{=} (i_{N_*})^*\circ T^{**} \co M^{**} \to N$ is the unique weak* continuous extension of $T$ given by \cite[Lemma A.2.2]{BLM} (and its proof). The formula of the weak* extension of the proof of \cite[Theorem 1]{Tom1} is formally different but equivalent to ours. Indeed, in \cite[Theorem 1]{Tom1}, the weak* continuous extension $\tilde{T}$ is given by $\tilde{T}=(T^*|N_*)^*$ and we have $(T^*|N_*)^*=(T^* \circ i_{N_*})^*= (i_{N_*})^*\circ T^{**}$.


\begin{prop}
\label{Prop-recover-weak-star-continuity}
Let $M$ and $N$ be von Neumann algebras. Then the map $P_{\w^*} \co \B(M,N) \to \B(M,N)$, $T \mapsto T_{\w^*}$ is a contractive projection. Moreover, if $T\co M \to N$ is completely positive then the map $P_{\w^*}(T)$ is completely positive. Finally, if $T \co M \to N$ is completely bounded then $P_{\w^*}(T)$ is also completely bounded and $P_{\w^*} \co \B(M,N) \to \B(M,N)(T)$ is contractive projection. 
\end{prop}

\begin{proof}
It is obvious that $P_{\w^*}$ is a projection. Note that by \cite[Lemma A.2.2]{BLM}, we have $\norm{\tilde{T}}_{M^{**} \to N}=\norm{T}_{M \to N}$. 
Now, it is clear that, by composition, $P_{\w^*}$ is contractive. If $T \co M \to N$ is completely positive, using Lemma \ref{Lemma-adjoint-cp}, it is immediate to see that $T_{\w^*}$ is completely positive. By \cite[Section 1.4.8]{BLM}, if $T \co M \to N$ is completely bounded, then $\tilde{T} \co M^{**} \to N$ is completely bounded with the same completely bounded norm. By composition, we deduce that $P_{\w^*}(T)$ is completely bounded and that $\norm{P_{\w^*}(T)}_{\cb,M \to N} \leq \norm{T}_{\cb,M \to N}$.
\end{proof}



\begin{lemma}
\label{Lemma-op-mappings-3}
Let $M$ and $N$ be von Neumann algebras equipped with semifinite faithful normal traces. Suppose $1 \leq p \leq \infty$. Let $T \co \L^p(M) \to \L^p(N)$ be a linear map. Then $T$ is decomposable if and only if $T^\op$ is decomposable. In the case, we have $\norm{T}_{\dec, \L^p(M) \to \L^p(N)} = \norm{T^\op}_{\dec,\L^p(M)^\op \to \L^p(N)^\op}$.
\end{lemma}

\begin{proof}
Assume that $T \co \L^p(M) \to \L^p(N)$ is decomposable. By \eqref{Norm-dec}, there exist linear maps $v_1,v_2 \co \L^p(M) \to \L^p(N)$ such that 
$\begin{bmatrix} 
v_1 & T \\ 
T^\circ & v_2 
\end{bmatrix} 
\co S^p_2(\L^p(M)) \to S^p_2(\L^p(N))$ is completely positive with $\max\{\norm{v_1},\norm{v_2}\} \leq \norm{T}_{\dec, \L^p(M) \to \L^p(N)} + \epsi$. We claim that $\begin{bmatrix} 
v_2 & T \\ 
T^\circ & v_1 
\end{bmatrix} 
\co S^p_2(\L^p(M)^\op) \to S^p_2(\L^p(N)^\op)$ is also completely positive. Indeed, let $b \in \M_n(S^p_2(\L^p(M)^\op))_+ =S_{2n}^p(\L^p(M^\op))_+$. Denoting $b^t$ the transposed matrix, where transposition is executed in $S^p_{2n}$, i.e. both in the $\M_n$ and in the $S^p_2$ component, an obvious computation gives
$$
\left(\Id_{\M_n} \ot 
\begin{bmatrix} 
v_2 & T \\ 
T^\circ & v_1 
\end{bmatrix}\right)(b) 
=\left(\left(\Id_{\M_n} \ot 
\begin{bmatrix} 
v_2 & T^\circ \\ 
T & v_1 
\end{bmatrix}\right)(b^t) \right)^t
$$
which is positive in $\M_n(S^p_2(\L^p(M)^\op))$ according to Lemma \ref{equ-1-proof-op-mappings}, applied twice, provided that we show that the map $\begin{bmatrix} 
v_2 & T^\circ \\ 
T & v_1 
\end{bmatrix} 
\co S^p_2(\L^p(M)) \to S^p_2(\L^p(N))$ is completely positive.
But this can be seen using the identity
$$
\begin{bmatrix} 
v_2  & T^\circ \\ 
T & v_1 
\end{bmatrix} 
=\mathcal{F}_N 
\begin{bmatrix} 
v_1 & T \\ 
T^\circ & v_2 
\end{bmatrix} \mathcal{F}_M ,
$$
where $\mathcal{F}_M \co S^p_2(\L^p(M)) \to S^p_2(\L^p(M))$ denotes the flip mapping 
$$
\begin{bmatrix} 
a & b \\ 
c & d 
\end{bmatrix} 
\mapsto 
\begin{bmatrix} 
0 & 1 \\ 
1 & 0 \end{bmatrix} 
\begin{bmatrix} 
a & b \\ 
c & d 
\end{bmatrix}  
\begin{bmatrix} 
0 & 1 \\ 
1 & 0 
\end{bmatrix} 
=\begin{bmatrix} 
d & c \\ 
b & a 
\end{bmatrix} 
$$
which is completely positive according to \eqref{conj-cp} (and similarly for $\mathcal{F}_N$). We infer that the linear map $T^\op \co \L^p(M)^\op \to \L^p(N)^\op$ is decomposable and that $\norm{T^\op}_{\dec,\L^p(M)^\op \to \L^p(N)^\op} \leq \max\{ \norm{v_2},\norm{v_1} \} \leq \norm{T}_{\dec,\L^p(M) \to \L^p(N)} + \epsi$. Letting $\epsi \to 0$ and using symmetry, we can finish the proof of the lemma.
\end{proof}

We will use the following easy\footnote{\thefootnote. The first part is a consequence of the following computation (and the second part can be proved similarly):
\begin{align*}
\MoveEqLeft
\left\langle 
\begin{bmatrix} 
    T_{11} & T_{12} \\
    T_{21} & T_{22} \\
\end{bmatrix}
\left(\begin{bmatrix} 
a & b \\ 
c & d
\end{bmatrix}
\right),
\begin{bmatrix} 
x & y \\ 
z & w 
\end{bmatrix}
\right\rangle_{S_2^p(\L^p(N)),S_2^{p^*}(\L^{p^*}(N))}
=\left\langle 
\begin{bmatrix} 
T_{11}(a) & T_{12}(b) \\ 
T_{21}(c) & T_{22}(d)
\end{bmatrix} 
,
\begin{bmatrix} 
x & y \\ 
z & w 
\end{bmatrix}
\right\rangle\\
&=\tau(T_{11}(a) x) + \tau(T_{12}(b) y) + \tau(T_{21}(c) z) + \tau(T_{22}(d) w) 
=\tau(a T_{11}^*(x)) + \tau(bT_{12}^*(y)) + \tau(cT_{21}^*(z)) + \tau(dT_{22}^*(w))  \\
&=\left\langle 
\begin{bmatrix} 
a & b \\ 
c & d
\end{bmatrix}
,
\begin{bmatrix} 
T_{11}^* & T_{12}^* \\ 
T_{21}^* & T_{22}^* 
\end{bmatrix} 
\left(\begin{bmatrix} 
x & y \\ 
z & w 
\end{bmatrix}\right)
\right\rangle_{S_2^p(\L^p(M)),S_2^{p^*}(\L^{p^*}(M))}.
\end{align*}}  lemma several times.

\begin{lemma}
\label{Lemma-dualite-calcul-adjoint}
Let $M$ and $N$ be von Neumann algebras equipped with semifinite faithful normal traces. Suppose $1 \leq p < \infty$. The Banach adjoint of a bounded operator 
\[\begin{bmatrix} 
    T_{11} & T_{12} \\
    T_{21} & T_{22} \\
\end{bmatrix} \co S_2^p(\L^p(M)) \to S_2^p(\L^p(N))\]
identifies to $\begin{bmatrix} 
    (T_{11})^* & (T_{12})^* \\
    (T_{21})^* & (T_{22})^* \\
\end{bmatrix} \co S_2^{p^*}(\L^{p^*}(N)) \to S_2^{p^*}(\L^{p^*}(M))$ and the Banach preadjoint of a weak* continuous operator $\begin{bmatrix} 
    T_{11} & T_{12} \\
    T_{21} & T_{22} \\
\end{bmatrix} \co \M_2(M) \to \M_2(N)$ identifies to the bounded operator $\begin{bmatrix} 
    (T_{11})_* & (T_{12})_* \\
    (T_{21})_* & (T_{22})_* \\
\end{bmatrix} \co S_2^1(\L^1(N)) \to S_2^1(\L^1(M))$.
\end{lemma}


The following complements \cite[Lemma 3.2]{JuR} and completes a gap in the proof of the case $p=1$.

\begin{prop}
\label{Prop-Duality-dec}
Let $M$ and $N$ be two von Neumann algebras equipped with faithful normal semifinite traces. Suppose $1 \leq p < \infty$. A bounded map $T \co \L^p(M) \to \L^p(N)$ is decomposable if and only if the Banach adjoint $T^* \co \L^{p^*}(N) \to \L^{p^*}(M)$ is decomposable. In this case, we have  
\begin{equation}
	\label{Duality-dec}
	\norm{T}_{\dec, \L^p(M) \to \L^p(N)}
	=\norm{T^*}_{\dec,\L^{p^*}(N) \to \L^{p^*}(M)}.
\end{equation}
\end{prop}

\begin{proof}
Suppose $1 \leq p<\infty$. Suppose that $T \co \L^p(M) \to \L^p(N)$ is decomposable.
There exist some maps $v_1,v_2 \co \L^p(M) \to \L^p(N)$ such that $
\begin{bmatrix} 
   v_1  &  T \\
   T^\circ  &  v_2  \\
\end{bmatrix}$ is completely positive. Using Lemma \ref{Lemma-dualite-calcul-adjoint}, we obtain
that $\left(\begin{bmatrix} 
   v_1  &  T\\
   T^\circ  &  v_2  \\
\end{bmatrix}\right)^*
=\begin{bmatrix} 
   v_1^*  &  T^* \\
   (T^\circ)^*  &  v_2^*  \\
\end{bmatrix}
=\begin{bmatrix} 
   v_1^*  &  T^* \\
   (T^*)^\circ  &  v_2^*  \\
\end{bmatrix}$. By Lemma \ref{Lemma-adjoint-cp}, this operator is completely positive as a map $S^{p^*}_2(\L^{p^*}(M))^\op \to S^{p^*}_2(\L^{p^*}(N))^\op$. So by Lemma \ref{lem-op-mappings}, it also define a completely positive map $S^{p^*}_2(\L^{p^*}(M)) \to S^{p^*}_2(\L^{p^*}(N))$. We conclude that $T^* \co \L^p(M) \to \L^p(N)$ is decomposable with $\norm{T^*}_{\dec, \L^p(M) \to \L^p(N)} \leq \max\{\norm{v_1^*},\norm{v_2^*}\}=\max\{\norm{v_1},\norm{v_2}\}$. Taking the infimum, we obtain $\norm{T^*}_{\dec, \L^p(M) \to \L^p(N)} \leq \norm{T}_{\dec, \L^p(M) \to \L^p(N)}$. If $p\not=1$, a symmetric argument gives the result.

Suppose $p=1$ and that the map $T^* \co N \to M$ is decomposable. There exist some maps $v_1,v_2 \co N \to M$ such that $
\begin{bmatrix} 
   v_1  &  T^* \\
   (T^*)^\circ  &  v_2  \\
\end{bmatrix}$ is completely positive. Note that $v_1$ and $v_2$ are not necessarily weak* continuous. However, it is not difficult to see by uniqueness that 
$
P_{\mathrm{w}^*}\left(\begin{bmatrix} 
   v_1  &  T^* \\
   (T^*)^\circ  &  v_2  \\
\end{bmatrix}\right)
=\begin{bmatrix} 
   (v_1)_{\w^*}  &     T^* \\
   (T^*)^\circ  &  (v_2)_{\w^*}  \\
\end{bmatrix}
$ 
where $P_{\w^*} \co \B(\M_2(N),\M_2(M)) \to \B(\M_2(N),\M_2(M))$ is the projection of Proposition \ref{Prop-recover-weak-star-continuity}. Moreover, the same result says that 
$\begin{bmatrix} 
   (v_1)_{\w^*}  &     T^* \\
   (T^*)^\circ  &  (v_2)_{\w^*}  \\
\end{bmatrix}$ is still completely positive and that $\max\{\norm{(v_1)_{\w^*}},\norm{(v_2)_{\w^*}}\} \leq \max\{\norm{v_1},\norm{v_2}\}$. Using Lemma \ref{Lemma-dualite-calcul-adjoint}, we obtain
that $\left(
\begin{bmatrix} 
   (v_1)_{\w^*}  &     T^* \\
   (T^*)^\circ  &  (v_2)_{\w^*}  \\
\end{bmatrix}
\right)_*
=\begin{bmatrix} 
   ((v_1)_{\mathrm{w}^*})_*  &     T \\
   T^\circ  &  ((v_2)_{\w^*})_*  \\
\end{bmatrix}$. By Lemma \ref{Lemma-adjoint-cp} and Lemma \ref{lem-op-mappings}, this operator is completely positive as a map $S^1_2(\L^1(M)) \to S^1_2(\L^1(N))$. We conclude that $T$ is decomposable with $\norm{T}_{\dec, \L^1(M) \to \L^1(N)} \leq \max\{\norm{((v_1)_{\w^*})_*},\norm{((v_2)_{\w^*})_*}\}=\max\{\norm{(v_1)_{\w^*}},\norm{(v_2)_{\w^*}}\} \leq \max\{\norm{v_1},\norm{v_2}\}$. Taking the infimum, we obtain the inequality $\norm{T}_{\dec, \L^1(M) \to \L^1(N)} \leq \norm{T^*}_{\dec,N \to M}$. 
\end{proof}

Let $M_1$, $M_2$ and $M_3$ be von Neumann algebras equipped with faithful normal semifinite traces. Suppose $1 \leq p \leq \infty$. Let $T_1 \co \L^p(M_1) \to \L^p(M_2)$ and $T_2 \co \L^p(M_2) \to \L^p(M_3)$ be some decomposable maps. It is easy to see
that the composition $T_2 \circ T_1$ is decomposable and that 
\begin{equation}
	\label{Composition-dec}
	\norm{T_2 \circ T_1}_{\dec} \leq \norm{T_2}_{\dec} \norm{T_1}_{\dec}.
\end{equation}

Let $M_1$, $M_2$ and $M_3$ be approximately finite-dimensional von Neumann algebras equipped with  normal semifinite faithful traces. Suppose $1 \leq p \leq \infty$. Let $T_1 \co \L^p(M_1) \to \L^p(M_2)$ and $T_2 \co \L^p(M_2) \to \L^p(M_3)$ be some regular maps. It is easy to see that the composition $T_2 \circ T_1$ is regular and that 
\begin{equation}
	\label{Composition-reg}
	\norm{T_2 \circ T_1}_{\reg} 
	\leq \norm{T_2}_{\reg} \norm{T_1}_{\reg}.
\end{equation}
Let $M$ and $N$ be approximately finite-dimensional von Neumann algebras equipped with normal semifinite faithful traces. Suppose $1 < p < \infty$. According to \cite[Corollary 3.3]{Pis4} and \cite[Theorem 3.7]{Pis4} (see also \cite[(6) page 264]{Pis14}), we have the isometric interpolation identity \footnote{\thefootnote. The compatibility means, roughly speaking, that the elements of $\CB(M,N) \cap \CB(\L^1(M),\L^1(N))$ are the maps simultaneous bounded from $M$ into $N$ and from $\L^1(M)$ into $\L^1(N)$.} 
\begin{equation}
\label{Regular-as-interpolation-space}
\Reg\big(\L^p(M),\L^p(N)\big) 
=\big(\CB_{\w^*}(M,N),\CB(\L^1(M),\L^1(N))\big)^\frac{1}{p}
\end{equation}
where we use the Cald\'eron's second method or upper method \cite[page~88]{BeL} and where the subscript w* means ``weak* continuous''. The replacement of the space $\CB(M,N)$ of \cite[Corollary 3.3]{Pis4} $\CB_{\w^*}(M,N)$ by is irrelevant thanks to Proposition \ref{Prop-recover-weak-star-continuity}. We prefer use weak* continuous maps on von Neumann algebras in the sequel.

By Lemma \ref{Lemma-op-mapping-1} and \eqref{Regular-as-interpolation-space}, note that we have isometrically
\begin{align*}
\MoveEqLeft
 \Reg(\L^p(M^\op),\L^p(N^\op))
=(\CB_{\w^*}(M^\op,N^\op),\CB(\L^1(M^\op),\L^1(N^\op)))^\frac{1}{p}\\
&=(\CB_{\w^*}(M,N),\CB(\L^1(M),\L^1(N)))^\frac{1}{p}
=\Reg(\L^p(M),\L^p(N)).           
\end{align*} 
So a map $T \co \L^p(M) \to \L^p(N)$ is regular if and only if the opposite map $T^\op \co \L^p(M^\op) \to \L^p(N^\op)$ is regular with equality of regular norms.

Suppose $1 \leq p < \infty$. Let $M$ and $N$ be hyperfinite von Neumann algebras equipped with normal faithful semifinite traces. A bounded map $T \co \L^p(M) \to \L^p(N)$ is regular if and only if the Banach adjoint map $T^* \co \L^{p^*}(N) \to \L^{p^*}(M)$ is regular. In this case, we have 
\begin{equation}
	\label{Duality-reg}
	\norm{T}_{\reg, \L^p(M) \to \L^p(N)}=\norm{T^*}_{\reg,\L^{p^*}(N)^\op \to \L^{p^*}(M)^\op}.
\end{equation}

\subsection{On the infimum of the decomposable norm}
\label{subsec-infimum-decomposable-norm}

\begin{prop}
\label{Prop-dec-inf-atteint}
Let $M$ and $N$ be two von Neumann algebras equipped with faithful normal semifinite traces. Suppose $1 \leq p \leq \infty$. Let $T \co \L^p(M) \to \L^p(N)$ be a decomposable map. Then the infimum in the definition of $\norm{T}_\dec$ is actually a minimum i.e. we can choose $v_1$ and $v_2$ in \eqref{Norm-dec} such that $\norm{T}_{\dec,\L^p(M) \to \L^p(N)}=\max\{\norm{v_1},\norm{v_2}\}$. 
\end{prop}

\begin{proof} 
See \cite[page 184]{Haa} for the case $p=\infty$. Suppose $1<p<\infty$. For any integer $n$, let $v_n, w_n \co \L^p(M) \to \L^p(N)$ be bounded maps such that the map $
\begin{bmatrix} 
v_n & T \\ 
T^\circ & w_n
\end{bmatrix} 
\co S^p_2(\L^p(M)) \to S^p_2(\L^p(N))
$ 
is completely positive with $\max\{\norm{v_n},\norm{w_n}\} \leq \norm{T}_\dec + \frac{1}{n}$. Note that since $\L^p(N)$ is reflexive, the closed unit ball of the space $\B(\L^p(M),\L^p(N))$ of bounded operators in the weak operator topology is compact. Hence the bounded sequences $(v_{n})$ and $(w_n)$ admit convergent subnets $(v_{\alpha})$ and $(w_{\alpha})$ in the weak operator topology which converge to some $v,w \in \B(\L^p(M),\L^p(N))$. Now, it is easy to see that 
$ 
\begin{bmatrix} 
v & T \\ 
T^\circ & w 
\end{bmatrix} 
= \lim_\alpha 
\begin{bmatrix} 
v_{\alpha} & T \\ 
T^\circ & w_{\alpha}
\end{bmatrix}
$ 
in the weak operator topology of $\B(S_2^p(\L^p(M)),S_2^p(\L^p(N)))$. By Lemma \ref{lem-completely-positive-weak-limit}, the operator on the left hand side is completely positive as a weak limit of completely positive mappings. Moreover, using the weak lower semicontinuity of the norm, we see that $\norm{v} \leq \liminf_\alpha \norm{v_{\alpha}} \leq \norm{T}_\dec$ and $\norm{w} \leq \liminf_\alpha \norm{w_{\alpha}} \leq \norm{T}_\dec$. Hence, we have $\max\{\norm{v},\norm{w}\}=\norm{T}_{\dec}$. 

The case $p=1$ can be proved by duality using the proof of Proposition \ref{Prop-Duality-dec}.
\end{proof}

\begin{remark} \normalfont
Suppose $1<p<\infty$. If $T \co \L^p(M) \to \L^p(N)$ is a contractively decomposable map, we ignore if we can find some linear maps $v_1,v_2$ such that the map $\Phi$ of \eqref{Matrice-2-2-Phi} is completely positive \textit{and} contractive.
\end{remark}

\subsection{The Banach space of decomposable operators}

\begin{prop}
\label{prop-dec-homogeneous}
Let $M$ and $N$ be von Neumann algebras equipped with faithful normal semifinite traces. Suppose $1 \leq p\leq \infty$. If $\lambda\in \C$ and $T \co \L^p(M) \to \L^p(N)$ is decomposable then the map $\lambda T$ is decomposable and $\norm{\lambda T}_{\dec,\L^p(M) \to \L^p(N)}=|\lambda|\, \norm{T}_{\dec,\L^p(M) \to \L^p(N)}$.
\end{prop}

\begin{proof}
By symmetry, it suffices to prove $\norm{\lambda T}_{\dec} \leq |\lambda| \norm{T}_{\dec}$, since then $\norm{T}_{\dec} = \bnorm{\frac{1}{\lambda} \lambda T}_{\dec} \leq \frac{1}{|\lambda|} \norm{\lambda T}_{\dec}$. We can write $\lambda=|\lambda| \theta$ where $\theta$ is a complex number such that $|\theta|=1$. Assume that $v_1,v_2 \co \L^p(M) \to \L^p(N)$ are linear maps such that the map 
$\begin{bmatrix} 
v_1 & T \\ 
T^\circ & v_2
\end{bmatrix} 
\co S^p_2(\L^p(M)) \to S^p_2(\L^p(N))$ 
is completely positive. By \eqref{conj-cp}, the linear map 
$\begin{bmatrix} 
1 & 0 \\ 
0 & \theta 
\end{bmatrix} ^* 
\begin{bmatrix} 
v_1(\cdot) & T(\cdot) \\ 
T^\circ(\cdot) & v_2(\cdot) 
\end{bmatrix} 
\begin{bmatrix} 
1 & 0 \\ 
0 & \theta 
\end{bmatrix} $ 
is also completely positive on $S^p_2(\L^p(M))$. But it is easy to check that the latter operator equals 
$\begin{bmatrix} 
v_1 & \theta T \\ 
\ovl{\theta}T^\circ & v_2
\end{bmatrix} $. Thus the map $|\lambda| 
\cdot 
\begin{bmatrix} 
v_1 & \theta T \\ 
\ovl{\theta}T^\circ & v_2
\end{bmatrix} 
=\begin{bmatrix} 
|\lambda| v_1 & \lambda T \\ 
(\lambda T)^\circ & |\lambda| v_2
\end{bmatrix}$ is also completely positive. We deduce that $T$ is decomposable and that $\norm{\lambda T}_{\dec} \leq \max\big\{ \norm{|\lambda| v_1}, \norm{|\lambda| v_2} \big\} = |\lambda| \max\big\{ \norm{v_1}, \norm{v_2} \big\}$. Passing to the infimum yields the desired inequality $\norm{\lambda T}_{\dec} \leq |\lambda| \norm{T}_{\dec}$.
\end{proof}

It is not proved in \cite{JuR} that $\norm{\cdot}_{\dec,\L^p(M) \to \L^p(N)}$ is a norm.

\begin{prop}
\label{Prop-norm-dec}
Let $M$ and $N$ be two von Neumann algebras equipped with faithful normal semifinite traces. Suppose $1 \leq p \leq \infty$. Then $\Dec(\L^p(M),\L^p(N))$ is a vector space and $\norm{\cdot}_{\dec,\L^p(M) \to \L^p(N)}$ is a norm on $\Dec(\L^p(M),\L^p(N))$.
\end{prop}

\begin{proof}
Let $T_1,T_2 \co \L^p(M) \to \L^p(N)$ be decomposable maps. There exist some linear maps $v_1,v_2,w_1,w_2 \co \L^p(M) \to \L^p(N)$ such that $\begin{bmatrix} 
   v_1  &  T_1 \\
   T_1^\circ  &  v_2  \\
\end{bmatrix}$ and $\begin{bmatrix} 
    w_1  &   T_2 \\
    T_2^\circ  &  w_2 \\
\end{bmatrix}$ are completely positive. We can write 
$
\begin{bmatrix} 
   v_1  &  T_1 \\
   T_1^\circ  &  v_2  \\
\end{bmatrix}
+
\begin{bmatrix} 
    w_1  &   T_2 \\
    T_2^\circ  &  w_2 \\
\end{bmatrix}
=
\begin{bmatrix} 
   v_1+ w_1  &  T_1+T_2 \\
   T_1^\circ + T_2^\circ  &   v_2 + w_2  \\
\end{bmatrix}
=
\begin{bmatrix} 
   v_1+w_1              &  T_1 + T_2 \\
   (T_1+T_2)^\circ  &   v_2+w_2  \\
\end{bmatrix}
$. 
Moreover, this map is completely positive. Hence $T_1+T_1$ is decomposable. Furthermore, we deduce that
\begin{align*}
\MoveEqLeft
  \norm{T_1+T_2}_{\dec}  
		\leq \max\big\{\norm{v_1+w_1},\norm{v_2+w_2}\big\}\\
		&\leq \max\big\{\norm{v_1}+\norm{w_1}, \norm{v_2}+\norm{w_2}\big\}
		\leq \max\big\{\norm{v_1},\norm{v_2}\big\}+ \max\big\{\norm{w_1},\norm{w_2}\big\}.
\end{align*}
Passing to the infimum, we conclude that the sum $T_1+T_2$ is decomposable and we obtain the inequality $\label{prop-dec-triangle-inequality} \norm{T_1+T_2}_{\dec} \leq \norm{T_1}_{\dec}+\norm{T_2}_{\dec}$. The absolute homogeneity is Proposition \ref{prop-dec-homogeneous}. For the separation property, we can use Proposition \ref{Prop-cb-leq-dec} if the von Neumann algebras are $\QWEP$. If it is not the case, suppose $\norm{T}_\dec=0$. By Proposition \ref{Prop-dec-inf-atteint}, the map $
\begin{bmatrix} 
0 & T \\ 
T^\circ & 0
\end{bmatrix} 
\co S^p_2(\L^p(M)) \to S^p_2(\L^p(N))
$ is completely positive. Now, let $b \in \L^p(M)$ with $\norm{b}_{\L^p(M)} \leq 1$. By Proposition \ref{Lemma-Matricial-inequality3} there exist some $a, c \in \L^p(M)$ with $\norm{a}_{\L^p(M)} \leq 1$ and $\norm{c}_{\L^p(M)} \leq 1$ such that the element
$
\begin{bmatrix}
a   & b\\
b^* & c
\end{bmatrix}
$
of $S^p_{2}(\L^p(M))$ is positive. We deduce that the element $
\begin{bmatrix}
0   & T(b)\\
T(b)^* & 0
\end{bmatrix}$ 
is also positive. Using Lemma \ref{Lemma-Matricial-inequality}, we infer that $T(b)=0$. We conclude that $T=0$.
\end{proof}

\begin{lemma}
\label{Lemma-T-cic-decomposable}
Let $M$ and $N$ be von Neumann algebras equipped with faithful normal semifinite traces. Suppose $1 \leq p \leq \infty$ and let $T \co \L^p(M) \to \L^p(N)$ be a decomposable map. Then $T^\circ \co  \L^p(M) \to \L^p(N)$ defined by $T^\circ(x) = (T(x^*))^*$ is also decomposable and we have $\norm{T}_{\dec} = \|T^{\circ}\|_{\dec}$.
\end{lemma}
\begin{proof}
Consider some completely positive maps $v_1,v_2 \co \L^p(M) \to \L^p(N)$ such that
$\begin{bmatrix}
v_1 & T \\ 

T^{\circ} & v_2
\end{bmatrix}$ is completely positive. Using \eqref{conj-cp}, note that the map
$$ 
\mathcal{F}_M \co S^p_2(\L^p(M)) \to S^p_2(\L^p(M)),\: 
\begin{bmatrix}
a & b \\ c & d
\end{bmatrix}
\mapsto
\begin{bmatrix}
0 & 1 \\ 
1 & 0
\end{bmatrix}
\begin{bmatrix}
a & b \\ 
c & d
\end{bmatrix}
\begin{bmatrix}
0 & 1 \\ 
1 & 0
\end{bmatrix}
=\begin{bmatrix}
d & c \\ 
b & a
\end{bmatrix}
$$
is completely positive and similarly $\mathcal{F}_N \co S^p_2(\L^p(N)) \to S^p_2(\L^p(N))$. We deduce that the map
$$
\begin{bmatrix}
v_2 & T^{\circ} \\ 
T & v_1
\end{bmatrix}
= \mathcal{F}_N \circ
\begin{bmatrix}
v_1 & T \\ 
T^{\circ} & v_2
\end{bmatrix}
\circ \mathcal{F}_M
$$
is completely positive. Hence $T^\circ$ is decomposable and $\norm{T^\circ}_{\dec} \leq \max\{\norm{v_1},\norm{v_2}\}$. Passing to the infimum gives $\|T^{\circ}\|_{\dec} \leq \norm{T}_{\dec}$. Since $(T^\circ)^\circ = T$, we even have $\|T^\circ\|_{\dec} = \norm{T}_{\dec}$.
\end{proof}

\begin{prop}
\label{prop-decomposable-Banach-space}
Let $M$ and $N$ be two von Neumann algebras equipped with faithful normal semifinite traces. Suppose $1 \leq p \leq \infty$. Then the space $\Dec(\L^p(M),\L^p(N))$ is a Banach space with respect to the norm $\norm{\cdot}_{\dec,\L^p(M) \to \L^p(N)}$.
\end{prop}

\begin{proof}
Note first that $\norm{T}_{\L^p(M) \to \L^p(N)} \leq \norm{T}_{\dec,\L^p(M) \to \L^p(N)}$ for any decomposable map $T$. Indeed, for given $\epsi > 0$, let $v_1,v_2 \co \L^p(M) \to \L^p(N)$ be completely positive maps such that $\begin{bmatrix} v_1 & T \\ T^\circ & v_2 \end{bmatrix}$ is completely positive and $\max\{\norm{v_1},\norm{v_2}\} \leq \norm{T}_{\dec} + \epsi$.
Let $b \in \L^p(M)$ of norm less than one.
According to Proposition \ref{Lemma-Matricial-inequality3}, there exist $a,c \in \L^p(M)$ of norm less than one such that $\begin{bmatrix} a & b \\b^* & c \end{bmatrix}$ is positive.
Thus, $\begin{bmatrix}v_1(a) & T(b) \\T^\circ(b^*) & v_2(c) \end{bmatrix}$ is positive.
Then by Lemma \ref{Lemma-Matricial-inequality}, $\norm{T(b)}_p \leq \sqrt{\norm{v_1(a)}_p \norm{v_2(c)}_p} \leq \max\{\norm{v_1},\norm{v_2}\}\sqrt{\norm{a}_p \norm{c}_p} \leq \norm{T}_{\dec} + \epsi$.
Letting $\epsi \to 0$ shows that $\norm{T}_{\L^p(M) \to \L^p(N)} \leq \norm{T}_{\dec,\L^p(M) \to \L^p(N)}$.

Thus, if $(T_n)$ is a sequence in $\Dec(\L^p(M),\L^p(N))$ such that $\sum_{n = 1}^\infty \norm{T_n}_{\dec} < \infty$, we have that $\sum_{n = 1}^\infty T_n$ converges in $\B(\L^p(M),\L^p(N))$ with sum $T$. Let $v_{1,n},v_{2,n}$ be maps such that $\begin{bmatrix} v_{1,n} & T_n \\ 
T_n^\circ & v_{2,n} \end{bmatrix}$ is completely positive with $\max\{\norm{v_{1,n}},\norm{v_{2,n}}\} \leq \norm{T_n}_{\dec} + \epsi 2^{-n}$.
Then the series $\dsp \sum_{n = 1}^\infty \begin{bmatrix} v_{1,n} & T_n \\T_n^\circ & v_{2,n} \end{bmatrix}$ converges in $\B(S^p_2(\L^p(M)),S^p_2(\L^p(N)))$ and is completely positive by Lemma \ref{lem-completely-positive-weak-limit}. 
 With $v_i \ov{\mathrm{def}}{=} \sum_{n = 1}^\infty v_{i,n}$ where $i=1,2$, we infer that $\begin{bmatrix} v_1 & T \\ 
T^\circ & v_2 \end{bmatrix}$ 
is completely positive. So $T$ is decomposable with $\norm{T}_{\dec} \leq \max\{\norm{v_1},\norm{v_2}\} \leq \sum_{n = 1}^\infty \max\{\norm{v_{1,n}},\norm{v_{2,n}}\} \leq \epsi + \sum_{n = 1}^\infty \norm{T_n}_{\dec}$.
Finally, replacing $T$ by $T - \sum_{n = 1}^N T_n$ in the previous argument shows that $\norm{T-\sum_{n=1}^N T_n}_\dec \leq \epsi+\sum_{n=N+1}^{\infty} \norm{T_n}_\dec$. Hence $(\sum_{n=1}^N T_n)$ converges in $\Dec(\L^p(M),\L^p(N))$ to $T$.
\end{proof}

\begin{prop}
\label{quest-cp-versus-dec1}
Let $M$ and $N$ be two von Neumann algebras equipped with faithful normal semifinite traces. Suppose $1 \leq p \leq \infty$. Let $T \co \L^p(M) \to \L^p(N)$ be a completely positive map. Then $T$ is decomposable and $\norm{T}_{\dec,\L^p(M) \to \L^p(N)} \leq \norm{T}_{\L^p(M) \to \L^p(N)}$. 
\end{prop}

\begin{proof}
Using Lemma \ref{Lemma-cp-S1E}, we see that the linear map $\begin{bmatrix} 
T & T \\ 
T & T
\end{bmatrix} 
\co S^p_2(\L^p(M)) \to S^p_2(\L^p(N))$ is completely positive. We infer that $T$ is decomposable and that the inequality is true.
\end{proof}

\begin{prop}
\label{prop-linearcp-imply-decomposable}\label{prop-decomposable-se-decompose-en-cp}
Let $M$ and $N$ be von Neumann algebras equipped with faithful normal semifinite traces. Suppose $1 \leq p \leq \infty$. Let $T \co \L^p(M) \to \L^p(N)$ be a linear map. Then the following are equivalent.
\begin{enumerate}
	\item The map $T$ is decomposable.
	\item The map $T$ belongs to the span of the completely positive maps from $\L^p(M)$ into $\L^p(N)$.
	\item There exist some completely positive maps $T_1,T_2,T_3,T_4 \co \L^p(M) \to \L^p(N)$ such that 
	$$
	T
	=T_1-T_2+\i(T_3-T_4).
	$$
\end{enumerate}
If the latter case is satisfied, we have $\norm{T}_{\dec,\L^p(M) \to \L^p(N)} \leq \norm{T_1+T_2+T_3+T_4}_{\L^p(M) \to \L^p(N)}$.
\end{prop}

\begin{proof}
If there exist some completely positive maps $T_1,T_2,T_3,T_4 \co \L^p(M) \to \L^p(N)$ such that $T=T_1-T_2+\mathrm{i}(T_3-T_4)$ then $T$ belongs to to the span of the completely positive maps from $\L^p(M)$ into $\L^p(N)$. If $T$ belongs to the span of the completely positive maps from $\L^p(M)$ into $\L^p(N)$, by Proposition \ref{quest-cp-versus-dec1} and  Proposition \ref{Prop-norm-dec}, we deduce that $T$ is decomposable. Moreover, the proof of these results shows that if $T=T_1-T_2+\i(T_3-T_4)$ for some completely positive maps $T_1,T_2,T_3,T_4$ then we can use\footnote{\thefootnote. The argument is similar to the one of \cite[Proposition 5.4.1]{ER} and use a straightforward generalization of a part of \cite[Proposition 1.3.5]{ER}.} $v_1=v_2=T_1+T_2+T_3+T_4$ in \eqref{Matrice-2-2-Phi}. Hence we have $\norm{T}_{\dec,\L^p(M) \to \L^p(N)} \leq \norm{T_1+T_2+T_3+T_4}_{\L^p(M) \to \L^p(N)}$.

Now, suppose that the map $T$ is decomposable. There exist some completely positive maps $v_1,v_2 \co \L^p(M) \to \L^p(N)$ such that
$\Phi=\begin{bmatrix}
v_1 & T \\ 
T^{\circ} & v_2
\end{bmatrix}$ is completely positive. By \eqref{conj-cp}, the maps 
$
T_1=\frac{1}{4}\begin{bmatrix}
1 & 1 \\ 
\end{bmatrix}
\Phi
\begin{bmatrix}
1  \\
1\\ 
\end{bmatrix}
$,
$
T_2=\frac{1}{4}\begin{bmatrix}
1 & -1 \\
\end{bmatrix}
\Phi
\begin{bmatrix}
1  \\
-1\\ 
\end{bmatrix}
$,
$
T_3=\frac{1}{4}\begin{bmatrix}
1 & \i \\ 
\end{bmatrix}
\Phi
\begin{bmatrix}
1  \\
-\i\\ 
\end{bmatrix}
$
and
$
T_4=\frac{1}{4}\begin{bmatrix}
1 & -\i \\ 
\end{bmatrix}
\Phi
\begin{bmatrix}
1  \\
\i\\ 
\end{bmatrix}
$ are completely positive from $\L^p(M)$ into $\L^p(N)$ and it is easy to check that $T=T_1-T_2+\i(T_3-T_4)$.
\end{proof}

\begin{remark} \normalfont
Suppose $1 \leq p \leq \infty$. Let $T \co \L^p(M) \to \L^p(N)$ be a decomposable operator. We can define
$$
\norm{T}_{[d]}
\ov{\mathrm{def}}{=} \inf \big\{\norm{T_1}+\norm{T_2}+\norm{T_3}+\norm{T_4}\big\}
$$
where the infimum runs over all the previous possible decompositions of $T$ as $T=T_1-T_2+\mathrm{i}(T_3-T_4)$ where each $T_i$ is completely positive. It is stated in \cite[page~230]{Pis7} that $\norm{\cdot}_{[d]}$ is a norm, but it is not correct. Indeed, let $M = \C$. We have $\L^p(M)=\C$. Let $T \co \C \to \C$, $x \mapsto x$. Then we will prove that $\norm{T}_{[d]} = 1$ and that $\norm{(1+\mathrm{i})T}_{[d]} =2 \neq \sqrt{2} = |1+\mathrm{i}| \, \norm{T}_{[d]}$. First, since $T$ is completely positive, we have 
$$
\bnorm{T}_{[d]}
= \inf \bigg\{ a_1 + a_2 + a_3 + a_4 :\: a_k \geq 0,\:  1=a_1-a_2+\mathrm{i}(a_3-a_4) \bigg\}.
$$
For such a decomposition, we have $1=\Re(a_1-a_2+\mathrm{i}(a_3-a_4)) = a_1 - a_2$. We deduce that $\norm{T}_{[d]} \geq a_1 = 1 + a_2 \geq 1$. The decomposition $1=1-0+\mathrm{i}(0-0)$ gives the reverse inequality. Moreover, we have 
$$
\bnorm{(1+\mathrm{i})T}_{[d]}
= \inf \bigg\{ a_1 + a_2 + a_3 + a_4 :\: a_k \geq 0,\:  1 + \mathrm{i}=a_1-a_2+\mathrm{i}(a_3-a_4) \bigg\}.
$$
For such a decomposition, we have $1=\Re(a_1-a_2+\mathrm{i}(a_3-a_4)) = a_1 - a_2$ and $1=\Im (a_1-a_2+\mathrm{i}(a_3-a_4)) = a_3 - a_4$. We deduce that $a_1 = 1 + a_2 \geq 1$ and $a_3 = 1 + a_4 \geq 1$. Then $\norm{(1+\mathrm{i})T}_{[d]} \geq a_1 + a_3 \geq 1 + 1 = 2$. The decomposition $1+\mathrm{i}=1-0+\mathrm{i}(1-0)$ gives the reverse inequality.

However, it seems that $\norm{\cdot}_{[d]}$ is a norm on the real vector space of decomposable operators. The verification is left to the reader.
\end{remark}

\begin{prop}
\label{Prop-dec-finite-rank}
Let $M$ and $N$ be von Neumann algebras equipped with faithful normal semifinite traces. Suppose $1 \leq p \leq \infty$. Any finite rank bounded map $T \co \L^p(M) \to \L^p(N)$ is decomposable.
\end{prop}

\begin{proof}
Suppose $1 \leq p<\infty$. It suffices to prove that a rank one operator $T=\tr(y\cdot) \ot x$ is decomposable where $y \in \L^{p^*}(M)$ and $x \in \L^p(N)$. We can write $x=x_1-x_2+\i(x_3-x_4)$ and $y=y_1-y_2+\mathrm{i}(y_3-y_4)$ with $x_k,y_k \geq 0$. Hence we can suppose that $y \geq 0$ and $x \geq 0$. By Proposition \ref{prop-positive-imply-cp}, we deduce that the linear form $\tr(y\cdot)  \co \L^p(M) \to \C$ is completely positive. It is easy to deduce that $\tr(y\cdot) \ot x$ is completely positive, hence decomposable by Proposition \ref{quest-cp-versus-dec1}. The case $p=\infty$ is similar.
\end{proof}

\subsection{Reduction to the adjoint preserving case}
\label{subsec-Reduction-to-the-selfadjoint-case}

\begin{lemma}
\label{lem-regulier-et-selfadjoint3}
Let $E$ be an operator space and suppose $1 \leq p \leq \infty$. Then for any $a,b,c,d \in E$, we have
$$
\left\| 
\begin{bmatrix} 
0 & b \\ 
c & 0 
\end{bmatrix} 
\right\|_{S^p_2(E)} 
\leq 
\left\| 
\begin{bmatrix} 
a & b \\ 
c & d 
\end{bmatrix} \right\|_{S^p_2(E)}.
$$
\end{lemma}

\begin{proof}
Consider the Schur multiplier $M_A \co S^\infty_2 \to S^\infty_2$ where $A=\begin{bmatrix} 0 & 1 \\ 1 & 0 \end{bmatrix}$. Using Lemma \ref{lem-regulier-et-selfadjoint1} with $E = \C$ and $p = \infty$, we note that for any $a,b,c,d \in \C$
\begin{align*}
\left\| \begin{bmatrix} 
a & b \\ 
c & d \end{bmatrix} \right\|_{S^\infty_2}^2 
& = \left\| \begin{bmatrix} 
a & b \\ 
c & d \end{bmatrix}^* 
\begin{bmatrix} 
a & b \\ 
c & d \end{bmatrix} \right\|_{S^\infty_2} 
= \left\| \begin{bmatrix} 
|a|^2 + |c|^2 & \ovl{a}b + \ovl{c}d \\ 
a \ovl{b} + c \ovl{d} & |b|^2 + |d|^2 
\end{bmatrix} \right\|_{S^\infty_2} \\
& \geq \max\big\{ |a|^2 + |c|^2 , |b|^2 + |d|^2 \big\} 
\geq \max\big\{|c|,|b|\big\}^2
=\left\| \begin{bmatrix} 
0 & b \\ 
c & 0 
\end{bmatrix} \right\|_{S^\infty_2}^2.
\end{align*}
We deduce that the Schur multiplier $M_A$ is a contraction, hence a complete contraction.  By duality, $M_A \co S^1_2 \to S^1_2$ is also a complete contraction. Using Lemma \ref{lem-sufficient-condition-reg}, we deduce that $M_A$ is contractively regular on $S^p_2$ and the lemma follows.
\end{proof}

\begin{lemma}
\label{lem-regulier-et-selfadjoint2}
Let $M$ and $N$ be approximately finite-dimensional von Neumann algebras equipped with faithful normal semifinite traces. Suppose $1 \leq p \leq \infty$ and let $T \co \L^p(M) \to \L^p(N)$ be a regular map. Then $T^\circ \co  \L^p(M) \to \L^p(N)$ defined by $T^\circ(x) = (T(x^*))^*$ is also regular and we have $\|T^\circ\|_{\reg} = \norm{T}_{\reg}$.
\end{lemma}

\begin{proof}
We recall that by \eqref{Regular-as-interpolation-space}, $\Reg(\L^p(M),\L^p(N))$ is a complex interpolation space following Calder\'on's upper method.
Choose now an analytic function $F \co S \to \CB(M,N) + \CB(\L^1(M),\L^1(N))$ of $\mathcal{G}$ defined on the usual complex interpolation strip $S = \{ z \in \C : 0 \leq \Re z \leq 1 \}$, such that $F'(\theta) = T$ with $\norm{F}_{\mathcal{G}} \leq \norm{T}_{\reg} + \epsi$. Put $G(z) = F(\ovl{z})^\circ$. Then the function $G$ also belongs to $\mathcal{G}$ with $\norm{G}_{\mathcal{G}} = \norm{F}_{\mathcal{G}}$ and we have $G'(\theta) = T^\circ$. Thus the map $T^\circ$ is regular and $\|T^\circ\|_{\reg} \leq \norm{T}_{\reg} + \epsi$. Letting $\epsi \to 0$ we obtain $\|T^\circ\|_{\reg} \leq \norm{T}_{\reg}$. Since $(T^\circ)^\circ = T$, we even have $\|T^\circ\|_{\reg} = \norm{T}_{\reg}$. 
\end{proof}

\begin{prop}
\label{prop-regulier-et-selfadjoint}
Let $M$ and $N$ be approximately finite-dimensional von Neumann algebras equipped with faithful normal semifinite traces. Suppose $1 \leq p \leq \infty$ and that $T \co \L^p(M) \to \L^p(N)$ is a linear mapping. Define $\tilde{T} \co S^p_2(\L^p(M)) \to S^p_2(\L^p(N))$ by
$$
\tilde{T} \left(\begin{bmatrix} 
a & b \\ 
c & d 
\end{bmatrix}\right) 
= \begin{bmatrix} 
0 & T(b) \\ 
T^\circ(c) & 0 
\end{bmatrix}.
$$
Then $\tilde{T}$ is adjoint preserving in the sense that $\tilde{T}(x^*) = \left( \tilde{T}(x) \right)^*$. Moreover, $T$ is regular if and only if the map $\tilde{T} \co S^p_2(\L^p(M)) \to S^p_2(\L^p(N))$ is regular and in this case, we have $\norm{T}_{\reg, \L^p(M) \to \L^p(N)}=\|\tilde{T}\|_{\reg, S^p_2(\L^p(M)) \to S^p_2(\L^p(N))}$.
\end{prop}

\begin{proof}
Let $x = 
\begin{bmatrix} 
a & b \\ 
c & d 
\end{bmatrix} \in S^p_2(\L^p(M))$.
We have 
$$
\tilde{T}(x^*) 
= \tilde{T} \left(\begin{bmatrix} 
a^* & c^* \\ 
b^* & d^* 
\end{bmatrix} \right)
= \begin{bmatrix} 
0 & T(c^*) \\ 
T^\circ(b^*) & 0 
\end{bmatrix} 
= \begin{bmatrix} 
0 & T(c^*) \\ 
T(b)^* & 0 
\end{bmatrix}
$$
and also
$$
\left(\tilde{T}(x)\right)^* 
= \begin{bmatrix} 
0 & T(b) \\ 
T^\circ(c) & 0 
\end{bmatrix}^* 
= \begin{bmatrix} 
0 & T^\circ(c)^* \\ 
T(b)^* & 0 
\end{bmatrix} =  
\begin{bmatrix} 

0 & T(c^*) \\ 
T(b)^* & 0 
\end{bmatrix}.
$$
We conclude that $\tilde{T}$ is adjoint preserving, i.e. $\tilde{T}^\circ = \tilde{T}$.
Assume first that $1 \leq p < \infty$.
Let $E$ be any operator space. Assume first that $T$ is regular. For any 
$\begin{bmatrix} 
a & b \\ 
c & d 
\end{bmatrix} \in S^p_2(\L^p(M,E))$, according to Lemma \ref{lem-regulier-et-selfadjoint1} with $E$ replaced by $\L^p(N,E)$, we have
\begin{align*}
\left\| \left( \tilde{T} \ot \Id_E \right) 
\begin{bmatrix}
a & b \\ 
c & d 
\end{bmatrix} \right\|_{S^p_2(\L^p(N,E))}
& = 
\left\| \begin{bmatrix} 
0 & (T \ot \Id_E)(b) \\ 
(T^\circ \ot \Id_E)(c) & 0 
\end{bmatrix} 
\right\|_{S^p_2(\L^p(N,E))} \\
& = \left( \bnorm{(T \ot \Id_E)(b)}_{\L^p(N,E)}^p + \bnorm{(T^\circ \ot \Id_E)(c)}_{\L^p(N,E)}^p \right)^{\frac1p}.
\end{align*}
The previous quantity can be estimated by $\norm{T}_{\reg} \left(\norm{b}_{\L^p(M,E)}^p + \norm{c}_{\L^p(M,E)}^p \right)^{\frac1p}$ due to Lemma \ref{lem-regulier-et-selfadjoint2}. According to Lemmas \ref{lem-regulier-et-selfadjoint1} and \ref{lem-regulier-et-selfadjoint3} with $E$ replaced by $\L^p(M,E)$, this in turn can be estimated by
$$
\norm{T}_{\reg} \left\| 
\begin{bmatrix} 
0 & b \\ 
c & 0 
\end{bmatrix} \right\|_{S^p_2(\L^p(M,E))} 
\leq \| T \|_{\reg} \left\| 
\begin{bmatrix} 
a & b \\ 
c & d 
\end{bmatrix} 
\right\|_{S^p_2(\L^p(M,E))}.
$$
This shows that $\|\tilde{T} \ot \Id_E\|_{S^p_2(\L^p(M,E)) \to S^p_2(\L^p(N,E))} \leq \norm{T}_{\reg}$. Passing to the supremum over all operator spaces $E$, we deduce that $\tilde{T}$ is regular and that $\|\tilde{T}\|_{\reg} \leq \norm{T}_{\reg}$.

For the converse inequality, assume that $\tilde{T}$ is regular and let $x \in \L^p(M,E)$. Applying Lemma \ref{lem-regulier-et-selfadjoint1} twice, we have
\begin{align*}
 \MoveEqLeft
\bnorm{(T \ot \Id_E)(x)}_{\L^p(N,E)} 
 = \left\| 
\begin{bmatrix} 
0 & (T \ot \Id_E)(x) \\ 
0 & 0 
\end{bmatrix} 
\right\|_{S^p_2(\L^p(N,E))} \\
&= \left\| 
\bigg(\begin{bmatrix} 
0 & T \\ 
T^\circ & 0 
\end{bmatrix} 
\ot \Id_E \bigg)
\begin{bmatrix} 
0 & x \\ 
0 & 0 
\end{bmatrix} \right\|_{S^p_2(\L^p(N,E))} 
\leq \| \tilde{T} \|_{\reg} 
\left\| \begin{bmatrix} 
0 & x \\ 
0 & 0 
\end{bmatrix} \right\|_{S^p_2(\L^p(M,E))} \\
&= \| \tilde{T} \|_{\reg} \norm{x}_{\L^p(M,E)}.
\end{align*}
We conclude that $T$ is regular and that $\norm{T}_{\reg} \leq \| \tilde{T} \|_{\reg}$.

The case $p = \infty$ is similar, using in the second part of Lemma \ref{lem-regulier-et-selfadjoint1} each time.
\end{proof}

\begin{prop}
\label{prop-decomposable-et-selfadjoint3}
Let $M$ and $N$ be von Neumann algebras equipped with faithful normal semifinite traces. Suppose $1 \leq p \leq \infty$. Let $T \co \L^p(M) \to \L^p(N)$ be a linear map. Then $T$ is decomposable if and only if the map $\tilde{T} \co S^p_2(\L^p(M)) \to S^p_2(\L^p(N))$ from Proposition \ref{prop-regulier-et-selfadjoint} is decomposable, and in this case, we have $\norm{T}_{\dec,\L^p(M) \to \L^p(N)} = \|\tilde{T}\|_{\dec,S^p_2(\L^p(M)) \to S^p_2(\L^p(N))}$.
\end{prop}

\begin{proof}
Suppose that $T$ is decomposable. Choose some maps $v_1,v_2 \co \L^p(M) \to \L^p(N)$ such that $
\begin{bmatrix} 
v_1 & T \\ 
T^\circ & v_2 
\end{bmatrix} 
\co S^p_2(\L^p(M)) \to S^p_2(\L^p(N))
$ is completely positive. By \eqref{conj-cp}, the mapping
\begin{align*}
& \begin{bmatrix} 0 & 0 & 1 & 0 \\ 0 & 1 & 0 & 0 \\ 1 & 0 & 0 & 0 \\ 0 & 0 & 0 & 1 \end{bmatrix} \begin{bmatrix} v_1 & T & 0 & 0 \\ T^\circ & v_2 & 0 & 0 \\ 0 & 0 & v_1 & T \\ 0 & 0 & T^\circ & v_2 \end{bmatrix} ( \cdot ) 
\begin{bmatrix} 
0 & 0 & 1 & 0 \\ 
0 & 1 & 0 & 0 \\ 
1 & 0 & 0 & 0 \\ 
0 & 0 & 0 & 1 
\end{bmatrix} 
 = \begin{bmatrix} 
\begin{bmatrix} v_1 & 0 \\ 
0 & v_2 \end{bmatrix} & \tilde{T} \\ 
\tilde{T} &  
\begin{bmatrix} 
v_1 & 0 \\ 
0 & v_2 
\end{bmatrix} 
\end{bmatrix}  
\end{align*}
is also completely positive from $S^p_4(\L^p(M))$ into $S^p_4(\L^p(N))$. Therefore the map $\tilde{T}$ is decomposable and $\| \tilde{T} \|_{\dec} \leq \left\| \begin{bmatrix} v_1 & 0 \\ 0 & v_2 \end{bmatrix} \right\| = \max\{ \norm{v_1}, \norm{v_2}\}$, the latter according to \cite[Corollary 1.3]{Pis5}. By passing to the infimum over all admissible $v_1,v_2$, we see that $\| \tilde{T} \|_{\dec} \leq \norm{T}_{\dec}$.

Now suppose that the map $\tilde{T}$ is decomposable. Let $v_1,v_2 \co S^p_2(\L^p(M)) \to S^p_2(\L^p(N))$ such that the map $
\begin{bmatrix} 
v_1 & \tilde{T} \\ 
\tilde{T} & v_2 
\end{bmatrix} 
\co S^p_4(\L^p(M)) \to S^p_4(\L^p(N))$ is completely positive. Put $w_1 \co \L^p(M) \to \L^p(N)$, $a \mapsto 
\left(v_1\left(\begin{bmatrix} 
a & 0 \\ 
0 & 0 
\end{bmatrix}\right) \right)_{11}$ and $w_2 \co \L^p(M) \to \L^p(N)$, $d \mapsto 
\left(v_2\left(\begin{bmatrix} 
0 & 0 \\ 
0 & d 
\end{bmatrix}\right) \right)_{22}$. Then each $w_i$ is also completely positive as a composition of completely positive mappings. We also define 
\begin{equation*}
\begin{array}{cccc}
   J   \co &   S^p_2(\L^p(M))  &  \longrightarrow   &  S^p_4(\L^p(M))  \\
           &  \begin{bmatrix} 
a & b \\ 
c & d 
\end{bmatrix}   &  \longmapsto       &  \begin{bmatrix} 
a & 0 & 0 & b\\
0 & 0 & 0 & 0\\
0 & 0 & 0 & 0\\
c & 0 & 0 & d\\
\end{bmatrix}   \\
\end{array}.
\end{equation*}
It is easy to see that $J$ is a completely positive and completely isometric embedding. Then an easy computation gives
\begin{align*}
\MoveEqLeft
\begin{bmatrix} 
1 & 0 & 0 & 0 \\ 
0 & 0 & 0 & 1 
\end{bmatrix} 
\cdot 
\Bigg(
\begin{bmatrix} 
v_1 & \tilde{T} \\ 
\tilde{T} & v_2 
\end{bmatrix} 
\left( J\left( 
\begin{bmatrix} 
a & b \\ 
c & d 
\end{bmatrix} 
\right)\right)\Bigg)\cdot \begin{bmatrix} 
1 & 0 \\ 0 & 0 \\ 
0 & 0 \\ 0 & 1 
\end{bmatrix}
\\
&=\begin{bmatrix} 
1 & 0 & 0 & 0 \\ 
0 & 0 & 0 & 1 
\end{bmatrix} 
\cdot 
\left(
\begin{bmatrix} 
v_1 & \tilde{T} \\ 
\tilde{T} & v_2 
\end{bmatrix} 
\left(  \begin{bmatrix} 
a & 0 & 0 & b\\
0 & 0 & 0 & 0\\
0 & 0 & 0 & 0\\
c & 0 & 0 & d\\
\end{bmatrix}\right)\right)\cdot \begin{bmatrix} 
1 & 0 \\ 0 & 0 \\ 
0 & 0 \\ 0 & 1 
\end{bmatrix}\\
&=\begin{bmatrix} 
1 & 0 & 0 & 0 \\ 
0 & 0 & 0 & 1 
\end{bmatrix} 
\cdot 
\begin{bmatrix} 
v_1\left(\begin{bmatrix} 
a & 0 \\ 0 & 0 \end{bmatrix}\right) & \begin{bmatrix} 
0 & T(b) \\ 
0 & 0 
\end{bmatrix} \\ 
 \begin{bmatrix} 
0 & 0 \\ 
T^\circ(c) & 0 
\end{bmatrix} &  v_2\left(\begin{bmatrix} 
0 & 0 \\ 0 & d \end{bmatrix}\right)
\end{bmatrix}\cdot 
\begin{bmatrix} 
1 & 0 \\ 0 & 0 \\ 
0 & 0 \\ 0 & 1 
\end{bmatrix}
=\begin{bmatrix} 
w_1(a) & T(b) \\ 
T^\circ(c) & w_2(d) 
\end{bmatrix}.
\end{align*}
Using \eqref{conj-cp}, we deduce by composition that the map $\begin{bmatrix} 
w_1 & T \\ 
\tilde{T} & w_2 
\end{bmatrix}$ is completely positive. We infer that $T$ is decomposable and that $\norm{T}_{\dec} \leq \max\{\norm{w_1},\norm{w_2}\} \leq \max\{\norm{v_1},\norm{v_2}\}$. Passing to the infimum over all admissible $v_1,v_2$, we obtain that $\norm{T}_{\dec} \leq \| \tilde{T} \|_{\dec}$.
\end{proof}

\begin{prop}
\label{prop-decomposable-et-selfadjoint}
Let $M$ and $N$ be von Neumann algebras equipped with faithful normal semifinite traces. Suppose $1 \leq p \leq \infty$. An adjoint preserving\footnote{\thefootnote. That means that $T(x^*)=T(x)^*$ for any $x \in \L^p(M)$.} map $T \co \L^p(M) \to \L^p(N)$ is decomposable if and only if one of the two following infimums is finite. In this case, we have
\begin{align*}
\norm{T}_{\dec,\L^p(M) \to \L^p(N)}
&=\inf\big\{\norm{S}\ :\ S \co \L^p(M) \to \L^p(N) \ \cp,\ -S \leq_{\cp} T \leq_{\cp} S \big\}\\    
&=\inf\big\{\norm{T_1+T_2}\ :\ T_1,T_2 \co \L^p(M) \to \L^p(N)\ \cp,\ T=T_1-T_2 \big\}.
\end{align*}
\end{prop}

\begin{proof}
The first equality is a consequence of Lemma \ref{Lemma-passage1} and Lemma \ref{Lemma-passage2}. To prove the second equality, first assume that there exists some completely positive map $S \co \L^p(M) \to \L^p(N)$ such that
$$
-S \leq_{\cp} T \leq_{\cp} S.
$$
Then $T_1 = \frac{1}{2}(S+T)$ and $T_2 = \frac{1}{2}(S-T)$ are completely positive and we have $T_1+T_2 =\frac{1}{2}(S+T)+\frac{1}{2}(S-T)= S$ and $T_1-T_2=\frac{1}{2}(S+T)-\frac{1}{2}(S-T)=T$. 

Conversely, suppose that we can write $T=T_1-T_2$ for some completely positive maps $T_1,T_2 \co \L^p(M) \to \L^p(N)$. Then we have
$$
-(T_1+T_2) \leq_{\cp} T \leq_{\cp} (T_1+T_2).
$$
This proves the second equality.
\end{proof}

\subsection{Decomposable vs regular on Schatten spaces}
\label{subsec-decomposable-vs-regular-on-Schatten-spaces}

Similarly to the commutative case, an absolute contraction between noncommutative $\L^p$-spaces is contractively regular.


\begin{lemma}
\label{lem-sufficient-condition-reg}
Let $M$ and $N$ be approximately finite-dimensional von Neumann algebras which are equipped with faithful normal semifinite traces. Let $T \co M \to N$ be a completely contractive map such that the restriction to $M \cap \L^1(M)$ induces a completely contractive map from $\L^1(M)$ into $\L^1(N)$. Then for any $1 \leq p \leq \infty$, we have $\norm{T}_{\reg,\L^p(M) \to \L^p(N)} \leq 1$.
\end{lemma}

\begin{proof}
Let $E$ be any operator space. According to \cite[Proposition 8.1.5]{ER}, the map $T \ot \Id_E \co \L^\infty(M,E) = M \ot_{\min} E \to \L^\infty(N,E) = N \ot_{\min} E$ is completely contractive. Moreover, by \cite[Corollary 7.1.3]{ER} the map $T \ot \Id_E \co \L^1(M,E) = \L^1(M) \otp E \to \L^1(N,E) = \L^1(N) \otp E$ is also completely contractive, where $\otp$ denotes the operator space projective tensor product. By interpolation, we infer that the map $T \ot \Id_E \co \L^p(M,E) \to \L^p(N,E)$ is completely contractive for any $1 \leq p \leq \infty$. Passing over the supremum of all operator spaces, we obtain the lemma.
\end{proof}

Suppose $1\leq p<\infty$. If $n$ and $d$ are integers then a particular case of \cite[Theorem 1.5]{Pis4} gives for any $x \in S^p_n(\M_d)$
\begin{equation}
	\label{Norm-Sp(Md)}
	\norm{x}_{S^p_n(\M_d)} 
= \inf\big\{ \norm{\alpha}_{S^{2p}_n} \norm{y}_{\M_n(\M_d)} \norm{\beta}_{S^{2p}_n} : x=(\alpha \ot \I_d)y(\beta \ot \I_d)\big\}.
\end{equation}

\begin{thm}
\label{prop-dec-reg-matrix-case}
Let $n,m \in \N$ and $1 \leq p \leq \infty$. Then any linear mapping $T \co S^p_m \to S^p_n$ satisfies
$$
\norm{T}_{\reg,S^p_m \to S^p_n} 
=\norm{T}_{\dec,S^p_m \to S^p_n}.
$$
\end{thm}

\begin{proof}
Assume that the theorem is true for all adjoint preserving maps $T \co S^p_m \to S^p_n$, i.e. $T(x^*) = T(x)^*$. Then we can deduce from Propositions \ref{prop-decomposable-et-selfadjoint3} and \ref{prop-regulier-et-selfadjoint}, with the adjoint preserving mapping $\tilde{T} \co S^p_{2m} \to S^p_{2n}$, that $\norm{T}_\dec = \|\tilde{T}\|_\dec = \|\tilde{T}\|_\reg = \norm{T}_\reg$. Hence we can assume in addition that $T$ is adjoint preserving.

First we show $\norm{T}_{\reg} \leq \norm{T}_{\dec}$.
The following proof is inspired by the proof of \cite[Lemma 2.3]{Pis4}. 
Let $\epsi > 0$. According to Proposition \ref{prop-decomposable-et-selfadjoint}, there exist some completely positive maps $T_1,T_2 \co S^p_n \to S^p_n$ such that $T = T_1 - T_2$ and $\norm{T_1 + T_2} \leq \norm{T}_\dec + \epsi$. According to Choi's characterization \cite[Theorem 1]{Choi}, there exist $a_1,\ldots,a_l,b_1,\ldots,b_l \in \M_{m,n}$ such that $T_1(x) = \sum_{k=1}^{l} a_k^* x a_k$ and $T_2(x) = \sum_{k=1}^{l} b_k^* x b_k$. Let $x$ be an element of $S^p_n(\M_d)$ with $\norm{x}_{S^p_n(\M_d)} < 1$. By \eqref{Norm-Sp(Md)}, there exists a decomposition $x=(\alpha \ot \I_d)y(\beta \ot \I_d)$ with $\alpha,\beta \in S^{2p}_m$ of norm less than 1 and $y \in \M_m(\M_d)$ which is also of norm less than 1. Using the notations 
$$
\alpha_1 \ov{\mathrm{def}}{=} [a_{1}^{*}\alpha,\ldots,a_{l}^{*}\alpha], \quad 
\beta_1 \ov{\mathrm{def}}{=} (a_1^* \beta^*, \ldots, a_l^* \beta^*),
$$
and 
$$
\alpha_2 \ov{\mathrm{def}}{=} [b_{1}^{*}\alpha,\ldots,b_{l}^{*}\alpha], \quad
\beta_2 \ov{\mathrm{def}}{=} (b_1^* \beta^*, \ldots,b_l^* \beta^*)
$$ 
of $\M_{1,l}(\M_{n,m})$, we can write
\begin{align*}
\MoveEqLeft
  (T \ot \Id_{\M_d})(x) 
	=(T \ot \Id_{\M_d})\big((\alpha \ot \I_d)y(\beta \ot \I_d)\big)\\
  & =(T \ot \Id_{\M_d})\Bigg((\alpha \ot \I_{d}) \bigg(\sum_{i,j=1}^{n} e_{ij} \ot y_{ij}\bigg) (\beta \ot \I_{d})\Bigg)
	=  \sum_{i,j=1}^{n} (T \ot \Id_{\M_d})(\alpha e_{ij}\beta \ot y_{ij}) \\ 
		&=  \sum_{i,j=1}^{n} T(\alpha e_{ij} \beta) \ot y_{ij} 
		=  \sum_{i,j=1}^{n} T_1(\alpha e_{ij}\beta) \ot y_{ij}-T_2(\alpha e_{ij}\beta) \ot y_{ij}\\
  & =  \sum_{i,j=1}^{n}\sum_{k=1}^{l} a_{k}^{*}\alpha e_{ij}\beta a_{k} \ot y_{ij}
	- b_{k}^{*}\alpha e_{ij}\beta b_{k} \ot y_{ij}\\
	&=  \sum_{i,j=1}^{n} \sum_{k=1}^{l} (a_{k}^*\alpha \ot \I_d) (e_{ij} \ot y_{ij}) (\beta a_{k} \ot \I_d) 
	-(b_{k}^*\alpha \ot \I_d) (e_{ij} \ot y_{ij}) (\beta b_{k} \ot \I_d)\\
  & =  \sum_{k=1}^{l} (a_{k}^{*}\alpha \ot \I_d) \Bigg(\sum_{i,j=1}^{n} e_{ij} \ot y_{ij} \Bigg) (\beta a_{k} \ot \I_d)
	-(b_{k}^{*}\alpha \ot \I_d) \Bigg(\sum_{i,j=1}^{n} e_{ij} \ot y_{ij} \Bigg) (\beta b_{k} \ot \I_d)\\
  & =  \sum_{k=1}^{l} (a_{k}^{*}\alpha \ot \I_d)y (\beta a_{k} \ot \I_d)
	-(b_{k}^{*}\alpha \ot \I_d)y (\beta b_{k} \ot \I_d)\\
& =  \Big([a_{1}^{*}\alpha,\ldots,a_{l}^{*}\alpha] \ot \I_d\Big)
\begin{bmatrix}
y       &   0     &  \cdots   &   0   \\
 0      &  \ddots &           &  \vdots     \\
\vdots  &         &  \ddots   &   0   \\
 0      &  \cdots &     0     &   y  \\
\end{bmatrix}
\left(
\begin{bmatrix}
\beta a_{1} \\
\vdots\\
\vdots\\
\beta a_{l} \\
\end{bmatrix}
\ot \I_d \right)\\
&-\Big([b_{1}^{*}\alpha,\ldots,b_{l}^{*}\alpha] \ot \I_d\Big)
\begin{bmatrix}
y       &   0     &  \cdots   &   0   \\
 0      &  \ddots &           &  \vdots     \\
\vdots  &         &  \ddots   &   0   \\
 0      &  \cdots &     0     &   y  \\
\end{bmatrix}
 \left(
\begin{bmatrix}
\beta b_{1} \\
\vdots\\
\vdots\\
\beta b_{l} \\
\end{bmatrix}
\ot \I_d \right) \\
&= (\alpha_1 \ot \I_d)\cdot (\I_{l} \ot y) \cdot (\beta_1^*\ot \I_d) - (\alpha_2 \ot \I_d)\cdot (\I_{l} \ot y) \cdot (\beta_2^*\ot \I_d).
\end{align*}
The matrix $\I_{l} \ot y\in \M_l(\M_n(\M_d))$ is of norm less than 1. A simple computation shows that 
\begin{align*}
\MoveEqLeft
   \begin{bmatrix} 
(T \ot \Id_{\M_d})(x) & 0 \\ 
0 & 0 
\end{bmatrix}
=\begin{bmatrix} 
(\alpha_1 \ot \I_d)\cdot (\I_{l} \ot y) \cdot (\beta_1^*\ot \I_d) - (\alpha_2 \ot \I_d)\cdot (\I_{l} \ot y) \cdot (\beta_2^*\ot \I_d)  & 0 \\ 
0 & 0 
\end{bmatrix}\\
&=\left(\begin{bmatrix} 
\alpha_1 & - \alpha_2 \\ 
0 & 0 
\end{bmatrix} \ot \I_d\right)
\begin{bmatrix} 
\I_{l} \ot y & 0 \\ 
0 & \I_{l} \ot y
\end{bmatrix} 
\left(\begin{bmatrix} 
\beta_1^* & 0 \\ 
\beta_2^* & 0 
\end{bmatrix} \ot \I_d\right).
\end{align*}
On the other hand, we have
\begingroup
\allowdisplaybreaks
\begin{align*}
\MoveEqLeft
\left\| \begin{bmatrix} 
\alpha_1 & - \alpha_2 \\ 
0 & 0 
\end{bmatrix} 
\right\|_{S_{2n}^{2p}}
 = \left\| 
\begin{bmatrix} 
\alpha_1 & - \alpha_2 \\ 
0 & 0 
\end{bmatrix}^* \right\|_{S^{2p}_{2n}}
=\tr\left( \left( 
\begin{bmatrix} 
\alpha_1 & - \alpha_2 \\ 
0 & 0 
\end{bmatrix} 
\begin{bmatrix} 
\alpha_1^* & 0 \\ 
- \alpha_2^* & 0 
\end{bmatrix} \right)^p \right)^{\frac{1}{2p}}\\
&= \tr \left( 
\begin{bmatrix} 
\alpha_1 \alpha_1^* + \alpha_2 \alpha_2^* & 0 \\ 
0 & 0 \end{bmatrix}^p \right)^{\frac{1}{2p}} 
 =\tr\big( (\alpha_1 \alpha_1^* + \alpha_2 \alpha_2^*)^p \big)^{\frac{1}{2p}} \\
 &= \tr\bigg( \bigg( \sum_{k = 1}^l a_k^* \alpha \alpha^* a_k + b_k^* \alpha \alpha^* b_k  \bigg)^p \bigg)^{\frac{1}{2p}} 
= \bnorm{T_1(\alpha \alpha^*) + T_2(\alpha \alpha^*)}_{S^p_n}^{\frac12} \\
& \leq \norm{T_1 + T_2}_{S^p_m \to S^p_n}^{\frac12} \norm{\alpha \alpha^*}_{S^p_m}^{\frac12} = \norm{T_1 + T_2}_{S^p_m \to S^p_n}^{\frac12} \norm{\alpha}_{S^{2p}_m} \\
& \leq \norm{T_1 + T_2}_{S^p_m \to S^p_n}^{\frac12}.
\end{align*}
\endgroup
In the same way, it follows that $ 
\left\| \begin{bmatrix} 
\beta_1^* & 0 \\ 
\beta_2^* & 0 
\end{bmatrix} \right\|_{S_{2n}^{2p}} 
\leq \norm{T_1 + T_2}_{S^p_m \to S^p_n}^{\frac12}$. Using \eqref{Norm-Sp(Md)}, we infer that
$$
\bnorm{(T \ot \Id_{\M_d})(x)}_{S^p_n(\M_d)} 
= \left\| \begin{bmatrix} 
(T \ot \Id_{\M_d})(x) & 0 \\ 
0 & 0 
\end{bmatrix} \right\|_{S^p_{2n}(\M_d)} 
\leq \norm{T_1 + T_2}_{S^p_m \to S^p_n}^{\frac12} \norm{T_1 + T_2}_{S^p_m \to S^p_n}^{\frac12}.
$$
This yields $\norm{T \ot \Id_{\M_d}}_{S^p_m(\M_d) \to S^p_n(\M_d)} \leq \norm{T_1 + T_2}_{S^p_m \to S^p_n} \leq \norm{T}_{\dec} + \epsi$, hence $\norm{T}_{\reg} \leq \norm{T}_{\dec} + \epsi$. Passing $\epsi \to 0$ yields one of the desired estimates $\norm{T}_{\reg} \leq \norm{T}_{\dec}$.

%

Finally we shall show $\norm{T}_{\dec} \leq \norm{T}_{\reg}$. Assume that $\norm{T}_{\reg} \leq 1$. According to \cite[Theorem 5.12]{Pis7}, note that we have isometrically
\begin{equation}
\label{Produits-tensoriels}
	\CB(S^\infty_n) = \M_n \ot_h \M_n
\end{equation}
where $\ot_h$ denotes the Haagerup tensor product. Moreover, using the properties of this tensor product \cite[pages 95-97]{Pis5}, we obtain
\begin{align*}
\MoveEqLeft
  \M_n^\op \ot_h \M_n^\op  
		=(\mathrm{C}_n \ot_h \mathrm{R}_n)^\op \ot_h ( \mathrm{C}_n\ot_h \mathrm{R}_n)^\op
		=\mathrm{R}_n^\op \ot_h \mathrm{C}_n^\op \ot_h \mathrm{R}_n^\op \ot_h \mathrm{C}_n^\op\\
		&=\mathrm{C}_n \ot_h \mathrm{R}_n \ot_h \mathrm{C}_n \ot_h \mathrm{R}_n
		=\mathrm{C}_n \ot_h S^1_n \ot_h \mathrm{R}_n
		=\M_n(S^1_n)
		=\M_n \ot_{\textrm{min}} S^1_n
		=\CB(S^1_n). \nonumber
\end{align*}
We have $\gamma_\theta(T) \leq 1$ with $\theta = \frac1p$ and $\gamma_\theta$ defined in \cite[Theorem 8.5]{Pis13}, according to \cite[Corollary 3.3]{Pis4}. Then since $T$ is adjoint preserving, \cite[Corollary 8.7]{Pis13} yields that $\norm{T}_\dec \leq \norm{T_1 + T_2}_{S^p_m \to S^p_n} \leq 1$ where $T = T_1 - T_2$ and $T_1, \: T_2$ are completely positive mappings $\M_n \to \M_n$ given there. The proof of the theorem is complete.
\end{proof}

\subsection{Decomposable vs regular on approximately finite-dimensional algebras}
\label{Decomposable-approx}

In this section, we will extend by approximation Theorem \ref{prop-dec-reg-matrix-case} to approximately finite-dimensional von Neumann algebras. We start with two lemmas which show that, under suitable assumptions, the decomposability or the regularity of maps is preserved under a passage to the limit. 

\begin{lemma}
\label{lem-decomposable-weak-limit}
Let $M$ and $N$ be von Neumann algebras equipped with faithful normal semifinite traces. Suppose $1 \leq p \leq \infty$. Let $(T_\alpha)$ be a net of decomposable operators from $\L^p(M)$ into $\L^p(N)$ such that $\norm{T_\alpha}_{\dec,\L^p(M) \to \L^p(N)} \leq C$ for some constant $C$ which converges to some $T \co \L^p(M) \to \L^p(N)$ in the weak operator topology (in the point weak* topology of $\B(M,N)$ if $p=\infty$). Then $T$ is decomposable and $\norm{T}_{\dec,\L^p(M) \to \L^p(N)} \leq \liminf_\alpha \norm{T_\alpha}_{\dec,\L^p(M) \to \L^p(N)}$.
\end{lemma}


\begin{proof}
We assume first that $1 < p < \infty$. By Proposition \ref{Prop-dec-inf-atteint}, for any $\alpha$, there exist some maps $v_\alpha, w_\alpha \co \L^p(M) \to \L^p(N)$ such that the map $
\begin{bmatrix} 
v_\alpha & T_\alpha \\ 
T_\alpha^\circ & w_\alpha 
\end{bmatrix} \co S^p_2(\L^p(M)) \to S^p_2(\L^p(N))$ 
is completely positive with $\max\{\norm{v_\alpha},\norm{w_\alpha}\} = \norm{T_\alpha}_\dec \leq C$. Note that since $\L^p(N)$ is reflexive, the closed unit ball of the space $\B(\L^p(M),\L^p(N))$ of bounded operators in the weak operator topology is compact. Hence the bounded nets $(v_{\alpha})$ and $(w_\alpha)$ admit convergent subnets $(v_{\beta})$ and $(w_{\beta})$ in the weak operator topology which converge to some $v,w \in \B(\L^p(M),\L^p(N))$. Now, it is easy to see that
$$ 
\begin{bmatrix} 
v & T \\ 
T^\circ & w 
\end{bmatrix} 
= \lim_\beta 
\begin{bmatrix} 
v_{\beta} & T_{\beta} \\ 
T_{\beta}^\circ & w_{\beta}
\end{bmatrix}
$$
in the weak operator topology of $\B(S_2^p(\L^p(M)),S_2^p(\L^p(N)))$. By Lemma \ref{lem-completely-positive-weak-limit}, the operator on the left hand side is completely positive as a weak limit of completely positive mappings. Hence the operator $T$ is decomposable. Moreover, using the weak lower semicontinuity of the norm, we see that $\norm{v} \leq \liminf_\beta \norm{v_{\beta}} \leq \liminf_\beta \norm{T_\beta}_\dec$ and $\norm{w} \leq \liminf_\beta \norm{w_{\beta}} \leq \liminf_\beta \norm{T_\beta}_\dec$. Hence, we have $\norm{T}_{\dec} \leq \max\{\norm{v},\norm{w}\} \leq \liminf_\beta \norm{T_\beta}_\dec$.
By considering a priori only subnets $\beta$ of $\alpha$ such that $\lim_\beta \norm{T_\beta}_\dec = \liminf_\alpha \norm{T_\alpha}_\dec$ (see \cite[Exercise 2.55 (f)]{Meg1}), we finish the proof in the case $1 < p < \infty$.

Assume now that $p = \infty$. Then the Banach space $\B(M,N)$ is still a dual space, namely that of the projective tensor product $M  \hat{\ot} \L^1(N)$. Consequently, the bounded nets $(v_{\alpha})$ and $(w_\alpha)$ admit convergent subnets $(v_{\beta})$ and $(w_{\beta})$ which converge in the weak* topology of $\B(M,N)$ to some $v,w$, where $v_\beta,w_\beta$ are constructed as previously. Note that the weak* convergence implies the point weak* convergence and thus allows us to apply Lemma \ref{lem-completely-positive-weak-limit} and deduce that $\begin{bmatrix} 
v & T \\ 
T^\circ & w 
\end{bmatrix}\co \M_2(M) \to \M_2(N)$ is completely positive. Using the weak* lower semicontinuity of the norm, we infer that $\norm{v} \leq \liminf_\beta \norm{v_{\beta}}\leq \liminf_\beta \norm{T_\beta}_\dec$ and similarly  $\norm{w} \leq \liminf_\beta \norm{T_\beta}_\dec$ and thus $\norm{T}_{\dec} \leq \max\{\norm{v},\norm{w}\} \leq \liminf_\beta \norm{T_\beta}_\dec = \liminf_\alpha \norm{T_\alpha}_{\dec}$, again under suitable choices of subnets $\beta$ of $\alpha$.

Assume finally that $p = 1$. According to \eqref{Duality-dec}, we note that the case $p = \infty$ is applicable\footnote{\thefootnote. \label{footnote-14}If $X$ is a dual Banach space $X$ with predual $X_*$, it is well-known that the mapping $\B(X_*) \to \B_{\mathrm{w}^*}(X)$, $T \mapsto T^*$ is a weak operator-point weak* homeomorphism onto the space $\B_{\mathrm{w}^*}(X)$ of weak* continuous operators of $\B(X)$ and the point weak* topology and the weak* topology coincide on bounded sets by \cite[Lemma 7.2]{Pau}.} to $T_\alpha^*$ and $T^*$ and thus $\norm{T}_{\dec,\L^1(M) \to \L^1(N)} = \norm{T^*}_{\dec,N \to M} \leq \liminf_\alpha \|T_\alpha^*\|_{\dec,N \to M} = \liminf_\alpha \norm{T_\alpha}_{\dec,\L^1(M) \to \L^1(N)}$ where we used again \eqref{Duality-dec} in the last equality.
\end{proof}

\begin{lemma}
\label{Lemma-regular-strong-limit}
Let $M$ and $N$ be approximately finite-dimensional von Neumann algebras which are equipped with faithful normal semifinite traces. Suppose $1 < p < \infty$. Let $(T_\alpha)$ be a net of maps from $\L^p(M)$ into $\L^p(N)$ such that $\norm{T_\alpha}_{\reg,\L^p(M) \to \L^p(N)} \leq C$ for some constant $C$ which converges to some $T \co \L^p(M) \to \L^p(N)$ in the strong operator topology. Then the map $T$ is regular and $\norm{T}_{\reg,\L^p(M) \to \L^p(N)} \leq \liminf_\alpha \norm{T_\alpha}_{\reg,\L^p(M) \to \L^p(N)}$.
\end{lemma}

\begin{proof}
Let $E$ be an operator space. For any $x \in \L^p(M) \ot E$, an easy computation gives\footnote{\thefootnote. If $\sum_{k=1}^{n} x_k \ot y_k \in \L^p(M) \ot E$ then
\begin{align*}
\MoveEqLeft
  \norm{(T_\alpha \ot \Id_E)\Bigg(\sum_{k=1}^{n} x_k \ot y_k\Bigg)-(T \ot \Id_E)\Bigg(\sum_{k=1}^{n} x_k \ot y_k\Bigg)}_{\L^p(M,E)}    \\
		&=\norm{\sum_{k=1}^{n} T_\alpha(x_k) \ot y_k-\sum_{k=1}^{n} T(x_k) \ot y_k}_{\L^p(M,E)} \\
		&=\norm{\sum_{k=1}^{n} (T_\alpha(x_k)-T(x_k)) \ot y_k}_{\L^p(M,E)}
		\leq \sum_{k=1}^{n}\norm{T_\alpha(x_k)- T(x_k)}_{\L^p(M)} \norm{y_k}_{E}
		\xra[\alpha]{} 0.
\end{align*}
}
$$
\lim_\alpha (T_\alpha \ot \Id_E)(x) 
= (T \ot \Id_E)(x). 
$$
We deduce\footnote{\thefootnote. Let $X$ be a Banach space and $D$ a dense subset of $X$.  Let $(T_\alpha)$ be a bounded net of bounded linear operators in $\B(X)$. Suppose that, for each $x \in D$, the net $(T_\alpha(x))$  is convergent in $X$. By \cite[page 55]{DuS}, there exists a bounded linear operator $T \co X \to X$ such that $(T_\alpha)$ converges strongly to $T$.} that $T \ot \Id_E$ induces a bounded operator on $\L^p(M,E)$ and that the net $(T_\alpha \ot \Id_E)$ converges strongly to $T \ot \Id_E$. By the strong lower semicontinuity of the norm, we deduce that
$$
\norm{T \ot \Id_E}_{\L^p(M,E) \to \L^p(N,E)} 
\leq \liminf_\alpha \norm{T_\alpha \ot \Id_E}_{\L^p(M,E) \to \L^p(N,E)}
\leq \liminf_\alpha \norm{T_\alpha}_{\reg,\L^p(M) \to \L^p(N)}.
$$
Taking the supremum, we get the desired conclusion.
\end{proof}


\begin{thm}
\label{thm-dec=reg-hyperfinite}
Let $M$ and $N$ be approximately finite-dimensional von Neumann algebras which are equipped with faithful normal semifinite traces. Suppose $1 \leq p \leq \infty$. Let $T \co \L^p(M) \to \L^p(N)$ be a linear mapping. Then $T$ is regular if and only if $T$ is decomposable. In this case, we have 
$$
\norm{T}_{\dec,\L^p(M) \to \L^p(N)} = \norm{T}_{\reg,\L^p(M) \to \L^p(N)}.
$$
\end{thm}

\begin{proof}
The case $p=\infty$ is \cite[Lemma 5.4.3]{ER} and a straightforward generalization of \cite[Remark following Definition 2.1]{Pis4} since $\L^\infty(M,E)=M \ot_{\min} E$. The case $p=1$ is also true by duality using Lemma \ref{lem-op-mappings}, \eqref{Duality-dec} and \eqref{Duality-reg}. 

Let us now turn to the case $1 < p < \infty$. We denote by $\tau$ and $\sigma$ the traces of $M$ and $N$.

\noindent
\textit{Case 1: $M$ and $N$ are finite-dimensional.}
By \cite[Theorem 11.2]{Tak1} and \cite[proof of Proposition 7 page 109, Theorem 5 page 105, Corollary page 103]{Dix2}, there exist $m_1,\ldots,m_K,n_1,\ldots,n_L \in \N$ and $\lambda_1,\ldots,\lambda_K, \mu_1,\ldots,\mu_L \in (0,\infty)$ such that $(M,\tau) = (\M_{m_1} \oplus \cdots \oplus \M_{m_K}, \lambda_1 \tr_{m_1} \oplus \cdots \oplus \lambda_K \tr_{m_K})$ and $(N,\sigma) = (\M_{n_1} \oplus \cdots \oplus \M_{n_L}, \mu_1 \tr_{n_1} \oplus \cdots \oplus \mu_L \tr_{n_L})$.

\vspace{0.2cm}

\noindent
\textit{Case 1.1: All $\lambda_k$ and $\mu_l$ belong to $\N$.}
Then let $m = \sum_{k = 1}^K \lambda_k m_k$ and $n = \sum_{l = 1}^L \mu_l n_l$. Let further $J \co M \to \M_m$ be the normal unital trace preserving $*$-homomorphism defined by
\[
J(x_1 \oplus \cdots \oplus x_K) 
= \begin{bmatrix} x_1 & & & & & &  \\ 
& \ddots & & & & & \\ & & x_1 & & &0 & \\ 
& & & \ddots & & & \\ & & & & x_K & & \\ 
& 0& & & & \ddots & \\ 
 & & & & & & x_K \end{bmatrix} ,
\]
where $x_k$ appears $\lambda_k$ times on the diagonal, $k = 1,\ldots, K$. Let moreover $\E \co \M_m \to M$ be the associated conditional expectation. Moreover, we introduce similar maps $J' \co N \to \M_n$ and $\E' \co \M_n \to N$. We denote by the same symbols the induced maps on the associated $\L^p$-spaces.


Lemma \ref{lem-sufficient-condition-reg} is applicable for both $J'$ and $\E$ and we obtain the estimates $\norm{J'}_{\reg,\L^p(N) \to S^p_n} \leq 1$ and $\norm{\E}_{\reg,S^p_m \to \L^p(M)} \leq 1$. Moreover, by Proposition \ref{quest-cp-versus-dec1}, we also infer that $\norm{J}_{\dec,\L^p(M) \to S^p_m} \leq 1$ and $\norm{\E'}_{\dec,S^p_n \to \L^p(N)} \leq 1$. Suppose that $T \co \L^p(M) \to \L^p(N)$ is regular. By Theorem \ref{prop-dec-reg-matrix-case} applied to $J'T\E \co S^p_m \to S^p_n$ together with \eqref{Composition-dec} and \eqref{Composition-reg}, we obtain that $T=\E'(J'T\E)J$ is decomposable and that
\begin{align*}
\MoveEqLeft
    \norm{T}_{\dec,\L^p(M) \to \L^p(N)}
= \bnorm{\E'J' T \E J}_\dec 
\leq \norm{\E'}_\dec \norm{J'T\E}_\dec \norm{J}_\dec \\
&\leq \norm{J'T\E}_\reg \leq \norm{J'}_\reg \norm{T}_\reg \norm{\E}_\reg 
\leq \norm{T}_{\reg,\L^p(M) \to \L^p(N)}.
\end{align*}
Let $T \co \L^p(M) \to \L^p(N)$ be a decomposable map. In a similar manner, we obtain the inequalities $\norm{J}_\reg,\norm{\E'}_\reg,\norm{J'}_\dec,\norm{\E}_\dec \leq 1$ and that $T$ is regular and we have
\begin{align*}
\MoveEqLeft
\norm{T}_{\reg,\L^p(M) \to \L^p(N)}
= \norm{\E' J' T \E J}_\reg \leq \norm{\E'}_\reg \norm{J'T\E}_\reg \norm{J}_\reg \\
&\leq \norm{J'T\E}_\dec \leq \norm{J'}_\dec \norm{T}_\dec \norm{\E}_\dec 
\leq \norm{T}_{\dec,\L^p(M) \to \L^p(N)}.
\end{align*}

\noindent
\textit{Case 1.2: All $\lambda_k$ and $\mu_l$ belong to $\Q_{+}$.}
Then there exists a common denominator of the $\lambda_k$'s and the $\mu_l$'s, that is, there exists $t \in \N$ such that $\lambda_k = \frac{\lambda_k'}{t}$, $\mu_l =\frac{\mu_l'}{t}$ for some integers $\lambda_k'$ and $\mu_l'$. Since we have $\norm{x}_{\L^p(M_1,t \tau_1)} = t^{\frac1p}\norm{x}_{\L^p(M_1,\tau_1)}$ for any semifinite von Neumann algebra $(M_1,\tau_1),$ it is easy to deduce that 
$$
\norm{T}_{\dec,\L^p(M,t\tau) \to \L^p(N,t\sigma)} 
= \norm{T}_{\dec,\L^p(M,\tau) \to \L^p(N,\sigma)}
$$ 
and also that $T \co \L^p(M,t\tau) \to \L^p(N,t\sigma)$ is regular if and only if $T \co \L^p(M,\tau) \to \L^p(N,\sigma)$ is regular with equal regular norms in this case. Thus, Case 1.2 follows from Case 1.1. 

\vspace{0.2cm}

\noindent

\textit{Case 1.3: $\lambda_k,\:\mu_l \in (0,\infty)$.}
For $\epsi > 0$, let $\lambda_{k,\epsi}$, $\mu_{l,\epsi} \in \Q_+$ be $\epsi$-close to $\lambda_k$ and $\mu_l$ in the sense that $(1+\epsi)^{-1} \lambda_k \leq \lambda_{k,\epsi} \leq (1+\epsi) \lambda_k$, and similarly for $\mu_l,\mu_{l,\epsi}$. We introduce the trace $\tau_\epsi = \lambda_{1,\epsi} \tr_{m_1} \oplus \cdots \oplus \lambda_{K,\epsi} \tr_{m_K}$ on $M=\M_{m_1} \oplus \cdots \oplus \M_{m_K}$. Consider the (non-isometric) identity mapping $\Id_{M}^\epsi \co \L^p(M,\tau) \to \L^p(M,\tau_\epsi)$. Note that for any element $x = x_1 \oplus \ldots \oplus x_K$ of $\L^p(M,\tau)$, the definition of multiplication and adjoint in the sum space $\M_{m_1} \oplus \cdots \oplus \M_{m_K}$ yields immediately that $|x|^p = |x_1|^p \oplus \ldots \oplus |x_K|^p$. Thus, $\norm{x}_{\L^p(M,\tau)}^p = \tau(|x|^p) = \sum_{k=1}^K \lambda_k \tr_{m_k}(|x_k|^p)$. By the same argument, $\|x\|_{\L^p(M,\tau_\epsi)} = \sum_{k=1}^K \lambda_{k,\epsi} \tr_{m_k}(|x_k|^p)$. Thus,
\begin{align*}
\MoveEqLeft
\norm{ \Id_{M}^\epsi}_{\L^p(M,\tau) \to \L^p(M,\tau_\epsi)}^p 
= \sup_{x \in \L^p(M,\tau) \backslash \{ 0 \}} \frac{ \sum_{k=1}^K \lambda_{k,\epsi} \tr_{m_k}(|x_k|^p) }{ \sum_{k=1}^K \lambda_k \tr_{m_k}(|x_k|^p) } \\
& \leq \sup_{x \in \L^p(M,\tau) \backslash \{ 0 \}} \frac{ \sum_{k=1}^K (1+\epsi)\lambda_{k} \tr_{m_k}(|x_k|^p) }{ \sum_{k=1}^K \lambda_k \tr_{m_k}(|x_k|^p) }
 = 1 + \epsi.
\end{align*}
In the same manner, using $(1+\epsi)^{-1} \lambda_k \leq \lambda_{k,\epsi}$, one obtains that $\norm{(\Id_{M}^\epsi)^{-1}}^p_{\cb,\L^p(M,\tau_\epsi) \to \L^p(M,\tau)} \leq 1 + \epsi$. We infer that $\|\Id_M^\epsi\|_{\cb},\norm{(\Id_M^\epsi)^{-1}}_{\cb} \to 1$ as $\epsi \to 0$. In the case $p = \infty$, this convergence also holds, since $\norm{x}_{\L^\infty(M,\tau^\epsi)} = \norm{x}_{\L^\infty(M,\tau)}$. We also define the trace $\sigma_\epsi = \mu_{1,\epsi} \tr_{n_1} \oplus \cdots \oplus \mu_{L,\epsi} \tr_{m_L}$ on the algebra $N$. Moreover, we also have a map $\Id_N^\epsi \co \L^p(N,\sigma) \to \L^p(N,\sigma^\epsi)$ and $\norm{\Id_{N}^\epsi}_{\cb},\norm{(\Id_{N}^\epsi)^{-1}}_{\cb}$ go to 1 when $\epsi$ approaches $0$. Since $\Id_{M}^\epsi$, $\Id_{N}^\epsi$ and their inverses are completely positive (since they are identity mappings and complete positivity is independent of the trace), by Proposition \ref{quest-cp-versus-dec1}, their decomposable norms approach $1$  when $\epsi$ approaches $0$. Moreover, interpolating between $p = 1$ and $p = \infty$, using Lemma \ref{lem-sufficient-condition-reg}, we also infer that their regular norms approach $1$ as $\epsi$ goes to $0$. Suppose that $T \co \L^p(M,\tau) \to \L^p(N,\sigma)$ is regular. Using Case 1.2 with the map $\Id_{N}^\epsi T (\Id_{M}^\epsi)^{-1} \co \L^p(M,\tau_\epsi)\to \L^p(N,\sigma_\epsi)$, \eqref{Composition-dec} and \eqref{Composition-reg}, we see that
\begin{align*}
\MoveEqLeft
 \norm{T}_{\dec,\L^p(M,\tau) \to \L^p(N,\sigma)} 
 = \bnorm{(\Id_{N}^\epsi)^{-1} \Id_{N}^\epsi T (\Id_{M}^\epsi)^{-1} \Id_{M}^\epsi}_{\dec,\L^p(M,\tau) \to \L^p(N,\sigma)}\\ 
 &\leq \norm{(\Id_{N}^\epsi)^{-1}}_\dec \norm{\Id_{N}^\epsi T (\Id_{M}^\epsi)^{-1}}_\dec \norm{\Id_{M}^\epsi}_\dec \\
& = \norm{(\Id_{N}^\epsi)^{-1}}_\dec \norm{\Id_{N}^\epsi T (\Id_{M}^\epsi)^{-1}}_\reg \norm{\Id_{M}^\epsi}_\dec \\
 &\leq \norm{(\Id_{N}^\epsi)^{-1}}_\dec \norm{\Id_{N}^\epsi}_\reg \norm{T}_\reg \norm{(\Id_{M}^\epsi)^{-1}}_\reg \norm{\Id_{M}^\epsi}_\dec.
\end{align*}
Going to the limit, we obtain $\norm{T}_\dec \leq \norm{T}_\reg$. In the same vein, one shows that any map $T \co \L^p(M,\tau) \to \L^p(N,\sigma)$ is regular and that we have $\norm{T}_\reg \leq \norm{T}_\dec$. The proof of Case 1.3, and thus of Case 1, is complete.

\vspace{0.2cm}

\noindent
\textit{Case 2: $M$ and $N$ are approximately finite-dimensional and finite.}
In this case \cite[page 291]{Bla}, $M = \ovl{\bigcup_{\alpha} M_\alpha}^{\textrm{w}^*}$ and $N = \ovl{\bigcup_{\beta} N_\beta}^{\textrm{w}^*}$ where $(M_\alpha)$ and $(N_\beta)$ are nets directed by inclusion of finite dimensional unital $*$-subalgebras (as in Case 1). Moreover, we denote by $J_\alpha \co M_\alpha \to M$, $J_\beta' \co N_\beta \to N$ the canonical unital $*$-homomorphisms and by $\E_\alpha \co M \to M_\alpha$ and $\E_\beta' \co N \to N_\beta$ the associated conditional expectations given by \cite[Corollary 10.6]{Str} since the traces are finite. All these maps induce completely contractive and completely positive maps on the associated $\L^p$-spaces denoted by the same notations such that\footnote{\thefootnote. 
Recall that $\cup_{\alpha} \L^p(M_\alpha)$ is dense in $\L^p(M)$. Let $x \in \L^p(M)$ and $\epsi >0$. There exists $\alpha_0$ and $y \in \L^p(M_{\alpha_0})$ such that $\norm{x-y}_{\L^p(M)} \leq \epsi$. Hence for any $\alpha \geq \alpha_0$, since $y \in \L^p(M_{\alpha})$,  we have
\begin{align*}
\MoveEqLeft
  \norm{x-J_\alpha \E_\alpha(x)}_{\L^p(M)}  
		\leq\norm{x-y}_{\L^p(M)}+\norm{y-J_\alpha \E_\alpha(x)}_{\L^p(M)}
		\leq \epsi+\norm{J_\alpha\E_\alpha(y-x)}_{\L^p(M)}
		\leq 2\epsi.
\end{align*}}
\begin{equation}
\label{equ-Lp-convergence-approx-finite-dim}
\lim_\alpha J_\alpha \E_\alpha(x) = x 
\quad \text{ and } \quad
\lim_\beta J_\beta' \E_\beta'(y) = y
\end{equation}
(for the $\L^p$-norm) for any $x \in \L^p(M)$ and any $y \in \L^p(N)$.
Let $T \co \L^p(M) \to \L^p(N)$ be a bounded map. The net\footnote{\thefootnote. The index set $A \times B$ is directed by letting $(\alpha,\beta) \leq (\alpha',\beta')$ if $\alpha \leq \alpha'$ and $\beta \leq \beta'$.} $(J_\beta' \E_\beta',J_\alpha \E_\alpha)_{(\alpha,\beta)}$ of $\B(\L^p(N)) \times \B(\L^p(M))$ is obviously convergent to $(\Id_{\L^p(N)},\Id_{\L^p(M)})$ where each factor is equipped with the strong topology. Using the strong continuity of the product on bounded sets, we infer that the net $(J_\beta'\E_\beta' T J_\alpha\E_\alpha)$ converges strongly to $T$. Suppose that $T$ is decomposable. Using Case 1 with the operator $\E_\beta' T J_\alpha \co \L^p(M_\alpha) \to \L^p(N_\beta)$, we deduce that $T$ is regular and that, using \eqref{Composition-dec} and \eqref{Composition-reg}
\begin{align*}
\MoveEqLeft
\bnorm{T}_{\reg,\L^p(M) \to \L^p(N)} 
 \leq \liminf_{\alpha,\beta} \norm{J_\beta'\E_\beta' T J_\alpha \E_\alpha}_{\reg}  
 \leq \liminf_{\alpha,\beta} \norm{J_\beta'}_\reg \norm{\E_\beta' T J_\alpha}_\reg \|\E_\alpha'\|_\reg \\
 &\leq \liminf_{\alpha,\beta} \norm{\E_\beta' T J_\alpha}_\dec \leq
\liminf_{\alpha,\beta} \norm{\E_\beta'}_{\dec} \norm{T}_\dec \norm{J_\alpha}_{\dec}
\leq \norm{T}_{\dec,\L^p(M) \to \L^p(N)}.
\end{align*}
For the converse inequality, suppose that the map $T \co \L^p(M) \to \L^p(N)$ is regular. Since $T = \lim_{\alpha,\beta} J_\beta' \E_\beta' T J_\alpha \E_\alpha$ is the strong, hence weak, limit of decomposable operators, hence decomposable by Proposition \ref{lem-decomposable-weak-limit}, we obtain, using again \eqref{Composition-dec} and \eqref{Composition-reg},
\begin{align*}
\MoveEqLeft
\norm{T}_{\dec,\L^p(M) \to \L^p(N)} 
\leq \liminf_{\alpha,\beta} \norm{J_\beta' \E_\beta' T J_\alpha \E_\alpha}_\dec 
\leq \liminf_{\alpha,\beta} \norm{J_\beta'}_\dec \|\E_\beta' T J_\alpha \|_\dec \|\E_\alpha' \|_\dec \\
& \leq \liminf_{\alpha,\beta} \norm{\E_\beta' T J_\alpha}_\reg 
\leq \liminf_{\alpha,\beta} \norm{\E_\beta'}_\reg \norm{T}_\reg \norm{J_\alpha}_{\reg} 
 \leq \norm{T}_{\reg,\L^p(M) \to \L^p(N)}.
\end{align*}
Thus, Case 2 is proved.

\vspace{0.2cm}

\noindent
\textit{Case 3: $M$ and $N$ are general approximately finite-dimensional semifinite von Neumann algebras.}
By \cite[page 57]{Sun}, there exist an increasing net of projections $(e_i)$ which is strongly convergent to $1$ with $\tau(e_i)<\infty$ for any $i$. We set $M_i \overset{\mathrm{def}}= e_iMe_i$. The trace $\tau|_{M_i}$ is obviously finite. Moreover, it is well-known\footnote{\thefootnote. This observation relies on the equivalence between ``injective'' and ``approximately finite-dimensional''.} that $M_i$ is approximately finite-dimensional. We conclude that $M_i$ is a von Neumann algebra satisfying the properties of Case 2. We also introduce the completely positive and completely contractive adjoint preserving normal map $Q_i \co M \to M_i$, $x \mapsto e_ixe_i$ and the canonical inclusion map $J_i \co M_i \to M$. We do the same construction on $N$ and obtain some maps $Q_j' \co N \to N_j$ and $J_j' \co N_j \to N$. All these maps induce completely positive and completely contractive maps on all $\L^p$ levels, $1 \leq p \leq \infty$. Moreover, for any $1 \leq p < \infty$ and any $x \in \L^p(M)$ we have\footnote{\thefootnote. Since the product of strongly convergent bounded nets of bounded operators on $\L^p(M)$ define a strongly convergent net, it suffices to prove that the net $(e_ix)$ converges to $x$ in $\L^p(M)$. Now using the GNS representation $\pi \co M \to \B(\L^2(M))$ and \cite[Corollary 7.1.16]{KaR}, we deduce that for any $x \in \L^2(M)$, the net $(e_ix)$ converges to $x$ in $\L^2(M)$. Using interpolation between 2 and $\infty$, we obtain the convergence for $2 < p <\infty$. For the case $1 \leq p<2$, it suffices to write an element $x \in \L^p(M)$ as $x=yz$ with $y,z \in \L^{2p}(M)$ and use H\"older inequality.} $x =\lim_{i} e_ixe_i =\lim_{i} J_iQ_i(x)$ and similarly $y = \lim_{j} J_j'Q_j'(y)$ for any $y \in \L^p(N)$. We conclude by the same arguments as in Case 2.
\end{proof}

\begin{remark}\normalfont
Using Proposition \ref{prop-linearcp-imply-decomposable}, this theorem also shows that the space of regular operators between $\L^p(M)$ and $\L^p(N)$ is precisely the span of the completely positive maps from $\L^p(M)$ into $\L^p(N)$. This assertion is alluded in \cite[Theorem 3.7]{Pis4} and proved\footnote{\thefootnote. The proof of \cite[Theorem 8.8]{Pis13} for Schatten spaces does not generalize in a straightforward manner  to the case of noncommutative $\L^p$-spaces. Indeed, the equality \eqref{Produits-tensoriels} is not true with a von Neumann algebra $M$ instead of $\M_n$. For example, by \cite[page 97]{Pis7}, the space $\ell^\infty_n \ot_h \ell^\infty_n$ is isometric to the space $\mathfrak{M}_n^\infty$ of Schur multipliers on $\M_n$ and the space $\CB(\ell^\infty_n)$ is isometric to $\B(\ell^\infty_n)$ by \cite[Proposition 2.2.6]{ER} and it is easy to see that $\mathfrak{M}_n^\infty$ is not isometric to $\B(\ell^\infty_n)$.} in \cite[Lemma 2.3]{Pis4} and \cite[Theorem 8.8]{Pis13} for $\L^p(M)=\L^p(N)=S^p$.
\end{remark}

With the same method, we can prove the particular case of Theorem \ref{quest-cp-versus-cb}. Using the same notations, we only indicate the changes.
 
\begin{thm}
\label{thm-norm=normcb-hyperfinite}
Let $M$ and $N$ be approximately finite-dimensional von Neumann algebras which are equipped with faithful normal semifinite traces. Suppose $1 \leq p \leq \infty$. Let $T \co \L^p(M) \to \L^p(N)$ be a completely positive map. Then $T$ is completely bounded and we have 
$$
\norm{T}_{\L^p(M) \to \L^p(N)}
=\norm{T}_{\cb,\L^p(M) \to \L^p(N)}.
$$
\end{thm}

\begin{proof}

\noindent
\textit{Case 1: $M$ and $N$ are finite-dimensional.}
Then as explained in the proof of Theorem \ref{thm-dec=reg-hyperfinite}, we can write $(M,\tau) = (\M_{m_1} \oplus \cdots \oplus \M_{m_K}, \lambda_1 \tr_{m_1} \oplus \cdots \oplus \lambda_k \tr_{m_K})$ and $(N,\sigma) = (\M_{n_1} \oplus \cdots \oplus \M_{n_L}, \mu_1 \tr_{n_1} \oplus \cdots \oplus \mu_L \tr_{n_L})$.

\noindent
\textit{Case 1.1: All $\lambda_k$ and $\mu_l$ belong to $\N$.}
We thus have, as in the proof of Theorem \ref{thm-dec=reg-hyperfinite}, unital trace preserving $*$-homomorphisms $J \co M \to \M_m$ and $J' \co N \to \M_n$ as well as associated conditional expectations $\E \co \M_m \to M$ and $\E' \co \M_n \to N$. Suppose that $T \co \L^p(M) \to \L^p(N)$ is completely positive. By a straightforward extension of \cite[Proposition 2.2 and Lemma 2.3]{Pis4} applied to $J'T\E \co S^p_m \to S^p_n$, we obtain that $T=\E'(J'T\E)J$ is completely bounded and that
\begin{align*}
\MoveEqLeft
\norm{T}_{\cb,\L^p(M) \to \L^p(N)}
=\bnorm{\E'J' T \E J}_\cb 
\leq \norm{\E'}_\cb \norm{J'T\E}_\cb \norm{J}_\cb 
\leq \norm{J'T\E} \leq \norm{J'} \norm{T} \norm{\E} 
\leq \norm{T}.
\end{align*}

\noindent
\textit{Case 1.2: All $\lambda_k$ and $\mu_l$ belong to $\Q_{+}.$}
It is easy to prove that $T \co \L^p(M,t\tau) \to \L^p(N,t\sigma)$ is bounded if and only if $T \co \L^p(M,\tau) \to \L^p(N,\sigma)$ is bounded with equal norms a similar result for the complete boundedness. Thus, Case 1.2 follows from Case 1.1. 

\noindent
\textit{Case 1.3: $\lambda_k,\:\mu_l \in (0,\infty).$}
Suppose that $T \co \L^p(M,\tau) \to \L^p(N,\sigma)$ is completely positive. Using Case 1.2 with the map $\Id_{N}^\epsi T (\Id_{M}^\epsi)^{-1} \co \L^p(M,\tau_\epsi)\to \L^p(N,\sigma_\epsi)$, we see that $T$ is completely bounded and that
\begin{align*}
\MoveEqLeft
 \norm{T}_{\cb,\L^p(M,\tau) \to \L^p(N,\sigma)} 
 = \bnorm{(\Id_{N}^\epsi)^{-1} \Id_{N}^\epsi T (\Id_{M}^\epsi)^{-1} \Id_{M}^\epsi}_{\cb,\L^p(M,\tau) \to \L^p(N,\sigma)}\\ 
 &\leq \norm{(\Id_{N}^\epsi)^{-1}}_\cb \norm{\Id_{N}^\epsi T (\Id_{M}^\epsi)^{-1}}_\cb \norm{\Id_{M}^\epsi}_\cb 
 = \norm{(\Id_{N}^\epsi)^{-1}}_\cb \norm{\Id_{N}^\epsi T (\Id_{M}^\epsi)^{-1}} \norm{\Id_{M}^\epsi}_\cb \\
 &\leq \norm{(\Id_{N}^\epsi)^{-1}}_\cb \norm{\Id_{N}^\epsi} \norm{T} \norm{(\Id_{M}^\epsi)^{-1}} \norm{\Id_{M}^\epsi}_\cb.
\end{align*}
Going to the limit, we obtain $\norm{T}_{\cb,\L^p(M) \to \L^p(N)} \leq \norm{T}_{\L^p(M) \to \L^p(N)}$. Thus Case 1 is complete.

\noindent
\textit{Case 2: $M$ and $N$ are approximately finite-dimensional and finite.} Let $T \co \L^p(M) \to \L^p(N)$ be a completely positive map. The net $(J_\beta'\E_\beta' T J_\alpha\E_\alpha)$ converges strongly to $T$. Using Case 1 with the operator $\E_\beta' T J_\alpha \co \L^p(M_\alpha) \to \L^p(N_\beta)$ and \cite[Theorem 7.4]{Pau} we deduce that $T$ is completely bounded and that
\begin{align*}
\MoveEqLeft
\bnorm{T}_{\cb,\L^p(M) \to \L^p(N)} 
 \leq \liminf_{\alpha,\beta} \norm{J_\beta'\E_\beta' T J_\alpha \E_\alpha}_{\cb}  
 \leq \liminf_{\alpha,\beta} \norm{J_\beta'}_\cb \norm{\E_\beta' T J_\alpha}_\cb \|\E_\alpha\|_\cb \\
 &\leq \liminf_{\alpha,\beta} \norm{\E_\beta' T J_\alpha} 
\leq \liminf_{\alpha,\beta} \norm{\E_\beta'} \norm{T} \norm{J_\alpha}
\leq \norm{T}_{\L^p(M) \to \L^p(N)}.
\end{align*}
Thus, Case 2 is proved. The Case 3 is similar to the Case 2.
\end{proof}

\subsection{Modulus of regular operators vs 2x2 matrix of decomposable operators}
\label{Modulus}

For any regular operator $T \co \L^p(\Omega) \to \L^p(\Omega')$ on classical $\L^p$-spaces, it is well-known that $\bnorm{|T|}_{\L^p(\Omega) \to \L^p(\Omega')} = \norm{T}_{\reg, \L^p(\Omega) \to \L^p(\Omega')}$, see e.g. \cite[Proposition 1.3.6]{MeN}. We recall that the modulus of a regular operator $T$ between real-valued $\L^p$-spaces is given by $|T| \ov{\mathrm{def}}{=} -T \vee T$, in the sense that $|T|$ is the supremum of the set $\{-T,T\}$ in $\B(\L^p(\Omega),\L^p(\Omega'))$, see \cite[page 229]{Sch1}. For any positive $f \in \L^p(\Omega)$, we have $|T|(f) = \sup\{|T(g)| : \: |g|\leq f\}$, see \cite[Theorem 1.3.2]{MeN} and \cite[Proposition 2.2.6]{MeN} in the case of complex-valued $\L^p$-spaces.


\begin{thm}
\label{prop-reg=decomp}
Let $\Omega$ and $\Omega'$ be (localizable) measure spaces. Suppose $1 \leq p < \infty$ (see Remark \ref{Rem-w-continuous-necessary} for the case $p =\infty$). Let $T \co \L^p(\Omega) \to \L^p(\Omega')$ be a regular operator. Then the map $\Phi 
=\begin{bmatrix}
|T| & T \\ 
T^\circ & |T| 
\end{bmatrix}
\co S^p_2(\L^p(\Omega)) \to S^p_2(\L^p(\Omega'))$ is completely positive, i.e. the infimum of \eqref{Matrice-2-2-Phi} is attained with $v_1=v_2=|T|$.
\end{thm}

\begin{proof}
%
We say that a finite collection $\alpha=\{A_1,\ldots,A_{n_\alpha}\}$ of disjoint measurable subsets of $\Omega$ with finite measures is a semipartition of $\Omega$. We introduce a preorder on the set $\mathcal{A}$ of semipartitions of $\Omega$ by letting $\alpha \leq \alpha'$ if each set in $\alpha$ is a union of some sets in $\alpha'$. It is not difficult to prove that $\mathcal{A}$ is a directed set. For any $\alpha \in \mathcal{A}$, we denote by $\{A_1,\ldots,A_{n_\alpha}\}$ the elements of $\alpha$ of measure $>0$. Similarly, we introduce the set $\mathcal{B}$ of semipartitions of $\Omega'$. It is not difficult to see\footnote{\thefootnote. Since the functions $1_{A_j}$ are disjoint, for any complex numbers $a_1,\ldots,a_{n_\alpha}$, we have
\begin{align*}
\MoveEqLeft
\norm{\sum_{j=1}^{n_\alpha} \frac{a_j}{\mu(A_j)^{\frac1p}}1_{A_j}}_{\L^p(\Omega)}  
		=\Bigg(\sum_{j=1}^{n_\alpha} \norm{\frac{a_j}{\mu(A_j)^{\frac1p}}1_{A_j}}_{\L^p(\Omega)}^p\Bigg)^{\frac{1}{p}}
		=\Bigg(\sum_{j=1}^{n_\alpha}\frac{|a_j|^p}{\mu(A_j)}\norm{1_{A_j}}_{\L^p(\Omega)}^p\Bigg)^{\frac{1}{p}}\\
		&=\Bigg(\sum_{j=1}^{n_\alpha} |a_j|^p\Bigg)^{\frac{1}{p}}
		=\norm{\sum_{j=1}^{n_\alpha} a_j e_j}_{\ell^p_{n_\alpha}}.
\end{align*}} that the operator $\ell^p_{n_\alpha} \to \mathrm{span} \{1_{A_1},\ldots,1_{A_{n_\alpha}} \}$, $e_j \mapsto \frac{1}{\mu(A_j)^{\frac1p}} 1_{A_j}$ is a positive isometric isomorphism onto the subspace $\mathrm{span} \{1_{A_1},\ldots,1_{A_{n_\alpha}}\}$ of $\L^p(\Omega)$. By composition with the canonical identification of $\mathrm{span} \{1_{A_1},\ldots,1_{A_{n_\alpha}}\}$ in $\L^p(\Omega)$, we obtain a positive isometric embedding $J_\alpha \co \ell^p_{n_\alpha} \to \L^p(\Omega)$. We equally define the average operator $\mathcal{P}_\alpha \co \L^p(\Omega) \to \ell^p_{n_\alpha}$ by
$$
\mathcal{P}_\alpha(f) 
\ov{\mathrm{def}}{=} \sum_{j=1}^{n_\alpha} \bigg(\frac{1}{\mu(A_j)^{1-\frac1p}}\int_{A_j} f \d\mu \bigg) e_j,
\quad f \in \L^p(\Omega).
$$ 
We need the following folklore lemma.

\begin{lemma}
\label{Lemma-average}
Suppose $1 \leq p < \infty$. 
\begin{enumerate} 
\item For any $\alpha \in \mathcal{A}$, the map $\mathcal{P}_\alpha$ is positive and contractive.

\item For any $f \in \L^p(\Omega)$, we have $\lim_{\alpha} J_{\alpha}\mathcal{P}_{\alpha}(f)=f$.
\end{enumerate}
\end{lemma}

\begin{proof}
1. The positivity is obvious. Using Jensen's inequality, it is elementary to check the contractivity.

2. Since $\norm{J_\alpha \mathcal{P}_\alpha}_{\L^p(\Omega) \to \L^p(\Omega)}$ is uniformly bounded by 1, by \cite[III 17.4, Proposition 5]{Bou5} it suffices to show this for $f$ in the dense class of integrable simple functions constructed with subsets of measure $>0$. So let $f$ be such a function, say with respect to some semipartition $\alpha_f$. For any $\alpha \in \mathcal{A}$ which refines $\alpha_f$, it is easy to see that $J_{\alpha}\mathcal{P}_\alpha (f) = f$. Hence, for this $f$, the assertion is true. 
\end{proof}

The net\footnote{\thefootnote. The index set $A \times B$ is directed by letting $(\alpha,\beta) \leq (\alpha',\beta')$ if $\alpha \leq \alpha'$ and $\beta \leq \beta'$.} $\left(\begin{bmatrix}  
J_{\beta} \mathcal{P}_{\beta} &  J_{\beta} \mathcal{P}_{\beta} \\ 
 J_{\beta} \mathcal{P}_{\beta}  &  J_{\beta} \mathcal{P}_{\beta}
\end{bmatrix} ,\begin{bmatrix}  
|T| J_{\alpha} \mathcal{P}_{\alpha} &  T J_{\alpha}\mathcal{P}_{\alpha} \\ 
 T^\circ J_{\alpha}\mathcal{P}_{\alpha} & |T| J_{\alpha}\mathcal{P}_{\alpha}  
\end{bmatrix} \right)_{(\alpha,\beta)}$ of the product $\B(S^p_2(\L^p(\Omega'))) \times \B(S^p_2(\L^p(\Omega)),S^p_2(\L^p(\Omega')))$ is obviously convergent to $\left(\Id_{S^p_2(\L^p(\Omega'))},\begin{bmatrix} 
|T|  &  T  \\ 
 T^\circ  & |T| 
\end{bmatrix} \right)$ where each factor is equipped with the strong operator topology. Using the strong continuity of the product on bounded sets (see \cite[Proposition C.19]{EFHN}),  we infer that the net 
\[
\left(\begin{bmatrix}  
J_{\beta} \mathcal{P}_{\beta} |T| J_{\alpha} \mathcal{P}_{\alpha} &  J_{\beta} \mathcal{P}_{\beta} T J_{\alpha}\mathcal{P}_{\alpha} \\ 
 J_{\beta} \mathcal{P}_{\beta} T^\circ J_{\alpha}\mathcal{P}_{\alpha} &  J_{\beta} \mathcal{P}_{\beta} |T| J_{\alpha}\mathcal{P}_{\alpha}  
\end{bmatrix} \right)_{(\alpha,\beta)}
\]
converges strongly to the map 
$
\begin{bmatrix} 
|T| & T \\ 
T^\circ & |T| 
\end{bmatrix} 
\co S^p_2(\L^p(\Omega)) \to S^p_2(\L^p(\Omega'))$. By Lemma \ref{lem-completely-positive-weak-limit}, since we have the equality
$$
\begin{bmatrix}  
J_{\beta} \mathcal{P}_{\beta} |T| J_{\alpha} P_{\alpha} &  J_{\beta} \mathcal{P}_{\beta} T J_{\alpha}\mathcal{P}_{\alpha} \\ 
 J_{\beta} \mathcal{P}_{\beta} T^\circ J_{\alpha}P_{\alpha} &  J_{\beta} \mathcal{P}_{\beta} |T| J_{\alpha}\mathcal{P}_{\alpha}  
\end{bmatrix}
=(\Id_{S^p_2} \ot J_{\beta})
\circ 
\begin{bmatrix} \mathcal{P}_{\beta} |T| J_{\alpha} & \mathcal{P}_{\beta} T J_{\alpha} \\ 
\mathcal{P}_{\beta} T^\circ J_{\alpha} & \mathcal{P}_{\beta} |T| J_{\alpha}  
\end{bmatrix} 
\circ
(\Id_{S^p_2} \ot \mathcal{P}_{\alpha}),
$$
it suffices to show that the three linear maps $\Id_{S^p_2} \ot J_{\beta} \co S^p_2(\ell^p_{n_\beta}) \to S^p_2(\L^p(\Omega'))$, 
$\Phi_{\alpha,\beta} 
=\begin{bmatrix} 
\mathcal{P}_{\beta} |T| J_{\alpha} & \mathcal{P}_{\beta} TJ_{\alpha}  \\ 
\mathcal{P}_{\beta} T^\circ J_{\alpha} & \mathcal{P}_{\beta} |T|  J_{\alpha} 
\end{bmatrix} 
\co S^p_2(\ell^p_{n_\alpha}) \to S^p_2(\ell^p_{n_\beta})$ and $\Id_{S^p_2} \ot \mathcal{P}_{\alpha} \co S^p_2(\L^p(\Omega)) \to S^p_2(\ell^p_{n_\alpha})$ are all completely positive. By Proposition \ref{prop-positive-imply-cp}, the positive maps $J_{\beta} \co \ell^p_{n_\beta} \to \L^p(\Omega')$ and $\mathcal{P}_{\alpha} \co \L^p(\Omega) \to \ell^p_{n_\alpha}$ are completely positive. It remains to show the second assertion. For any $1 \leq j \leq n_\alpha$, we have
\begin{align*}
(\mathcal{P}_\beta T J_{\alpha})(e_j)
& =(\mathcal{P}_\beta T) \bigg(\frac{1}{\mu(A_j)^{\frac1p}} 1_{A_j}\bigg)
=\frac{1}{\mu(A_j)^{\frac1p}} \mathcal{P}_\beta\big(T(1_{A_j})\big) \\
& = \frac{1}{\mu(A_j)^{\frac1p}} \sum_{i=1}^{n_\beta} \frac{1}{\nu(B_i)^{1 - \frac1p}}\bigg(\int_{B_i} T(1_{A_j}) \d\mu' \bigg) e_i.
\end{align*}
We deduce that the matrix $[t_{\alpha,\beta,ij}]$ of the linear map $\mathcal{P}_\beta T J_{\alpha} \co \ell^p_{n_\alpha} \to \ell^p_{n_\beta}$ in the canonical basis is $\big[\frac{1}{\mu(A_j)^{\frac1p}} \frac{1}{\nu(B_i)^{1 - \frac1p}}\int_{B_i} T(1_{A_j}) \d\mu' \big]$. Moreover, we have
\begin{align*}
(\mathcal{P}_\beta T^\circ J_{\alpha})(e_j)
& =(\mathcal{P}_\beta T^\circ) \bigg(\frac{1}{\mu(A_j)^{\frac1p}} 1_{A_j}\bigg)
=\frac{1}{\mu(A_j)^{\frac1p}} \mathcal{P}_\beta\big(\ovl{T(1_{A_j})}\big) \\
& =\ovl{\frac{1}{\mu(A_j)^{\frac1p}} \sum_{i=1}^{n_\beta} \frac{1}{\nu(B_i)^{1-\frac1p}}\bigg(\int_{B_i} T(1_{A_j}) \d\mu' \bigg)} e_i.
\end{align*}
Hence the matrix of $P_{\alpha} T^\circ J_{\alpha}$ is $[\ovl{t_{\alpha,\beta,ij}}]_{ij}$. Finally, we equally have
\begin{align*}
(\mathcal{P}_\beta |T| J_{\alpha})(e_j)
& =(\mathcal{P}_\beta |T|) \bigg(\frac{1}{\mu(A_j)^{\frac1p}} 1_{A_j}\bigg)
=\frac{1}{\mu(A_j)^{\frac1p}} \mathcal{P}_\beta\big(|T|(1_{A_j})\big) \\
& =\frac{1}{\mu(A_j)^{\frac1p}} \sum_{i=1}^{n_\beta} \frac{1}{\nu(B_i)^{1-\frac1p}}\bigg(\int_{B_i} |T|(1_{A_j}) \d\mu' \bigg) e_i.
\end{align*} 
Now, we note that
$$
\int_{B_i} |T|(1_{A_j}) \d \mu' 
\geq \int_{B_i} |T(1_{A_j})|\d \mu'
\geq \left| \int_{B_i} T(1_{A_j})\d \mu'\right| 
= \mu(A_j)^{\frac1p} \nu(B_i)^{1 - \frac1p} |t_{\alpha,\beta,ij}|.
$$
Thus, the map $\mathcal{P}_\beta |T| J_{\alpha}$ is associated with some matrix $[s_{\alpha,\beta,ij}]$ with $s_{\alpha,\beta,ij} = |t_{\alpha,\beta,ij}| + r_{\alpha,\beta,ij}$ where $r_{\alpha,\beta,ij} \geq 0$ for any $i,j$. Further, let $\psi_{\alpha,\beta,ij} \in \C$ such that $t_{\alpha,\beta,ij} = |t_{\alpha,\beta,ij} |\psi_{\alpha,\beta,ij}$.

We denote by $i_{\alpha} \co \ell^p_{n_\alpha} \hookrightarrow S^p_{n_\alpha}$ the canonical diagonal embedding, $\tilde{J}_\alpha \ov{\mathrm{def}}{=} \Id_{S^p_2} \ot i_\alpha \co S^p_2(\ell^p_{n_\alpha}) \to S^p_2(S^p_{n_\alpha})$ and by $Q_{\alpha} \co S_2^p(S^p_{n_\alpha}) \to S^p_2(\ell^p_{n_\alpha})$ the canonical projection. Note that $Q_{\alpha} \tilde{J}_\alpha=\Id_{S^p_2(\ell^p_{n_\alpha})}$. Now, we show that the map $\tilde{J}_\beta \Phi_{\alpha,\beta} Q_{\alpha} \co S^p_2(S^p_{n_\alpha}) \to S^p_2(S^p_{n_\beta})$ is completely positive. If we take $a_{ij} 
= \begin{bmatrix} 
\sqrt{|t_{\alpha,\beta,ij}|} \psi_{\alpha,\beta,ij} e_{ij} & 0 \\ 
0 & \sqrt{|t_{\alpha,\beta,ij}|} e_{ij} 
\end{bmatrix}
,\, b_{ij}^{(1)} 
= \begin{bmatrix}
\sqrt{r_{\alpha,\beta,ij}} e_{ij} & 0 \\ 
0 & 0 \end{bmatrix}$ and $b_{ij}^{(2)} 
= \begin{bmatrix}
0 & 0 \\ 
0 & \sqrt{r_{\alpha,\beta,ij}} e_{ij} 
\end{bmatrix}$, 
we obtain for any $x \in S^p_2(S^p_{n_\alpha})$
\begingroup
\allowdisplaybreaks
\begin{align*}
\MoveEqLeft
 (\tilde{J}_\beta \Phi_{\alpha,\beta} Q_\alpha)(x)
=(\tilde{J}_\beta \Phi_{\alpha,\beta} Q_\alpha)
\left( \begin{bmatrix} 
x_{11} & x_{12} \\ 
x_{21} & x_{22}
\end{bmatrix}\right) \\
&=\left(\tilde{J}_\beta \begin{bmatrix} 
\mathcal{P}_{\beta} |T| J_{\alpha} & \mathcal{P}_{\beta} TJ_{\alpha}  \\ 
\mathcal{P}_{\beta} T^\circ J_{\alpha} & \mathcal{P}_{\beta} |T|  J_{\alpha} 
\end{bmatrix} \right)\left( \begin{bmatrix} 
\sum_{j=1}^{n_\alpha} x_{11jj} e_j & \sum_{j=1}^{n_\alpha} x_{12jj}e_j \\ 
\sum_{j=1}^{n_\alpha} x_{21jj}e_j & \sum_{j=1}^{n_\alpha} x_{22jj}e_j
\end{bmatrix}\right)\\
&=\tilde{J}_\beta\left( \begin{bmatrix} 
\sum_{j=1}^{n_\alpha} x_{11jj} \mathcal{P}_{\beta} |T| J_{\alpha}e_j & \sum_{j=1}^{n_\alpha} x_{12jj} \mathcal{P}_{\beta} TJ_{\alpha}e_j \\ 
\sum_{j=1}^{n_\alpha} x_{21jj} \mathcal{P}_{\beta} T^\circ J_{\alpha}e_j & \sum_{j=1}^{n_\alpha} x_{22jj} \mathcal{P}_{\beta} |T|  J_{\alpha}e_j
\end{bmatrix}\right)\\
&=\tilde{J}_\beta\left( \begin{bmatrix} 
\sum_{j=1}^{n_\alpha} x_{11jj} \sum_{i=1}^{n_\beta} s_{\alpha,\beta,ij} e_i & \sum_{j=1}^{n_\alpha} x_{12jj} \sum_{i=1}^{n_\beta} t_{\alpha,\beta,ij} e_i \\ 
\sum_{j=1}^{n_\alpha} x_{21jj}\sum_{i=1}^{n_\beta} \ovl{t_{\alpha,\beta,ij}}e_i & \sum_{j=1}^{n_\alpha} x_{22jj} \sum_{i=1}^{n_\beta} s_{\alpha,\beta,ij}e_i
\end{bmatrix}\right)\\
&=\sum_{j=1}^{n_\alpha}\sum_{i=1}^{n_\beta} \begin{bmatrix} 
x_{11jj}  s_{\alpha,\beta,ij} e_{ii} &  x_{12jj}  t_{\alpha,\beta,ij} e_{ii} \\ 
 x_{21jj} \ovl{t_{\alpha,\beta,ij}}e_{ii} & x_{22jj}  s_{\alpha,\beta,ij}e_{ii}
\end{bmatrix}\\
&=\sum_{j=1}^{n_\alpha}\sum_{i=1}^{n_\beta} \left(\begin{bmatrix} 
x_{11jj} |t_{\alpha,\beta,ij}| e_{ii} &  x_{12jj} t_{\alpha,\beta,ij} e_{ii} \\ 
x_{21jj}\ovl{t_{\alpha,\beta,ij}}e_{ii} & x_{22jj}  |t_{\alpha,\beta,ij}|e_{ii}
\end{bmatrix}
+\begin{bmatrix} 
 x_{11jj} r_{\alpha,\beta,ij}e_{ii} & 0 \\ 
0 &  x_{22jj}  r_{\alpha,\beta,ij}e_{ii}
\end{bmatrix}\right)\\
&=\sum_{j=1}^{n_\alpha}\sum_{i=1}^{n_\beta}\left( \begin{bmatrix} 
\sqrt{|t_{\alpha,\beta,ij}|} \psi_{\alpha,\beta,ij} e_{ij} & 0 \\ 
0 & \sqrt{|t_{\alpha,\beta,ij}|} e_{ij} 
\end{bmatrix} \begin{bmatrix} 
x_{11} & x_{12} \\ 
x_{21} & x_{22}
\end{bmatrix} 
 \right.\\
&\begin{bmatrix} 
\sqrt{|t_{\alpha,\beta,ij}|}\, \ovl{\psi_{\alpha,\beta,ij}} e_{ji} & 0 \\ 
0 & \sqrt{|t_{\alpha,\beta,ij}|} e_{ji} 
\end{bmatrix} + \begin{bmatrix}
\sqrt{r_{\alpha,\beta,ij}} e_{ij} & 0 \\ 
0 & 0 \end{bmatrix}
\begin{bmatrix} 
x_{11} & x_{12} \\ 
x_{21} & x_{22}
\end{bmatrix} 
\begin{bmatrix}
\sqrt{r_{\alpha,\beta,ij}} e_{ji} & 0 \\ 
0 & 0 \end{bmatrix} \\
&\left.+ \begin{bmatrix}
0 & 0 \\ 
0 & \sqrt{r_{\alpha,\beta,ij}} e_{ij} 
\end{bmatrix}
\begin{bmatrix} 
x_{11} & x_{12} \\ 
x_{21} & x_{22}
\end{bmatrix} 
\begin{bmatrix}
0 & 0 \\ 
0 & \sqrt{r_{\alpha,\beta,ij}} e_{ji} 
\end{bmatrix}\right)\\
&=\sum_{j=1}^{n_\alpha}\sum_{i=1}^{n_\beta} \big(a_{ij} x a_{ij}^* + b_{ij}^{(1)} x b_{ij}^{(1)*} + b_{ij}^{(2)} x b_{ij}^{(2)*}\big).
\end{align*}
\endgroup
We infer that $\tilde{J}_\beta \Phi_{\alpha,\beta} Q_\alpha$ is completely positive. Since $\Phi_{\alpha,\beta} = Q_\beta (\tilde{J}_\beta \Phi_{\alpha,\beta} Q_\alpha) \tilde{J}_\alpha$, we conclude that $\Phi_{\alpha,\beta}$ is completely positive.
The case $1 \leq p < \infty$ is proved.

\end{proof}

\begin{remark} \normalfont
\label{Rem-w-continuous-necessary}
Theorem \ref{prop-reg=decomp} seems to us to be true equally in the case $p=\infty$. That is, if $T \co \L^\infty(\Omega) \to \L^\infty(\Omega')$ is a (regular) operator, then the map $\Phi = \begin{bmatrix} |T|& T \\ T^\circ & |T| \end{bmatrix} \co S^\infty_2(\L^\infty(\Omega)) \to S^\infty_2(\L^\infty(\Omega'))$ is completely positive.
To prove this, replace the mapping $\mathcal{P}_\alpha \co \L^\infty(\Omega) \to \ell^\infty_n$ by $\mathcal{P}_\alpha(f) = \sum_{i=1}^n \phi_{A_i}(f|_{A_i}) e_i$, where $\phi_{A_i}$ is an arbitrary state on $\L^\infty(A_i)$ and $\Omega$ is partitioned (not semipartitioned) into $\Omega = \bigcup_{i=1}^n A_i$.
We equally take $J_\alpha \co \ell^\infty_n \to \L^\infty(\Omega),\: e_i \mapsto 1_{A_i}$.
Then Lemma \ref{Lemma-average} admits an $\L^\infty$-variant (the verification is entirely left to the reader), in particular $J_\alpha \mathcal{P}_\alpha$ converges strongly to the identity on $\L^\infty(\Omega)$ (the partitions are of course directed by refinement).
Also the proof of Theorem \ref{prop-reg=decomp} works in a similar way. If $T$ is in addition weak* continuous, we can use a duality argument\footnote{\thefootnote. Assume in addition that $T \co \L^\infty(\Omega) \to \L^\infty(\Omega')$ is weak* continuous with pre-adjoint $T_* \co \L^1(\Omega') \to \L^1(\Omega)$. Then by \eqref{Duality-reg} and by the case $p = 1$ proved previously, the map 
$\begin{bmatrix}
|T_*| & T_* \\ 
(T_*)^\circ & |T_*| 
\end{bmatrix} 
\co S^1_2(\L^1(\Omega')) \to S^1_2(\L^1(\Omega))$ is completely positive. Note that $|T_*|^*=|(T_*)^*|=|T|$ where we use \cite[Theorem 2.28 page 85]{AbA} in the first equality and it is easily checked that $((T_*)^\circ)^* = T^\circ$. So by Lemma \ref{Lemma-adjoint-cp}, its adjoint
$\begin{bmatrix} 
|T_*|^* & (T_*)^* \\ 
((T_*)^\circ)^* & |T_*|^* 
\end{bmatrix}
=\begin{bmatrix} 
|T| & T \\ 
T^\circ & |T| 
\end{bmatrix}
 \co S^\infty_2(\L^\infty(\Omega)) \to S^\infty_2(\L^\infty(\Omega'))$ is also completely positive. }.
\end{remark}

 %
%
%
 %

\subsection{Decomposable vs completely bounded}
\label{section-Decomposable-cb}

The authors of \cite{JuR} say that the following result is true without the $\QWEP$ assumption (and without proof). However, we think that $\QWEP$ is necessary\footnote{\thefootnote. Another point of view is to replace the formula of definition \eqref {Norm-dec} by $
\norm{T}_{\dec,\L^p(M) \to \L^p(N)}
=\inf\big\{\max\{\norm{v_1}_{\cb},\norm{v_2}_{\cb}\}\big\}$.} for $1<p< \infty$. 

\begin{prop}
\label{Prop-cb-leq-dec}
Let $M$ and $N$ be two $\QWEP$ von Neumann algebras which are equipped with faithful normal semifinite traces. Suppose $1 \leq p \leq \infty$. Let $T \co \L^p(M) \to \L^p(N)$ be a decomposable map. Then $T$ is completely bounded and $\norm{T}_{\cb,\L^p(M) \to \L^p(N)} \leq \norm{T}_{\dec,\L^p(M) \to \L^p(N)}$.
\end{prop}

\begin{proof}
By Proposition \ref{Prop-dec-inf-atteint}, there exist linear maps $v_1,v_2 \co \L^p(M) \to \L^p(N)$ such that the map
$
\Phi
\ov{\mathrm{def}}{=} \begin{bmatrix}
   v_1  &  T \\
   T^\circ  &  v_2  \\
\end{bmatrix} 
\co S^p_2(\L^p(M)) \to S^p_2(\L^p(N))
$
is completely positive with $\max\big\{\norm{v_1}_{},\norm{v_2}_{}\big\} =\norm{T}_{\dec}$. Let $b$ be an element of $S^p_n(\L^p(M))$ with $\norm{b}_{S^p_n(\L^p(M))} \leq 1$. By Lemma \ref{Lemma-Matricial-inequality3}, we can find $a, c \in S^p_n(\L^p(M))$ with $\norm{a}_{S^p_n(\L^p(M))} \leq 1$ and $\norm{c}_{S^p_n(\L^p(M))} \leq 1$ such that
$
\begin{bmatrix}
a   & b\\
b^* & c
\end{bmatrix}
$
is a positive element of $S^p_{2n}(\L^p(M))$. We deduce that
\begin{align*}
\MoveEqLeft
 \begin{bmatrix}
(\Id_{S^p_n} \ot v_{1})(a)   & (\Id_{S^p_n} \ot T)(b)\\
(\Id_{S^p_n} \ot T)(b)^* & (\Id_{S^p_n} \ot v_{2})(c)
\end{bmatrix}
=
\begin{bmatrix}
(\Id_{S^p_n} \ot v_{1})(a)       & (\Id_{S^p_n} \ot T)(b)\\
(\Id_{S^p_n} \ot T)^{\circ}(b^*) & (\Id_{S^p_n} \ot v_{2})(c)
\end{bmatrix}\\
&=
\begin{bmatrix}
(\Id_{S^p_n} \ot v_{1})(a)       & (\Id_{S^p_n} \ot T)(b)\\
(\Id_{S^p_n} \ot T^{\circ})(b^*) & (\Id_{S^p_n} \ot v_{2})(c)
\end{bmatrix}
=
(\Id_{S^p_n} \ot \Phi)\bigg(
\begin{bmatrix}
a   & b\\
b^* & c
\end{bmatrix}\bigg)   
\end{align*}
is a positive element of $S^p_{2n}(\L^p(N))$. By Lemma \ref{Lemma-Matricial-inequality}, using Theorem \ref{quest-cp-versus-cb}, we obtain
\begin{align*}
\MoveEqLeft
 \bnorm{(\Id_{S^p_n} \ot T)(b)}_{S^p_n(\L^p(N))}    
		\leq\frac{1}{2^{\frac{1}{p}}} \Big(\bnorm{(\Id_{S^p_n} \ot v_{1})(a)}_{S^p_{n}(\L^p(N))}^p + \bnorm{(\Id_{S^p_n} \ot v_{2})(c)}_{S^p_n(\L^p(N))}^p \Big)^{\frac1p} \\
		&\leq \frac{1}{2^{\frac{1}{p}}} \Big(\norm{v_1}_{\cb}^p \norm{a}_{S^p_n(\L^p(M))}^p + \norm{v_2}_{\cb}^p \norm{c}_{S^p_{n}(\L^p(M))}^p \Big)^{\frac1p}\\
		&\leq\max\big\{ \norm{v_1}_{},\norm{v_2}_{}\}\frac{1}{2^{\frac{1}{p}}} \Big(\norm{a}_{S^p_n(\L^p(M))}^p +  \norm{c}_{S^p_{n}(\L^p(M))}^p \Big)^{\frac1p} \\
		& \leq \max\big\{ \norm{v_1}_{},\norm{v_2}_{} \big\}
		= \norm{T}_{\dec}.
\end{align*}
We obtain $\bnorm{\Id_{S^p_n} \ot T}_{S^p_n(\L^p(M)) \to S^p_n(\L^p(N))} \leq \norm{T}_{\dec}$. We conclude that $\norm{T}_{\cb} \leq \norm{T}_{\dec}$.
\end{proof}

\begin{prop}
\label{quest-cp-versus-dec}
Let $M$ and $N$ be two $\QWEP$ von Neumann algebras equipped with faithful normal semifinite traces. Suppose $1 \leq p \leq \infty$. Let $T \co \L^p(M) \to \L^p(N)$ be a completely positive map. Then $T$ is decomposable and we have $\norm{T}_{\cb}=\norm{T}_{\dec}=\norm{T}$.
\end{prop}

\begin{proof}
By Proposition \ref{quest-cp-versus-dec1}, we know that $T$ is decomposable and that $\norm{T}_{\dec} \leq \norm{T}$. If $M$ and $N$ are $\QWEP$, by Proposition \ref{Prop-cb-leq-dec}, we have $\norm{T}_{\cb} \leq \norm{T}_{\dec}$.
\end{proof}

To complement the previous proposition, we observe that completely bounded operators are not decomposable in general. For that, we give a result on group von Neumann algebras of discrete groups, see Section \ref{section-twisted-von-Neumann} for background.

\begin{prop}
\label{Prop-cb-mais-pas-dec-Fourier-mult}
\begin{enumerate}
	\item Let $G$ be a non-amenable weakly amenable discrete group. Then there exists a completely bounded Fourier multiplier $M_\varphi \co \VN(G) \to \VN(G)$ which is not decomposable.
	\item Suppose $1<p< \infty$. Let $G$ be a non-amenable discrete group with $\mathrm{AP}$ and such that $\VN(G)$ has $\mathrm{QWEP}$. Then there exists a completely bounded Fourier multiplier $M_\varphi \co \L^p(\VN(G)) \to \L^p(\VN(G))$ which is not decomposable.
\end{enumerate}
\end{prop}

\begin{proof}
1. By the proofs of \cite[Theorem 12.3.10]{BrO} and \cite[Theorem 4.4]{JR1}, there exists a net $\big(M_{\varphi_\alpha}\big)$ of finite-rank completely bounded Fourier multipliers on $\VN(G)$ with $\norm{M_{\varphi_\alpha}}_{\cb} \leq C$ such that $M_{\varphi_\alpha} \to \Id_{\VN(G)}$ in the point weak* topology. If all the Fourier multipliers were decomposable, since two comparable complete norms on a linear space are in fact equivalent, the von Neumann algebra $\VN(G)$ would have the bounded normal decomposable approximation property of \cite[Theorem 4.3 (iv)]{LM} (see also \cite[page 355]{JuR}) and $\VN(G)$ would be injective. By \cite[Theorem 3.8.2]{SS}, we conclude that $G$ is amenable. This is the desired contradiction. 

2. By \cite[Theorem 4.4]{JR1}, there exists a net of completely contractive finite-rank Fourier multipliers $M_{\varphi_\alpha} \co \L^p(\VN(G))\to \L^p(\VN(G))$ such that $M_{\varphi_\alpha} \to \Id_{\L^p(\VN(G))}$ in the point-norm topology. If all the Fourier multipliers were decomposable, again since two comparable complete norms on a linear space are in fact equivalent, the space $\L^p(\VN(G))$ would have the bounded decomposable approximation property of \cite[page 356]{JuR}. By \cite[Theorem 5.2]{JuR} the von Neumann algebra $\VN(G)$ would be injective. By \cite[Theorem 3.8.2]{SS}, we conclude that $G$ is amenable. This is a second contradiction. 
\end{proof}

\begin{remark} \normalfont
Note that we can use the free group $\F_n$ where $2\leq n \leq \infty$ ($n$ countable) with the two parts of the last result. Indeed, by \cite[Theorem 1.8]{Haa1} (see also \cite[Corollary 3.11]{DCH}), the group $\F_n$ is weakly amenable, hence has AP by \cite[page 677]{HK}. Moreover, it is well-known that $\VN(\F_n)$ has QWEP, see e.g. \cite[Theorem 9.10.4]{Pis7}.
\end{remark}

We will describe in Theorem \ref{thm-comparaison-cb-dec-free-group} an explicit result in the same vein. For that, we need intermediate results.

\begin{lemma}
\label{Lemma-Some-cp-maps}
Let $M$ be a von Neumann algebra equipped with a faithful normal semifinite trace. Suppose $1 \leq p \leq \infty$. For any integer $n \geq 2$, the maps

\begin{equation*}
\begin{array}{cccc}
  \alpha_n    \co &   \L^p(M)  &  \longrightarrow   &  S_n^p(\L^p(M))  \\
           &   x  &  \longmapsto       &  \begin{bmatrix} 
   x   &  \cdots &   x\\ 
\vdots &        & \vdots\\ 
   x   & \cdots  &   x
\end{bmatrix}  \\  \\
\end{array}
\end{equation*}
and
\begin{equation*}
\begin{array}{cccc}
   \sigma_{n}    \co &   S_{n^2}^p(\L^p(M))   &  \longrightarrow   & S_n^p(\L^p(M))   \\
           &   \begin{bmatrix}  
  \begin{bmatrix} b_{11}^{11} & \cdots & b_{11}^{1n} \\ \vdots & & \vdots \\ b_{11}^{n1} & \cdots & b_{11}^{nn} \end{bmatrix}
& \cdots & \begin{bmatrix} b_{1n}^{11} & \cdots & b_{1n}^{1n} \\ \vdots & & \vdots \\ b_{1n}^{n1} & \cdots & b_{1n}^{nn} \end{bmatrix} \\
\vdots &  & \vdots \\
\begin{bmatrix}  b_{n1}^{11} & \cdots & b_{n1}^{1n} \\ \vdots & & \vdots \\ b_{n1}^{n1} & \cdots & b_{n1}^{nn} \end{bmatrix} & \cdots & \begin{bmatrix}  b_{nn}^{11} & \cdots & b_{nn}^{1n} \\ \vdots & & \vdots \\ b_{nn}^{n1} & \cdots & b_{nn}^{nn} \end{bmatrix}
 \end{bmatrix}   &  \longmapsto       &  \begin{bmatrix}
   b_{11}^{11}  & \cdots  & b_{1n}^{1n} \\
\vdots & & \vdots \\
   b_{n1}^{n1}  & \cdots & b_{nn}^{nn} \\
 \end{bmatrix}   \\
\end{array}
\end{equation*}
are completely positive. 
\end{lemma}

\begin{proof}
For any $x \in \L^p(M)$, we have $\alpha_{n}(x)
=\begin{bmatrix} 
   x   &  \cdots &   x\\ 
\vdots &        & \vdots\\ 
   x   & \cdots  &   x
\end{bmatrix} 
=
\begin{bmatrix}
1  \\
\vdots\\
1\\ 
\end{bmatrix} 
x
\begin{bmatrix}
1 & \cdots &1 \\ 
\end{bmatrix}$. 
Moreover, for any $b \in S_{n^2}^p(\L^p(M))$, we have $\sigma_{n}(b) =  A b A^*$ where $A \in \M_{n,n^2}$ is defined by 
{\renewcommand{\arraystretch}{1.2}
$$
A
= 
\begin{bmatrix} 
{\renewcommand{\arraystretch}{1}\begin{bmatrix} 
1 & 0 &  \cdots & 0 
\end{bmatrix}} & 
{\renewcommand{\arraystretch}{1}\begin{bmatrix} 
0 & 0 &\cdots  & 0 
\end{bmatrix}} 
&  \cdots & 
{\renewcommand{\arraystretch}{1}\begin{bmatrix} 
0 & 0 &\cdots &  0 
\end{bmatrix}} \\
{\renewcommand{\arraystretch}{1}\begin{bmatrix} 
0 & 0 & \cdots& 0 
\end{bmatrix}} & 
{\renewcommand{\arraystretch}{1}\begin{bmatrix} 
0 & 1 & \cdots & 0 
\end{bmatrix}} 
& \ldots & 
{\renewcommand{\arraystretch}{1}\begin{bmatrix} 
0 & 0 &\cdots & 0 
\end{bmatrix}} \\
\vdots & & & \vdots \\
{\renewcommand{\arraystretch}{1}\begin{bmatrix} 
0 & 0 &\cdots & 0 
\end{bmatrix}} 
&  \cdots & 
{\renewcommand{\arraystretch}{1}\begin{bmatrix} 
0 & 0 &\cdots & 0 
\end{bmatrix}} & 
{\renewcommand{\arraystretch}{1}\begin{bmatrix} 
0 & \cdots & 0 & 1 
\end{bmatrix}} 
\end{bmatrix}.
$$}
Now, we appeal to \eqref{conj-cp}.
\end{proof}

\begin{prop}
\label{prop-w-et-v}
Let $M$ and $N$ be von Neumann algebras equipped with faithful normal semifinite traces. Suppose $1 \leq p \leq \infty$. Let $n \geq 2$ be an integer and consider some bounded maps $T_{ij} \co \L^p(M) \to \L^p(N)$ where $1 \leq i,j \leq n$. If $\alpha_n$ is the completely positive map from Lemma \ref{Lemma-Some-cp-maps} then the map
\begin{equation*}
\begin{array}{cccc}
    \Phi  \co & S_n^p(\L^p(M)) &  \longrightarrow &S_n^p(\L^p(N))  \\
           &   \begin{bmatrix}  
   a_{11}  & \cdots   & a_{1n}  \\
   \vdots  &   & \vdots  \\
   a_{n1}  & \cdots   & a_{nn}  \\
 \end{bmatrix}   &  \longmapsto       &  \begin{bmatrix}  
   T_{11}(a_{11})  & \cdots   & T_{1n}(a_{1n})  \\
   \vdots  &   & \vdots  \\
   T_{n1}(a_{n1})  & \cdots   & T_{nn}(a_{nn})  \\
 \end{bmatrix}  \\
\end{array}
\end{equation*}
is completely positive if and only if the map $\Phi \circ \alpha_n$ is completely positive. 
\end{prop}


\begin{proof}
One direction is obvious. For the reverse direction, we have 
\begingroup
\allowdisplaybreaks
\begin{align*}
\MoveEqLeft
  \sigma_{n} \circ \big(\Id_{S^p_n} \ot (\Phi \circ \alpha_n)\big)\left(\begin{bmatrix}  
   a_{11}  & \cdots   & a_{1n}  \\
   \vdots  &   & \vdots  \\
   a_{n1}  & \cdots   & a_{nn}  \\
 \end{bmatrix}\right)
=\sigma_{n}\left(\begin{bmatrix}  
   \Phi \circ \alpha_n(a_{11})  & \cdots   & \Phi \circ \alpha_n(a_{1n})  \\
   \vdots  &   & \vdots  \\
   \Phi \circ \alpha_n(a_{n1})  & \cdots   & \Phi \circ \alpha_n(a_{nn})  \\
 \end{bmatrix}\right)\\
&=
\sigma_{n}\left(\begin{bmatrix}
\begin{bmatrix}  
   T_{11}(a_{11})  & \cdots   & T_{1n}(a_{11})  \\
   \vdots  &   & \vdots  \\
   T_{n1}(a_{11})  & \cdots   & T_{nn}(a_{11})  \\
 \end{bmatrix}
&\cdots&\begin{bmatrix}  
   T_{11}(a_{1n})  & \cdots   & T_{1n}(a_{1n})  \\
   \vdots  &   & \vdots  \\
   T_{n1}(a_{1n})  & \cdots   & T_{nn}(a_{1n})  \\
 \end{bmatrix}\\
\vdots&&\vdots\\
\begin{bmatrix}  
   T_{11}(a_{n1})  & \cdots   & T_{1n}(a_{n1})  \\
   \vdots  &   & \vdots  \\
   T_{n1}(a_{n1})  & \cdots   & T_{nn}(a_{n1})  \\
 \end{bmatrix}
&\cdots
&\begin{bmatrix}  
   T_{11}(a_{nn})  & \cdots   & T_{1n}(a_{nn})  \\
   \vdots  &   & \vdots  \\
   T_{n1}(a_{nn})  & \cdots   & T_{nn}(a_{nn})  \\
 \end{bmatrix}
\end{bmatrix}\right) \\
&=
\begin{bmatrix}  
   T_{11}(a_{11})  & \cdots   & T_{1n}(a_{1n})  \\
   \vdots  &   & \vdots  \\
   T_{n1}(a_{n1})  & \cdots   & T_{nn}(a_{nn})  \\
 \end{bmatrix}
=
\Phi\left(\begin{bmatrix}  
   a_{11}  & \cdots   & a_{1n}  \\
   \vdots  &   & \vdots  \\
   a_{n1}  & \cdots   & a_{nn}  \\
 \end{bmatrix} \right).
\end{align*}
\endgroup
Hence $\Phi=\sigma_{n} \circ (\Id_{S^p_n} \ot (\Phi \circ \alpha_n))$. Note that if $\Phi \circ \alpha_n$ is completely positive then $\Id_{S^p_n} \ot (\Phi \circ \alpha_n)$ is also completely positive by Lemma \ref{Lemma-cp-S1E}. In this case, since $\sigma_{n}$ is completely positive we deduce that $\Phi$ is completely positive.
\end{proof}

\begin{prop}
\label{prop-decomposable-et-selfadjoint3-prime}
Let $M$ and $N$ be von Neumann algebras equipped with faithful normal semifinite traces. Suppose $1 \leq p \leq \infty$. Let $T \co \L^p(M) \to \L^p(N)$ be a linear map. Then $T$ is decomposable if and only if the map $\tilde{T} \circ \alpha_2 \co \L^p(M) \to S^p_2(\L^p(N))$ where $\tilde{T}$ is the map from Proposition \ref{prop-regulier-et-selfadjoint} is decomposable. Moreover, in this case, we have 
$$
\norm{T}_{\dec,\L^p(M) \to \L^p(N)} 
\leq \|\tilde{T}\circ \alpha_2\|_{\dec,\L^p(M) \to S^p_2(\L^p(N))} 
\leq 2^{\frac1p} \norm{T}_{\dec, \L^p(M) \to \L^p(N)}.
$$
Furthermore, $\tilde{T} \circ \alpha_2$ is adjoint preserving.
\end{prop}

\begin{proof}
Let $x  \in \L^p(M)$. We have 
$$
\tilde{T} \circ \alpha_2(x^*)
=\tilde{T}\left(\begin{bmatrix} 
x^* & x^* \\ 
x^* & x^* 
\end{bmatrix}\right) 
=\begin{bmatrix} 
0 & T(x^*) \\ 
T^\circ(x^*) & 0 
\end{bmatrix} 
=\begin{bmatrix} 
0 & T(x^*) \\ 
T(x)^* & 0 
\end{bmatrix}
$$

and also
$$
\left(\tilde{T}\circ \alpha_2(x)\right)^* 
=\left(\tilde{T}\left(
\begin{bmatrix} 
x & x \\ 
x & x 
\end{bmatrix}
\right)\right)^* 
=\begin{bmatrix} 
0 & T(x) \\ 
T^\circ(x) & 0 
\end{bmatrix}^* 
= \begin{bmatrix} 
0 & T^\circ(x)^* \\ 
T(x)^* & 0 
\end{bmatrix} =  
\begin{bmatrix} 
0 & T(x^*) \\ 
T(x)^* & 0 
\end{bmatrix}.
$$
We conclude that $\tilde{T} \circ \alpha_2$ is adjoint preserving, i.e. $(\tilde{T}\circ \alpha_2)^\circ = \tilde{T}\circ \alpha_2$.

Suppose that $T$ is decomposable. By Proposition \ref{Prop-dec-inf-atteint}, there exist some maps $v_1,v_2 \co \L^p(M) \to \L^p(N)$ such that $
\begin{bmatrix}
   v_1  &  T \\
   T^\circ  &  v_2  \\
\end{bmatrix}$ is completely positive with $\max\big\{\norm{v_1},\norm{v_2} \big\}
= \norm{T}_{\dec}$. Using \eqref{conj-cp}, we note that the map
$$ 
S^p_2(\L^p(M)) \to S^p_2(\L^p(M)),\: 
\begin{bmatrix}
a & b \\ c & d
\end{bmatrix}
\mapsto
\begin{bmatrix}
1 & 0 \\ 
0 & -1
\end{bmatrix}
\begin{bmatrix}
a & b \\ 
c & d
\end{bmatrix}
\begin{bmatrix}
1 & 0 \\ 
0 & -1
\end{bmatrix}
=\begin{bmatrix}
a & -b \\ 
-c & d
\end{bmatrix}
$$
is completely positive. By composition, we deduce that the map $
\begin{bmatrix}
v_1 & -T\\ 
-T^{\circ}  & v_2
\end{bmatrix}
\circ \alpha_2$ is completely positive. We define the map $S\overset{\textrm{def}}=\begin{bmatrix}
   v_1  &  0 \\
   0  &  v_2  \\
\end{bmatrix} \circ \alpha_2 \co \L^p(M) \to S^p_2(\L^p(N))$. Then In the light of the foregoing, $S$ is completely positive and is easy to check using \eqref{Block-diagonal} that $\norm{S} \leq 2^{\frac1p} \norm{T}_{\dec}$. Moreover, $-S \leq_{\cp} \tilde{T} \circ \alpha_2 \leq_{\cp} S$. By Proposition \ref{prop-decomposable-et-selfadjoint}, we conclude that $\norm{\tilde{T} \circ \alpha_2}_{\dec} \leq 2^{\frac1p} \norm{T}_{\dec}$.

Now suppose that the map $\tilde{T} \circ \alpha_2 \co \L^p(M) \to S^p_2(\L^p(N))$ is decomposable. Moreover let $v_1,v_2 \co \L^p(M) \to S^p_2(\L^p(N))$ such that the map $
\begin{bmatrix} 
v_1 & \tilde{T} \circ \alpha_2 \\ 
\tilde{T} \circ \alpha_2 & v_2 
\end{bmatrix} \co S^p_2(\L^p(M)) \to S^p_4(\L^p(N))$ is completely positive. Put $w_1 \co \L^p(M) \to \L^p(N)$, $a \mapsto (v_1(a))_{11}$ and $w_2 \co \L^p(M) \to \L^p(N)$, $a \mapsto (v_2(a))_{22}$. Then each $w_i$ is also completely positive as a composition of completely positive mappings. Then an easy computation gives
\begin{align*}
\MoveEqLeft
\begin{bmatrix} 
1 & 0 & 0 & 0 \\ 
0 & 0 & 0 & 1 
\end{bmatrix} 
\cdot 
\Bigg(
\begin{bmatrix} 
v_1 & \tilde{T} \circ \alpha_2 \\ 
\tilde{T} \circ \alpha_2 & v_2 
\end{bmatrix} 
\left( 
\begin{bmatrix} 
a & b \\ 
c & d 
\end{bmatrix} 
\right)\Bigg)\cdot \begin{bmatrix} 
1 & 0 \\ 0 & 0 \\ 
0 & 0 \\ 0 & 1 
\end{bmatrix}\\
&=\begin{bmatrix} 
1 & 0 & 0 & 0 \\ 
0 & 0 & 0 & 1 
\end{bmatrix} 
\cdot 
\begin{bmatrix} 
v_1(a) & \tilde{T}\circ \alpha_2(b) \\ 
\tilde{T} \circ \alpha_2(c) &  v_2(d)
\end{bmatrix} \cdot 
\begin{bmatrix} 
1 & 0 \\ 0 & 0 \\ 
0 & 0 \\ 0 & 1 
\end{bmatrix}
=\begin{bmatrix} 
w_1(a) & T(b) \\ 
T^\circ(c) & w_2(d) 
\end{bmatrix}. 
\end{align*}
Using \eqref{conj-cp}, we deduce by composition that the map $\begin{bmatrix} 
w_1 & T \\ 
T^\circ & w_2 
\end{bmatrix}$ is completely positive. We infer that $T$ is decomposable and that $\|T\|_{\dec} \leq \max\{\norm{w_1},\norm{w_2}\} \leq \max\{\norm{v_1},\norm{v_2}\}$ and passing to the infimum over all admissible $v_1,v_2$ shows that $\norm{T}_{\dec} \leq \| \tilde{T} \|_{\dec}$.
\end{proof}

In the following result, we generalize the results of \cite[Theorem 5.4.7]{ER} and \cite[page 204]{Haa} done for $p=\infty$.

\begin{thm}
\label{thm-computation-norm-finite-factor}
Let $M$ be a von Neumann algebra equipped with a normal finite faithful normalized trace and let $u_1,\ldots,u_n \in M$ be arbitrary unitaries. Suppose $1\leq p \leq \infty$. Consider the map $T \co \ell^p_n \to \L^p(M)$ defined by $T(e_k)=u_k$. Then $\norm{T}_{\dec,\ell^p_n \to \L^p(M)}=n^{1-\frac{1}{p}}$.  
\end{thm}

\begin{proof}
As observed, we can suppose $1\leq p<\infty$. Note that the unit element $1$ of $M$ belongs to $\L^p(M)$ since $M$ is finite. The map $\varphi \co \ell^p_n \to \C$, $ \sum_{k=1}^{n} c_k e_k \mapsto \sum_{k=1}^{n} c_k$ is a positive linear functional. Since $\ell^p_n$ is a commutative $\L^p$-space, by Proposition \ref{prop-positive-imply-cp-if-depart-commutative}, we deduce that the linear map 
\begin{equation*}
\begin{array}{cccc}
    v  \co &  \ell^p_n   &  \longrightarrow   &  \L^p(M)  \\
              &  \sum_{k=1}^{n} c_k e_k   &  \longmapsto       &   \big(\sum_{k=1}^{n} c_k\big)1 \\
\end{array}
\end{equation*}
is completely positive. Moreover, using the normalization  of the trace in the third equality and H\"older's inequality in the last inequality, we have
\begin{align*}
\MoveEqLeft
 \Bgnorm{v\bigg(\sum_{k=1}^{n} c_k e_k\bigg)}_{\L^p(M)}   
		=\Bgnorm{\bigg(\sum_{k=1}^{n} c_k\bigg)1}_{\L^p(M)}
		=\left|\sum_{k=1}^{n} c_k\right|\norm{1}_{\L^p(M)}\\
		&=\left|\sum_{k=1}^{n} c_k\right|\leq \sum_{k=1}^{n} |c_k|
		\leq n^{1-\frac{1}{p}}\bigg(\sum_{k=1}^{n} |c_k|^p\bigg)^{\frac{1}{p}}.
\end{align*}
We infer that $\norm{v} \leq n^{1-\frac{1}{p}}$. 

We consider the map 
$
\tilde{T}
=
\begin{bmatrix}
     0        &  T \\
    T^\circ   &  0 \\
  \end{bmatrix} \co S^p_2(\ell^p_n) \to S^p_2(\L^p(M))
$ and the map $\alpha_4 \co \ell^p_n \to S^p_4(\ell^p_n)$ of Lemma \ref{Lemma-Some-cp-maps} with $M=\ell^\infty_n$.
Since $e^*_k=e_k$, we have
\begin{align*}
\MoveEqLeft
 \left(\begin{bmatrix} 
\begin{bmatrix} 
v & 0 \\ 
0 & v 
\end{bmatrix}
& \tilde{T}  \\ 
\tilde{T} &  
\begin{bmatrix} 
v & 0 \\ 
0 & v
\end{bmatrix} 
\end{bmatrix} 
\circ \alpha_4\right)(e_k)  
=\begin{bmatrix} 
v(e_k) & 0 & 0 & T(e_k)\\
0 & v(e_k) & T^\circ(e_k) & 0\\
0 & T(e_k) & v(e_k) & 0\\
T^\circ(e_k) & 0 & 0 & v(e_k)\\
\end{bmatrix} 
=\begin{bmatrix} 
1 & 0 & 0 & u_k\\
0 & 1 & u_k^* & 0\\
0 & u_k & 1 & 0\\
u_k^* & 0 & 0 & 1\\
\end{bmatrix}
.
\end{align*}
%
The 2x2 matrix $\tilde{u}_k
=\begin{bmatrix}
  0   & u_k  \\
  u_k^*  &  0 \\
  \end{bmatrix}$ is a (selfadjoint) unitary. Hence we have $\norm{
  \begin{bmatrix}
  0   & u_k  \\
  u_k^*  &  0 \\
  \end{bmatrix}}_{\M_2(M)} \leq 1$. By \cite[Proposition 1.3.2]{ER}, we conclude that the matrix on the right hand side of the previous equation is positive. Thus the map $\begin{bmatrix} 
\begin{bmatrix} 
v & 0 \\ 
0 & v 
\end{bmatrix}
& \tilde{T}  \\ 
\tilde{T} &  
\begin{bmatrix} 
v & 0 \\ 
0 & v
\end{bmatrix} 
\end{bmatrix} \circ \alpha_4$ is positive. Using again Proposition \ref{prop-positive-imply-cp-if-depart-commutative}, we obtain that this map is indeed completely positive. By Proposition \ref{prop-w-et-v}, we deduce that the map $\begin{bmatrix} 
\begin{bmatrix} 
v & 0 \\ 
0 & v 
\end{bmatrix}
& \tilde{T}  \\ 
\tilde{T} &  
\begin{bmatrix} 
v & 0 \\ 
0 & v
\end{bmatrix} 
\end{bmatrix}$ is completely positive. Hence $\tilde{T}$ is decomposable with $\norm{\tilde{T}}_{\dec} \leq \norm{
\begin{bmatrix} 
v & 0 \\ 
0 & v 
\end{bmatrix}}
\leq \norm{v}$ where the last inequality is easy to prove using \cite[Corollary 1.3]{Pis5}. Using Proposition \ref{prop-decomposable-et-selfadjoint3}, we conclude that $T$ is decomposable and that $\norm{T}_{\dec}=\norm{\tilde{T}}_{\dec} \leq \norm{v} \leq n^{1-\frac{1}{p}}$.

On the other hand, let $S \co \ell^p_n \to S^p_2(\L^p(M))$ be a completely positive map satisfying $-S \leq_{\cp} \tilde{T} \circ \alpha_2 \leq_{\cp} S$ where $\alpha_2 \co \ell^p_n \to S^p_2(\ell^p_n)$. If we let $x_k\overset{\textrm{def}}=S(e_k)$ and $\tilde{u}_k\overset{\textrm{def}}=\tilde{T} \circ \alpha_2(e_k)$, then $\tilde{u}_k
=\begin{bmatrix} 
0 & u_k \\ 
u_k^* & 0 
\end{bmatrix}$ is a selfadjoint unitary with $-x_k \leq \tilde{u}_k \leq x_k$. Thus we have
$$
x_k=\frac{1}{2}\big[(x_k-\tilde{u}_k)+(x_k+\tilde{u}_k)\big],
$$
with $x_k\pm \tilde{u}_k \geq 0$. 
Consider the finite trace $\tau_1 \overset{\textrm{def}}=\tr \ot\ \tau$ on $\M_2(M)$ where $\tau$ is the normalized trace on $M$. Then it follows that
\begin{align*}
\MoveEqLeft
 \tau_1(x_k)   
		=\tau_1\bigg(\frac{1}{2}\big[(x_k-\tilde{u}_k)+(x_k+\tilde{u}_k)\big]\bigg)
		=\frac{1}{2}\Big[\tau_1\big(x_k-\tilde{u}_k\big)+\tau_1\big(x_k+\tilde{u}_k\big)\Big]\\
		&=\frac{1}{2}\Big[\norm{x_k-\tilde{u}_k}_{1}+\norm{x_k+\tilde{u}_k}_{1}\Big]
		\geq \frac{1}{2}\norm{x_k-\tilde{u}_k-(x_k+\tilde{u}_k)}_{1}=\norm{\tilde{u}_k}_{S^1_2(\L^1(M))}
\end{align*}
where $\dsp
\norm{\tilde{u}_k}_{S^1_2(\L^1(M))}=\tau_1\Big(\big(\tilde{u}_k^*\tilde{u}_k\big)^{\frac{1}{2}}\Big)=\tau_1(\I_2 \ot 1)=2$. Moreover, we have $\norm{\I_2 \ot 1}_{S^{p^*}_2(\L^{p^*}(M))}=2^{\frac{1}{p^*}}$. By duality, we obtain
$$
\norm{x_1 + \cdots + x_n}_{S^p_2(\L^p(M))} 
\geq \frac{\langle x_1 + \cdots + x_n,\I_2 \ot 1 \rangle}{\norm{\I_2 \ot 1}_{S^{p^*}_2(\L^{p^*}(M))}}
=\frac{\tau_1(x_1 + \cdots + x_n)}{\norm{\I_2 \ot 1}_{S^{p^*}_2(\L^{p^*}(M))}}
=2^{1-\frac{1}{p^*}}n.
$$
We deduce that
\begin{align*}
\MoveEqLeft
    \norm{S}_{\ell^p_n \to S^p_2(\L^p(M))}
		\geq \frac{\bnorm{S(1_{})}_{S^p_2(\L^p(M))}}{\norm{1_{}}_{\ell^p_n}} 
		=n^{-\frac1p} \bnorm{S(e_1)+\cdots +S(e_n)}_{S^p_2(\L^p(M))}\\ 
		&=n^{-\frac1p}\norm{x_1+\cdots+x_n}_{S^p_2(\L^p(M))}
		\geq n^{-\frac1p}2^{1-\frac{1}{p^*}}n
		=n^{1-\frac1p}2^{1-\frac{1}{p^*}}.
\end{align*} 
Using Proposition \ref{prop-decomposable-et-selfadjoint3-prime} in the first inequality and Proposition \ref{prop-decomposable-et-selfadjoint} in the second inequality, we conclude that 
$$
\norm{T}_{\dec,\ell^p_n \to \L^p(M)}
\geq 2^{-\frac1p} \norm{\tilde{T} \circ \alpha_2}_{\dec,\ell^p_n \to S^p_2(\L^p(M))} \geq 2^{-\frac1p}n^{1-\frac1p}2^{1-\frac{1}{p^*}}
=n^{1-\frac1p}.
$$
\end{proof}

Let $n \geq 1$ be an integer and let $G=\mathbb{F}_n$ be a free group with $n$ generators denoted by $g_1,\ldots,g_n$.

\begin{thm}
\label{thm-comparaison-cb-dec-free-group}
Suppose $1 \leq p \leq \infty$. Let $n \geq 2$ be an integer. Consider the map $T_n \co \ell^p_n \to \L^p(\VN(\F_n))$ defined by $T_n(e_k)=\lambda_{g_k}$. We have $\norm{T_n}_{\cb} \leq \big( 2\sqrt{n-1}\big)^{1-\frac{1}{p}}$ 
and $\norm{T_n}_{\dec}=n^{1-\frac{1}{p}}$. In particular, if $1<p\leq \infty$ we have $\frac{\norm{T_n}_{\dec}}{\norm{T_n}_{\cb}}\xra[n \to+\infty]{} +\infty$.
\end{thm}

\begin{proof}
The equality is a consequence of Theorem \ref{thm-computation-norm-finite-factor}. For any $1 \leq k \leq n$, using the normalized trace $\tau_{\F_n}$, note that
$$
\bnorm{\lambda_{g_k}}_{\L^1(\VN(\F_n))} 
=\tau_{\F_n}(|\lambda_{g_k}|)
=\tau_{\F_n}\Big(\big(\lambda_{g_k}^*\lambda_{g_k}\big)^{\frac{1}{2}}\Big)
=\tau_{\F_n}(1)
=1.
$$
For any $A_1,\ldots,A_l \in S^1_n$, using the isometry $S^1_n(\ell^1_n)=\ell^1_n(S^1_n)$ in the last equality, we deduce that
\begin{align*}
\MoveEqLeft
  \Bgnorm{\big(\Id_{S^1_n} \ot T_n \big)\bigg(\sum_{k=1}^{l} A_k \ot e_k\bigg)}_{S^1_n(\L^1(\VN(\F_n)))}  
		=\Bgnorm{\sum_{k=1}^{l} A_k\ot \lambda_{g_k}}_{S^1_n(\L^1(\VN(\F_n)))}\\
		&\leq \sum_{k=1}^{l} \norm{A_k}_{S^1_n} \bnorm{\lambda_{g_k}}_{\L^1(\VN(\F_n))}
		= \sum_{k=1}^{l} \norm{A_k}_{S^1_n}
		=\Bgnorm{\sum_{k=1}^{l} A_k \ot e_k}_{S^1_n(\ell^1_n)}.
\end{align*}
We deduce that $\norm{T_n}_{\cb, \ell^1_n \to \L^1(\VN(\F_n))} \leq 1$. Note that \cite[Theorem 5.4.7]{ER} gives the estimate $\norm{T_n}_{\cb, \ell^\infty_n \to \VN(\F_n)} \leq 2\sqrt{n-1}$. Hence, by interpolation, we deduce that
\begin{align*}
\MoveEqLeft
   \norm{T_n}_{\cb,\ell^p_n \to \L^p(\VN(\F_n))} 
		\leq \big(\norm{T_n}_{\cb, \ell^1_n \to \L^1(\VN(\F_n))}\big)^{\frac{1}{p}}\big( \norm{T_n}_{\cb, \ell^\infty_n \to \VN(\F_n)}\big)^{1-\frac{1}{p}}
		\leq \big( 2\sqrt{n-1}\big)^{1-\frac{1}{p}}.
\end{align*}
\end{proof}

In Chapter \ref{sec:Existence-of-strongly}, we will continue these investigations.

\section{Decomposable Schur multipliers and Fourier multipliers on discrete groups}
\label{sec:Decomposable-Schur-multipliers-and-Fourier-multipliers}

In this section, we give a generalization of the average argument of Haagerup. This construction simultaneously gives a complementation for spaces of completely bounded Schur multipliers and completely bounded Fourier multipliers on \textit{discrete} groups, possibly deformed by a 2-cocycle and the independence of the completely bounded norm and the complete positivity with respect to the 2-cocycle. In Section \ref{subsec-description-decomposable-norm} below, we give our first results on decomposable Fourier multipliers (and Schur multipliers).

\subsection{Twisted von Neumann algebras}
\label{section-twisted-von-Neumann}
 
A basic reference on this subject is \cite{ZM}. See also \cite{BC1} and references therein. Let $G$ be a discrete group. We first recall that a 2-cocycle on $G$ with values in $\T$ is a map $\sigma \co G \times G \to \T$ such that 
\begin{equation}
\label{equation-2-cocycle}
\sigma(s,t)\sigma(st,r)
=\sigma(t,r)\sigma(s,tr)	
\end{equation}
for any $s, t, r \in G$. We will consider only normalized 2-cocycles, that is, satisfying $
\sigma(s,e)
=\sigma(e,s)
=1$ for any $s \in G$. This implies that $\sigma(s,s^{-1}) = \sigma(s^{-1},s)$ for any $s \in G$. The set $\mathrm{Z}^2(G,\T)$ of all normalized 2-cocycles becomes an abelian group under pointwise product, the inverse operation corresponding to conjugation: $\sigma ^{-1} = \ovl{\sigma},$ where $\ovl{\sigma}(s,t) = \ovl{\sigma(s,t)}$, and the identity element being the trivial cocycle on $G$ denoted by $1$.

Now, suppose that $G$ is equipped with a $\T$-valued $2$-cocycle. For any $s \in G$, we define the bounded operator $\lambda_{\sigma,s} \in \B(\ell^2_G)$ defined by 
\begin{equation}
\label{def-lambdas}
\lambda_{\sigma,s} \epsi_t 
\ov{\mathrm{def}}{=} \sigma(s,t) \epsi_{s t}	
\end{equation}
where $(\epsi_t)_{t \in G}$ is the canonical basis of $\ell^2_G$. We define the twisted group von Neumann algebra $\VN(G,\sigma)$ as the von Neumann subalgebra of $\B(\ell^2_G)$ generated by the $*$-algebra
$$
\C(G,\sigma) 
\ov{\mathrm{def}}{=} \mathrm{span}\big\{\lambda_{\sigma,s} \ : \ s \in G\big\}.
$$
For example, let $d \in \{2,3,\ldots,\infty\}$ and set $G=\Z^d$. To each $d \times d$ real skew symmetric matrix $\theta$, one may associate $\sigma_\theta \in \mathrm{Z}^2(\Z^d,\T)$ by $\sigma_\theta (m,n)=\e^{2\mathrm{i}\pi \langle m, \theta n\rangle}$ where $m,n \in \Z^d$. The resulting algebras $\T^d_{\theta}=\VN(\Z^d,\sigma_\theta)$ are the so-called $d$-dimensional noncommutative tori. See \cite{CXY} for a study of harmonic analysis on this algebra.

If $\sigma=1$, we obtain the left regular representation $\lambda \co G \to \B(\ell^2_G)$ and the group von Neumann algebra $\VN(G)$ of $G$. 

The von Neumann algebra $\VN(G,\sigma)$ is a finite algebra with trace given by $\tau_{G,\sigma}(x)=\big\langle\epsi_{e},x(\epsi_{e})\big\rangle_{\ell^2_G}$ where $x \in \VN(G,\sigma)$. In particular $\tau_{G,\sigma}(\lambda_{\sigma,s}) = \delta_{s,e}$. The generators $\lambda_{\sigma,s}$ satisfy the relations
\begin{equation}
\label{product-adjoint-twisted}
\lambda_{\sigma,s} \lambda_{\sigma,t} 
= \sigma(s,t) \lambda_{\sigma,st}, 
\qquad 
\big(\lambda_{\sigma,s}\big)^* 
= \ovl{\sigma(s,s^{-1})} \lambda_{\sigma,s^{-1}}.	
\end{equation}
Moreover, we have
$$
\tau_{G,\sigma}\big(\lambda_{\sigma,s}\lambda_{\sigma,t}\big) 
= \sigma(s,t) \delta_{s,t^{-1}},\quad s,t \in G.
$$

Given a discrete group $G$ and a $\T$-valued 2-cocycle $\sigma$, we can consider the fundamental unitary $W \co \epsi_t \ot \epsi_r  \mapsto \epsi_t \ot \epsi_{tr}$ on $\ell^2_G \ot_2 \ell^2_G$ and another unitary operator $\tilde\sigma \co \epsi_t \ot \epsi_r \to \sigma(t,r)\epsi_t \ot \epsi_r$ representing $\sigma$. We define the $\sigma$-fundamental unitary as the unitary operator
\begin{equation}
\label{eq:reg-omega-rep-unitary}
 W^{(\sigma)} 
= W \tilde\sigma \co \varepsilon_t \ot \varepsilon_r 
\mapsto \sigma(t,r) \varepsilon_t \ot \varepsilon_{t r}.
\end{equation}

\begin{lemma}
\label{Lemma-Twisted-coproduct}
Suppose that $\sigma$ and $\omega$ are $\T$-valued $2$-cocycles on a discrete group $G$. Then, for any $s \in G$ we have
$$
W^{(\omega)} \big(\lambda_{\sigma\cdot\omega,s} \ot \Id_{\ell^2_G}\big) \big(W^{(\omega)}\big)^* 
= \lambda_{\sigma,s} \ot \lambda_{\omega,s}.
$$
\end{lemma}

\begin{proof}
On the one hand, for any $s,t,r \in G$, using \eqref{def-lambdas} in the second equality and \eqref{eq:reg-omega-rep-unitary} in the third equality, we have
\begin{align*}
\MoveEqLeft
  W^{(\omega)} \big(\lambda_{\sigma\cdot\omega,s} \ot \Id_{\ell^2_G}\big)(\epsi_t \ot \epsi_r)   
		= W^{(\omega)} \big(\lambda_{\sigma\cdot\omega ,s}\epsi_t \ot \epsi_r\big)\\
		&=(\sigma\cdot\omega)(s,t)W^{(\omega)} \big(\epsi_{st} \ot \epsi_r\big)
		=\sigma(s,t) \omega(s,t)\omega(st,r)\epsi_{st} \ot \epsi_{str}.
\end{align*}
On the other hand, using \eqref{eq:reg-omega-rep-unitary} in the first equality and \eqref{def-lambdas} in the third equality, we have
\begin{align*}
\MoveEqLeft
   (\lambda_{\sigma,s} \ot \lambda_{\omega,s}) W^{(\omega)}(\epsi_t \ot \epsi_r)
		=(\lambda_{\sigma,s} \ot \lambda_{\omega,s})(\omega(t,r)\epsi_t \ot \epsi_{tr})\\
		&=\omega(t,r)(\lambda_{\sigma,s}\epsi_t \ot \lambda_{\omega,s} \epsi_{t r})
		=\sigma(s,t)\omega(t,r)\omega(s,tr)\epsi_{st} \ot \epsi_{st r}.
\end{align*}
Using \eqref{equation-2-cocycle} with $\omega$ instead of $\sigma$, we conclude that these quantities are equal.
\end{proof}

Using this lemma, we obtain a well-defined kind of ``twisted coproduct'' which is a unital normal $*$-monomorphism:
\begin{equation}
\label{Twisted-coproduct}
\begin{array}{cccc}
  \Delta_{\sigma,\omega}    \co &  \VN(G,\sigma \cdot \omega)   &  \longrightarrow   &  \VN(G,\sigma) \otvn \VN(G,\omega)  \\
           &  \lambda_{\sigma \cdot \omega,s}    &  \longmapsto & \lambda_{\sigma,s} \ot \lambda_{\omega,s}   \\
\end{array}
\end{equation}
A very particular case of this construction is considered in \cite[Corollary 2.2]{CXY} for noncommutative tori with $\sigma=1$, under the notation $x \mapsto \tilde{x}$.

Suppose $1 \leq p \leq \infty$. Then a linear map $T \co \L^p(\VN(G,\sigma)) \to \L^p(\VN(G,\sigma))$ is a (completely) bounded Fourier multiplier on $\L^p(\VN(G,\sigma))$ if $T$ is (completely) bounded (and normal if $p=\infty$) and if there exists a complex function $\varphi \co G \to \C$ such that $T\big(\lambda_{\sigma,s}\big)=\varphi_s\lambda_{\sigma,s}$ for any $s \in G$. In this case, we denote $T$ by
\begin{equation*}
\begin{array}{cccc}
   M_\varphi \co   & \L^p(\VN(G,\sigma))      &  \longrightarrow   & \L^p(\VN(G,\sigma))  \\
          & \lambda_{\sigma,s}  &  \longmapsto & \varphi_s\lambda_{\sigma,s}.  \\
\end{array}
\end{equation*}
We denote by $\mathfrak{M}^{p}(G,\sigma)$ the space of bounded Fourier multipliers on $\L^p(\VN(G,\sigma))$ and by $\mathfrak{M}^{p,\cb}(G,\sigma)$ the space of completely bounded Fourier multipliers on $\L^p(\VN(G,\sigma))$.

More generally, if $I$ is a set, we denote by $\mathfrak{M}_I^{p,\cb}(G,\sigma)$ the space of (normal if $p=\infty$) completely bounded operators $\Phi \co \L^p(\B(\ell^2_I) \otvn \VN(G,\sigma)) \to \L^p(\B(\ell^2_I) \otvn \VN(G,\sigma))$ such that $\Phi=[M_{\varphi_{ij}}]_{i,j \in I}$ for some functions $\varphi_{ij} \co G \to \C$. For a (normal if $p=\infty$) bounded operator $\Phi$, this is equivalent to  the existence of a family of functions $(\varphi_{ij} \co G \to \C)_{i,j \in I}$ such that 
\begin{equation}
\label{equ-cocycle-Fourier-multiplier}
(\tr \ot \tau_{G,\sigma})\big(T(e_{ij} \ot \lambda_{\sigma,s})(e_{kl} \ot \lambda_{\sigma,t})^*\big)=\varphi_{ij}(s)\delta_{s,t} \delta_{i,k}\delta_{j,l}
\end{equation}
for any $s,t \in G$ and any $i,j,k,l \in I$.

If $\sigma$ is a $\T$-valued 2-cocycle on a discrete group $G$ and if $H$ is a subgroup of $G$, we denote by $\sigma|H \co H \times H \to \T$ the restriction of $\sigma$ to $H \times H$. It follows from \cite[Section 4.26]{ZM} 
that there is a canonical normal unital $*$-monomorphism $J$ of $\VN(H,\sigma|H)$ into $\VN(G,\sigma)$ sending $\lambda_{\sigma|H,s}$ to $\lambda_{\sigma,s}$ for each $s \in H$ which is trace preserving. Its $\L^p$-extension $J_p \co \L^p(\VN(H,\sigma|H)) \to \L^p(\VN(G,\sigma)),\: \lambda_{\sigma|H,s} \mapsto \lambda_{\sigma,s}$ is a complete contraction for $1 \leq p \leq \infty$. 

Moreover, it is easy to see for $1 \leq p \leq \infty$ that the adjoint of $J_{p^*}$ (preadjoint if $p=1$) is given by $(J_{p^*})^* \co \L^p(\VN(G,\sigma)) \to \L^p(\VN(H,\sigma|H)),\: \lambda_{\sigma,s} \mapsto \delta_{s \in H} \lambda_{\sigma|H,s}$, which is again a complete contraction. Thus, for an element 
$$
T 
=[M_{\varphi_{ij}}]_{i,j \in I} \co S^p_I(\L^p(\VN(H,\sigma|H))) \to S^p_I(\L^p(\VN(H,\sigma|H)))
$$ 
of $\mathfrak{M}_I^{p,\cb}(H,\sigma|H)$, we can consider the completely bounded map 
$$
S 
\ov{\mathrm{def}}{=} (\Id_{S^p_I} \ot J_p) T (\Id_{S^p_I} \ot (J_{p^*})^*) \co S^p_I(\L^p(\VN(G,\sigma))) \to S^p_I(\L^p(\VN(G,\sigma))).
$$ 
We clearly have $\norm{S}_{\cb} \leq \norm{T}_{\cb}$ and using $(J_{p^*})^* J_p = \Id_{\L^p(\VN(H,\sigma|H))}$, we also have $\norm{T}_{\cb} \leq \norm{S}_{\cb}$. Thus we can identify isometrically $\mathfrak{M}_I^{p,\cb}(H,\sigma|H)$ as a subspace of the Banach space $\CB(\L^p(\B(\ell^2_I) \otvn \VN(G,\sigma)))$ by identifying $[M_{\varphi_{ij}}]_{i,j \in I}$ to $[M_{\tilde{\varphi}_{ij}}]_{i,j \in I}$ where $\widetilde{\varphi} \co G \to \C$ denotes the extension of $\varphi \co H \to \C$ on $G$ which is zero off $H$.
Moreover, we have a canonical contraction $\mathfrak{M}_I^{p,\cb}(G,\sigma) \to \mathfrak{M}_I^{p,\cb}(H,\sigma|H)$, sending $[M_{\varphi_{ij}}]_{ij}$ to $[M_{\varphi_{ij}|H}]_{ij}$.
Indeed, note that $[M_{\varphi_{ij}|H}]_{ij} = \Id_{S^p_I} \ot (J_{p^*})^* \cdot [M_{\varphi_{ij}}] \cdot \Id_{S^p_I} \ot J_p$.

\subsection{Complementation for Schur multipliers and Fourier multipliers on discrete groups}
\label{Complementation-for-discrete-Schur-multipliers-and-Fourier-multipliers}

The following theorem generalizes an average trick of Haagerup \cite[proof of Lemma 2.5]{Haa3}\footnote{\thefootnote. We warn the reader that the assumption ``normal'' is lacking in \cite[Lemma 2.5]{Haa3} for maps defined on $\mathfrak{M}(\Gamma)$.}. The important point of the proof (for $1 \leq p \leq \infty$) is the fact that the map $\Delta$ below is trace preserving. 

\begin{thm}
\label{Prop-complementation-Schur-Fourier}
Let $I$ be an index set equipped with the counting measure. Let $G$ be a discrete group equipped with two normalized $\T$-valued 2-cocycles $\sigma,\omega$. Suppose $1 \leq p \leq \infty$. If $p\not=\infty$, we suppose that $\VN(G,\omega)$ has $\QWEP$. Let $T \co S^p_I(\L^p(\VN(G,\sigma))) \to S^p_I(\L^p(\VN(G,\sigma)))$ be a completely bounded operator. For any $i,j \in I$, we define the complex function $\varphi_{ij} \co G \to \C$ by
$$
\varphi_{ij}(s)
\ov{\mathrm{def}}{=} (\tr \ot \tau_{G,\sigma})\big(T(e_{ij} \ot \lambda_{\sigma,s})(e_{ij} \ot \lambda_{\sigma,s})^*\big), \qquad s \in G.
$$
Then the map
\begin{equation*}
\begin{array}{cccc}
P_{I,G}^p \co & \CB(S^p_I(\L^p(\VN(G,\sigma)))) &  \longrightarrow & \CB(S^p_I(\L^p(\VN(G,\sigma \cdot \omega)))) \\
  &   T  &  \longmapsto  &  [M_{\varphi_{ij}}]  \\
\end{array}
\end{equation*}
is a well-defined contractive map into $\mathfrak{M}_I^{p,\cb}(G,\sigma \cdot\omega)$.
There are the following additional properties of $P_{I,G}^p$.
\begin{enumerate}
\item If $\omega=1$, the map $P_{I,G}^p$ is a projection onto $\mathfrak{M}_I^{p,\cb}(G,\sigma)$. 

\item For $p=\infty$, the same assertions are true by replacing $\CB(S^p_I(\L^p(\VN(G,\sigma))))$ by the space $\CB_{\w^*}(\B(\ell^2_I) \otvn \VN(G,\sigma))$.

\item If $T$ is completely positive then the map $P_{I,G}^p(T)$ is completely positive. 

\item For any values $p,q \in [1,\infty]$ and any $T \in \CB(S^p_I(\L^p(\VN(G,\sigma)))) \cap \CB(S^q_I(\L^q(\VN(G,\sigma))))$ we have $(P_{I,G}^p(T))([x_{ij}]) = (P_{I,G}^q(T))([x_{ij}])$ for any element $[x_{ij}]$ of $S^p_I(\L^p(\VN(G,\sigma\cdot\omega))) \cap S^q_I(\L^q(\VN(G,\sigma\cdot\omega)))$. So the mappings $P_{I,G}^p$, $1 \leq p \leq \infty$, are compatible.

\item Furthermore, if $p=\infty$ and if $T$ is selfadjoint then $P_{I,G}^\infty(T)$ is selfadjoint. If $T=[T_{ij}]$ is a normal operator where $T_{ij} \co \VN(G,\sigma) \to \VN(G,\sigma)$ and if each $T_{ii}$ is unital then $P_{I,G}^\infty(T)$ is unital.

\item We have an isometry 
\[
\mathfrak{M}_I^{p,\cb}(G,\sigma) 
= \mathfrak{M}_I^{p,\cb}(G,\sigma \cdot \omega) .
\]

\end{enumerate}
\end{thm}

\begin{proof}
Using the map \eqref{Twisted-coproduct}, it is easy to see that we can define a well-defined unital normal $*$-isomorphism 
$$
\Delta \co \M_I(\VN(G,\sigma \cdot \omega)) \to \M_I(\VN(G,\sigma)) \otvn \M_I(\VN(G,\omega))
$$ 
onto the sub von Neumann algebra $\Delta\big(\M_I(\VN(G,\sigma \cdot \omega))\big)$ of $\M_I(\VN(G,\sigma)) \otvn \M_I(\VN(G,\omega))$ such that
$$
\Delta(e_{ij} \ot \lambda_{\sigma \cdot \omega,s})
=e_{ij} \ot \lambda_{\sigma,s} \ot e_{ij} \ot \lambda_{\omega,s},\qquad  s \in G.
$$ 
Using the flip $\M_I \otvn \VN(G,\sigma) \otvn \M_I \otvn \VN(G,\omega) \to \M_I \otvn \M_I \otvn \VN(G,\sigma) \otvn \VN(G,\omega), $ $x \ot y \ot z \ot t\mapsto x \ot z \ot y \ot t$, it is not difficult to check with \cite[Theorem 6.2]{Str} that the operator $\Delta$ preserves the traces. Consequently $\Delta$ is a Markov map in the sense of Section \ref{sec:Markov-maps} and admits a canonical extension $\Delta_p \co S^p_I(\L^p(\VN(G,\sigma\cdot \omega))) \to \L^p(\B(\ell^2_I) \otvn \VN(G,\sigma) \otvn \B(\ell^2_I) \otvn \VN(G,\omega))$ which is completely contractive and completely positive (and normal if $p=\infty$).

Suppose that $T \co S^p_I(\L^p(\VN(G,\sigma))) \to S^p_I(\L^p(\VN(G,\sigma)))$ is a completely bounded operator. If $\VN(G,\omega)$ is QWEP then by \eqref{multiplicativity-T} the operator 
$$
P_{I,G}^p(T)
=(\Delta^*)_{p}\big(T \ot \Id_{S^p_I(\L^p(\VN(G,\omega)))}\big)\Delta_p
$$ 
is a completely bounded map on the space $S^p_I(\L^p(\VN(G,\sigma \cdot \omega)))$. Moreover, we have
\begin{align*}
\MoveEqLeft
\bnorm{P_{I,G}^p(T)}_{\cb, S^p_I(\L^p(\VN(G,\sigma \cdot \omega)))\to S^p_I(\L^p(\VN(G,\sigma\cdot \omega)))} 
\leq \bnorm{(\Delta^*)_{p} \big(T \ot \Id_{S^p_I(\L^p(\VN(G,\omega)))}\big)\Delta_p}_{\cb}\\
& \leq \norm{T}_{\cb,S^p_I(\L^p(\VN(G,\sigma))) \to S^p_I(\L^p(\VN(G,\sigma)))}.
\end{align*}
Thus $P_{I,G}^p$ is contractive. For any $i,j,k,l \in I$ and any $s,s' \in G$, we have
\begin{align*}
\MoveEqLeft
(\tr \ot \tau_{G,\sigma \cdot \omega})\Big(\big((\Delta^*)_{p}\big(T \ot \Id_{S^p_I(\L^p(\VN(G,\omega)))} \big)\Delta_p(e_{ij}\ot \lambda_{\sigma \cdot \omega,s})\big) (e_{kl} \ot \lambda_{\sigma \cdot \omega,s'})^*\Big)  \\
		&=(\tr \ot \tau_{G,\sigma \cdot \omega})\big((\Delta^*)_{p}\big(T \ot \Id_{S^p_I(\L^p(\VN(G,\omega)))}\big)(e_{ij} \ot \lambda_{\sigma,s} \ot e_{ij} \ot \lambda_{\omega,s})(e_{kl}^* \ot \lambda_{\sigma \cdot \omega,s'}^*)\big) \\
		&=(\tr \ot \tau_{G,\sigma \cdot \omega})\Big(\big((\Delta^*)_{p}\big(T(e_{ij} \ot \lambda_{\sigma,s}) \ot e_{ij} \ot \lambda_{\omega,s} \big)\big) \big(e_{lk} \ot \ovl{(\sigma\cdot \omega)(s',s'^{-1})}\lambda_{\sigma \cdot \omega,s'^{-1}} \Big)\\
		&=\ovl{(\sigma \cdot \omega)(s',s'^{-1})}(\tr \ot \tau_{G,\sigma} \ot \tr \ot \tau_{G,\omega})\Big(\big(T(e_{ij} \ot \lambda_{\sigma,s}) \ot e_{ij} \ot \lambda_{\omega,s}\big) \Delta_{p^*}\big(e_{lk} \ot \lambda_{\sigma\cdot \omega,s'^{-1}}\big)\Big) \\
		&=\ovl{(\sigma \cdot \omega)(s',s'^{-1})}(\tr \ot \tau_{G,\sigma} \ot \tr \ot \tau_{G,\omega})\big((T(e_{ij} \ot \lambda_{\sigma,s}) \ot e_{ij} \ot \lambda_{\omega,s})\\
		&\qquad \qquad \qquad \qquad\qquad \qquad \qquad \qquad\qquad \qquad \qquad \qquad (e_{lk} \ot \lambda_{\sigma,s'^{-1}} \ot e_{lk} \ot \lambda_{\omega,s'^{-1}})\big) \\
		&=\ovl{(\sigma \cdot \omega)(s',s'^{-1})}(\tr \ot \tau_{G,\sigma})\big(T(e_{ij} \ot \lambda_{\sigma,s}) (e_{lk} \ot \lambda_{\sigma,s'^{-1}})\big)(\tr \ot \tau_{G,\omega})\big(e_{ij}e_{lk}  \ot \lambda_{\omega,s}\lambda_{\omega,s'^{-1}} \big) \\
		&=\ovl{(\sigma\cdot \omega)(s',s'^{-1})}\omega(s',s'^{-1})(\tr \ot \tau_{G,\sigma})\big(T(e_{ij} \ot \lambda_{\sigma,s}) (e_{lk} \ot \lambda_{\sigma,s'^{-1}})\big) \delta_{i,k} \delta_{j,l}\delta_{s,s'}\\
		&=(\tr \ot \tau_{G,\sigma})\big(T(e_{ij} \ot \lambda_{\sigma,s}) (e_{lk} \ot \ovl{\sigma(s',s'^{-1})}\lambda_{\sigma,s'^{-1}})\big)\delta_{i,k} \delta_{j,l} \delta_{s,s'}\\
		&=(\tr \ot \tau_{G,\sigma})\big(T(e_{ij} \ot \lambda_{\sigma,s}) (e_{kl} \ot \lambda_{\sigma,s'})^*\big) \delta_{i,k} \delta_{j,l} \delta_{s,s'}.	
\end{align*}
Hence according to \eqref{equ-cocycle-Fourier-multiplier}, $P_{I,G}^p(T)$ is the operator $[M_{\varphi_{ij}}]$. 

1. If we choose $\omega = 1$, according to the discussion at the end of Section \ref{section-twisted-von-Neumann}, $P^p_{I,G}(T)$ belongs to $\mathfrak{M}^{p,\cb}_I(G,\sigma) \subset \CB(S^p_I(\L^p(\VN(G,\sigma))))$. If $T = [M_{\psi_{ij}}]$ right from the beginning, for some symbols $\psi_{ij} \co G \to \C$, then for $s \in G$
\begin{align*}
\MoveEqLeft
\varphi_{ij}(s) = (\tr \ot \tau_{G,\sigma})\big(T(e_{ij} \ot \lambda_{\sigma,s})(e_{ij} \ot \lambda_{\sigma,s})^*\big) \\
& = \psi_{ij}(s) (\tr \ot \tau_{G,\sigma})\big((e_{ij} \ot \lambda_{\sigma,s})(e_{ij} \ot \lambda_{\sigma,s})^* \big) \\
& = \psi_{ij}(s) \tau_{G,\sigma}(\lambda_{\sigma,s} \lambda_{\sigma,s}^*) = \psi_{ij}(s) \overline{\sigma(s,s^{-1})} \tau_{G,\sigma}(\lambda_{\sigma,s} \lambda_{\sigma,s^{-1}}) \\
& = \psi_{ij}(s) \overline{\sigma(s,s^{-1})} \sigma(s,s^{-1}) \delta_{s,s} = \psi_{ij}(s).
\end{align*}
Thus, in this case $P_{I,G}^p(T) = T$, so that $P_{I,G}^p$ is indeed a projection onto $\mathfrak{M}^{p,\cb}_I(G,\sigma)$.

2. We turn to the case $p = \infty$.
Since multipliers on the level $p = \infty$ are by definition normal mappings, we need to define $P^\infty_{I,G}(T) = (\Delta^*)_\infty (P_{\w^*}(T) \ot \Id_{\B(\ell^2_I) \otvn \VN(G,\omega)}) \Delta_\infty$, where $P_{\w^*} \co \CB(M) \to \CB(M)$ with $M = \B(\ell^2_I) \otvn \VN(G,\sigma)$ is the projection onto normal maps from Proposition \ref{Prop-recover-weak-star-continuity}.
Note that $P_{\w^*}$ is contractive and preserves complete positivity according to this proposition.
Moreover, then $P_{\w^*}(T) \ot \Id_{\B(\ell^2_I) \otvn \VN(G,\omega)}$ is also normal,  and since $(\Delta^*)_\infty$ and $\Delta_\infty$ are normal, and normality is preserved under composition, we finally infer that $P^\infty_{I,G}(T)$ is normal.
Moreover, as $e_{ij} \ot \lambda_{\sigma,s} \in S^1_I(\L^1(\VN(G,\sigma)))$ for any $i,j \in I$ and $s \in G$, $P_{\w^*}(T)(e_{ij} \ot \lambda_{\sigma,s}) = T(e_{ij} \ot \lambda_{\sigma,s})$, so that $P^\infty_{I,G}(T)$ is the multiplier with symbol $\varphi_{ij}$ from the statement.

3. Note that if $T$ is completely positive then $P_{I,G}^p(T)$ is also a completely positive map by composition.

4. The statement about the compatibility of $P_{I,G}^p$ for different values of $p \in [1,\infty]$ follows directly from the defining formula of $P_{I,G}^p$ and the fact that $(\Delta^*)_p,\Delta_p$ and $\Id_{S^p_I(\L^p(\VN(G,\sigma)))}$ are all compatible for two different values of $p$.

5. Suppose $p=\infty$. If $T \co \M_I(\VN(G,\sigma)) \to \M_I(\VN(G,\sigma))$ is selfadjoint then for any $s \in G$ and any $i,j \in I$ we have
\begin{align*}
\varphi_{ij}(s)
& =(\tr \ot \tau_{G,\sigma})\big(T(e_{ij} \ot \lambda_{\sigma,s})(e_{ij} \ot \lambda_{\sigma,s})^*\big)
=(\tr \ot \tau_{G,\sigma})\big(e_{ij} \ot \lambda_{\sigma,s}(T(e_{ij} \ot \lambda_{\sigma,s}))^*\big) \\
&		=\ovl{\varphi_{ij}(s)}.
\end{align*}
It is not difficult to conclude that $P_{I,G}^\infty(T)$ is selfadjoint.

Suppose that $T=[T_{ij}]$ is a matrix of operators such that each $T_{ii}$ is unital, i.e. $T(e_{ii} \ot \lambda_{\sigma,e})=e_{ii} \ot \lambda_{\sigma,e}$. We have
\begin{align*}
\MoveEqLeft
   \varphi_{ii}(e)
=(\tr \ot \tau_{G,\sigma})\big(T(e_{ii} \ot \lambda_{\sigma,e})(e_{ii} \ot \lambda_{\sigma,e})^*\big)
		=(\tr \ot \tau_{G,\sigma})\big((e_{ii} \ot \lambda_{\sigma,e})(e_{ii} \ot \lambda_{\sigma,e})^*\big)=1.
\end{align*} 
We conclude that $P_{I,G}^\infty(T)$ is unital.

6. It suffices to use the map $P^p_{I,G}|_{\mathfrak{M}_I^{p,\cb}(G,\sigma)}$ and a symmetry argument.
\end{proof}

\begin{remark}\normalfont \label{Rem-quantum-groups}
This result admits a generalization for \textit{unimodular discrete} quantum groups. We warn the reader that the formula given in \cite[Remark 7.6]{DFSW} for unimodular locally compact quantum groups does not make sense\footnote{\thefootnote. With the notations of \cite[Remark 7.6]{DFSW}, if we identify $\L^\infty(\hat{\G})$ with $\L^\infty(\R)$, and $x$ with a function $f$, we obtain $L(f)=\int_\R  
 [\Phi(f_t)\big]_{-t} \d\mu_{\R}(t)$ where $\Phi \co \L^\infty(\R) \to \L^\infty(\R)$ and where we use translations by $t$ and $-t$. This integral is meaningless. We would like to thank Adam Skalski for his confirmation of this problem by email \textit{on his own initiative}.} already in the case of the locally compact group $\R$ of real numbers.  
\end{remark}

The case $G=\{e\}$ gives the following complementation for the space of completely bounded Schur multipliers. Compare to \cite[Proposition 2.6]{Arh2}.

\begin{cor}
\label{Prop-complementation-Schur}
Suppose that $I$ is equipped with the counting measure. Let $T \co S^p_I \to S^p_I$ be a completely bounded operator. We define the matrix $\varphi$ by
\begin{equation}
	\label{Symbole-complementation-Schur}
	\varphi_{ij}
	=\tr\big(T(e_{ij})e_{ij}^*\big), \qquad i,j \in I.
\end{equation}
Then the map $P_{I}^p \co  \CB(S^p_I) \to \CB(S^p_I)$, $T \mapsto M_\varphi$ is a well-defined contractive projection onto the subspace $\mathfrak{M}_{I}^{p,\cb}$ of completely bounded Schur multipliers. Moreover, if $T$ is completely positive then the Schur multiplier $P_{I}^p(T)$ is completely positive. 
For $p=\infty$ the same assertions are true by replacing $\CB(S^p_I)$ by the space $\CB_{\w^*}(\B(\ell^2_I))$.
\end{cor}

The case where $I$ contains one element and a symmety argument show that the complete positivity of a multiplier is independent from the $\T$-valued $2$-cocycle $\sigma$  (this first point can be proved as the point 6 of the Theorem \ref{Prop-complementation-Schur-Fourier}).

\begin{cor}
\label{Prop-caracterisation-cp-discrete-G-sigma}
Let $G$ be a discrete group. Let $\sigma$ be a $\T$-valued $2$-cocycle on $G$. Suppose $1 \leq p \leq \infty$. If $p \not=\infty$, we suppose that $\VN(G)$ and $\VN(G,\sigma)$ have $\QWEP$. Let $\varphi \co G \to \C$ be a complex function.
\begin{enumerate}
	\item Then $\varphi$ induces a completely positive multiplier $M_{\varphi} \co \L^p(\VN(G,\sigma)) \to \L^p(\VN(G,\sigma))$ if and only if $\varphi$ induces a completely positive multiplier $M_{\varphi} \co \L^p(\VN(G)) \to \L^p(\VN(G))$.
	\item Then $\varphi$ induces a completely bounded multiplier $M_{\varphi} \co \L^p(\VN(G,\sigma)) \to \L^p(\VN(G,\sigma))$ if and only if $\varphi$ induces a completely bounded multiplier $M_{\varphi} \co \L^p(\VN(G)) \to \L^p(\VN(G))$. In this case, we have the equality
\begin{equation}
	\label{Normes-cb-cocycl-et-sans}
	\norm{M_{\varphi}}_{\cb,\L^p(\VN(G,\sigma)) \to \L^p(\VN(G,\sigma))}=\norm{M_{\varphi}}_{\cb,\L^p(\VN(G)) \to \L^p(\VN(G))}.
\end{equation}
\end{enumerate}
\end{cor}

Note that \cite[Proposition 4.3]{BC1} gives a proof of \eqref{Normes-cb-cocycl-et-sans} for $p=\infty$.

In the following result, $P_{I,G}^p$ is the map of Theorem \ref{Prop-complementation-Schur-Fourier} with $\omega=1$.

\begin{thm}
Let $I$ be an index set equipped with the counting measure. Let $G$ be a discrete group equipped with a normalized $\T$-valued 2-cocycle $\sigma$ and $H$ be a subgroup of $G$. Suppose $1 \leq p \leq \infty$. If $p \not=\infty$, we suppose that $\VN(G,\sigma)$ has $\QWEP$. Then we have a natural contraction $Q^p_H \co \mathfrak{M}_I^{p,\cb}(G,\sigma) \to \mathfrak{M}_I^{p,\cb}(H,\sigma|H)$ sending $[M_{\varphi_{ij}}] \mapsto [M_{\varphi_{ij}|H}]$ and an isometric embedding $J^p_H \co \mathfrak{M}_I^{p,\cb}(H,\sigma|H) \to \mathfrak{M}_I^{p,\cb}(G,\sigma)$ sending $[M_{\varphi_{ij}}] \mapsto [M_{\varphi_{ij} \delta_{\cdot \in H}}]$, so that $J^p_H \circ Q^p_H$ is a projection.
Then $P_{I,H}^p \ov{\mathrm{def}}{=} J^p_H \circ Q^p_H \circ P_{I,G}^p$ defines a projection $\CB(S^p_I(\L^p(\VN(G,\sigma)))) \to \CB(S^p_I(\L^p(\VN(G,\sigma))))$ having its image in $J^p_H(\mathfrak{M}_I^{p,\cb}(H,\sigma|H))$ and satisfying the same properties that the previous map $P_{I,G}^p$: it preserves complete positivity, is compatible for different values of $p$, and preserves selfadjointness and unital mappings.
\end{thm}

\begin{proof}
For the fact that $Q^p_H$ is a contraction and that $J^p_H$ is an isometric embedding, we refer to the end of Section \ref{section-twisted-von-Neumann}. It is elementary to check that $J^p_H Q^p_H$ is a projection. We have $(P^p_{I,H})^2 = J^p_H Q^p_H P^p_{I,G} J^p_H Q^p_H P^p_{I,G} = J^p_H Q^p_H J^p_H Q^p_H P^p_{I,G} = J^p_H Q^p_H P^p_{I,G} = P^p_{I,H}$, since $P^p_{I,G}$ is the identity on multipliers. Thus, $P^p_{I,H}$ is a projection. As $J^p_H Q^p_H([M_{\varphi_{ij}}]) = \Id_{S^p_I} \ot J_p (J_{p^*})^* \cdot [M_{\varphi_{ij}}] \cdot \Id_{S^p_I} \ot J_p (J_{p^*})^*$ and the mapping $J_p$ from the end of Section \ref{section-twisted-von-Neumann} is completely positive, thus by Lemma \ref{Lemma-adjoint-cp} also $(J_{p^*})^*$, we infer that $P^p_{I,H}$ preserves complete positivity. The compatibility of $P^p_{I,H}$ for different values of $p$ follows from that of $P^p_{I,G}$ and of $J^p_H$ and $Q^p_H$. If $p = \infty$ and $T$ is selfadjoint, $P^p_{I,G}(T)$ is selfadjoint, i.e. its symbol $\varphi_{ij}(s)$ takes real values for all $i,j \in I$ and $s \in G$. Then the symbol of $P^\infty_{I,H}(T)$ is $\varphi_{ij}(s) \cdot 1_{H}(s)$ which also has real values, so that $P^\infty_{I,H}$ preserves selfadjointness. In a similar way, if $T$ is normal and all $T_{ii}$ are unital, then $P^\infty_{I,G}(T)$ is unital, which amounts in $\varphi_{ii}(e) = 1$ for all $i \in I$. Since $e \in H$, we conclude that $P^\infty_{I,H}(T)$ is unital.
\end{proof}

The case where $I$ contains one element and where $\sigma=\omega=1$ gives the following.
 
\begin{cor}
\label{Prop-complementation-Fourier-subgroups}
Let $G$ be a discrete group and $H$ be a subgroup of $G$. Suppose $1 \leq p<\infty$. If $p \not=\infty$, we suppose that $\VN(G)$ has $\QWEP$. Let $T \co \L^p(\VN(G)) \to \L^p(\VN(G))$ be a completely bounded operator.  We define the complex function $\varphi \co H \to \C$ by
$$
\varphi(s)
=\tau_{G}\big(T(\lambda_{s})(\lambda_{s})^*\big), \qquad s \in H.
$$
Then the map $P_{H}^p \co \CB(\L^p(\VN(G))) \to \CB(\L^p(\VN(G)))$, $T \mapsto  M_{\varphi}$ is a well-defined contractive projection onto the subspace $\mathfrak{M}^{p,\cb}(H)$ (identified as a subspace of $\CB(\L^p(\VN(G)))$). Moreover, if $T$ is completely positive then the map $P_{H}^p(T)$ is completely positive. For $p=\infty$ the same assertions are true by replacing $\CB(\L^p(\VN(G)))$ by the space $\CB_{\w^*}(\VN(G))$.
\end{cor}

\subsection{Description of the decomposable norm of multipliers}
\label{subsec-description-decomposable-norm}

The following theorem is our first result describing decomposable multipliers on noncommutative $\L^p$-spaces. 


\begin{thm}
\label{prop-non-amenable-discrete-Fourier-multiplier-dec-infty}
Let $G$ be a discrete group equipped with a normalized $\T$-valued 2-cocycle $\sigma$. Suppose $1 \leq p \leq \infty$. We suppose that $\VN(G)$ and $\VN(G,\sigma)$ have $\QWEP$. Then a function $\phi \co G \to \C$ induces a decomposable Fourier multiplier on $\L^p(\VN(G,\sigma))$ if and only if it induces a decomposable Fourier multiplier on $\VN(G)$. 
\end{thm}

\begin{proof}
$\Rightarrow$: Let $M_\phi \co \L^p(\VN(G,\sigma)) \to \L^p(\VN(G,\sigma))$ be a decomposable Fourier multiplier. By Proposition \ref{prop-decomposable-se-decompose-en-cp}, we can write 
$
M_\phi
=T_1-T_2+\i(T_3-T_4)
$ 
where each $T_j$ is a completely positive map on $\L^p(\VN(G,\sigma))$. Using the projection $P_G^p$ of Theorem \ref{Prop-complementation-Schur-Fourier}  with $G=H$, $I=\{0\}$ and $\omega = 1$, we obtain that
$$
M_\phi
=P_G^p(M_\phi)
=P_G^p\big(T_1-T_2+\i(T_3-T_4)\big)
=P_G^p(T_1)-P_G^p(T_2)+\i\big(P_G^p(T_3)-P_G^p(T_4)\big)
$$
and that each $P_G^p(T_j)=M_{\phi_j}$ is a completely positive Fourier multiplier on $\L^p(\VN(G,\sigma))$. By Corollary \ref{Prop-caracterisation-cp-discrete-G-sigma}, each $\phi_j$ also induces a completely positive Fourier multiplier on $\L^p(\VN(G))$. By the proof of \cite[Proposition 4.2]{DCH}, we see that the (continuous) function $\phi_j$ is\footnote{\thefootnote. Here we use the inclusion $\VN(G) \subset \L^p(\VN(G))$ and the realization of $\L^p(\VN(G))$ as a subspace of measurable operators. See also Proposition \ref{Th-cp-Fourier-multipliers} which is a more general result.} positive definite. Hence it induces a completely positive Fourier multiplier on $\VN(G)$ again by \cite[Proposition 4.2]{DCH}. We conclude that $\phi$ induces a decomposable Fourier multiplier on $\VN(G)$.

$\Leftarrow$: Let $M_\phi \co \VN(G) \to \VN(G)$ be a decomposable Fourier multiplier. Similarly, with Theorem \ref{Prop-complementation-Schur-Fourier}, we can write $M_\phi=M_{\phi_1}-M_{\phi_2}+\i(M_{\phi_3}-M_{\phi_4})$ where each $M_{\phi_j} \co \VN(G) \to \VN(G)$ is completely positive. By \cite[page 216]{Har}\footnote{\thefootnote. See also Lemma \ref{lemma-inclusion-Fourier-multipliers} which is a generalization.}, each Fourier multiplier $\phi_j$ induces a completely positive\footnote{\thefootnote. We use here the fact, left to the reader, that if $T \co M \to N$ is a completely positive map which induces a bounded map $T_p \co \L^p(M) \to \L^p(N)$ then $T_p$ is also completely positive.} multiplier on $\L^p(\VN(G))$ and also on $\L^p(\VN(G,\sigma))$ by Corollary \ref{Prop-caracterisation-cp-discrete-G-sigma}. Using Proposition \ref{prop-decomposable-se-decompose-en-cp}, we conclude that $\phi$ induces a decomposable Fourier multiplier on $\L^p(\VN(G,\sigma))$.
\end{proof}

The following is essentially \cite[Section 1.17.1 Theorem 1]{Tri}, see also \cite[page 58]{Pis7}.

\begin{lemma}
\label{Lemma-interpolation}
Let $(E_0,E_1)$ be an interpolation couple and let $C$ be a complemented subspace of $E_0+E_1$. We assume that the corresponding bounded projection $P \co E_0+E_1 \to E_0+E_1$ satisfies $P(E_i) \subset E_i$ and that the restriction $P \co E_i \to E_i$ is bounded for $i=0,1$. Then $(E_0 \cap C,E_1 \cap C)$ is an interpolation couple and the canonical inclusion $J \co C \to E_0+E_1$ induces an isomorphim $\tilde{J}$ from $(E_0 \cap C,E_1 \cap C)^\theta$ onto the subspace $P((E_0,E_1)^\theta)=(E_0,E_1)^\theta \cap C$ of $(E_0,E_1)^\theta$. More precisely, if $x \in (E_0 \cap C,E_1 \cap C)^\theta$, we have
$$
\norm{\tilde{J}(x)}_{(E_0,E_1)^\theta} 
\leq \norm{x}_{(E_0 \cap C,E_1 \cap C)^\theta} 
\leq \max\big\{\norm{P}_{E_0 \to E_0},\norm{P}_{E_1 \to E_1} \big\} \norm{\tilde{J}(x)}_{(E_0,E_1)^\theta}.
$$
In particular, if $\max\{\norm{P}_{E_0 \to E_0},\norm{P}_{E_1 \to E_1}\} = 1$ then $\tilde{J}$ is an isometry.
\end{lemma}

Let $(E_0,E_1)$ be an interpolation couple. If $T_0 \co E_0 \to E_0$, $T_1 \co E_1 \to E_1$ are (completely) bounded maps such that $T_0$ and $T_1$ agree on $E_0 \cap E_1$, then we say that $T_0$ and $T_1$ are compatible. In this case, it is elementary and well-known that there exists a unique (completely) bounded map $T_0+T_1 \co E_0 + E_1 \to E_0 + E_1$ 
which extends $T_0$ and $T_1$ and we have $\norm{T_0+T_1}_{E_0 + E_1 \to E_0 + E_1} \leq \max\{\norm{T_0}_{E_0 \to E_0},\norm{T_1}_{E_1 \to E_1}\}$ and similarly for the completely bounded norms. Moreover, if $T_0$ and $T_1$ are projections onto $F_0$ and $F_1$ then $T_0+T_1$ is a projection onto $F_0+F_1$. 

It allows us to deduce the following description of decomposable Fourier multipliers on amenable groups.

Let $G$ be a discrete group. Recall that the group von Neumann algebra $\VN(G)$ is approximately finite-dimensional if and only if $G$ is amenable, see \cite[Theorem 3.8.2]{SS}. Using Corollary \ref{Prop-complementation-Fourier-subgroups} with $H=G$, we obtain the following result.

\begin{thm}
\label{prop-amenable-discrete-Fourier-multiplier-dec-infty}
Let $G$ be an amenable discrete group. Suppose $1 \leq p \leq \infty$. Then a function $\phi \co G \to \C$ induces a decomposable Fourier multiplier $M_\phi \co \L^p(\VN(G)) \to \L^p(\VN(G))$ if and only if it induces a (completely) bounded Fourier multiplier $M_\phi \co \VN(G) \to \VN(G)$. In this case, we have the isometric identity
$$
\norm{M_\phi}_{\dec,\L^p(\VN(G)) \to \L^p(\VN(G))} 
= \norm{M_\phi}_{\cb, \VN(G) \to \VN(G)}
=\norm{M_\phi}_{\VN(G) \to \VN(G)}.
$$
\end{thm}

\begin{proof}
By \cite[Corollary 1.8]{DCH}, since $G$ is amenable, we have $\mathfrak{M}^{\infty}(G)=\mathfrak{M}^{\infty,\cb}(G)$ isometrically. The first part is Theorem \ref{prop-non-amenable-discrete-Fourier-multiplier-dec-infty} using \cite[Theorem 2.1]{Haa} (which says that the decomposable norm and the completely bounded norm coincide for operators on approximately finite-dimensional von Neumann algebras). By \cite{Har}, we have $\mathfrak{M}^{\infty}(G)=\mathfrak{M}^{1}(G)$ isometrically. Now, we use Lemma \ref{Lemma-interpolation} with the interpolation couple \eqref{Regular-as-interpolation-space} and with $C=\mathfrak{M}^{\infty}(G)$ and we also use the projection Corollary \ref{Prop-complementation-Fourier-subgroups} with $H=G$. Note that we have isometrically
\begin{align*}
\MoveEqLeft
\big(\CB_{\mathrm{w}^*}(\VN(G)) \cap \mathfrak{M}^{\infty}(G),\CB(\L^1(\VN(G))) \cap \mathfrak{M}^{\infty}(G)\big)^{\frac{1}{p}}
		=\big(\mathfrak{M}^{\infty}(G),\mathfrak{M}^{\infty}(G)\big)^{\frac{1}{p}}
		=\mathfrak{M}^{\infty}(G).
\end{align*}
We infer that the space 
$
\Reg(\L^p(\VN(G))) \cap \mathfrak{M}^{\infty}(G)
=(\CB_{\mathrm{w}^*}(\VN(G)),\CB(\L^1(\VN(G))))^{\frac1p} \cap \mathfrak{M}^{\infty}(G)
$ 
equipped with the regular norm $\norm{\cdot}_{\reg,\L^p(\VN(G)) \to \L^p(\VN(G))}$ is isometric to the space $\mathfrak{M}^{\infty,\cb}(G)$.
We finally employ Theorem \ref{thm-dec=reg-hyperfinite} to pass isometrically from regular operators to decomposable operators.
\end{proof}

Similarly, we obtain the following description of decomposable Schur multipliers with the projection of Corollary \ref{Prop-complementation-Schur}.

\begin{thm}
\label{prop-continuous-Schur-multiplier-dec-infty}
Suppose $1 \leq p \leq \infty$. Then a function $\phi \co I \times I \to \C$ induces a decomposable Schur multiplier on $S^p_I$ if and only if it induces a (completely) bounded Schur multiplier on $\B(\ell^2_I)$. In this case, we have the isometric identity
$$
\norm{M_\phi}_{\dec, S^p_I \to S^p_I} 
=\norm{M_\phi}_{\reg, S^p_I \to S^p_I} 
=\norm{M_\phi}_{\cb,\B(\ell^2_I) \to \B(\ell^2_I)}
=\norm{M_\phi}_{\B(\ell^2_I) \to \B(\ell^2_I)}.
$$
\end{thm}

\section{Approximation by discrete groups}
\label{sec:approximation-discrete-groups}

The complementation Theorem \ref{Prop-complementation-Schur-Fourier} from Chapter \ref{sec:Decomposable-Schur-multipliers-and-Fourier-multipliers} is stated only for \textit{discrete} groups $G$. In order to exhibit a suitable class of admissible \textit{non-discrete} locally compact groups, approximations by discrete subgroups of $G$ become important. In this section, we introduce and study several notions of approximation which are of independent interest, but which will be important in the subsequent Chapter \ref{sec:Locally-compact-groups}.

\subsection{Preliminaries}

\paragraph{Chabauty-Fell topology.} For a topological space $Y$, let $\mathscr{F}(Y)$ denote the set of closed subsets of $Y$. For a compact subset $K$ and an open subset $U$ of $Y$, set\footnote{\thefootnote. Note that $\mathcal{O}_{K_1} \cap \cdots \cap \mathcal{O}_{K_m} = \mathcal{O}_{K_1 \cup \cdots \cup K_m}$.}
$$
\mathcal{O}_K \ov{\mathrm{def}}{=} \{F \in \mathscr{F}(Y) \ : \  F \cap K = \emptyset \} 
\quad \text{ and } \quad
\mathcal{O}_{U}' \ov{\mathrm{def}}{=} \{F \in \mathscr{F}(Y) \ : \ F \cap U \not= \emptyset\}.
$$
The finite intersections $\mathcal{O}_{K_1} \cap \cdots \cap \mathcal{O}_{K_m} \cap \mathcal{O}_{U_1}' \cap \cdots \cap \mathcal{O}_{U_n}'$ constitute a basis of a topology on $\mathscr{F}(Y)$, called the Chabauty-Fell topology, introduced in \cite[page 472]{Fell} under the name of H-topology. By \cite[Theorem 1]{Fell}, if $Y$ is locally compact then $\mathscr{F}(Y)$ is a (Hausdorff) compact space. See also \cite{BeR} and \cite{Harp} for more information.

\paragraph{Geometric convergence.} The Chabauty-Fell topology is related to the geometric convergence of Thurston. By \cite[Proposition E.1.2]{BeR}, if $Y$ is a locally compact metrizable space then a sequence $(F_n)$ of closed subsets of $Y$ converges to an element $F$ of $\mathscr{F}(Y)$ if and only if the two following conditions are satisfied:
\begin{itemize}
	\item Let $(F_{n_k})$ be a subsequence of $(F_n)$ and let $x_k \in F_{n_k}$ such that the sequence $(x_k)$ converges in $Y$ to some $x$ in $Y$. Then we have $x \in F$.
	
	\item Any point in $F$ is the limit in $Y$ of a sequence $(x_n)$ with $x_n \in F_n$ for each $n$.
\end{itemize}

\paragraph{Spaces of closed subgroups.} By \cite[IV page 474]{Fell} (see also \cite[Chapitre VIII, \S5, no. 3, Th\'eor\`eme 1]{Bou1}), if $Y=G$ is a locally compact group, the space $\mathscr{C}(G)$ of closed subgroups of $G$ equipped with the induced topology is closed in $\mathscr{F}(G)$, hence compact. Moreover, in this case, it is folklore but not entirely obvious that a basis of neighborhoods of a closed subgroup $H \in \mathscr{C}(G)$ is  given by the sets
\begin{equation}
\label{Chabauty-def}
\mathcal{N}_U^K(H) 
\ov{\mathrm{def}}{=} \big\{H' \in \mathscr{C}(G)\ : \ H' \cap K \subset H U \text{ and } H \cap K \subset H' U\big\}
\end{equation}
where $K$ runs over the compact subsets of $G$ and $U$ runs over the neighborhoods of $e_G$. In words, $H'$ is very close to $H$ if, on a large compact set $K$, the elements of $H'$ belong uniformly to a small neighborhood of $H$, and conversely. In this specific case, the convergence of a sequence was introduced by Chabauty \cite[page 147]{Cha} to generalize Mahler's well-known compactness criterion to lattices in locally compact groups. The following is folklore, see e.g. \cite[Appendix A]{BHP18}.

\begin{prop} 
\label{Chabauty-equivalence}
Let $G$ be a locally compact group. The sets $\mathcal{N}_U^K(H)$ generate the neighborhood filter of $H$ in the Chabauty-Fell topology.
\end{prop}

\paragraph{Lattices and fundamental domains.} A lattice $\Gamma$ in a locally compact group $G$ is a discrete subgroup for which $G/\Gamma$ has a bounded $G$-invariant Borel measure \cite[Definition B.2.1 page 332]{BHV}. 
A locally compact $G$ that admits a lattice is necessarily unimodular \cite[Proposition B.2.2 page 332]{BHV}. 
The same reference says that if $\Gamma$ is a cocompact\footnote{\thefootnote. The word uniform is also used.} (i.e. $G/\Gamma$ is compact) discrete subgroup of a locally compact group $G$ then $\Gamma$ is a lattice of $G$. 

Let $\Gamma$ be a discrete group of a locally compact group $G$. If $A$ is a subset of $G$ and $\gamma \in \Gamma$, then the set $A\gamma$ is called an image of $A$. A fundamental domain $\X$ relative to $\Gamma$ is a Borel measurable subset of $G$ satisfying the following two properties: 
\begin{flalign}
&\label{DGamma=G} \X\Gamma=G, \\
&\label{Dgamma=Dgamma2} \X\gamma \cap \X\gamma'= \emptyset \text{ for any distinct elements }\gamma, \gamma' \text{ of } \Gamma.       
\end{flalign}   
These properties say that every element $x \in G$ is covered by one and only one image of $\X$. These conditions are equivalent to the following statement: $\X$ is a Borel measurable subset of $G$ such that the restriction of the canonical mapping $G \to G/\Gamma$ of $G$ onto left cosets, restricted to $\X$, becomes a bijection onto $G/\Gamma$. We obtain a set $\X$ with these two properties, if we select a representative $s$ from every left coset $s\Gamma$ of $\Gamma$ relative to $G$. However, in general, such a set $\X$ is not a Borel set. If $G$ is $\sigma$-compact the result \cite[Proposition B.2.4 page 333]{BHV} (see also \cite[Lemma 2]{Sie}) gives the existence of a fundamental domain for any discrete subgroup $\Gamma$ and if in addition $\Gamma$ is a lattice in $G$ then every fundamental domain for $\Gamma$ has finite Haar measure \cite[Proposition B.2.4 page 333]{BHV}. 








\subsection{Different notions of groups approximable by discrete groups}
\label{subsec-different-notions-of-groups-approximable-by-discrete-groups}

Recall that a locally compact group $G$ is approximable by a sequence $(\Gamma_j)$ of discrete subgroups \cite[Definition 1]{Kur} \cite[page 36]{Toy} if for any non-empty open set $O$ of $G$, there exists an integer $j_0$ such that for any $j \geq j_0$ we have $O\cap \Gamma_j\not=\emptyset$. We say that a locally compact group $G$ is approximable by discrete subgroups (ADS) if $G$ is approximable by some sequence $(\Gamma_j)$ of discrete subgroups. It is obvious that a second countable locally compact group $G$ is approximable by a sequence $(\Gamma_j)$ of discrete subgroups if and only if $(\Gamma_j)$ converges to $G$ for the Chabauty-Fell topology. Using the definition of the geometric convergence we obtain the following characterization.

\begin{prop}
\label{thm-approximable-discrete-subgroups}
Let $G$ be second countable locally compact group.
Let $(\Gamma_j)$ be a sequence of discrete subgroups of $G$.
The following are equivalent.
\begin{enumerate}
	\item The group $G$ is approximable by the sequence $(\Gamma_j)$.
	\item Any $s \in G$ is the limit in $G$ of a sequence $(\gamma_j)$ with $\gamma_j \in \Gamma_j$ for any integer $j$.
\end{enumerate}
\end{prop}
Moreover, note that a connected ADS locally compact group $G$ is necessarily nilpotent (see \cite[Theorem 2.18]{HaK2}) and that a connected simply connected Lie group is ADS if and only if $G$ is nilpotent and if it admits a discrete cocompact subgroup (\cite[Theorem 1.6, 1.7 and 1.9]{HaS}. We refer to \cite{HaS}, \cite{HaK}, \cite{HaK2}, \cite{Kur}, \cite{Toy} and \cite{Wang} for more information on this notion. Now, we introduce different notions of approximation by discrete groups. These will be used in Chapter \ref{sec:Locally-compact-groups}.

\begin{defi}
\label{defi-sequentially-ALS}\label{defi-ALS}
Let $G$ be a second countable locally compact group.
\begin{enumerate}
\item The group $G$ is said to be approximable by lattice subgroups ($\ALS$) if there exists a sequence $(\Gamma_j)$ of lattices in $G$ such that $(\Gamma_j)$ converges to $G$ for the Chabauty-Fell topology.

\item The group $G$ is said to be (right) uniformly approximable by a sequence $(\Gamma_j)$ of discrete subgroups if there exists a right invariant metric $\dist$ such that for any $\epsi > 0$, there exists an integer $j_0$ such that for all $j \geq j_0$ and all $s \in G$ there exists $\gamma_j \in \Gamma_j$ such that $\dist(s,\gamma_j) < \epsi$. The group $G$ is said to be uniformly $\ADS$ if $G$ is uniformly approximable by a sequence $(\Gamma_j)$ of discrete subgroups. We also define the notion ``uniformly $\mathrm{ALS}$'' where ``discrete groups'' is replaced by ``lattice subgroups''.

\item The group $G$ is said to be approximable by shrinking by a sequence $(\Gamma_j)$ of lattice subgroups with associated fundamental domains $(\mathrm{X}_j)$ if for any neighborhood $V$ of the identity $e_G$ (equivalently, for any ball $V = B(e_G,\epsi)$ with $\epsi > 0$, associated with a right invariant metric generating the topology of $G$) there exists some integer $j_0$ such that $\X_j \subset V$ for any $j \geq j_0$. The group $G$ is said to be approximable by lattice subgroups by shrinking ($\ALSS$) if there exists a sequence $(\Gamma_j)_{j \geq 1}$ of lattice subgroups in $G$ and some associated fundamental domains $(\X_j)$ such that $G$ is approximable by shrinking by $(\Gamma_j)$ and $(\X_j)$.
\end{enumerate}
\end{defi}

\begin{remark}\normalfont
\label{rem-after-defi-ALS}
\begin{enumerate}
\item If we assume in Part 3 of Definition \ref{defi-ALS} that the subgroups $\Gamma_j$ are only discrete subgroups instead of being lattices, we obtain  the same definition. Indeed, for any sufficiently small $\epsi > 0$ and any sufficiently large $j$, we have $\X_j \subset B(e_G,\epsi)$ where $B(e_G,\epsi)$ is relatively compact according to the local compactness of $G$. Thus the closure $\ovl{\X_j}$ is compact. The canonical mapping $\pi \co G \to G/\Gamma_j$ being continuous, $\pi(\ovl{\X_j})$ is also compact. But since $\X_j$ is a fundamental domain, we have $\pi(\X_j) = G/\Gamma_j$ and a fortiori $\pi(\ovl{\X_j})= G/\Gamma_j$. Therefore, $G/\Gamma_j$ is compact, and so by \cite[Proposition B.2.2]{BHV}, the discrete subgroup $\Gamma_j$ is automatically a lattice.

\item We shall see in Part 3 of Proposition \ref{prop-approx-ALS} that a second countable locally compact group which is uniformly ADS with respect to a sequence $(\Gamma_j)$ of discrete subgroups admits fundamental domains which are almost all included in small balls. Therefore, combined with the first part of this remark, we deduce that if $G$ is uniformly ADS then $G$ is uniformly $\ALS$.

\item
Part 3 of Definition \ref{defi-ALS} is inspired by the notion $\ADS$ from \cite[page 3]{CPPR}. It is formally slightly weaker since we assume that the $\X_j$ are becoming smaller and smaller around $e_G$ instead of forming a neighborhood basis of $e_G$ as in \cite{CPPR}. Moreover, the authors of \cite{CPPR} use only lattice subgroups. However, we shall see in Part 3 of Proposition \ref{prop-approx-ALS} that our notion of $\ALSS$ is equivalent to $\ADS$ from \cite[page 3]{CPPR}.

\item It is obvious that the property uniformly ADS implies the property ADS, that uniformly ALS implies ALS and that ALS implies ADS.
\end{enumerate}
\end{remark}

Recall that any locally compact group $G$ which contains a lattice subgroup $\Gamma$ is unimodular by \cite[Proposition B.2.2]{BHV} and that the subset of unimodular closed subgroups of $G$ is closed in $\mathscr{C}(G)$ for the Chabauty topology, see \cite[Chapitre VIII, \S5, no. 3, Th\'eor\`eme 1]{Bou1}.


We start with a result giving the existence of fundamental domains satisfying some inclusion constraint. In this proposition and the subsequent lemma, we equip the group $G$ with a left invariant metric $\dist$ generating its topology and we consider the balls $B(e_G,r) \ov{\mathrm{def}}{=} \{ s \in G :\: \dist(s,e_G) < r \}$. However, note that the statement in Proposition \ref{lem-Dirichlet-cell} remains valid if one replaces the distance $\dist$ by a \textit{right} invariant one $\dist'$, generating the topology of $G$, together with balls $\tilde{B}(e_G,r) \ov{\mathrm{def}}{=} \{ s \in G :\: \dist'(s,e_G) < r \}$. Indeed, note that since both $\dist$ and $\dist'$ generate the same topology, if $D$ contains a ball $\tilde{B}(e_G,\tilde{r})$, it will contain a ball $B(e_G,r)$, so $\X$ will contain a ball $B(e_G,r')$ and thus also a ball $\tilde{B}(e_G,r'')$. 

\begin{prop}
\label{lem-Dirichlet-cell}
Let $G$ be a second countable locally compact group together with a discrete subgroup $\Gamma \subset G$. Let $D \subset G$ be a measurable subset satisfying $\bigcup_{\gamma \in \Gamma} D \gamma = G$. Then there exists a fundamental domain $\X \subset D$ associated with $\Gamma$. Moreover, if $D$ contains a ball $B(e_G,r)$ then $\X$ contains a ball $B(e_G,r')$.
\end{prop}
\begin{proof}
Note first that since $G$ is second countable, $\Gamma$ endowed with the trace topology is again second countable. Since $\Gamma$ is discrete, this implies that $\Gamma$ is at most countable, and we choose one enumeration $(\gamma_j)$ of $\Gamma$. 

Consider the canonical map $p \co G \to G/\Gamma$. 
Since $G$ is second countable, there exists by \cite[Lemma 1.1]{Mac52} (see also the discussions \cite[page 11]{Bor21} and \cite[page 67]{Fur05}) a locally bounded Borel section $q \co G/\Gamma \to G$. By \cite[Corollary 4.49]{Wil07}, the map $\gamma \co G \to \Gamma$, $s \mapsto (q(s\Gamma))^{-1}s$ is a (locally bounded) Borel function. 


%

\begin{lemma}
There exists some $\rho > 0$ such that $\gamma(B(e,\rho)) \subset \{e\}$.
\end{lemma} 

\begin{proof}
Let $\dist$ be a left invariant metric on $G$ generating its topology as a locally compact group and $\dist_{G/\Gamma}$ the associated distance on $G/\Gamma$. Consider the strictly\footnote{\thefootnote. The subset $\{e\}$ is open in $\Gamma$, so $\Gamma \backslash \{e\}$ is closed.} positive number $r_0 =\dist(\Gamma \backslash \{e\},e)>0$. Since $B(e,r_0) \cap \Gamma= \{e\}$, for any $s \in G$, the condition $\dist(\gamma(s),e) < r_0$ implies that $\gamma(s) = e$. Now by definition of $\gamma$, we have $\dist(\gamma(s),e) < r_0$ if and only if $\dist(q(s\Gamma)^{-1} s,e) < r_0$ and finally if and only if $\dist(s,q(s\Gamma)) < r_0$ by left invariance. Since $q$ is continuous in a neighborhood of $e \Gamma$, there exists $r_1 >0$ such that $\dist_{G/\Gamma}(s\Gamma,e\Gamma)<r_1$ implies $\dist(q(s\Gamma),e)< \frac{r_0}{2}$. If $\dist(s,e) < \min\{r_1,\frac{r_0}{2}\}$ we have 
$$
\dist_{G/\Gamma}(s\Gamma,e\Gamma) 
\leq \dist(s,e)
< r_1,
$$ 
hence $\dist(e,q(s\Gamma)) < \frac{r_0}{2}$. Thus the triangle inequality gives
$$
\dist\big(s,q(s\Gamma)\big) 
\leq \dist(s,e)+\dist\big(e,q(s\Gamma)\big)  
< \frac{r_0}{2} + \frac{r_0}{2} 
= r_0.
$$ 
The lemma is proved.
\end{proof}

Define now $A_1 \ov{\mathrm{def}}{=} \{ s \in D :\: \gamma(s) = \gamma_1 \} = D \cap \gamma^{-1}(\{\gamma_1\})$, which is measurable as the intersection of two measurable sets. Assuming without loss of generality that $\gamma_1 = e$, we have that $B(e,r') \subset A_1$ for $r' = \min(r,\rho)$ since $B(e,r) \subset D$. Define then recursively for $k \geq 2$, the subsets
\begin{align*} 
A_k & \ov{\mathrm{def}}{=} \big\{ s \in D :\: \gamma(s) = \gamma_k ,\: \centernot{\exists}\: j \in \{ 1, \ldots, k-1\},\: \centernot{\exists} \: l \in \N :\: s \gamma_l \in A_j \big\} \\
& = D \cap \gamma^{-1}(\{\gamma_k\}) \cap \bigcap_{j = 1}^{k-1} \bigcap_{l \in \N} A_j^c \gamma_l^{-1}.
\end{align*}
It can easily be shown recursively that $A_k$ is measurable as the countable intersection of measurable sets. Define finally $\X \ov{\mathrm{def}}{=} \bigcup_{k = 1}^\infty A_k$.

We claim that $\X$ is a (measurable) fundamental domain of $\Gamma$ which is contained in $D$. First, it is measurable as a countable union of measurable sets. Since by definition, we have $A_k \subset D$ for any integer $k \geq 1$, we also have $\X \subset D$. 

\begin{lemma}
For any $\gamma \in \Gamma \backslash \{e\}$, we have $\X \gamma \cap \X = \emptyset$.
\end{lemma} 

\begin{proof}
Indeed, let $s \in \X$, so that $s \in A_{k_0}$ for some $k_0 \in \N$. This implies that $\gamma(s) = \gamma_{k_0}$. Put $t = s \gamma$. Since $\gamma \not= e$, we cannot have $\gamma(t) = \gamma(s)$, because otherwise $Y(t)=(t\Gamma,\gamma(t)) = (s \Gamma,\gamma(s))=Y(s)$, and since $Y$ is bijective, we obtain $t = s$, which is a contradiction. So $\gamma(t) = \gamma_{k_1}$ for some $k_1 \not= k_0$.

If $k_1 > k_0$, then $t$ cannot belong to $A_{k_1}$. Indeed, $t \in A_{k_1}$ implies that we cannot find $l \in \N$ such that $t \gamma_l \in A_{k_0}$ since $k_0<k_1$. This implies with $\gamma_l = \gamma^{-1}$ that $s = t \gamma^{-1} \not\in A_{k_0}$, which is a contradiction. 

If $k_1 < k_0$, then $t$ cannot belong to $A_{k_1}$ either. Indeed, since $s \in A_{k_0},$ we cannot find $l \in \N$ such that $s \gamma_l \in A_{k_1}$ since $k_1<k_0$. This implies with $\gamma_l = \gamma$ that $t = s \gamma \not\in A_{k_1}$. Thus $t \not\in \X$, so we have $\X \gamma \cap \X = \emptyset$.
\end{proof}

\begin{lemma}
We have
\begin{equation}
\label{equ-proof-lem-Dirichlet-cell}
\bigcup_{k = 1}^\infty A_k \Gamma 
= \big\{s \in D:\: \gamma(s) \in \{ \gamma_1, \gamma_2, \ldots \} \big\} \Gamma.
\end{equation}
\end{lemma}

\begin{proof}
For the inclusion $\subset$, we note that if $s \in A_k$ for some $k \in \N$, then in particular $s \in D$ and $\gamma(s) = \gamma_k$, so that $s \Gamma$ is contained in the right hand side of \eqref{equ-proof-lem-Dirichlet-cell}. For the inclusion $\supset$, if $s \in D$ and $\gamma(s) = \gamma_k$ for some $k \in \N$, then either $s \in A_k$, which implies that $s \Gamma$ is contained in the left hand side of \eqref{equ-proof-lem-Dirichlet-cell} or there exists $l \in \N$ and $j \in \{ 1, \ldots, k-1 \}$ such that $s \gamma_l \in A_j$. Then  $s \Gamma=s\gamma_l\gamma_l^{-1}\Gamma \subset A_j \Gamma$, so it is also contained in the left hand side of \eqref{equ-proof-lem-Dirichlet-cell}. Whence, \eqref{equ-proof-lem-Dirichlet-cell} is shown.
\end{proof}

The left hand side of \eqref{equ-proof-lem-Dirichlet-cell} equals clearly $\X \Gamma$, and the right hand side equals $D\Gamma$, since $\gamma(s)$ must belong to $\{\gamma_1,\gamma_2,\ldots \}$ for any $s \in D$. Since $D\Gamma=G$, we obtain $\X \Gamma = G$, so that $\X$ is a fundamental domain. Since $B(e,r') \subset A_1$, we also have $B(e,r') \subset \X$.
\end{proof}



\begin{prop}
\label{prop-approx-ALS}
Let $G$ be a second countable locally compact group.

\begin{enumerate}
\item If the group $G$ is $\ALSS$ with respect to $(\Gamma_j)$ and $(\X_j)$ then $G$ is uniformly $\ALS$ with respect to $(\Gamma_j)$.

\item Let $G$ be an $\ADS$ group with respect to a sequence $(\Gamma_j)$ of discrete subgroups. Suppose that for some $j_0 \in \N$, some compact $K \subset G$ and any $j \geq j_0$ there exists a fundamental domain $\X_j$ with respect to $\Gamma_j$ such that $\X_j \subset K$. Then the group $G$ is uniformly $\ADS$ with respect to $(\Gamma_j)$. We have a similar property for $\ALS$ and uniformly $\ALS$.

\item If the group $G$ is uniformly $\ADS$ with respect to discrete subgroups $(\Gamma_j)$ then $G$ is $\ALSS$ with respect to $(\Gamma_j)$ and some particular sequence $(\X_j)$ of fundamental domains. Moreover, the $\X_j$ can be chosen to be neighborhoods of $e_G$ if $j$ is large enough. In particular, if $G$ is uniformly $\ALS$ then $G$ is $\ALSS$.

\item The group $G$ is uniformly $\ADS$ if and only if it is uniformly $\ALS$ if and only if it is $\ALSS$.

\end{enumerate}
\end{prop}

\begin{proof}
1. First assume that $G$ is ALSS with respect to a sequence of lattice subgroups $(\Gamma_j)$ with associated fundamental domains $(\X_j)$. Take a right invariant metric $\dist$ on $G$ generating its topology as a locally compact group. Fix $\epsi > 0$. By the $\ALSS$ property, there exists some integer $j_0$ such that the fundamental domains $\X_j$ are contained in $B(e,\epsi)$ for any $j \geq j_0$. For any $s \in G$ and any $j$, there exists $x \in X_j$ and $\gamma \in \Gamma_j$ such that $s = x \gamma$. For any $j \geq j_0$, we conclude that
$$
\dist(s,\gamma) 
= \dist(x \gamma,\gamma) 
= \dist(x,e) < \epsi.
$$
Thus, the group $G$ is uniformly ALS.

2. Let $G$ be an ADS group with respect to a sequence $(\Gamma_j)$ of discrete subgroups in $G$. Suppose that for some $j_0 \in \N$, some compact $K \subset G$ and any $j \geq j_0,$ there exists a fundamental domain $\X_j$ with respect to $\Gamma_j$ such that $\X_j \subset K$. Fix a right invariant metric $\dist$ on $G$. The compact subset $K$ is totally bounded. Then for any $\epsi > 0$, there exist some $s_1,\ldots,s_N \in K$ such that for $j \geq j_0$,
$$
\X_{j} \subset K \subset \bigcup_{k=1}^N B\bigg(s_k,\frac{\epsi}{2}\bigg).
$$
Moreover, since $G$ is ADS, for any $1 \leq k \leq N$, there exists some $j_k \in \N$ such that for all $i \geq j_k$ there is some $\gamma_i \in \Gamma_i$ with $\dist(s_k,\gamma_i) < \frac{\epsi}{2}$. Note that this implies that if $x \in B(s_k,\frac{\epsi}{2})$ we have 
$$
\dist(x,\gamma_i) 
\leq \dist(x,s_k)+\dist(s_k,\gamma_i) 
< \frac{\epsi}{2}+\frac{\epsi}{2}
=\epsi.
$$ 
Thus, for $j_{\max} \overset{\textrm{def}}= \max\{j_0,j_1,\ldots,j_N\}$, any $j \geq j_{\max}$, any $x \in \X_{j}$ and any $i \geq j_{\max}$, there exists some $\gamma_i \in \Gamma_i$ such that $\dist(x,\gamma_i) < \epsi$. 

For an arbitrary $s \in G$ and any $j \geq j_{\max}$, we write $s = x_j \widetilde{\gamma}_j$ with $x_j \in \X_j$ and $\widetilde{\gamma}_j \in \Gamma_{j}$ and we have (setting $i =j$) $\dist(x_j,\gamma_j) < \epsi$ for some $\gamma_j \in \Gamma_j$ so also 
$$
\dist(s,\gamma_j \widetilde{\gamma}_j) 
=\dist(x\widetilde{\gamma}_j,\gamma_j \widetilde{\gamma}_j)
=\dist(x,\gamma_j) < \epsi.
$$ 
Note that $\gamma_j \widetilde{\gamma}_j$ belongs to $\Gamma_j$. Thus the group $G$ is uniformly $\ADS$. The proof of the second property is identical.

3. Now assume that $G$ is uniformly ADS with respect to a sequence $(\Gamma_j)$ of discrete subgroups. We fix a right invariant metric $\dist$ of $G$ which generates the topology of $G$ and with respect to which the uniformly ADS property holds. There exists $\delta > 0$ such that any closed ball of radius $<\delta$ is compact. 

For any $j$, we introduce the Dirichlet cell 
$$ 
D_{\Gamma_j} 
=\big\{s \in G \ : \ \dist(s,e) \leq \dist(s,\gamma) \text{ for any } \gamma \in \Gamma_j \big\}.
$$

We first show that for given $\epsi > 0,$ there exists $j_0 \in \N$ such that $D_{\Gamma_j} \subset B(e,\epsi)$ for $j \geq j_0$. Note that by the uniformly ADS property there exists a $j_0 \in \N$ such that for all $s \in G$ and any $j \geq j_0$ there exists $\gamma_j \in \Gamma_j$ such that $\dist(s,\gamma_j) \leq \frac{\epsi}{2}$. If $s \in B(e,\epsi)^c$ and if $j \geq j_0$ we obtain
$$
\dist(s,\gamma_j) 
\leq \frac{\epsi}{2} 
< \epsi 
\leq \dist(s,e).
$$
Hence $s$ does not belong to $D_{\Gamma_j}$. We deduce that $B(e,\epsi)^c \subset D_{\Gamma_j}^c$ if $j \geq j_0$. The claim is proved.

Now we prove that the Dirichlet cell $D_{\Gamma_j}$ satisfies $\bigcup_{\gamma \in \Gamma_j} D_{\Gamma_j} \gamma = G$ if $j$ is large enough. Let $s \in G$. For any $j$, consider the positive real number 
$$
r_j
\ov{\mathrm{def}}{=} \inf_{\gamma' \in \Gamma_j} \dist(s,\gamma').
$$  
There exists $j_1$ such that for any $j \geq j_1$ and any $s \in G$ there exists $\gamma_j \in \Gamma_j$ such that $\dist(s,\gamma_j)< \frac{\delta}{3}$, hence $r_j <\frac{\delta}{3}$. 

\begin{lemma}
For any $j \geq j_1$, there exists $\gamma \in \Gamma_j$ such that $\dist(s,\gamma) \leq \dist(s,\gamma')$ for any $\gamma' \in \Gamma_j$.
\end{lemma}

\begin{proof}
If $s \in \Gamma_j$, it is obvious that the infimum is a minimum. Suppose $s \notin \Gamma_j$. We have $r_j>0$. We let $K=B'(x,2r_j) \cap \Gamma_j$. This subset is nonempty and compact. If $\gamma' \in \Gamma_j \backslash K$ we have $\dist(x,\gamma') > 2r_j$. We deduce that
$$
r_j
=\inf_{\gamma' \in \Gamma_j} \dist(s,\gamma')
=\inf_{\gamma' \in K} \dist(s,\gamma').
$$
Finally, the map $\gamma' \mapsto \dist(s,\gamma')$ is continuous on the compact $K$, hence attains its infimum on $K$. 
\end{proof}
%

In particular, for any $\gamma'' \in \Gamma_j$, using the right-invariance of the distance, we obtain
$$
\dist(s\gamma^{-1},e)
=\dist(s,\gamma) \leq \dist(s,\gamma''\gamma)
=\dist(s\gamma^{-1},\gamma'').
$$  
Therefore, $s \gamma^{-1} \in D_{\Gamma_j}$, that is $s \in D_{\Gamma_j} \gamma$.

Moreover, $D_{\Gamma_j} = \bigcap_{\gamma \in \Gamma_j} \{s \in G :\: \dist(s,e) \leq \dist(s,\gamma) \}$ is an intersection of closed sets, and hence itself closed, hence measurable. 

Note that $\Gamma_j \backslash \{e\}$ is closed. Hence we have $r_j'=\dist(e,\Gamma_j \backslash \{e\})>0$. Thus the ball $B(e,\frac{r_j'}{2})$ is contained in $D_{\Gamma_j}$. According to Proposition \ref{lem-Dirichlet-cell}, there exists some fundamental domain $\X_j \subset D_{\Gamma_j}$ associated with $\Gamma_j$, which is a neighborhood of $e \in G$. Furthermore, if $j \geq j_0$ we have $\X_j \subset D_{\Gamma_j} \subset B(e,\epsi)$. Hence we conclude that the group $G$ is $\ALSS$ with respect to $(\Gamma_j)$ and $(\X_j)$. The proof of the second property is identical.

4. This statement is now obvious.
\end{proof}

%

\subsection{The case of second countable compactly generated locally compact groups}

The following uses a trick of the proof of \cite[Lemma 5.7]{Wang}. For the sake of completeness, we give all the details. Recall that a topological group is compactly generated if it has a compact generating set \cite[Definition 5.12]{HR}. For example, a connected locally compact group is compactly generated \cite[Proposition 2.C.3 (2)]{dCHa}.

\begin{lemma}
\label{Lemma-uniforme-large-i}
Let $G$ be a compactly generated locally compact group and $(\Gamma_i)$ a sequence of subgroups of $G$ which converges to $G$ for the Chabauty-Fell topology. Then there exists a compact subset $K$ of $G$ and $i_0$ such that $G =K \Gamma_i$ for any $i \geq i_0$. 
\end{lemma}

\begin{proof}
By the proof of \cite[Theorem 5.13]{HR}, there exists an open subset $V$ of $G$ containing $e$ with $G=\bigcup_{n \geq 1} (V \cup V^{-1})^n$ such that $\ovl{V}$ is compact. We let $U=V \cup V^{-1}$. The subset $U$ is open and contains $e$. Moreover, the set $K \ov{\mathrm{def}}{=} \ovl{U}=\ovl{V \cup V^{-1}}=\ovl{V} \cup \ovl{V}^{-1}$ is compact and we have $G=\bigcup_{n \geq 1} K^n$. Since $e$ belongs to $U$, we have $UG=G$. Moreover, by \cite[Theorem 4.4]{HR}, the subset $K^3$ is compact and included in $UG$. Using \cite[Theorem 4.4]{HR} again, we deduce that $(U s)_{s \in G}$ is an open covering of $K^3$. By compactness there exist some elements $s_1,\ldots,s_m \in G$ such that 
$$
K^3 \subset \bigcup_{j=1}^{m} Us_j.
$$ 
Since $(\Gamma_i)$ approximates the group $G$, there exists some $i_0$ such that for any $i \geq i_0$ we have $\{s_1,\ldots,s_m\} \subset U\Gamma_i$. For $i \geq i_0$, we deduce that $
K^3 \subset U^2 \Gamma_i \subset K^2 \Gamma_i$. By induction\footnote{\thefootnote. If $K^n \subset K^2\Gamma_i$ for some $n \geq 3$ then we have $K^{n+1}=KK^n \subset KK^2\Gamma_i=K^3\Gamma_i \subset K^2\Gamma_i \Gamma_i=K^2\Gamma_i$.}, we obtain $K^n \subset K^2\Gamma_i$ for any $n \geq 3$. Moreover, we have $K^2 \subset K^2\Gamma_i$. For any $i \geq i_0$, we deduce that
\begin{align*}
\MoveEqLeft
  G \backslash K \subset \bigcup_{n \geq 2} K^n \subset K^2 \Gamma_i.
\end{align*}
Note that $K \subset K\Gamma_i$. Thus the compact $K \cup K^2$ has the desired property.
\end{proof}

\begin{cor}
\label{Cor-lattice-large-i}
Let $G$ be a compactly generated locally compact group and $(\Gamma_i)$ a sequence of discrete subgroups which converges to $G$ for the Chabauty-Fell topology. For any large enough $i$, the subgroup $\Gamma_i$ is a cocompact lattice. 
\end{cor}

\begin{proof}
Use the previous Lemma \ref{Lemma-uniforme-large-i} and recall that a discrete subgroup $\Gamma$ which is cocompact\footnote{\thefootnote. \label{footnote-35}If $G = K \Gamma_i$ for a compact $K$, then for the canonical and continuous $q \co G \to G/\Gamma_i$, we have $q(K) = G/\Gamma_i$, so that $G/\Gamma_i$ is compact.} is a lattice.
\end{proof}

\begin{thm}
\label{Th-compactly-tout-equivalent}
Let $G$ be a second countable compactly generated locally compact group. The following are equivalent.
\begin{enumerate}
	\item $G$ is $\ADS$.
	
	\item $G$ is $\ALS$.
	
	\item $G$ is uniformly $\ALS$.
	
	\item $G$ is $\ALSS$.
\end{enumerate}
\end{thm}

\begin{proof}
The implications 2. $\Rightarrow$ 1. and 3. $\Rightarrow$ 2. are obvious. By Corollary \ref{Cor-lattice-large-i}, we have the implication 1. $\Rightarrow$ 2. By the part 3 of Proposition \ref{prop-approx-ALS}, the properties 3. and 4. are equivalent. 

Suppose that $G$ is ALS with respect to a sequence $(\Gamma_j)$ of lattice subgroups in $G$. Then by Lemma \ref{Lemma-uniforme-large-i}, there exists a compact subset $K$ of $G$ and $i_0$ such that $G =K \Gamma_i$ for any $i \geq i_0$. By Proposition \ref{lem-Dirichlet-cell}, there exists\footnote{\thefootnote. If $G$ is a second countable locally compact group and if $\Gamma$ is a cocompact lattice in $G$ then there exists a relatively compact fundamental domain $\X$ for $\Gamma$ in $G$. This result \cite[8]{Sie} of Siegel does not suffice here.} a fundamental domain $\X_i$ for $\Gamma_i$ in $G$ such that $\X_i \subset K$ for any $i \geq i_0$. From part 2 of Proposition \ref{prop-approx-ALS}, we conclude that $G$ is ALSS and thus 2. implies 3. 
\end{proof}



\section{Decomposable Fourier multipliers on non-discrete locally compact groups}
\label{sec:Locally-compact-groups}

In this section, we start by giving general results on Fourier multipliers on noncommutative $\L^p$-spaces. After this, we construct our projections by approximation. Then we study (classes of) examples, including direct and semi-direct products of groups, the semi-discrete Heisenberg group, groups acting on trees and pro-discrete groups.
We conclude by drawing the relevant consequences for decomposable multipliers.

\subsection{Generalities on Fourier multipliers on unimodular groups}
\label{sec:Generalities-Fourier-multipliers}



\paragraph{Group von Neumann algebras of locally compact groups.} 
Let $G$ be a locally compact group equipped with a left invariant Haar measure $\mu_G$. For a complex function $g \co G \to \C$, we write $\lambda(g)$ for the left convolution operator (in general unbounded) by $g$ on $\L^2(G)$. This means that the domain of $\lambda(g)$ consists of all $f$ of $\L^2(G)$ for which the integral $(g*f)(t) \ov{\mathrm{def}}{=} \int_G g(s)f(s^{-1}t) \d\mu_G(s)$ exists for almost all $t \in G$ and for which the resulting function $g*f$ belongs to $\L^2(G)$, and for such $f$, we let $\lambda(g)f \ov{\mathrm{def}}{=} g*f$. Finally, by \cite[Corollary 20.14]{HR}, each $g \in \L^1(G)$ induces a bounded operator $\lambda(g) \co \L^2(G) \to \L^2(G)$. 

Let $\VN(G)$ be the von Neumann algebra generated by the set $\big\{\lambda(g) : g \in \L^1(G)\big\}$. It is called the group von Neumann algebra of $G$ and is equal to the von Neumann algebra generated by the set $\{\lambda_s : s \in G\}$ where 
\begin{equation}
\label{Left-translation}
   \lambda_s  \co  \begin{cases}
  \L^2(G)   &  \longrightarrow    \L^2(G)  \\
            f   &  \longmapsto        (t \mapsto f(s^{-1}t))
\end{cases}
\end{equation}
is the left translation by $s$. 
Recall that for any $g \in \L^1(G)$ we have $\lambda(g)=\int_{G} g(s)\lambda_s \d\mu_G(s)$ where the latter integral is understood in the weak operator sense\footnote{\thefootnote. That means (see e.g. \cite[Theorem 5 page 289]{Gaa}) that $\lambda(g) \co \L^2(G) \to \L^2(G)$ is the unique bounded operator such that
$$
\langle \lambda(g)f, h\rangle_{\L^2(G)}
=\int_{G} g(s)\langle \lambda_s f,h \rangle_{\L^2(G)} \d\mu_G(s),\quad f,h \in \L^2(G).
$$}. 

Let $H$ be a closed subgroup of $G$ equipped with a left Haar measure. The prescription $\lambda_{H,s} \mapsto \lambda_{G,s}$, $s \in H$ (where $\lambda_{H,s}$ denotes the left translation by $h$ on $\L^2(H)$ and $\lambda_{G,s}$ the corresponding left translation by $h$ on $\L^2(G)$) extends to a normal injective $*$-homomorphism from $\VN(H)$ to $\VN(G)$, see e.g. \cite[Proposition 2.6.6]{KaL1}, \cite[Theorem 2 page 113]{Der4} and \cite{DKSS} for generalizations to quantum groups. 

We also use the notation $\lambda(\mu) \co \L^2(G) \to \L^2(G)$ for the convolution operator by the measure $\mu$.

\paragraph{Plancherel weights.} Let $G$ be a locally compact group. A function $g \in \L^2(G)$ is called left bounded \cite[Definition 2.1]{Haa2} if the convolution operator $\lambda(g)$ induces a bounded operator on $\L^2(G)$. The Plancherel weight $\tau_G \co \VN(G)^+\to [0,\infty]$ is\footnote{\thefootnote. This is the natural weight associated with the left Hilbert algebra $\mathrm{C}_c(G)$.} defined by the formula
$$
\tau_G(x)
\ov{\mathrm{def}}{=} \begin{cases}
\norm{g}^2_{\L^2(G)} & \text{if }x^{\frac{1}{2}}=\lambda(g) \text{ for some left bounded function } g \in \L^2(G)\\
+\infty & \text{otherwise}
\end{cases}.
$$

By \cite[Proposition 2.9]{Haa2} (see also \cite[Theorem 7.2.7]{Ped}), the canonical left ideal $\mathfrak{n}_{\tau_G}=\big\{x \in \VN(G)\ : \  \tau_G(x^*x) <\infty\big\}$ is given by
$$
\mathfrak{n}_{\tau_G}
=\big\{\lambda(g)\ :\ g \in \L^2(G)\text{ is left bounded}\big\}.
$$
Recall that $\mathfrak{m}_{\tau_G}^+$ denotes the set $\big\{x \in \VN(G)^+ : \tau_G(x)<\infty\big\}$ and that $\mathfrak{m}_{\tau_G}$ is the complex linear span of $\mathfrak{m}_{\tau_G}^+$ which is a $*$-subalgebra of $\VN(G)$. By \cite[Proposition 2.9]{Haa2} and \cite[Proposition page 280]{Str}, we have 
$$
\mathfrak{m}_{\tau_G}^+
=\big\{\lambda(g) : g \in \L^2(G) \text{ continuous and left bounded}, \ \lambda(g)\geq 0\big\}.
$$

By \cite[page 125]{Haa2} or \cite[Proposition 7.2.8]{Ped}, the Plancherel weight $\tau_G$ on $\VN(G)$ is tracial if and only if G is unimodular, which means that the left Haar measure of $G$ and the right Haar measure of $G$ coincide. Now, in the sequel, we suppose that the locally compact group $G$ is unimodular. 


We will use the involution $f^*(t) \ov{\mathrm{def}}{=} \ovl{f(t^{-1})}$. By \cite[Theorem 4]{Kun}, if $f,g \in \L^2(G)$ are left bounded then $f*g$ and $f^*$ are left bounded and we have 
\begin{equation}
\label{composition-et-lambda}
\lambda(g)\lambda(f)
=\lambda(g*f) 
\quad \text{and} \quad 
\lambda(f)^*=\lambda(f^*).
\end{equation}

If $f,g \in \L^2(G)$ it is well-known \cite[Corollaire page 168 and (17) page 166]{Bou1} that the function $f*g$ is continuous and that we have $(f*g)(e_G)=(g*f)(e_G)=\int_G \check{g} f \d\mu_G$ where $e_G$ denotes the identity element of $G$ and where $\check{g}(s) \overset{\textrm{def}}= g(s^{-1})$. By \cite[(4) page 282]{StZ}, if $f,g \in \L^2(G)$ are left bounded, the operator $\lambda(g)^*\lambda(f)$ belongs to $\mathfrak{m}_{\tau_G}$ and we have the fundamental ``noncommutative Plancherel formula''
\begin{equation}
\label{Formule-Plancherel}
\tau_G\big(\lambda(g)^*\lambda(f)\big)
=\langle g,f\rangle_{\L^2(G)}
\quad \text{which gives} \quad 
\tau_G\big(\lambda(g)\lambda(f)\big)
=\int_G \check{g} f \d\mu_G
=(g*f)(e_G).
\end{equation}
In particular, this formula can be used with any functions $f,g$ of $\L^1(G) \cap \L^2(G)$. By \eqref{Formule-mtau}, if we consider the subset $\mathrm{C}_e(G) \ov{\mathrm{def}}{=} \Span\big\{g^**f : g,f \in \L^2(G)\text{ left bounded}\big\}$ of $\mathrm{C}(G)$, we have
\begin{equation}
	\label{Def-mtauG}
	\mathfrak{m}_{\tau_G}
=\lambda\big(\mathrm{C}_e(G)\big)
\end{equation}
and we can see $\tau_G$ as the functional that evaluates functions of $\mathrm{C}_e(G)$ at $e_G \in G$. Although the formula $\tau_G\big(\lambda(h)\big)=h(e)$ seems to make sense for every function $h$ in $\mathrm{C}_c(G)$, we warn the reader that it is not true\footnote{\thefootnote. \label{footnote-39}In fact, suppose that $G$ is compact. Since $\L^2(G) \subset \L^1(G)$, any function of $\L^2(G)$ is left bounded. Moreover, the group $G$ is unimodular so the map $f \mapsto f^*$ is an anti-unitary operator on $\L^2(G)$. We infer that $\L^2(G)^* = \L^2(G)$ and consequently that
$$
\mathrm{C}_e(G) 
= \mathrm{span} \, \L^2(G)*\L^2(G).
$$ 
As already noted, we always have $\mathrm{C}_e(G) \subset \mathrm{C}(G)$. If in addition $\lambda(\mathrm{C}(G)) \subset \lambda(\mathrm{C}_e(G))$, we have $\mathrm{C}(G) = \mathrm{C}_c(G) \subset \mathrm{C}_e(G)$ (if $f,g \in \L^1(G)$ and $\lambda(f)=\lambda(g)$, we have $f=g$ almost everywhere since the regular representation $\lambda \co \L^1(G) \to \B(\L^2(G))$, $f \mapsto \lambda(f)$ is injective by \cite[page 285]{Dix}), then we obtain $\mathrm{span}\, \L^2(G)*\L^2(G) = \mathrm{C}(G)$. But this is true only if $G$ is finite (see \cite[34.16, 34.40 (ii) and 37.4]{HeR2}).} in general that $\lambda\bigl(\mathrm{C}_c(G)\bigr) \subset \mathfrak{m}_{\tau_G}$ contrary to what is unfortunately too often written in the literature.

\paragraph{Averaging projections.} If $K$ is a compact subgroup of a locally compact group $G$ equipped with its \textit{normalized} Haar measure $\mu_K$, we can consider the element $p_K \ov{\mathrm{def}}{=} \lambda_K(\mu_K)$ of $\VN(K)$. It is easy to see that it identifies to the element $\lambda_G(\mu_K^0)$ of $\VN(G)$ where $\mu_K^0$ is the canonical extension of the measure $\mu_K$ on the locally compact space $G$. We say that it is the averaging projection associated with $K$. The following lemma is folklore. For the sake of completeness, we give a short proof.

\begin{lemma}
\label{Lemma-iso-GK-VN}
If $K$ is a normal compact subgroup then the averaging projection $p_K$ associated with $K$ is a central projection in $\VN(G)$ and finally the map
\begin{equation}
\label{Iso-averaging-projection}
\begin{array}{cccc}
    \pi  \co &   \VN(G/K)  &  \longrightarrow   &  \VN(G)p_K  \\
           &  \lambda_{sK}   &  \longmapsto       &  \lambda_s p_K  \\
\end{array}
\end{equation}
is a well-defined $*$-isomorphism.
\end{lemma}

\begin{proof}
For any $s \in G$, we have $sK=Ks$ and consequently
$
   \lambda_s\lambda(\mu_K^0) 
		=\lambda(\delta_s*\mu_K^0)
		=\lambda(\mu_K^0)\lambda_{s}
$. Hence $p_K$ is central. For any $s \in G$, if $sK=s'K$, we have 
$
\lambda_s p_K
=\lambda_s \lambda(\mu_K^0)
=\lambda(\delta_s *\mu_K^0)
=\lambda(\delta_{s'}*\mu_K^0)
=\lambda_{s'} \lambda(\mu_K^0)
=\lambda_{s'} p_K
$. Hence $\pi$ is well-defined. Other statements are obvious.
\end{proof}


If $K$ is in addition an \textit{open} subgroup, the following allows us to consider maps on the associated noncommutative $\L^p$-spaces.

\begin{lemma}
\label{Lemma-trace-preserving-averaging}
Let $K$ be a compact open normal subgroup of a unimodular locally compact group $G$. We suppose that $G$ is equipped with a Haar measure $\mu_G$ and that $K$ is equipped with its normalized Haar measure $\mu_K$. We have $p_K=\frac{1}{\mu_G(K)}\lambda(1_{K})$ and the map $\mu_G(K)\pi \co \VN(G/K) \to \VN(G)p_K$ is trace preserving. Finally if $1 \leq p \leq \infty$, the $*$-isomorphism $\pi$ induces a complete isometry $\mu_G(K)^{\frac1p}\pi_p$ from $\L^p(\VN(G/K))$ into $\L^p(\VN(G)p_K)$. In particular $\pi_p$ is of completely bounded norm less than $\frac{1}{\mu_G(K)^{\frac1p}}$.
\end{lemma}

\begin{proof}
The subgroup $K$ is open, so $\mu_G|_{K}$ is a Haar measure on $K$ and $\mu_K = \frac{1}{\mu_G(K)} \mu_G|_{K}$. So
$$
p_K
=\lambda\big(\mu_K^0\big)
=\lambda\bigg(\bigg(\frac{1}{\mu_G(K)} \mu_G|_{K}\bigg)^0\bigg)
=\frac{1}{\mu_G(K)}\lambda\big((\mu_G|_{K})^0\big)
=\frac{1}{\mu_G(K)}\lambda(1_K\mu_G)
=\frac{1}{\mu_G(K)}\lambda(1_{K}).
$$
Note that the group $G/K$ is discrete by \cite[Theorem 5.21]{HR} since $K$ is open and that $p_K=p_Kp_K^*$. For any $s \in G$, using Plancherel formula \eqref{Formule-Plancherel} in the second equality, we obtain
\begin{align*}
\label{Calcul-trace-preserving-averaging}
\MoveEqLeft
\tau_G(\pi(\lambda_{sK})) 
=\tau_G(\lambda_sp_K) 
=\tau_G(p_K^* \lambda_s p_K) 
=\frac{1}{\mu_G(K)^2} \tau_G(\lambda(1_{K})^* \lambda_s \lambda(1_{K})) \\
&=\frac{1}{\mu_G(K)^2} \langle 1_{K}, 1_{sK} \rangle_{\L^2(G)} 
=\frac{1}{\mu_G(K)} 1_{K}(s) 
=\frac{1}{\mu_G(K)} \tau_{G/K}(\lambda_{sK}). 
\end{align*}
The statements on induced maps by $\pi$ between $\L^p$-spaces are now standard using interpolation. Indeed, if $x \in \L^p(\VN(G/K))$ we have $\big(\tau_G(|\pi(x)|^p)\big)^{\frac{1}{p}}=\big(\frac{1}{\mu_G(K)} \tau_{G/K}(|x|^p)\big)^{\frac{1}{p}}$.
\end{proof}

\paragraph{Noncommutative $\L^p$-spaces of group von Neumann algebras.}

By \eqref{Formule-Plancherel}, the linear map $\L^1(G) \cap \L^2(G) \to \L^2(\VN(G))$, $g \mapsto  \lambda(g)$ is an isometric map which can be extended to an isometry between $\L^2(G)$ and $\L^2(\VN(G))$ using \cite[Corollary 9.3]{Sti}.

We need a convenient dense subspace of $\L^p(\VN(G))$. If $p=\infty$, \cite[Corollary 7 page 51]{Der4} says\footnote{\thefootnote.  Note that $\mathrm{PM}_2(G)=\VN(G)$.} that $\lambda(\mathrm{C}_c(G))$ is weak* dense in $\VN(G)$, so by Kaplansky's density theorem, the closed unit ball of $\lambda(\mathrm{C}_c(G))$ is weak* dense in the closed unit ball of $\VN(G)$. Moreover, it is proved in \cite[Proposition 4.7]{Daws4} (see \cite[Proposition 3.4]{Eym} for the case $p=1$) that $\lambda(\Span \mathrm{C}_c(G) \ast \mathrm{C}_c(G))$ is dense in $\L^p(\VN(G))$ in the case $1 \leq p < \infty$.



\paragraph{Fourier multipliers on noncommutative $\L^p$-spaces.}

Note that if $\phi \in \L^1_{\loc}(G)$ is a locally integrable function and if $f \in \mathrm{C}_c(G)$ then the product $\phi f$ belongs to $\L^1(G)$ and consequently induces a bounded operator $\lambda(\phi f) \co \L^2(G) \to \L^2(G)$. Recall that this operator is equal to the weak integral $\int_{G} \phi(s) f(s) \lambda_s \d \mu_G(s)$. Finally, recall that $\L^2_{\loc}(G) \subset \L^1_{\loc}(G)$.

\begin{defi}
\label{defi-Lp-multiplier}
Let $G$ be a unimodular locally compact group. Suppose $1 \leq p \leq \infty$. Then  we say that a (weak* continuous if $p=\infty$) bounded operator $T \co \L^p(\VN(G)) \to \L^p(\VN(G))$ is a ($\L^p$) Fourier multiplier if there exists a locally 2-integrable function $\phi \in \L^2_{\loc}(G)$ such that for any $f \in \mathrm{C}_c(G) * \mathrm{C}_c(G)$ ($f \in \mathrm{C}_c(G)$ if $p=\infty$) the element $\int_{G} \phi(s) f(s) \lambda_s \d \mu_G(s)$ belongs to $\L^p(\VN(G))$ and 
\begin{equation}
\label{equ-def-Fourier-mult}
T\bigg(\int_{G} f(s) \lambda_s \d \mu_G(s)\bigg) 
= \int_{G} \phi(s) f(s) \lambda_s \d \mu_G(s), \quad \text{i.e.} \quad T(\lambda(f)) =\lambda(\phi f).
\end{equation}
In this case, we let $T=M_\phi$.
\end{defi}
Then $\mathfrak{M}^p(G)$ is defined to be the space of all bounded $\L^p$ Fourier multipliers and $\mathfrak{M}^{p,\cb}(G)$ to be the subspace consisting of completely bounded $\L^p$ Fourier multipliers.

Note that we take symbols in $\L^2_{\loc}(G)$ to use Plancherel formula \eqref{Formule-Plancherel} in the sequel of this section. We will see in Proposition \ref{prop-M2-Fourier-multipliers} combined with Lemma \ref{lemma-inclusion-Fourier-multipliers} that the symbol $\phi$ of a bounded Fourier multiplier necessarily belongs to the smaller space $\L^\infty(G)$. So, we could replace $\L^2_{\loc}(G)$ by $\L^\infty(G)$ in the definition. It is not clear if we can replace $\L^2_{\loc}(G)$ by $\L^1_{\loc}(G)$ for an arbitrary group $G$. Recall that the space $\L^1(\VN(G))$ canonically identifies to the Fourier algebra $\mathrm{A}(G)$. Using the regularity of the Fourier algebra \cite[Th 2.3.8]{KaL1}, it is not difficult in the case $p=1$ to see that a Fourier multiplier $\phi$ is equal almost everywhere to a continuous complex function defined on $G$. Moreover, there exists\footnote{\thefootnote. This is clear since the regular representation $\lambda \co \L^1(G) \to \B(\L^2(G))$, $f \mapsto \lambda(f)$ is injective by \cite[page 285]{Dix}.} at most one function $\phi$ (up to identity almost everywhere) such that $T = M_\phi$ and we say that $\phi$ induces the bounded Fourier multiplier $M_\phi$. Finally, it is obvious that the linear map $\mathrm{M}\mathrm{A}(G) \to \mathfrak{M}^1(G)$, $\varphi \mapsto M_{\varphi}$ is an isometry, where the space $\mathrm{M}\mathrm{A}(G)$ of multipliers of the Fourier algebra $\mathrm{A}(G)$ is defined in \cite[pages 153-154]{KaL1}.*

Finally, note that we can see $\mathfrak{M}^\infty(G)$ as a subset of the space $\B(\mathrm{C}_\lambda^*(G),\VN(G))$ where $\mathrm{C}_\lambda^*(G)$ is the reduced $\mathrm{C}^*$-algebra of $G$. See \cite[Remark 1.3]{KaL1}.



 %

The following results generalize the alluded observations of \cite{Har} done for discrete groups.

\begin{lemma}
\label{lemma-duality-Fourier-multipliers}
Let $G$ be a unimodular locally compact group. Suppose $1 \leq p \leq \infty$. We have the isometries $\mathfrak{M}^p(G) \to \mathfrak{M}^{p^*}(G)$, $M_\phi \mapsto M_\phi$ and $\mathfrak{M}^{p,\cb}(G) \to \mathfrak{M}^{p^*,\cb}(G)$, $M_\phi \mapsto M_\phi$. Moreover, the Banach adjoint $(M_\phi)^* \co \L^{p^*}(\VN(G))\to \L^{p^*}(\VN(G))$ (preadjoint if $p=\infty$) of $M_\phi \co \L^p(\VN(G))\to \L^p(\VN(G))$ identifies to the Fourier multiplier whose symbol is $\check{\phi}$. Moreover, the maps $\mathfrak{M}^p(G) \to \mathfrak{M}^{p}(G)$, $M_\phi \mapsto M_{\ovl{\phi}}$ and $\mathfrak{M}^{p,\cb}(G) \to \mathfrak{M}^{p^*,\cb}(G)$, $M_\phi \mapsto M_{\ovl{\phi}}$ are isometries. Finally, we can replace $M_{\ovl{\phi}}$ by $M_{\check{\phi}}$ in the last sentence.
\end{lemma}

\begin{proof}
Let $M_\phi \co \L^p(\VN(G)) \to \L^p(\VN(G))$ be an element of $\mathfrak{M}^p(G)$. For any $f, g \in \mathrm{C}_c(G) \ast \mathrm{C}_c(G)$ ($f \in \mathrm{C}_c(G)$ and $g \in \mathrm{C}_c(G)*\mathrm{C}_c(G)$ if $p=\infty$ and  $f \in \mathrm{C}_c(G)*\mathrm{C}_c(G)$ and $g \in \mathrm{C}_c(G)$ if $p=1$), we have $g,\phi f\in \L^1(G) \cap \L^2(G)$ since $\phi \in \L^2_{\loc}(G)$. Using Plancherel formula \eqref{Formule-Plancherel} in the second and third equalities,  we deduce that
\begin{align*}
\MoveEqLeft
\label{Ultim-eq2}
\tau\big(M_\phi(\lambda(f))\lambda(g)\big)  
=\tau\big(\lambda(\phi f)\lambda(g)\big) 
=\int_G \check{\phi} \check{f} g\d\mu_G 
=\tau\big(\lambda(f) \lambda\big(\check{\phi} g\big)\big) 
=\tau\big(\lambda(f)M_{\check{\phi}}\big(\lambda(g)\big)\big). 
\end{align*}
We conclude that the adjoint $(M_\phi)^* \co \L^{p^*}(\VN(G)) \to \L^{p^*}(\VN(G))$ (preadjoint if $p=\infty$) identifies to the multiplier $M_{\check{\phi}}$. Thus the map $M_{\phi} \mapsto M_{\check{\phi}}$ provides an isometry $\mathfrak{M}^p(G) \to \mathfrak{M}^{p^*}(G)$. 

On the other hand, note that the map $\kappa \co \VN(G) \to \VN(G)$, $\lambda_s \mapsto \lambda_{s^{-1}}$ is a $*$-anti-automorphism of the algebra $\VN(G)$, hence weak* continuous. For any $g \in \mathrm{C}_c(G)$, using \cite[VI.3 Proposition 1]{Bou2} in the second equality, we see that
\begin{align*}
\MoveEqLeft
\kappa(\lambda(g))
=\kappa\bigg(\int_{G} g(s)\lambda_s \d\mu_G(s)\bigg)
=\int_{G} g(s)\kappa(\lambda_s) \d\mu_G(s) \\
&=\int_{G} g(s)\lambda_{s^{-1}} \d\mu_G(s)
=\int_{G} g(s^{-1})\lambda_{s} \d\mu_G(s)
=\lambda(\check{g})           
\end{align*}  
where we use that $\int_G f(s)\lambda_s \d \mu_G(s)$ is a well-defined weak* integral (by \cite[Lemma 2.2]{Arh9} and \cite[Corollary 2, III.38]{Bou2}). For any $f,g \in \mathrm{C}_c(G)$, we deduce that
\begin{align*}
\MoveEqLeft
\tau(\kappa(\lambda(g)\lambda(f)))
=\tau(\kappa(\lambda(g*f)))
=\tau(\lambda(\overbrace{g*f}^{\check{}}))
=\tau\big(\lambda(\check{f}*\check{g})\big) \\
&=\int_G f\check{g} \d\mu_G
=\int_G \check{f} g \d\mu_G
=\tau(\lambda(g*f))
=\tau(\lambda(g)\lambda(f))
\end{align*} 
We conclude with \cite[Theorem 6.2]{Str} that $\kappa$ preserves the trace. Hence, it induces an isometric map $\kappa_{p^*} \co \L^{p^*}(\VN(G))\to \L^{p^*}(\VN(G))$. Now, if $M_{\phi}$ belongs to $\mathfrak{M}^{p^*}(G)$ note that the map $\kappa_{p^*}^\op \circ M_{\varphi} \circ \kappa_{p^*} \co \L^{p^*}(\VN(G))\to \L^{p^*}(\VN(G))$ identifies to the multiplier $M_{\check{\phi}}$. We conclude that the map $\mathfrak{M}^{p^*}(G) \to \mathfrak{M}^{p^*}(G)$, $M_{\varphi} \mapsto M_{\check{\varphi}}$ is an isometry. We conclude by composition that the map $\mathfrak{M}^p(G) \to \mathfrak{M}^{p^*}(G)$, $M_\phi \mapsto M_\phi$ is an isometry. To show the isometry $\mathfrak{M}^{p,\cb}(G) = \mathfrak{M}^{p^*,\cb}(G)$, we proceed in the same way using Lemma \ref{Lemma-op-mapping-1} observing that $\kappa_{p^*} \co \L^{p^*}(\VN(G))^\op \to \L^{p^*}(\VN(G))$ is completely isometric. Finally, with the isometric map $\Theta \co \L^{p}(\VN(G)) \to \L^{p}(\VN(G))$, $x \mapsto x^*$, it is easy to check that the map $\Theta M_{\phi}\Theta \co \L^{p}(\VN(G)) \to \L^{p}(\VN(G))$ identifies to the multiplier $M_{\check{\ovl{\phi}}}$. Moreover, recall that $\Theta \co \ovl{\L^{p}(\VN(G))^\op} \to \L^{p}(\VN(G))$ is a complete isometry. Then it is not difficult to obtain the final assertions.
\end{proof}

\begin{lemma}
\label{prop-M2-Fourier-multipliers}
Let $G$ be a unimodular locally compact group. We have the following isometries 
$$
\mathfrak{M}^2(G)
=\mathfrak{M}^{2,\cb}(G)
=\L^\infty(G).
$$
\end{lemma}

\begin{proof}
Suppose that $\phi \in \L^2_{\loc}(G)$ induces a bounded Fourier multiplier. Using the Plancherel isometry $\L^2(\VN(G)) \cong \L^2(G)$, for any function $f \in \mathrm{C}_c(G)*\mathrm{C}_c(G)$, we obtain (since $\phi f \in \L^1(G) \cap \L^2(G)$) that $\bnorm{M_\phi(\lambda(f))}_{\L^2(\VN(G))} 
=\bnorm{\lambda(\phi f)}_{\L^2(\VN(G))}
=\norm{\phi f}_{\L^2(G)}$. We deduce that 
$$
\norm{M_\phi}_{\L^2(\VN(G)) \to \L^2(\VN(G))} 
=\sup_{f \in \mathrm{C}_c(G)*\mathrm{C}_c(G),\norm{f}_{\L^2(G)} \leq 1} \norm{\phi f}_{\L^2(G)} 
=\norm{\phi}_{\L^\infty(G)}.
$$ 
Conversely, if $\phi \in \L^\infty(G)$ then for any $f \in \mathrm{C}_c(G)*\mathrm{C}_c(G)$ we have $\phi f \in \L^1(G) \cap \L^2(G)$ and consequently $\lambda(\phi f) \in \L^2(\VN(G))$. Moreover, we have $\bnorm{\lambda(\phi f)}_{\L^2(\VN(G))}=\norm{\phi f}_{\L^2(G)} \leq \norm{\phi}_{\L^\infty(G)}\norm{f}_{\L^2(G)}$. So $\phi$ induces a bounded Fourier multiplier on $\L^2(\VN(G))$. This shows that $\mathfrak{M}^2(G) = \L^\infty(G)$. 

Moreover, the operator space structure of $\L^2(\VN(G))$ turns it into an operator Hilbert space \cite[page~139]{Pis7}, so that the completely bounded mappings on $\L^2(\VN(G))$ coincide with the bounded ones by \cite[page~127]{Pis7}. We conclude that $\mathfrak{M}^{2,\cb}(G) = \mathfrak{M}^2(G) = \L^\infty(G)$.
\end{proof}

\begin{lemma}
\label{lemma-inclusion-Fourier-multipliers}
Let $G$ be a unimodular locally compact group. Suppose $1 \leq p \leq q \leq 2$. We have the contractive inclusions $\mathfrak{M}^1(G) \subset \mathfrak{M}^p(G) \subset \mathfrak{M}^q(G) \subset \mathfrak{M}^2(G)$ and $\mathfrak{M}^{1,\cb}(G) \subset \mathfrak{M}^{p,\cb}(G) \subset \mathfrak{M}^{q,\cb}(G) \subset \mathfrak{M}^{2,\cb}(G)$.
\end{lemma}

\begin{proof}
Note that the first inclusion is a particular case of the second inclusion. If $M_\phi$ belongs to $\mathfrak{M}^p(G)$ then by Lemma \ref{lemma-duality-Fourier-multipliers}, it also belongs to $\mathfrak{M}^{p^*}(G)$, consequently, by complex interpolation, $M_\phi$ belongs to $\mathfrak{M}^2(G)$. Using again interpolation between 2 and $p$, we deduce that $M_\phi$ belongs to $\mathfrak{M}^q(G)$. The second chain is proved in the same manner.
\end{proof}

The first part of the following result generalizes \cite[Lemma 5.1.4]{KaL1}.

\begin{lemma}
\label{Lemma-Symbol-weakstar-convergence-weakLp-vrai}
Let $G$ be a unimodular locally compact group. Suppose $1 \leq p \leq \infty$. Let $(M_{\phi_j})$ be a bounded net of bounded Fourier multipliers on $\L^p(\VN(G))$ and suppose that $\phi$ is an element of $\L^\infty(G)$ such that $(\phi_j)$ converges to $\phi$ for the weak* topology of $\L^\infty(G)$. Then $\phi$ induces a bounded Fourier multiplier on $\L^p(\VN(G))$. In addition if $1<p<\infty$, the net $(M_{\phi_j})$ converges to $M_{\phi}$ for the weak operator topology of $\B(\L^p(\VN(G)))$ and
$$
\norm{M_\phi}_{\L^p(\VN(G)) \to \L^p(\VN(G))} 
\leq \liminf_{j \to \infty} \norm{M_{\phi_j}}_{\L^p(\VN(G)) \to \L^p(\VN(G))}.
$$
If $p=\infty$, for any functions $f \in \mathrm{C}_c(G)$ and $g \in \mathrm{C}_c(G) \ast \mathrm{C}_c(G)$, we have 
$$\big\langle M_{\phi_j}(\lambda(f)), \lambda(g) \big\rangle_{\VN(G),\L^1(\VN(G))}\xra[\ j \ ]{} \big\langle M_{\phi}(\lambda(f)), \lambda(g) \big\rangle_{\VN(G),\L^1(\VN(G))}.$$ 
A similar statement is true by replacing ``bounded'' by ``completely bounded'' and the norms by $\norm{\cdot}_{\cb,\L^p(\VN(G)) \to \L^p(\VN(G))}$.
\end{lemma}

\begin{proof}
For any functions $f, g \in \mathrm{C}_c(G) \ast \mathrm{C}_c(G)$ (to adapt if if $p=1$), we have $f\check{g} \in \L^1(G)$. For any $j$, we have
\begin{align*}
\MoveEqLeft
\left|\int_G \phi_j f \check{g}\d \mu_G\right| 
=\left|\big\langle \lambda(\phi_j f), \lambda(g) \big\rangle_{\L^p(\VN(G)),\L^{p^*}(\VN(G))}\right|\\
&=\left|\big\langle M_{\phi_j}(\lambda(f)), \lambda(g) \big\rangle_{\L^p(\VN(G)),\L^{p^*}(\VN(G))}\right|\\
&\leq \norm{M_{\phi_j}}_{\L^p(\VN(G)) \to \L^p(\VN(G))} 
\norm{\lambda(f)}_{\L^p(\VN(G))} \norm{\lambda(g)}_{\L^{p^*}(\VN(G))}.
\end{align*}
Passing to the limit, we obtain
$$
\left|\int_G \phi f \check{g}\d \mu_G\right| 
\leq \liminf_{j \to \infty} \norm{M_{\phi_j}}_{\L^p(\VN(G)) \to \L^p(\VN(G))}
\norm{\lambda(f)}_{\L^p(\VN(G))} \norm{\lambda(g)}_{\L^{p^*}(\VN(G))}.
$$
By density if $p<\infty$, we conclude that $\phi$ induces a bounded Fourier multiplier on $\L^p(\VN(G))$ with the estimate on the norm (use duality if $p=\infty$).

Using again Plancherel formula \eqref{Formule-Plancherel} and the weak* convergence of the net $(\phi_j)$, we deduce that for any functions $f, g \in \mathrm{C}_c(G) \ast \mathrm{C}_c(G)$
\begin{align*}
\MoveEqLeft
\big\langle(M_{\phi}-M_{\phi_j})(\lambda(f)), \lambda(g) \big\rangle_{\L^p(\VN(G)),\L^{p^*}(\VN(G))}
=\tau\big((M_{\phi}-M_{\phi_j})(\lambda(f))\lambda(g)\big) \\
&=\tau\big(\lambda((\phi-\phi_j) f)\lambda(g)\big) 
=\int_G (\phi-\phi_j) f\check{g} \d\mu_G
=\big\langle \phi-\phi_j,f\check{g} \big\rangle_{\L^\infty(G),\L^1(G)} 
\xra[\ j \ ]{} 0.            
\end{align*} 
By density, using a $\frac{\epsi}{4}$-argument and the boundedness of the net, we conclude\footnote{\thefootnote. More precisely, if $X$ is a Banach space, if $E_1$ is dense subset of $X$, if $E_2$ is a dense subset of $X^*$ and if $(T_j)$ is a bounded net of $\B(X)$ with an element $T$ of $\B(X)$ such that $\langle T_j(x),x^*\rangle \xra[\ j\ ]{}  \langle T(x), x^*\rangle$ for any $x \in E_1$ and any $x^* \in E_2$, then the net $(T_j)$ converges to $T$ for the weak operator topology of $\B(X)$.} the proof. The case $p=\infty$ is similar.

Now, we prove the last sentence, it suffices to show that $\phi$ induces a completely bounded Fourier multiplier. For any $f_{kl},g_{kl} \in \mathrm{C}_c(G)*\mathrm{C}_c(G)$ ($f_{kl} \in \mathrm{C}_c(G)$ if $p=\infty$) where $1 \leq k,l \leq N$, we have $f_{kl} \check{g}_{kl}\in \L^1(G)$ and for any $j$
\begin{align*}
\MoveEqLeft
\left| \big\langle \big[ M_{\phi_{j}}(\lambda(f_{kl}))\big] , \big[\lambda(g_{kl})\big] \big\rangle_{\M_N(\L^p(\VN(G))),S^1_N(\L^{p^*}(\VN(G)))} \right|  \\
& \leq \norm{M_{\phi_j}}_{\cb,\L^p(\VN(G)) \to \L^p(\VN(G))} \bnorm{\big[\lambda(f_{kl})\big]}_{\M_N(\L^p(\VN(G)))} \bnorm{\big[\lambda(g_{kl})\big]}_{S^1_N(\L^{p^*}(\VN(G)))},
\end{align*}
that is, using Plancherel formula \eqref{Formule-Plancherel},
\begin{align*}
\MoveEqLeft
\left|\sum_{k,l=1}^N \int_G \phi_j(s) f_{kl}(s) \check{g}_{kl}(s) \d\mu_G(s) \right| \\
& \leq \norm{M_{\phi_j}}_{\cb,\L^p(\VN(G)) \to \L^p(\VN(G))} \bnorm{\big[\lambda(f_{kl})\big]}_{\M_N(\L^p(\VN(G)))} \bnorm{\big[\lambda(g_{kl})\big]}_{S^1_N(\L^{p^*}(\VN(G)))}.
\end{align*}
Passing to the limit, we obtain
\begin{align*}
\MoveEqLeft
\left|\sum_{k,l=1}^N \int_G \phi(s) f_{kl}(s) \check{g}_{kl}(s) \d\mu_G(s) \right| \\
&\leq \liminf_{j \to \infty} \norm{M_{\phi_j}}_{\cb,\L^p(\VN(G)) \to \L^p(\VN(G))} \bnorm{\big[\lambda(f_{kl})\big]}_{\M_N(\L^p(\VN(G)))} \bnorm{\big[\lambda(g_{kl})\big]}_{S^1_N(\L^{p^*}(\VN(G)))}.            
\end{align*}  
We deduce that $\phi$ induces a completely bounded Fourier multiplier on $\L^p(\VN(G))$ with the suitable estimate on the completely bounded norm.
\end{proof}

\begin{lemma}
\label{lem-Fourier-multiplier-weak-star-closed}
Let $G$ be a unimodular locally compact group and $1 < p < \infty$. Then the space $\mathfrak{M}^{p,\cb}(G)$ is weak* closed in $\CB(\L^p(\VN(G)))$. Similarly, the space $\mathfrak{M}^p(G)$ is weak* closed in the space $\B(\L^p(\VN(G)))$. Finally, the spaces $\mathfrak{M}^{\infty}(G)$ and $\mathfrak{M}^{\infty,\cb}(G)$ are weak* closed in the spaces $\B(\mathrm{C}_\lambda^*(G),\VN(G))$ and $\CB(\mathrm{C}_\lambda^*(G),\VN(G))$.
\end{lemma}

\begin{proof}
By the Banach-Dieudonn\'e theorem \cite[page 154]{Hol1}, it suffices to show that the closed unit ball of $\mathfrak{M}^{p,\cb}(G)$ is weak* closed in $\CB(\L^p(\VN(G)))$. Let $(M_{\phi_j})$ be a net in that unit ball converging for the weak* topology to some completely bounded map $T \co \L^p(\VN(G)) \to \L^p(\VN(G))$. By Lemma \ref{prop-M2-Fourier-multipliers} and Lemma \ref{lemma-inclusion-Fourier-multipliers}, for any $j$, we have 
$$
\norm{\phi_j}_{\L^\infty(G)} 
\leq \norm{M_{\phi_j}}_{\cb,\L^p(\VN(G))\to \L^p(\VN(G))} 
\leq 1.
$$ 
Hence by Banach-Alaoglu's theorem there exists a subnet of $(\phi_j)$ converging for the weak* topology to some $\phi \in \L^\infty(G)$. It remains to show that $T = M_\phi$. Recall that the predual of the space $\CB(\L^p(\VN(G)))$ is given by $\L^p(\VN(G)) \widehat{\ot} \L^{p^*}(\VN(G))^\op$, where $\widehat{\ot}$ denotes the operator space projective tensor product and the duality bracket is given by
$$ 
\langle T, x \ot y \rangle_{\CB(\L^p(\VN(G))),\L^p(\VN(G)) \widehat{\ot} \L^{p^*}(\VN(G))} 
=\big\langle T(x), y \big\rangle_{\L^p(\VN(G)), \L^{p^*}(\VN(G))}.
$$
This implies that $\big\langle M_{\phi_j}(x), y\big\rangle \xra[j]{} \langle T(x),y \rangle$ for any $x \in \L^p(\VN(G))$ and any $y \in \L^{p^*}(\VN(G))$. By Lemma \ref{Lemma-Symbol-weakstar-convergence-weakLp-vrai}, the net $(M_{\phi_j})$ converges in addition to $M_{\phi}$ for the weak operator topology. So by uniqueness of the limit, we obtain that $T = M_\phi$.

For the last sentence, we use a similar proof where here $T \co \mathrm{C}_\lambda^*(G) \to \VN(G)$. On the one hand, we have $\big\langle M_{\phi_j}(x), y\big\rangle \xra[j]{} \langle T(x),y \rangle$ for any $x \in \mathrm{C}_\lambda^*(G)$ and any $y \in \L^1(\VN(G))$. On the other hand by Lemma \ref{Lemma-Symbol-weakstar-convergence-weakLp-vrai}, for any $f \in \mathrm{C}_c(G)$ and any $g \in \mathrm{C}_c(G) \ast \mathrm{C}_c(G)$, we have $\big\langle M_{\phi_j}(\lambda(f)), \lambda(g) \big\rangle \to \big\langle M_{\phi}(\lambda(f)), \lambda(g) \big\rangle$. By uniqueness of the limit, we deduce that $\big\langle T(\lambda(f)),\lambda(g) \big\rangle_{\VN(G),\L^1(\VN(G))}=\big\langle M_{\phi}(\lambda(f)), \lambda(g) \big\rangle_{\VN(G),\L^1(\VN(G))}$. Consequently, we obtain $M_{\phi}(\lambda(f))= T(\lambda(f))$ for any function $f \in \mathrm{C}_c(G)$. Finally $T=M_{\phi}$. 

The statement on the space $\mathfrak{M}^p(G)$ can be proved in a similar manner, using the predual $\L^p(\VN(G)) \hat{\ot} \L^{p^*}(\VN(G))$ of $\B(\L^p(\VN(G)))$ where $\hat{\ot}$ denotes the Banach space projective tensor product. The last sentence is similar.
\end{proof}

\begin{remark}\normalfont
\label{Max-com-subset}
We do not know if $\mathfrak{M}^{p,\cb}(G)$ and $\mathfrak{M}^{p}(G)$ are maximal commutative subsets of $\CB(\L^p(\VN(G)))$ and $\B(\L^p(\VN(G)))$ which is a stronger assertion. 
\end{remark}

If $G$ is an \textit{abelian} locally compact group and if $M_\varphi \co \L^p(G) \to \L^p(G)$ is a positive multiplier in $\mathfrak{M}^p(\hat{G})$, note that $\varphi$ is equal almost everywhere to a function of the Fourier-Stieltjes algebra $\B(\hat{G})$, thus to a continuous function. The next lemma extends this result to the noncommutative context. 

\begin{lemma}
\label{Lemma-positive-Fourier-multipliers}
Let $G$ be a unimodular locally compact group. Suppose $1 \leq p \leq \infty$. Let $\varphi \co G \to \C$ be a complex function which induces a positive Fourier multiplier $M_\varphi \co \L^p(\VN(G)) \to \L^p(\VN(G))$. Then $\varphi$ is equal almost everywhere to a continuous function.
\end{lemma}

\begin{proof}
We can suppose $1 <p <\infty$. Let $g \in \mathrm{C}_c(G)$. Then the operator $\lambda(g^* \ast g)=\lambda(g)^*\lambda(g) \co \L^2(G) \to \L^2(G)$ is positive. Moreover, by \eqref{Def-mtauG}, it belongs to $\mathfrak{m}_{\tau_G} \subset \L^p(\VN(G))$. We conclude that $\lambda(g^**g)$ belongs to $\L^p(\VN(G))_+$. We deduce that $M_\varphi\big(\lambda(g^* *g)\big)$ is a positive element of $\L^p(\VN(G))$. Since $\varphi(g^**g)$ belongs to $\L^1(G) \cap \L^2(G)$, the operator $M_\varphi\big(\lambda(g^**g)\big)=\lambda(\varphi(g^**g))$ is bounded on $\L^2(G)$. 
Now, for any $\xi \in \L^2(G)$, by positivity,
\begin{align*}
\MoveEqLeft
0  \leq \big\langle M_\varphi \big(\lambda(g^* \ast g)\big) \xi, \xi \big\rangle_{\L^2(G)}
=\bigg\langle\bigg(\int_G  \varphi(s)(g^**g)(s)\lambda_s \d\mu_G(s)\bigg) \xi, \xi\bigg\rangle_{\L^2(G)}\\
&=\int_G  \Big\langle\big(\varphi(s)(g^**g)(s)\lambda_s\big)\xi , \xi\Big\rangle_{L^2(G)}\d\mu_G(s)\\
&=\int_G  \bigg(\int_G \ovl{g(t^{-1})} g(t^{-1}s) \d\mu_G(t)\bigg) \varphi(s)\big\langle \lambda_s \xi, \xi \big\rangle_{\L^2(G)} \d\mu_G(s) \\
&=\int_G \bigg(\int_G  \ovl{g(t)} g(ts) \d\mu_G(t)\bigg) \varphi(s)\big\langle \lambda_s \xi, \xi \big\rangle_{\L^2(G)} \d\mu_G(s)\\  
&=\int_G \int_G  \ovl{g(t)} g(s) \varphi(t^{-1}s)\big\langle \lambda_{t^{-1}s} \xi, \xi \big\rangle_{\L^2(G)} \d\mu_G(s) \d\mu_G(t).
\end{align*}
Hence the function $s \mapsto \varphi(s) \big\langle \lambda_s \xi, \xi \big\rangle_{\L^2(G)}$ of $\L^\infty(G)$ is positive definite \cite[VII.3, Definition 3.20]{Tak2}, \cite[page 296]{Dix}. By \cite[VII.3, Corollary 3.22]{Tak2}, we deduce that it coincides almost everywhere with a continuous function on $G$. To conclude the lemma, it suffices now to show that there exists a neighborhood $K_1$ of $e \in G$ such that for any $s_0 \in G$, there exists $\xi \in \L^2(G)$ such that $\big\langle \lambda_s \xi, \xi \big\rangle_{\L^2(G)}$ does not vanish for $s \in K_1 s_0$. To this end, let $K_0$ be a compact neighborhood of $e$ and set $K = K_1^{-1} \cdot K_0,$ which is also compact. Let $\xi_0 \in \L^2(G)$ such that $\xi_0 \geq 0$ almost everywhere and $\xi_0 > 0$ on $K$. Put $\xi = \xi_0 + \lambda_{s_0^{-1}} \xi_0$. Then 
\begin{align*}
\big\langle \lambda_s \xi, \xi \big\rangle_{\L^2(G)} 
& = \Big\langle \lambda_s \big(\xi_0 + \lambda_{s_0^{-1}} \xi_0\big), \xi_0 + \lambda_{s_0^{-1}} \xi_0 \Big\rangle_{L^2(G)} 
\geq \Big\langle \lambda_{ss_0^{-1}} \xi_0 , \xi_0 \Big\rangle_{\L^2(G)} \\ 
&= \int_G \xi_0\big(s_0 s^{-1} t\big) \xi_0(t) \d\mu_G(t) 
 \geq \int_{K_0} \xi_0\big(s_0 s^{-1} t\big) \xi_0(t) \d\mu_G(t).
\end{align*}
For $t \in K_0$ and $s \in K_1s_0$, we have $s_0 s^{-1} t \in K_1^{-1} K_0=K$, so that $\xi_0(s_0 s^{-1} t) > 0$. Also, $\xi_0(t) > 0$ for such $t$. Thus, the last integral is strictly positive for $s \in K_1 s_0$, and the lemma is shown.
\end{proof}

\begin{prop}
\label{Th-cp-Fourier-multipliers}
Let $G$ be a unimodular locally compact group. Suppose $1 \leq p \leq \infty$. The following are equivalent for a complex measurable function $\varphi \co G \to \C$\footnote{\thefootnote. This proposition admits a generalization for $n$-positive maps.}.
\begin{enumerate}
	\item $\varphi$ induces a completely positive Fourier multiplier $M_\varphi \co \L^p(\VN(G)) \to \L^p(\VN(G))$. 
	\item $\varphi$ induces a completely positive Fourier multiplier $M_\varphi \co \VN(G) \to \VN(G)$.
	\item $\varphi$ is equal almost everywhere to a continuous positive definite function.
\end{enumerate} 
\end{prop}

\begin{proof}
3. $\Rightarrow$ 2.: This is \cite[Proposition 4.3]{DCH}.

2. $\Rightarrow$ 1.: Suppose first that $M_\varphi \co \VN(G) \to \VN(G)$ is completely positive. Since $M_\varphi$ is bounded on $\VN(G)$, by Lemma \ref{lemma-inclusion-Fourier-multipliers}, $\varphi$ induces a Fourier multiplier on $\L^p(\VN(G))$ which is\footnote{\thefootnote. We use here the fact, left to the reader, that if $T \co M \to N$ is a completely positive map which induces a bounded map $T_p \co \L^p(M) \to \L^p(N)$ then $T_p$ is also completely positive.} completely positive.

1. $\Rightarrow$ 3.: According to Lemma \ref{Lemma-positive-Fourier-multipliers}, the function $\varphi$ is continuous almost everywhere, so we can assume that $\varphi$ is continuous without changing the operator $M_\varphi$.
For $i = 1 ,\ldots, n$ let $f_i \in \mathrm{C}_c(G)$.
Note that by \cite[Proposition 2.1]{Lan} the matrix $[\lambda(f_i^**f_j)]=[\lambda(f_i)^*\lambda(f_j)]$ is a positive element of $\M_n(\VN(G))$ and an element of $\M_n(\L^p(\VN(G)))$ by \eqref{Def-mtauG}, hence a positive element of $\M_n(\L^p(\VN(G)))$. Consequently, $(\Id_{\M_n} \ot M_\varphi)[\lambda(f_i^**f_j)]=[\lambda(\varphi(f_i^**f_j))]$ is an element of $\M_n(\L^p(\VN(G))_+ \cap \M_n(\VN(G))$. In particular, for any $g_1,\ldots,g_n \in \mathrm{C}_c(G)$ we have
$$
\sum_{i,j=1}^{n} \big\langle\lambda\big(\varphi(f_i^**f_j)\big)\ovl{g_j}, \ovl{g_i}\big\rangle_{\L^2(G)} \geq 0
$$
that is
$$
\sum_{i,j=1}^{n} \int_{G} \varphi(s)(f_i^**f_j)(s)(g_i* \tilde{g}_j)(s) \d \mu_G(s)
\geq 0.
$$
By \cite[Proposition 4.3 and Proposition 4.2]{DCH}, we conclude that the function $\varphi$ is continuous and positive definite.
\end{proof}

\begin{prop}
\label{Prop-Convergence-of-multipliers}
Let $G$ be a unimodular locally compact group. Suppose $1 \leq p < \infty$. Let $(M_{\phi_n})$ be a bounded sequence of bounded Fourier multipliers on $\L^p(\VN(G))$ such that $(\phi_n)$ converges almost everywhere to some function $\phi \in \L^\infty(G)$. Then $\phi$ induces a bounded Fourier multiplier $M_\phi$ on $\L^p(\VN(G))$ and 
$$
\norm{M_\phi}_{\L^p(\VN(G)) \to \L^p(\VN(G))} 
\leq \liminf_{n \to +\infty} \norm{M_{\phi_n}}_{\L^p(\VN(G)) \to \L^p(\VN(G))}.
$$
A similar result is true for completely bounded multipliers.
\end{prop}

\begin{proof}
By Lemma \ref{prop-M2-Fourier-multipliers}, the sequence $(\phi_n)$ of functions is  uniformly bounded in the norm $\norm{\cdot}_{\L^\infty(G)}$. Note that if $f \in \L^1(G)$, the sequence $(\int_G \phi_n f \d \mu_G)$ converges to $\int_G \phi f \d \mu_G$ by the dominated convergence theorem. Hence $(\phi_n)$ converges to $\phi$ for the weak* topology of $\L^\infty(G)$. The conclusion is a consequence of Lemma \ref{Lemma-Symbol-weakstar-convergence-weakLp-vrai}. 
\end{proof}
\subsection{The completely bounded homomorphism theorem for Fourier multipliers}
\label{sec-The completely bounded homomorphism theorem}

Suppose $1 \leq p <\infty$. Let us remind the definition of a Schur multiplier on $S^p_\Omega=\L^p(\B(\L^2(\Omega)))$ where $(\Omega,\mu)$ is a ($\sigma$-finite) measure space \cite[Section 1.2]{LaS}. If $f \in \L^2(\Omega \times \Omega)$, we denote by $K_f \co \L^2(\Omega) \to \L^2(\Omega)$, $u \mapsto \int_\Omega f(z,\cdot) u(z) \d z$ the integral operator with kernel $f$. We say that a measurable function $\phi \co \Omega \times \Omega \to \C$ induces a bounded Schur multiplier on $S^p_\Omega$ if for any $f \in \L^2(\Omega \times \Omega)$ satisfying $K_f \in S^p_\Omega$ we have $K_{\phi f} \in S^p_\Omega$ and if the map $S^2_\Omega \cap S^p_\Omega \to S^p_\Omega$, $K_f \mapsto K_{\phi f}$ extends to a bounded map $M_\phi$ from $S^p_\Omega$ into $S^p_\Omega$ called the Schur multiplier associated with $\phi$. We denote by $\mathfrak{M}_{\Omega}^{p,\cb}$ the space of completely bounded Schur multipliers on $S^p_\Omega$. We refer to the surveys \cite{ToT1} and \cite{Tod1} for the case $p=\infty$. 

Let $G$ be a unimodular locally compact group.  The right regular representation $\rho \co G \to \B(\L^2(G))$ is given by $(\rho_t\xi)(s) = \xi(st)$. Recall that $\rho$ is a strongly continuous unitary representation. We will use the notation $\mathrm{Ad}_{\rho_s}^p \co S^p_G \to S^p_G$, $x \mapsto \rho_s x \rho_{s^{-1}}$. A bounded Schur multiplier $M_\phi \co S^p_G \to S^p_G$ is a Herz-Schur multiplier if $M_\phi\mathrm{Ad}_{\rho_s}^p =\mathrm{Ad}_{\rho_s}^p M_\phi$ for any $s \in G$. In this case, there exists a measurable function $\varphi \co G \to \C$ such that $\phi(r,s)=\varphi(rs^{-1})$ for almost every $r,s \in G$ and we let $M_{\varphi}^{\HS}=M_{\phi}$. We denote by $\frak{M}^{p,\cb,\HS}_G$ the subspace of $\frak{M}^{p,\cb}_G$ of completely bounded Herz-Schur multipliers.

In the sequel $G_{\disc}$ stands for the group $G$ equipped with the discrete topology.

\begin{prop}
\label{prop-homomorphism-Schur-multipliers-cb}
Let $G$ and $H$ be second countable locally compact groups and $\sigma \co G \to H$ be a continuous homomorphism. Suppose $1 \leq p \leq \infty$. If $\varphi \co H \to \C$ is a  continuous function which induces a completely bounded Herz-Schur multiplier $M_\varphi^{\mathrm{HS}} \co S^p_H \to S^p_H$, then the continuous function $\varphi \circ \sigma \co G \to \C$ induces a completely bounded Herz-Schur multiplier $M_{\varphi \circ \sigma}^{\mathrm{HS}} \co S^p_G \to S^p_G$ and
$$
\bnorm{M_{\varphi \circ \sigma}^{\mathrm{HS}}}_{\cb,S^p_G \to S^p_G} 
\leq \bnorm{M_\varphi^{\mathrm{HS}}}_{\cb,S^p_H \to S^p_H}.
$$
Moreover, if $\sigma(G)$ is dense in $H$, we have an isometry\footnote{\thefootnote. The proof shows that if $M_{\varphi \circ \sigma}$ is completely bounded then $M_{\varphi}$ is completely bounded.} $M_\varphi^{\mathrm{HS}} \mapsto M_{\varphi \circ \sigma}^{\mathrm{HS}}$.
\end{prop}

\begin{proof}
Let $G \xra{\pi} G/\ker(\sigma) \xra{\tilde \sigma} \Ran \sigma \xra{i} H$ be the canonical decomposition of the homomorphism $\sigma$. By \cite[Lemma 9.2]{CPPR} (see also the arxiv version, Lemma 8.2), we have
$$
\norm{M_{\varphi \circ i \circ \tilde \sigma \circ \pi}^{\mathrm{HS}}}_{\cb,S^p_G \to S^p_G}
= \norm {M_{\varphi \circ i \circ \tilde \sigma}^{\mathrm{HS}}}_{\cb,S^p_{G/\ker \sigma} \to S^p_{G/\ker \sigma}}.
$$
We have a natural isomorphism $J_{\tilde \sigma} \co S^p_{(G/\ker\sigma)_\disc} \to S^p_{(\Ran \sigma)_\disc}$, $e_{s_1,s_2} \mapsto e_{\tilde{\sigma}(s_1),\tilde{\sigma}(s_2)}$ where the $e_{s_1,s_2}$'s are the matrix units. Therefore, the group isomorphism $\tilde\sigma \co G/\ker\sigma \to \Ran \sigma$ yields an isometric isomorphism from the space of completely bounded Herz-Schur multipliers over $S^p_{(\Ran \sigma)_\disc}$ to the space of completely bounded Herz-Schur multipliers over $S^p_{(G/\ker\sigma)_\disc}$ by sending each $M_{\psi}^{\mathrm{HS}}$ to $M_{\psi \circ \tilde{\sigma}}^{\mathrm{HS}} = J_{\tilde{\sigma}^{-1}} M_{\psi}^{\mathrm{HS}} J_{\tilde\sigma}$. Thus, we obtain using \cite[Lemma 9.2]{CPPR} three times
\begin{align*}
\MoveEqLeft
 \norm{M_{\varphi \circ i \circ \tilde \sigma}^{\mathrm{HS}}}_{\cb,S^p_{G/\ker \sigma} \to S^p_{G/\ker \sigma}}
= \norm{M_{\varphi \circ i \circ \tilde \sigma}^{\mathrm{HS}}}_{\cb,S^p_{(G/\ker \sigma)_{\disc}} \to S^p_{(G/\ker \sigma)_{\disc}}} \\
 &= \norm{M_{\varphi \circ i}^{\mathrm{HS}}}_{\cb,S^p_{(\Ran \sigma)_\disc} \to S^p_{(\Ran \sigma)_\disc}} 
\leq \norm{M_{\varphi}^{\mathrm{HS}}}_{\cb,S^p_{H_\disc} \to S^p_{H_\disc}}
 = \norm{M_\varphi^{\mathrm{HS}}}_{\cb,S^p_H \to S^p_H}.
\end{align*}
This shows the first part of the proposition.

It remains to show the isometric statement in the case where $\Ran \sigma$ is dense in $H$. In the light of the foregoing, we only need to show that
\begin{equation}
	\label{Eq-Conclusion-Schur}
	\norm{M_\varphi^{\mathrm{HS}}}_{\cb,S^p_H \to S^p_H} 
\leq \norm{M_{\varphi \circ i}^{\mathrm{HS}}}_{\cb,S^p_{(\Ran \sigma)_\disc} \to S^p_{(\Ran \sigma)_\disc}}.
\end{equation}
According to \cite[Theorem 1.19]{LaS}, we have
$$
\bnorm{M_\varphi^{\mathrm{HS}}}_{\cb,S^p_H \to S^p_H} 
= \sup_{F \subset H \text{ finite}} \bnorm{M_{\varphi}^{\mathrm{HS}}|_F}_{\cb,S^p_F \to S^p_F}.
$$
Here, the restriction to $F$ means that one considers the mapping $\sum_{s_1,s_2 \in F} a_{s_1,s_2} e_{s_1,s_2} \mapsto \sum_{s_1,s_2 \in F} \varphi(s_1^{-1} s_2) a_{s_1,s_2} e_{s_1,s_2}$. We fix some finite subset $F = \{s_1,\ldots,s_N \} \subset H$ and some $\epsi > 0$. Then for any $1 \leq k,l \leq N$, by continuity of $\varphi$, there exist a neighborhood $V_{k,l}$ of $s_k^{-1} s_l$ such that $|\varphi(t) - \varphi(s_k^{-1}s_l)| < \epsi$ if $t \in V_{k,l}$. Since the mapping $G \times G \to G$, $(s,t) \mapsto s^{-1}t$ is continuous, there exist neighborhoods $W_{k,l}$ of $s_k$ and $W_{k,l}'$ of $s_l$ such that $(W_{k,l})^{-1} W_{k,l}' \subset V_{k,l}$. For any $1 \leq k \leq N$, let now $U_k = \bigcap_{l=1}^N W_{k,l} \cap \bigcap_{l = 1}^N W_{l,k}'$ which is a neighborhood $U_k$ of $s_k$. Since $\Ran \sigma$ is dense in $H$, there exists $t_k \in \Ran \sigma \cap U_k$ and we obtain a subset $\tilde{F}_\epsi = \{t_1,\ldots,t_N\}$ of $\Ran \sigma$ with the same cardinality as $F$. Moreover, $t_k^{-1} t_l$ belongs to $U_k^{-1} U_l \subset V_{k,l}$ and consequently, $|\varphi(t_k^{-1} t_l) - \varphi(s_k^{-1} s_l)| < \epsi$ for any $k,l \in \{1,\ldots,N\}$.

Denote $M_A,M_B \co S^p_N \to S^p_N$ the Schur multipliers with symbols $A=[\varphi(t_k^{-1}t_l)]$ and $B=[\varphi(s_k^{-1}s_l)]$. Then, we obtain using the identifications $S^p_{\tilde{F}_\epsi}=S^p_N$ and $S^p_F=S^p_N$ in the first equality
\begin{align*}
\MoveEqLeft
\left| \norm{M_{\varphi \circ i}^{\mathrm{HS}}|_{\tilde{F}_\epsi}}_{\cb,S^p_{\tilde{F}_\epsi} \to S^p_{\tilde{F}_\epsi}} - \norm{M_{\varphi}^{\mathrm{HS}}|_F }_{\cb,S^p_F \to S^p_F} \right| 
=\left| \norm{M_A}_{\cb,S^p_N \to S^p_N} - \norm{M_B}_{\cb,S^p_N \to S^p_N} \right| \\
&\leq \norm{M_A-M_B}_{\cb,S^p_N \to S^p_N}
= \norm{\sum_{k,l=1}^{N} \big(\varphi(t_k^{-1} t_l) - \varphi(s_k^{-1} s_l)\big)M_{e_{kl}}}_{\cb,S^p_N \to S^p_N}\\
&=\sum_{k,l = 1}^N \left|\varphi(t_k^{-1} t_l) - \varphi(s_k^{-1} s_l)\right| \norm{M_{e_{kl}}}_{\cb,S^p_N \to S^p_N} < N^2\epsi.
\end{align*}
We have shown
$$
\bnorm{M_{\varphi \circ i}^{\mathrm{HS}}|_{\tilde{F}_\epsi}}_{\cb,S^p_{\tilde{F}_\epsi} \to S^p_{\tilde{F}_\epsi}} 
\xra[\epsi \to 0]{}  
\bnorm{ M_{\varphi}^{\mathrm{HS}} |_F}_{\cb,S^p_F \to S^p_F}.
$$
But again according to \cite[Theorem 1.19]{LaS}, the left hand side is dominated by 
$$
\norm{M_{\varphi \circ i}^{\mathrm{HS}}}_{\cb,S^p_{(\Ran \sigma)_\disc} \to S^p_{(\Ran \sigma)_\disc}}.
$$
Hence we obtain \eqref{Eq-Conclusion-Schur}.
\end{proof}

Now, we state a completely bounded version of the classical homomorphism theorem \cite[page 184]{EdG}. 

\begin{thm}
\label{prop-homomorphism-Fourier-multipliers-cb}
Let $G$ and $H$ be locally compact groups and $\sigma \co G \to H$ be a continuous homomorphism. Suppose $1 \leq p \leq \infty$. We suppose that $G$ and $H$ are second countable and amenable if $1 < p <\infty$. If $\varphi \co H \to \C$ is a continuous function which induces a completely bounded Fourier multiplier $M_\varphi \co \L^p(\VN(H)) \to \L^p(\VN(H))$, then the continuous function $\varphi \circ \sigma \co G \to \C$ induces a completely bounded Fourier multiplier $M_\varphi \co \L^p(\VN(G)) \to \L^p(\VN(G))$ and
$$
\norm{M_{\varphi \circ \sigma}}_{\cb,\L^p(\VN(G)) \to \L^p(\VN(G))} 
\leq \norm{M_\varphi}_{\cb,\L^p(\VN(H)) \to \L^p(\VN(H))}.
$$
Moreover, if $\sigma(G)$ is dense in $H$, we have an isometry\footnote{\thefootnote. The proof shows that if $M_{\varphi \circ \sigma}$ is completely bounded then $M_{\varphi}$ is completely bounded.} $M_{\varphi} \mapsto M_{\varphi \circ \sigma}$. Finally, if $M_\varphi$ is completely positive then $M_{\varphi \circ \sigma}$ is also completely positive.
\end{thm}

\begin{proof}
The case $p=\infty$ is \cite[Theorem 6.2]{Spr1}. By duality, we obtain the case $p=1$. Now, we suppose that $1<p<\infty$. Note that by Lemma \ref{prop-M2-Fourier-multipliers} and Lemma \ref{lemma-inclusion-Fourier-multipliers}, the function $\varphi$ is bounded. Then by amenability of $G$ and $H$, using \cite[Theorem 4.2 and Corollary 5.3]{CDS}\footnote{\thefootnote. We warn the reader that the proof of \cite[Theorem 5.2]{CDS} is only valid for second countable groups. The proof uses Lebesgue's dominated convergence theorem in the last line of page 7007 and this result does not admit a generalization for nets. See \cite{KaL} for more information.} with \cite[Remark 9.3]{CPPR} and Proposition \ref{prop-homomorphism-Schur-multipliers-cb}, we obtain
\begin{align*}
\MoveEqLeft
\norm{M_{\varphi \circ \sigma}}_{\cb,\L^p(\VN(G)) \to \L^p(\VN(G))} 
=\norm{M_{\varphi \circ \sigma}^{\mathrm{HS}}}_{\cb,S^p_G \to S^p_G} 
 \leq \norm{M_{\varphi}^{\mathrm{HS}}}_{\cb,S^p_H \to S^p_H}\\ 
&= \norm{M_{\varphi}}_{\cb,\L^p(\VN(H)) \to \L^p(\VN(H))}.
\end{align*}
The isometric statement is proved in the same way. 

Finally, suppose that $M_\varphi$ is completely positive. By Proposition \ref{Th-cp-Fourier-multipliers}, we deduce that its symbol $\varphi$ is a continuous positive definite function. Since $\sigma$ is continuous, the function $\varphi \circ \sigma$ is also continuous. Moreover, if $\alpha_1,\ldots,\alpha_n \in \C$ and $s_1,\ldots,s_n \in G$, we infer that
\begin{align*}
\sum_{k,l=1}^n \alpha_k \ovl{\alpha_l} \varphi \circ \sigma (s_k s_l^{-1}) 
=\sum_{k,l=1}^n \alpha_k \ovl{\alpha_l} \varphi \big(\sigma(s_k) \sigma(s_l)^{-1}\big) 
 \geq 0.
\end{align*}
We conclude that $\varphi \circ \sigma$ is positive definite. We conclude by using again Proposition \ref{Th-cp-Fourier-multipliers}.
\end{proof}

\subsection{Extension of Fourier multipliers}
\label{sec-Extension-of-multipliers}

The following is an extension of \cite[Lemma 2.1 (2)]{Haa3} and a variant of \cite[Theorem B.1]{CPPR}. In \cite[Theorem B.1]{CPPR}, we warn the reader that a factor $\mu_G(\X)^{-1}$ is missing. Contrary to what is said, the alluded method does not give constant 1.

\begin{thm}
\label{thm-Jodeit-version-cb-Theorem-B.1}
Let $\Gamma$ be a lattice of a second countable unimodular locally compact group $G$ and $\X$ be a fundamental domain associated with $\Gamma$. We denote by $\gamma \co G \to \Gamma$ and $x \co G \to \X$ the measurable mappings uniquely determined by the decomposition $s = \omega(s) \gamma(s)$ for any $s \in G$. Suppose $1 \leq p \leq \infty$. We assume that $G$ is amenable if $1<p<\infty$. Let $\phi \co \Gamma \to \C$ be a complex function which induces a completely bounded Fourier multiplier $M_\phi \co \L^p(\VN(\Gamma)) \to \L^p(\VN(\Gamma))$. Then the complex function $\widetilde{\phi} \ov{\mathrm{def}}{=} \frac{1}{\mu_G(\X)} 1_{\X} \ast (\phi \mu_\Gamma) \ast 1_{\X^{-1}} \co G \to \C$ where $\mu_\Gamma$ is the counting measure on $\Gamma$ defined by
\begin{equation}
\label{symbole-tilde-varphi}
\widetilde{\phi}(s) 
\ov{\mathrm{def}}{=} \frac{1}{\mu_G(\X)} \int_\X \phi(\gamma(s\omega)) \d\mu_G(\omega),\quad s \in G
\end{equation}
is continuous and induces a completely bounded Fourier multiplier $M_{\widetilde{\phi}} \co \L^p(\VN(G)) \to \L^p(\VN(G))$ and we have
\begin{equation}
\label{Estimation-Lp-Haagerup}
\bnorm{M_{\widetilde{\phi}}}_{\cb,\L^p(\VN(G)) \to \L^p(\VN(G))} 
\leq \bnorm{M_\phi}_{\cb,\L^p(\VN(\Gamma)) \to \L^p(\VN(\Gamma))}.	
\end{equation}
Finally, if $M_\phi$ is completely positive then $M_{\widetilde{\phi}}$ is also completely positive.
\end{thm}

\begin{proof}
The case $p=\infty$ is \cite[Lemma 2.1 (2)]{Haa3} and the case $p=1$ follows by duality. The continuity of $\widetilde{\phi}$ is alluded\footnote{\thefootnote. We have
$$
\widetilde{\phi}(s) 
= \frac{1}{\mu_G(\X)} \int_{\X} \phi(\gamma(s\omega)) \d\mu_G(\omega) 
= \frac{1}{\mu_G(\X)}\int_G \phi(\gamma(t)) 1_{\X}(s^{-1}t) \d\mu_G(t).
$$
Then for any $s_1,s_2 \in G$, we have
\begin{align*}
\MoveEqLeft
 \left|\widetilde{\phi}(s_1) - \widetilde{\phi}(s_2)\right| 
 \leq \frac{1}{\mu_G(\X)} \int_G |\phi(\gamma(t))| \, |1_{\X}(s_1^{-1}t) - 1_{\X}(s_2^{-1}t)| \d\mu_G(t) 
\leq \frac{\norm{\phi}_{\L^\infty(G)}}{\mu_G(\X)} \int_G |1_{s_1\X}(t) - 1_{s_2\X}(t) | \d\mu_G(t) \\
& = \norm{\phi}_{\L^\infty(G)}\, \frac{\mu_G\big(s_1 \X \Delta s_2 \X\big)}{\mu_G(\X)} 
=  \norm{\phi}_{\L^\infty(G)}\,\frac{\mu_G\big((s_2^{-1} s_1 \X) \Delta \X\big)}{\mu_G(\X)} \xra[s_2 \to s_1]{} 0
\end{align*}
where the last line follows from \cite[Theorem A page 266]{Hal}.
} in \cite{Haa3} and in the proof of \cite[Lemma 2.1]{Haa3}, the formula \eqref{symbole-tilde-varphi} is shown.

Now, we consider the remaining case $1<p<\infty$. Since $G$ and $\Gamma$ are both amenable, we obtain using \cite[Theorem 4.2, Corollary 5.3]{CDS}\footnote{\thefootnote. We warn the reader that the proof of \cite[Theorem 5.2]{CDS} is only valid for second countable groups. The proof uses Lebesgue's dominated convergence theorem in the last line of page 7007 and this result does not admit a generalization for nets. See \cite{KaL} for more information.} in the first and in the last equality together with \cite[Remark 9.3]{CPPR}, and \cite[Lemma 2.6]{LaS} in the inequality
\begin{align*}
\MoveEqLeft
\bnorm{M_{\widetilde{\phi}}}_{\cb,\L^p(\VN(G)) \to \L^p(\VN(G))} 
= \bnorm{M_{\widetilde{\phi}}^{\HS}}_{\cb,S^p_G \to S^p_G} 
 \leq \bnorm{M_{\phi}^{\HS}}_{\cb, S^p_\Gamma \to S^p_\Gamma} \\ 
&= \bnorm{M_{\phi}}_{\cb, \L^p(\VN(\Gamma)) \to \L^p(\VN(\Gamma))}.
\end{align*}
Suppose that $M_\phi$ is completely positive. According to the proof of \cite[Lemma 2.1]{Haa3}, for any $s,t \in G$, we have
\begin{align}
\label{Formule-Haagerup}
\MoveEqLeft
\widetilde{\phi}(s t^{-1}) 
 = \frac{1}{\mu_G(\X)} \int_\X \phi\big(\gamma(s \omega') \gamma(t \omega')^{-1}\big) \d\mu_G(\omega').
\end{align}
We will show that $\widetilde{\phi}$ is positive definite. Let $\alpha_1,\ldots,\alpha_n \in \C$ and $s_1,\ldots,s_n \in G$. Since $\phi$ is positive definite by Proposition \ref{Th-cp-Fourier-multipliers}, we obtain
\begin{align*}
\sum_{k,l=1}^n \alpha_k \ovl{\alpha_l} \widetilde{\phi}(s_k s_l^{-1}) 
&=\frac{1}{\mu_G(\X)} \sum_{k,l=1}^{n} \alpha_k \ovl{\alpha_l} \int_{\X} \phi\big(\gamma(s_k \omega') \gamma(s_l \omega')^{-1}\big) \d\mu_G(\omega') \\
&=\frac{1}{\mu_G(\X)} \int_{\X} \sum_{k,l=1}^{n} \alpha_k \ovl{\alpha_l} \phi\big(\gamma(s_k \omega') \gamma(s_l \omega')^{-1}\big) \d\mu_G(\omega') 
\geq 0.
\end{align*}
Since the function $\widetilde{\phi}$ is continuous, we conclude that $M_{\widetilde{\phi}}$ is completely positive by using again Proposition \ref{Th-cp-Fourier-multipliers}.
\end{proof}




\subsection{Groups approximable by lattice subgroups}
\label{sec-Groups-approximable-by-lattice-subgroups}

If $(Y,\dist_Y)$ and $(Z,\dist_Z)$ are metric spaces and if $f \co Y \to Z$ is uniformly continuous, we denote by $\omega(f,\cdot) \co [0,+\infty[ \to [0,+\infty[$ a modulus of continuity of $f$. We have $\lim_{\delta \to 0} \omega(f,\delta)=0$ and $\omega(f,0)=0$. The function $\omega(f,\cdot)$ is increasing and for any $s,t \in Y$ we have
\begin{equation}
\label{inequality-modulus}
\dist_Z\big(f(s),f(t)\big) 
\leq \omega\big(f,\dist_Y(s,t)\big).	
\end{equation}
Let $G$ be a topological group. We denote by $\nu \co G \to G$, $s \mapsto s^{-1}$ the inversion map.

The following theorem gives a variant of Theorem \ref{Prop-complementation-Schur-Fourier} for a particular class of unimodular groups. 

\begin{thm}
\label{thm-complementation-Fourier-ADS-amenable} 
Let $G$ be a second countable unimodular locally compact group which satisfies $\mathrm{ALSS}$ with respect to a sequence of lattices $(\Gamma_j)_{j \geq 1}$ and associated fundamental domains $(\X_j)_{j \geq 1}$. Suppose $1 \leq p \leq \infty$. We assume that $G$ is amenable if $1<p<\infty$. Suppose that for some constant $c > 0$ and any compact subset $K$ of $G$ we have
\begin{equation}
\label{equ-Haar-measure-convergence}
\lim_{j \to \infty} \sup_{\gamma \in \Gamma_j \cap K}\left|\frac{1}{\mu(\X_j)} \int_G \frac{\mu^2(\X_j \cap \gamma \X_j s)}{\mu^2(\X_j)} \d\mu(s) - c\right|
=0
\end{equation}
where $\mu=\mu_G$ is a Haar measure of $G$. Then for $1 \leq p \leq \infty$, there exists a linear mapping
$$
P_G^p \co \CB(\L^p(\VN(G))) \to \mathfrak{M}^{p,\cb}(G)
$$
of norm at most $\frac{1}{c}$ with the properties:
\begin{enumerate}
\item If $T \co \L^p(\VN(G)) \to \L^p(\VN(G))$ is completely positive, then $P_G^p(T)$ is completely positive.
\item If $T = M_\psi$ is a Fourier multiplier on $\L^p(\VN(G))$ with bounded continuous symbol $\psi \co G \to \C$, then $P_G^p(M_\psi) = M_\psi$.
Moreover, if we have $\gamma \X_j = \X_j \gamma$ for any $j \in \N$ and any $\gamma \in \Gamma_j$, or alternatively, if $\X_j$ is symmetric in the sense that $\mu(\X_j \Delta \X_j^{-1}) = 0$ for any $j \in \N$, then $P_G^p(M_\psi) = M_\psi$ for any bounded measurable symbol such that $M_\psi \in \mathfrak{M}^{p,\cb}(G)$.
\end{enumerate}
For an element $T$ belonging to $\CB(\L^p(\VN(G)))$ and to $\CB(\L^q(\VN(G)))$ for two values $p,q \in [1,\infty]$, we have $P_G^p(T)x = P_G^q(T)x$ for $x \in \L^p(\VN(G)) \cap \L^q(\VN(G))$. 

In the preceding lines, if $p = \infty$, we can take $\CB_{\w^*}(\VN(G))$ as the domain space of $P_G^\infty$.
\end{thm}

\begin{proof}
If $G$ is amenable, note that each $\Gamma_j$ is amenable by \cite[Proposition G.2.2]{BHV}. So each $\VN(\Gamma_j)$ is hyperfinite, hence $\QWEP$. 

For any $j$, we consider the element $h_j \ov{\mathrm{def}}{=} \lambda(1_{\X_j}) = \int_{\X_j} \lambda_s \d\mu(s)$ of the group von Neumann algebra $\VN(G)$ and define for $1 \leq p \leq \infty$ the (normal\footnote{\thefootnote. Recall that the product of a von Neumann algebra is separately weak* continuous, e.g. see \cite[Proposition 2.7.4 (1)]{BLM}.}if $p=\infty$) completely positive map
$$ 
\Phi_j^p \co \L^p(\VN(\Gamma_j)) \to \L^p(\VN(G)),\: \lambda_\gamma \mapsto \mu(\X_j)^{-2 + \frac1p} h_j^* \lambda_\gamma h_j.
$$
It is noted and shown in \cite[page~19]{CPPR} that each $\Phi_j^p$ is completely contractive. For any $1 \leq p \leq \infty$, we also consider the adjoint (preadjoint if $p=1$) $\Psi_j^p=\big(\Phi_j^{p^*}\big)^* \co \L^p(\VN(G)) \to \L^p(\VN(\Gamma_j))$ of $\Phi_j^{p^*}$ which is also completely contractive and completely positive by Lemma \ref{Lemma-adjoint-cp}. 
Now, use Theorem \ref{Prop-complementation-Schur-Fourier} for the discrete group $\Gamma_j$ and define for some completely bounded map $T \co \L^p(\VN(G)) \to \L^p(\VN(G))$, the Fourier multiplier $M_{\phi_j} \co \L^p(\VN(\Gamma_j))\to \L^p(\VN(\Gamma_j))$ defined by
$$ 
M_{\phi_j} 
\ov{\mathrm{def}}{=} \frac1c P_{\Gamma_j}^p\big(\Psi_j^p T \Phi_j^p\big) 
\quad \text{if $1 \leq p<\infty$ and} \quad
M_{\phi_j} 
\ov{\mathrm{def}}{=} \frac1c P_{\Gamma_j}^\infty\big(\Psi_j^\infty P_{\w^*}(T) \Phi_j^\infty\big) \quad \text{if $p = \infty$}, 
$$
where the contractive map $P_{\w^*} \co \CB(\VN(G)) \to \CB(\VN(G))$ is described in Proposition \ref{Prop-recover-weak-star-continuity}, whose symbol is (if $T$ is normal in the case $p=\infty)$ 
\begin{align}
\label{Symbol-phi_j-first}
\phi_j(\gamma) 
& =\frac1c \tau_{\Gamma_j}\big(\Psi_j^p T \Phi_j^p (\lambda_\gamma) \lambda_{\gamma^{-1}}\big)
=\frac1c \tau_{G}\Big(T\Phi_j^p (\lambda_\gamma) \Phi_j^{p^*}(\lambda_{\gamma^{-1}})\Big) \\
& =\frac{1}{c\, \mu(\X_j)^{3}} \tau_G\big(T(h_j^* \lambda_\gamma h_j) h_j^* \lambda_{\gamma^{-1}} h_j\big). \nonumber
\end{align}
Then we have for $1 \leq p<\infty$
\begin{align*}
\MoveEqLeft
\norm{M_{\phi_j}}_{\cb,\L^p(\VN(\Gamma_j)) \to \L^p(\VN(\Gamma_j))} 
=\norm{\frac1c P_{\Gamma_j}^p(\Psi_j^p T \Phi_j^p)}_{\cb,\L^p(\VN(\Gamma_j)) \to \L^p(\VN(\Gamma_j))} \\
&\leq \frac{1}{c} \norm{\Psi_j^p T \Phi_j^p}_{\cb, \L^p(\VN(\Gamma_j)) \to \L^p(\VN(\Gamma_j))} 
\leq \frac{1}{c} \norm{T}_{\cb,\L^p(\VN(G)) \to \L^p(\VN(G))}
\end{align*}
and similarly for $p=\infty$. Let further
\begin{equation}
\label{equ1-Fourier-complementation}
\widetilde{\phi_j} 
\ov{\mathrm{def}}{=} \frac{1}{\mu(\X_j)} 1_{\X_j} \ast (\phi_j \mu_{\Gamma_j}) \ast 1_{\X_j^{-1}} \co G \to \C
\end{equation}
where $\mu_{\Gamma_j}$ is the counting measure on the discrete subset $\Gamma_j$ of $G$. According to Theorem \ref{thm-Jodeit-version-cb-Theorem-B.1}, $M_{\widetilde{\phi_j}} \co \L^p(\VN(G)) \to \L^p(\VN(G))$ is a completely bounded Fourier multiplier with
\begin{equation}
	\label{inegalite-1-sur-c}
\bnorm{M_{\widetilde{\phi_j}}}_{\cb,\L^p(\VN(G)) \to \L^p(\VN(G))} 
\leq \norm{M_{\phi_j}}_{\cb,\L^p(\VN(\Gamma_j)) \to \L^p(\VN(\Gamma_j))}
\leq \frac1c \norm{T}_{\cb,\L^p(\VN(G)) \to \L^p(\VN(G))}.	
\end{equation} 
If $1 < p \leq \infty$, note that $\B(\CB(\L^p(\VN(G))))$ is a dual Banach space and admits the predual
\begin{equation}
\label{equ-predual-bracket}
\CB(\L^p(\VN(G))) \hat \ot \big(\L^p(\VN(G)) \widehat{\ot} \L^{p^*}(\VN(G))^\op \big),
\end{equation}
where $\hat{\ot}$ denotes the Banach space projective tensor product and where $\widehat{\ot}$ denotes the operator space projective tensor product. The duality bracket is given by
\begin{equation}
\label{Derniere-eq}
\big\langle P , T \ot (x \ot y) \big\rangle
=\big\langle P(T) x, y \big\rangle_{\L^p(\VN(G)),\L^{p^*}(\VN(G))}.
\end{equation}

The mappings $P_j^p \co T \mapsto M_{\widetilde{\phi_j}}$ are linear and uniformly bounded in $\B(\CB(\L^p(\VN(G))))$ (we use $\B(\CB(\VN(G)), \CB(\mathrm{C}^*_\lambda(G),\VN(G)))$ if $p=\infty$). 
From now on, we restrict to the case $1 < p \leq \infty$ and we will return to the case $p = 1$ only at the end of the proof. The elements $P_j^p$ belong to the space $Y_p \ov{\mathrm{def}}{=} \frac1c \text{Ball}(\B(\CB(\L^p(\VN(G)))))$ for $p \in (1,\infty]$. By Banach-Alaoglu's theorem, note that each $Y_p$ is compact with respect to the weak* topology of the underlying Banach space. Then by Tychonoff's theorem, $\prod_{p \in (1,\infty]} Y_p$ is also compact. Thus, the net $\big((P_j^p)_{p \in (1,\infty]}\big)$ admits a convergent subnet $\big((P_{j(k)}^p)_{p \in (1,\infty]}\big)$, which converges to some element $((P_G^p)_{p \in (1,\infty)},P_G^\infty)$ of $\prod_{p \in (1,\infty]} Y_p$, i.e. for any $p$ the net $(P_{j(k)}^p)$ converges to $P_G^p$ for the weak* topology. With \eqref{Derniere-eq}, we see that this implies that $(P_{j(k)}^{p}(T))$ converges for the weak operator topology (in the point weak* topology if $p=\infty$) to $P_G^p(T)$. Observe that the weak* topology on $\CB(\L^p(\VN(G)))$ coincides on bounded subsets with the weak operator topology (the point weak* topology if $p=\infty$) essentially by the same argument as the one of the proof of \cite[Lemma 7.2]{Pau} (which uses \cite[Proposition 1.21]{Dou98}). We conclude by Lemma \ref{lem-Fourier-multiplier-weak-star-closed} that $P_G^p(T)$ is itself a Fourier multiplier. 


Note that we clearly have 
$$
\bnorm{P_G^p}_{\CB(\L^p(\VN(G))) \to \CB(\L^p(\VN(G)))} 
\leq \liminf_{k \to +\infty} \bnorm{P_{j(k)}^p}_{\CB(\L^p(\VN(G))) \to \CB(\L^p(\VN(G)))} 
\leq \frac1c.
$$ 

We next show that $P_G^p$ preserves the complete positivity. Suppose that $T$ is (normal if $p=\infty)$ completely positive. Since $\Phi_j^p$ and $\Psi_j^p$ are completely positive, $\Psi_j^p T \Phi_j^p$ is also completely positive and thus, by Theorem \ref{Prop-complementation-Schur-Fourier}, $M_{\phi_j} = \frac1c P_{\Gamma_j}^p(\Psi_j^p T \Phi_j^p)$ is completely positive. Using Theorem \ref{thm-Jodeit-version-cb-Theorem-B.1}, we conclude that $M_{\widetilde{\phi_j}}$ is completely positive. Since $P_G^p(T)$ is the weak operator topology limit of $M_{\widetilde{\phi_j}}$ (point weak* topology limit if $p=\infty$), the complete positivity of $M_{\widetilde{\phi_j}}$ carries over to that of $P_G^p(T)$ by Lemma \ref{lem-completely-positive-weak-limit}.

We claim that $P_G^p$ has the compatibility property stated in the theorem. Note that the symbol $\widetilde{\phi_j}$ of $P_j^p(T)$ does not depend on $p$ if $T$ belongs to two different spaces $\CB(\L^p(\VN(G)))$ and $\CB(\L^q(\VN(G)))$. If in addition $x$ belongs to both $\L^p(\VN(G))$ and $\L^q(\VN(G))$ and if $y$ belongs to both $\L^{p^*}(\VN(G))$ and $\L^{q^*}(\VN(G))$, then we have
\begin{align*}
\MoveEqLeft
\big\langle P_G^p(T)x, y \big\rangle_{} 
= \lim_k \big\langle P_{j(k)}^p(T)x, y \big\rangle_{} 
= \lim_k \big\langle P_{j(k)}^q(T)x, y \big\rangle_{} 
= \big\langle P_G^q(T)x, y \big\rangle_{}.  
\end{align*}
Then it is immediate that the $P_G^p$'s are compatible as stated in the theorem.


We finally will show now that $P_G^p(M_{\psi}) = M_{\psi}$ for any bounded continuous symbol $\psi \co G \to \C$ (or $\psi$ bounded measurable under the additional symmetry/commutativity assumption on $\X_j$) giving rise to a completely bounded $\L^p$-multiplier. We start by computing the symbol $\phi_j$. For any $\gamma \in \Gamma_j$, note that $\lambda_\gamma h_j=\lambda_\gamma \lambda(1_{X_j})=\lambda(1_{\gamma \X_j})$ and similarly $\lambda_{\gamma^{-1}} h_j=\lambda(1_{\gamma^{-1} \X_j})$. Consequently, we have
\begin{align*}
\MoveEqLeft
\phi_j(\gamma) 
=  \frac{1}{c\, \mu(\X_j)^{3}} \tau_G\big( M_\psi (h_j^* \lambda_\gamma h_j) h_j^* \lambda_{\gamma^{-1}} h_j\big) \\
& =  \frac{1}{c\, \mu(\X_j)^{3}} \tau_G \left( M_\psi \lambda\big(1_{\X_j^{-1}} \ast 1_{\gamma \X_j}\big) \lambda\big(1_{\X_j^{-1}} \ast 1_{\gamma^{-1} X_j}\big) \right) \\
& =  \frac{1}{c\, \mu(\X_j)^{3}} \tau_G\left( \lambda\big(\psi(1_{\X_j^{-1}} \ast 1_{\gamma \X_j})\big) \lambda\big( 1_{\X_j^{-1}} \ast 1_{\gamma^{-1} X_j}\big) \right) \\
&=\frac{1}{c\, \mu(\X_j)^{3}} \int_{G} \psi(s) \big(1_{\X_j^{-1}} \ast 1_{\gamma \X_j}\big)(s) \big(1_{\X_j^{-1}} \ast 1_{\gamma^{-1} \X_j}\big)(s^{-1})\d \mu(s)
\end{align*}
where the last equality follows from the Plancherel formula \eqref{Formule-Plancherel} and from the fact that the functions $\psi (1_{\X_j^{-1}} \ast 1_{\gamma \X_j})$ and $1_{\X_j^{-1}} \ast 1_{\gamma^{-1} \X_j}$ belong to the space $\L^1(G) \cap \L^2(G)$, and thus are left bounded. Now, using \cite[Theorem 20.10 (iv)]{HR}, note that for any $s \in G$
$$
\big(1_{\X_j^{-1}} \ast 1_{\gamma \X_j}\big)(s) 
=\int_{G} 1_{\X_j^{-1}}(t^{-1}) 1_{\gamma \X_j}(ts) \d\mu(t) 
=\int_{X_j}  1_{\gamma \X_j}(ts) \d\mu(t) 
=\mu(\X_j \cap \gamma \X_j s^{-1})	
$$
and 
$$
\big(1_{\X_j^{-1}} \ast 1_{\gamma^{-1} \X_j}\big)(s^{-1}) 
= \mu(\X_j \cap \gamma^{-1} \X_j s) 
= \mu(\gamma \X_j s^{-1} \cap \X_j).
$$
Thus, for any $\gamma \in \Gamma_j$, we conclude that
\begin{equation}
\label{Symbol-Phij}
\phi_j(\gamma) 
=\frac{1}{c\, \mu(\X_j)^{3}} \int_G \psi(s) \mu\big(\X_j \cap \gamma \X_j s^{-1}\big)^2 \d\mu(s).	
\end{equation}
Now, we examine the asymptotic behaviour of the sequence of symbols $\phi_j$. Since $G$ is second countable, it admits a right-invariant metric $\dist(\cdot,\cdot)$, i.e. $\dist(s,t) = \dist(sr,tr)$ for $r,s,t \in G$, such that the closed balls are compact \cite{HaP}. We denote by $B(x,r)$ the open ball centered on $x$ with radius $r$ and $B'(x,r)$ the closed ball. We need the following lemmas. 



\begin{lemma}
\label{Lemma-topo}
For any neighborhood $V$ of the identity $e$ in $G$, any compact subset $K$ of $G$, any $j$ sufficiently large and any $\gamma \in K$, we have 
\begin{equation}
\label{equ3-Fourier-complementation}
\X_j \cap \gamma \X_j s^{-1} 
= \emptyset, \quad s \in G \backslash \gamma V.
\end{equation}
\end{lemma}

\begin{proof}
Since $K$ is compact, we have $K \subset B(e,R_K)$ for some $R_K > 0$. Let $j$ be so large that $\X_j \subset B(e,\frac13)$. If $s \in G \backslash B(e,R_K + 1)$, then we have for $\omega \in \X_j$ and $\gamma \in K$ 
\begin{align*}
\MoveEqLeft
  \dist(e,\gamma \omega s^{-1}) 
	\geq \dist(e,s^{-1}) - \dist(s^{-1},\gamma \omega s^{-1}) 
	= \dist(s,e) - \dist(e,\gamma \omega) \\
	& \geq \dist(s,e) - \dist(e,\omega) - \dist(\omega,\gamma \omega) \geq \dist(s,e) - \dist(e,\omega) - \dist(e,\gamma) \\
	& \geq R_K + 1 - \frac13 - R_K \geq  \frac{2}{3}.  
\end{align*}
Thus, for such an $s$, we have $\X_j \cap \gamma \X_j s^{-1} = \emptyset$, since $\X_j \subset B(e,\frac13)$. So from now on, we can assume $s \in B(e,R_K+1)$, in other words, varying in a compact set. 

Let $\epsi>0$ such that $B(e,\epsi) \subset V$. By \cite[Theorem 4.9]{HR}, there exists $\epsi' > 0$ such that $\gamma B(e,\epsi) \gamma^{-1}$ contains the ball $B(e,\epsi')$ for any $\gamma \in K$. Let $\gamma \in K$ and $s \in B(e,R_K+1) \backslash \gamma V$. Since $s \not\in \gamma B(e,\epsi)$, we have $\gamma s^{-1} \not\in \gamma B(e,\epsi)^{-1} \gamma^{-1}$ and finally $\dist(\gamma,s)=\dist(e, \gamma s^{-1}) \geq \epsi'$. Consider the compact $K'=B'(e,1) \cdot B'(e,R_K+1)^{-1} $ and some $0<\epsi'' \leq \min\{\frac{1}{2}\epsi',1\}$ such that $\omega\big(\nu|K',\epsi''\big) \leq \frac{1}{2}\epsi'$. Consider $j$ so large that $\X_j \subset B(e,\epsi'')$. Let $\omega \in \X_j$. Then
$$
\dist(e,\gamma \omega s^{-1}) = \dist(e,s \omega^{-1} \gamma^{-1}) = \dist(\gamma,s \omega^{-1}) \geq \dist(\gamma,s) - \dist(s,s\omega^{-1}).
$$
Note that $s^{-1}$ and $\omega s^{-1}$ vary in the compact $K'$ for $\omega$ varying in $X_j$. Now, using \eqref {inequality-modulus}, we have
\begin{align*}
\MoveEqLeft
\dist(s,s \omega^{-1}) 
\leq \omega\big(\nu|K',\dist(s^{-1},\omega s^{-1})\big) 
=\omega\big(\nu|K',\dist(e,\omega) \big)
\leq \omega\big(\nu|K',\epsi''\big)
\leq \frac{1}{2}\epsi'.
\end{align*}
We deduce that
  $\dist(e,\gamma\omega s^{-1})  
		\geq \epsi' -\frac{1}{2}\epsi' 
		=\frac{1}{2}\epsi' 
		\geq \epsi''
$, so that $\gamma \omega s^{-1} \not\in B(e,\epsi'')$ and thus $\X_j \cap \gamma \X_j s^{-1} = \emptyset$ since $\X_j \subset B(e,\epsi'')$. We have shown \eqref{equ3-Fourier-complementation}.
\end{proof}

\begin{lemma}
Assume in addition that $\psi$ is a continuous symbol. Then for any compact subset $K$ of $G$, we have
\begin{equation}
\label{equ2-Fourier-complementation}
\sup_{\gamma \in \Gamma_j \cap K} |\phi_j(\gamma) - \psi(\gamma)| 
\xra[j \to +\infty]{} 0.
\end{equation}
\end{lemma}

\begin{proof}
We fix a compact subset $K$ of $G$ and a compact neighborhood $V$ of $e$. Then, for any $j$ sufficiently large and any $\gamma \in K$, Lemma \ref{Lemma-topo} implies the existence of the integral $\int_G \psi(\gamma)\mu\big(\X_j \cap \gamma \X_j s^{-1}\big)^2 \d\mu(s)$. By definition of $c$, for any $\gamma \in \Gamma_j \cap K$, using \eqref{Symbol-Phij} in the first equality, we have
\begingroup
\allowdisplaybreaks
\begin{align*}
\MoveEqLeft
\left|\phi_j(\gamma) -\psi(\gamma) \right| 
=\left|\frac{1}{c\, \mu(\X_j)^{3}} \int_G \psi(s) \mu\big(\X_j \cap \gamma \X_j s^{-1}\big)^2 \d\mu(s) -\psi(\gamma) \right|\\
&=\frac{1}{c\, \mu(\X_j)^{3}} \left| \int_G \psi(s) \mu\big(\X_j \cap \gamma \X_j s^{-1}\big)^2 \d\mu(s) -c \mu(\X_j)^{3} \psi(\gamma) \right|\\
&=\frac{1}{c\, \mu(\X_j)^{3}} \left| \int_G \psi(s) \mu\big(\X_j \cap \gamma \X_j s^{-1}\big)^2 \d\mu(s) -\int_G \psi(\gamma)\mu\big(\X_j \cap \gamma \X_j s^{-1}\big)^2\d\mu(s)\right.\\
&\quad \left.+\int_G \psi(\gamma)\mu\big(\X_j \cap \gamma \X_j s^{-1}\big)^2\d\mu(s) - c \mu(\X_j)^{3} \psi(\gamma) \right|\\
&\leq \frac{1}{c\, \mu(\X_j)^{3}} \int_G \left|\psi(s)-\psi(\gamma) \right|\mu\big(\X_j \cap \gamma \X_j s^{-1}\big)^2 \d\mu(s)\\
&\quad +\frac{1}{c\, \mu(\X_j)^{3}}|\psi(\gamma)|\left|\int_G \mu\big(\X_j \cap \gamma \X_j s^{-1}\big)^2\d\mu(s)-c \mu(\X_j)^{3} \right|\\
&\leq \frac{1}{c\, \mu(\X_j)^{3}} \int_G |\psi(s) - \psi(\gamma)|\, \mu\big(\X_j \cap \gamma \X_j s^{-1}\big)^2 \d\mu(s) \\  
&\quad + \frac1c |\psi(\gamma)| \left|\frac{1}{\mu(\X_j)^{}} \int_G \frac{\mu\big(\X_j \cap \gamma \X_j s^{-1}\big)^2}{\mu(\X_j)^{2}} \d\mu(s)-c \right|.
\end{align*}
\endgroup
The last summand converges to $0$ as $j \to \infty$ uniformly in $\gamma \in \Gamma_j \cap K$ according to the assumption \eqref{equ-Haar-measure-convergence} and the boundedness of $\psi$. It remains to treat the first summand. Then, for  and $j$ sufficiently large and $\gamma \in \Gamma_j \cap K$, using Lemma \ref{Lemma-topo} in the first equality, we obtain 
\begin{align*}
\MoveEqLeft
\sup_{\gamma \in \Gamma_j \cap K} \frac{1}{c\, \mu(\X_j)^{3}} \int_G |\psi(s) - \psi(\gamma)|\, \mu\big(\X_j \cap \gamma \X_j s^{-1}\big)^2 \d\mu(s) \\
& = \frac{1}{c\, \mu(\X_j)^{3}} \sup_{\gamma \in \Gamma_j \cap K} \int_{\gamma V} |\psi(s) - \psi(\gamma)|\, \mu\big(\X_j \cap \gamma \X_j s^{-1}\big)^2 \d\mu(s) \\
&\leq \frac{1}{c\, \mu(\X_j)^{3}}\bigg(\sup_{\gamma \in \Gamma_j \cap K}\int_{\gamma V} \mu\big(\X_j \cap \gamma \X_j s^{-1}\big)^2 \d\mu(s) \bigg)
\bigg(\sup_{s \in \gamma V,\:\gamma \in \Gamma_j \cap K} |\psi(s) - \psi(\gamma)| \bigg)\\
&=\bigg(\sup_{\gamma \in \Gamma_j \cap K}\frac{1}{c\, \mu(\X_j)}\int_{G} \frac{\mu\big(\X_j \cap \gamma \X_j s^{-1}\big)^2}{\mu(\X_j)^2} \d\mu(s) \bigg)
\bigg(\sup_{s \in \gamma V,\:\gamma \in \Gamma_j \cap K} |\psi(s) - \psi(\gamma)| \bigg).
\end{align*}
We will show that for $V=B'(e,\epsi')$ the last supremum converges to $0$ as $\epsi' \to 0$ uniformly in $j$. Since it is not difficult to see that the first factor is uniformly bounded for $j \geq 1$ and $\gamma \in \Gamma_j \cap K$ by the assumption \eqref{equ-Haar-measure-convergence} of the theorem, \eqref{equ2-Fourier-complementation} follows. Consider some $0<\epsi \leq 1$. Define the compact $K'=K\cdot B'(e,1)$. Let $0 < \epsi' \leq 1$ such that $\omega(\nu|K'^{-1},\epsi') \leq \epsi$. If $s,t \in K'$ and $\dist(s^{-1},t^{-1}) \leq \epsi'$, we have by \eqref{inequality-modulus}
$$
\dist(s,t) 
\leq \omega\big(\nu|K'^{-1},\dist(s^{-1},t^{-1})\big)
\leq \omega(\nu|K'^{-1},\epsi')
\leq \epsi.
$$
Note that the restriction $\psi|K'$ of the continuous function $\psi$ on $K'$ is uniformly continuous. 
For any $j$, using \eqref{inequality-modulus} in the first inequality, we deduce that
\begin{align*}
\MoveEqLeft
  \sup_{s \in \gamma B'(e,\epsi'),\:\gamma \in \Gamma_j \cap K} |\psi(s) - \psi(\gamma)|  
		\leq \sup_{s^{-1} \in B'(\gamma^{-1},\epsi'),\:\gamma \in \Gamma_j \cap K} \omega\big(\psi|K',\dist(s,\gamma)\big)\\
		&\leq \sup_{s \in \gamma V,\:\gamma \in \Gamma_j \cap K} \omega\big(\psi|K',\epsi\big)= \omega\big(\psi|K',\epsi\big)\xra[\epsi \to 0]{} 0.
\end{align*}
\end{proof}

We continue with the asymptotic behaviour of the symbols $\widetilde{\phi_j}$. 

\begin{lemma}
Assume in addition that $\psi$ is a continuous symbol.
Then for any $s \in G$, we have
\begin{equation}
\label{equ4-Fourier-complementation}
\widetilde{\phi_j}(s) \xra[j \to +\infty]{} \psi(s).
\end{equation}
\end{lemma}

\begin{proof}
Let $s \in G$. Recall that we have a unique decomposition $s = \omega_j(s) \gamma_j(s)$ with $\omega_j(s) \in \X_j$ and $\gamma_j(s) \in \Gamma_j$. Then, by \eqref{symbole-tilde-varphi}, we have
\begin{align*}
\left| \widetilde{\phi_j}(s) - \psi(s)\right| 
&=\left|\frac{1}{\mu(\X_j)} \int_{\X_j} \phi_j(\gamma_j(st))  \d\mu(t)- \psi(s)\right| = \frac{1}{\mu(\X_j)} \left| \int_{\X_j} \big(\phi_j(\gamma_j(st)) - \psi(s)\big) \d\mu(t) \right| \\
& \leq  \frac{1}{\mu(\X_j)}\int_{\X_j} \big(\left| \phi_j(\gamma_j(st)) - \psi(\gamma_j(st)) \right| + \left| \psi(\gamma_j(st)) - \psi(s) \right|\big) \d\mu(t)\\
&\leq \frac{1}{\mu(\X_j)} \int_{\X_j} \left| \phi_j(\gamma_j(st)) - \psi(\gamma_j(st)) \right|\d\mu(t) + \frac{1}{\mu(\X_j)}\int_{\X_j} \left| \psi(\gamma_j(st)) - \psi(s) \right| \d\mu(t).
\end{align*}
We start to prove that the first summand converges to $0$ as $j \to \infty$. Indeed, according to \eqref{equ2-Fourier-complementation}, it suffices to show that $\gamma_j(st)$ remains in a fixed compact set independent of $j$, for $t$ varying in $\X_j$. We will even show that $\dist(\gamma_j(st),s) \to 0$ as $j \to \infty$ uniformly in $t \in \X_j$. 

Let $\epsi >0$. Consider the compact $K_s=(s\cdot B'(e,1))^{-1}$. There exists $0< \epsi' \leq \min\{1,\epsi\}$ such that $\omega\big(\nu|K_s,\epsi'\big) \leq \epsi$. Then for some $j_0 \in \N$, we have $\X_{j} \subset B(e,\epsi')$ for all $j \geq j_0$. Note that $s^{-1}$ and $(st)^{-1}$ and vary in the compact set $K_s$ for $j\geq j_0$ and $t$ varying in $\X_j$. For these $j$ and any $t \in \X_j$, using \eqref{inequality-modulus}, we see that
\begin{align*}
\MoveEqLeft
\dist(\gamma_j(st),s)  
\leq \dist(\gamma_j(st),st) + \dist(st,s) 
= \dist(\omega_j(st)^{-1} st, st) + \dist(st,s) \\
&=\dist(\omega_j(st)^{-1},e)+\dist(st,s)
\leq \dist(e,\omega_j(st)) + \omega\big(\nu|K_s,\dist\big((st)^{-1}, s^{-1}\big)\big) \\
&\leq \epsi' + \omega\big(\nu|K_s,\dist\big(t^{-1},e\big)\big)
\leq \epsi + \omega\big(\nu|K_s,\epsi'\big)
\leq \epsi + \epsi.
\end{align*}
We conclude that $\sup_{t \in \X_j}\dist(\gamma_j(st),s) \to 0$ as $j \to \infty$. 

For the second summand, consider $\epsi>0$. Note that the restriction $\psi|B'(s,1)$ is uniformly continuous. There exists $0 < \epsi'\leq 1$ such that $\omega\big(\psi|B'(s,1),\epsi'\big) \leq \epsi$ and there exists $j_0$ such that $\sup_{t \in \X_j} \dist(\gamma_j(st),s) \leq \epsi'$ for any $j \geq j_0$. For these $j$, using \eqref{inequality-modulus}, we deduce that 
\begin{align*}
\MoveEqLeft
   \sup_{t \in \X_j} |\psi(\gamma_j(st))-\psi(s)| 
		\leq \sup_{t \in \X_j} \omega\big(\psi|B'(s,1),\dist(\gamma_j(st),s)\big) \\
		&\leq \sup_{t \in \X_j} \omega\big(\psi|B'(s,1),\epsi'\big)= \omega\big(\psi|B'(s,1),\epsi'\big)
		\leq \epsi.
\end{align*}
That means that $\sup_{t \in \X_j} |\psi(\gamma_j(st)) - \psi(s)| \to 0$ as $j \to \infty$. Thus \eqref{equ4-Fourier-complementation} follows.
\end{proof}

If $f \in \L^\infty(G)$, the particular case $p=2$ of \eqref{inegalite-1-sur-c}  applied to $M_\psi$ instead of $T$ together with Lemma \ref{prop-M2-Fourier-multipliers} allows us to define a well-defined operator $\Xi_j \co \L^\infty(G) \to \L^\infty(G)$, $\psi \mapsto \widetilde{\phi_j}$ for any $j$ with
\begin{equation}
\label{equ-2-proof-lem-ALSS-weak-star}
\bnorm{\Xi_j(\psi)}_{\L^\infty(G)} 
\leq \frac{1}{c} \norm{\psi}_{\L^\infty(G)}.
\end{equation}

\begin{lemma}
\label{lem-ALSS-estimation-L1}
Assume that $\gamma \X_j = \X_j \gamma$ for any $j \in \N$ and any $\gamma \in \Gamma_j$ or that $\mu(\X_j \Delta \X_j^{-1}) = 0$ for any $j \in \N$. 
\begin{enumerate}
	\item If $\psi \in \L^1(G)$ then the formula \eqref{Symbol-Phij} gives a well-defined function $\phi_j \co \Gamma_j \to \C$ for any $j$.
	\item For any $j$, we have a well-defined bounded operator $\Xi_j \co \L^1(G) \to \L^1(G)$, $\psi \mapsto \widetilde{\phi_j}$ where $\widetilde{\phi_j}$ is defined by the formula 
\begin{equation}
\label{eq-def-phij}
\widetilde{\phi_j} 
=\frac{1}{\mu(\X_j)} 1_{\X_j} \ast (\phi_j \mu_{\Gamma_j}) \ast 1_{\X_j^{-1}}.	
\end{equation}
Moreover, for any $\psi \in \L^1(G)$ and any $j$, we have
\begin{equation}
\label{equ-1-proof-lem-ALSS-weak-star}
\bnorm{\Xi_j(\psi)}_{\L^1(G)} 
\leq \frac{1}{c} \norm{\psi}_{\L^1(G)}.
\end{equation}
\end{enumerate}
\end{lemma}

\begin{proof}
1. If $\gamma \X_j = \X_j \gamma$ for any $\gamma \in \Gamma_j$, then using \eqref{DGamma=G} in the second equality
\begin{align*}
\MoveEqLeft
\mu(\X_j \cap \gamma \X_j s^{-1})
=\mu(\X_j \cap \X_j \gamma s^{-1})
\leq\mu(\X_j \cap \X_j \Gamma_j s^{-1})
=\mu(\X_j \cap G s^{-1})
=\mu(\X_j).
\end{align*}
If $\mu(\X_j \Delta \X_j^{-1}) = 0$, then using unimodularity in the last equality, we see that
\begin{align}
\MoveEqLeft
\mu(\X_j \cap \gamma \X_j s^{-1})
\leq \mu(\X_j \cap \X_j^{-1} \cap \gamma \X_j s^{-1})+\mu\big((\X_j -\X_j^{-1}) \cap \gamma \X_j s^{-1}\big) \nonumber\\
&\leq \mu(\X_j \cap \X_j^{-1} \cap \gamma \X_j s^{-1})+\mu\big((\X_j \Delta \X_j^{-1}) \cap \gamma \X_j s^{-1}\big)\nonumber\\
&\leq \mu\big(\X_j \cap \X_j^{-1} \cap \gamma (\X_j \cap \X_j^{-1}) s^{-1}\big)
+\mu\big(\X_j \cap \X_j^{-1} \cap \gamma (\X_j-\X_j^{-1}) s^{-1}\big)+\mu\big(\X_j \Delta X_j^{-1})\nonumber\\
&\leq  \mu\big(\X_j \cap \X_j^{-1} \cap \gamma (\X_j \cap \X_j^{-1}) s^{-1}\big) + \overbrace{\mu\big(\gamma (\X_j \Delta \X_j^{-1}) s^{-1}\big)}^{=0}+ \overbrace{\mu(\X_j \Delta \X_j^{-1})}^{=0} \nonumber\\
&\leq  \mu(\X_j^{-1} \cap \gamma \X_j^{-1} s^{-1})
= \mu(\X_j \cap s \X_j \gamma^{-1}) \label{estimation-1}.
\end{align}
Using \eqref{DGamma=G}, we obtain
$$
\mu(\X_j \cap \gamma \X_j s^{-1}) 
\leq \mu(\X_j \cap s \X_j \Gamma_j)
=\mu(\X_j).
$$
So the integrand of \eqref{Symbol-Phij} is integrable in both cases since $\psi \in \L^1(G)$. We deduce that the function $\phi_j$ is well-defined. 

2. For any $j$, using \eqref{Symbol-Phij} in the first equality,  we have
\begin{align}
\sum_{\gamma \in \Gamma_j} |\phi_j(\gamma)| 
&=\sum_{\gamma \in \Gamma_j} \left|\frac{1}{c\, \mu(\X_j)^{3}} \int_G \psi(s) \mu\big(\X_j \cap \gamma \X_j s^{-1}\big)^2 \d\mu(s)\right|\nonumber \\
&\leq \frac{1}{c\, \mu(\X_j)^{3}}\int_G \sum_{\gamma \in \Gamma_j} |\psi(s)| \mu(\X_j \cap \gamma \X_j s^{-1})^2 \d \mu(s) \nonumber \\
&\leq \frac{1}{c \mu(\X_j)^3} \norm{\psi}_{\L^1(G)}  \sup_{s \in G} \sum_{\gamma \in \Gamma_j} \mu(\X_j \cap \gamma \X_j s^{-1})^2 \nonumber\\
&=\frac{1}{c\mu(\X_j)} \norm{\psi}_{\L^1(G)}  \sup_{s \in G} \sum_{\gamma \in \Gamma_j} \frac{\mu(\X_j \cap \gamma \X_j s^{-1})^2}{\mu(\X_j)^2} \nonumber \\
&\leq \frac{1}{c\mu(\X_j)} \norm{\psi}_{\L^1(G)} \sup_{s \in G} \Bigg( \sum_{\gamma \in \Gamma_j} \frac{\mu(\X_j \cap \gamma \X_j s^{-1})}{\mu(\X_j)} \Bigg)^2. \label{equ-3-proof-lem-ALSS-weak-star}
\end{align}
If $\gamma \X_j = \X_j \gamma$ for any $\gamma \in \Gamma_j$, then we estimate \eqref{equ-3-proof-lem-ALSS-weak-star} further with the pairwise disjointness \eqref{Dgamma=Dgamma2} of the sets $\X_j \gamma s^{-1}$ for different values of $\gamma \in \Gamma_j$ in the second equality and \eqref{DGamma=G} in the third equality
\begin{align*}
\MoveEqLeft
\sum_{\gamma \in \Gamma_j} \frac{\mu(\X_j \cap \gamma \X_j s^{-1})}{\mu(\X_j)} 
=\sum_{\gamma \in \Gamma_j} \frac{\mu(\X_j \cap \X_j \gamma s^{-1})}{\mu(\X_j)} 
=\frac{\mu(\X_j \cap \X_j \Gamma_j s^{-1})}{\mu(\X_j)}\\
&=\frac{\mu(\X_j \cap G s^{-1})}{\mu(\X_j)} 
=\frac{\mu(\X_j)}{\mu(\X_j)} 
=1.
\end{align*}
If $\mu(\X_j \Delta \X_j^{-1}) = 0$, then we estimate \eqref{equ-3-proof-lem-ALSS-weak-star} using \eqref{estimation-1} in the first inequality and \eqref{Dgamma=Dgamma2} in the first equality and \eqref{DGamma=G} in the last equality, giving
\begin{align*}
\MoveEqLeft
\sum_{\gamma \in \Gamma_j} \frac{\mu(\X_j \cap \gamma \X_j s^{-1})}{\mu(\X_j)}
\leq \sum_{\gamma \in \Gamma_j} \frac{\mu(\X_j \cap s \X_j \gamma^{-1})}{\mu(\X_j)}
=\frac{\mu(\X_j \cap s \X_j \Gamma_j)}{\mu(\X_j)} 
=1.
\end{align*}
By \cite[Theorem 19.15]{HR}, we conclude that the measure $\phi_j \mu_{\Gamma_j}$ is bounded with $\norm{\phi_j \mu_{\Gamma_j}}_{\M(G)} \leq \frac{1}{c\mu(\X_j)} \norm{\psi}_{\L^1(G)}$. Therefore, using \eqref{Symbol-Phij} and \cite[Theorem 20.12]{HR} in the first inequality and the unimodularity of $G$ to write $\mu(\X_j^{-1}) = \mu(\X_j)$ in the third inequality, we obtain
\begin{align*}
\bnorm{\widetilde{\phi_j}}_{\L^1(G)}
&\leq \frac{1}{\mu(\X_j)} \norm{1_{\X_j}}_{\L^1(G)} \norm{\phi_j \mu_{\Gamma_j}}_{\M(G)} \bnorm{1_{\X_j^{-1}}}_{\L^1(G)}\\
&\leq \frac{1}{c\,\mu(\X_j)} \norm{\psi}_{\L^1(G)} \bnorm{1_{\X_j^{-1}}}_{\L^1(G)}
\leq \frac{1}{c} \norm{\psi}_{\L^1(G)}.
\end{align*}
Thus, \eqref{equ-1-proof-lem-ALSS-weak-star} is shown. 
\end{proof}

Next, observe that if $\psi$ has a support away from the origin $e \in G$ then $\widetilde{\phi_j}(r) = 0$ for $r$ close to $e$. More precisely, we have the following observation. This lemma is not useful if $G$ is compact.

\begin{lemma} 
\label{Null-phij}
Suppose that $\psi(s)=0$ a.e. if $\dist(s,e) < R$ for some $R>4$. Then we have $(\Xi_j\psi)(r) = 0$ for any $r \in B'(e,R-4)$ and any $j$ large enough. 
\end{lemma}

\begin{proof}
We pick $j_0 \in \N$ and take $j \geq j_0$ such that $\X_j \subset B'(e,1)$ for these $j$. By \eqref{symbole-tilde-varphi} (the computation of \cite[Lemma 2.1 (2)]{Haa3} is valid) and \eqref{Symbol-Phij}, we have
$$
\widetilde{\phi_j}(r) 
=\frac{1}{\mu(\X_j)} \int_{\X_j} \phi_j(\gamma_j(rt)) \d\mu(t)
=\frac{1}{c\mu(\X_j)^4} \int_{\X_j} \int_{G} \psi(s) \mu\big(\X_j \cap \gamma_j(rt) \X_j s^{-1}\big)^2 \d\mu(s) \d\mu(t).
$$
Let $r \in B'(e,R-4)$. If $\dist(s,e) <R$ the integrand is zero. On the other hand, if $\dist(s,e) \geq R$, writing $rt = \omega_j(rt) \gamma_j(rt)$ where $\omega_j(rt) \in \X_j$, we have for any $\omega_j' \in \X_j$
\begin{align*}
\dist(\gamma_j(rt)\omega_j's^{-1},e) 
&=\dist(\omega_j(rt)^{-1} rt \omega_j' s^{-1},e) 
=\dist(\omega_j(rt)^{-1} rt \omega_j',s) \\
&\geq \dist(s,e) - \dist(\omega_j(rt)^{-1} rt \omega_j',e) \\
&\geq \dist(s,e) - \dist(\omega_j(rt)^{-1} rt \omega_j',\omega_j') - \dist(\omega_j',e) \\
&\geq \dist(s,e) - \dist(\omega_j(rt)^{-1}rt,e) - 1 \\
&\geq \dist(s,e) - \dist(\omega_j(rt)^{-1}rt,t) - \dist(t,e) - 1\\
&\geq \dist(s,e) - \dist(\omega_j(rt)^{-1}r,e) - 2 \\
&\geq \dist(s,e) - \dist(\omega_j(rt)^{-1}r,r) - \dist(r,e) - 2\\
&=\dist(s,e) -\dist(\omega_j(rt)^{-1},e)- \dist(r,e)- 2 \\
&=\dist(s,e) -\dist(e,\omega_j(rt))- \dist(r,e)- 2 \\
&\geq \dist(s,e) - \dist(r,e) - 3
\geq R-R+4-3
=1.
\end{align*}
So the integrand is also zero. We infer that we have $\widetilde{\phi_j}(r) = 0$.
\end{proof}

We turn to the weak* convergence\footnote{\thefootnote. Note that if $G$ is compact, the proof is more simple. No need to use $\chi$.} of the symbol $\widetilde{\phi_j}$. 

\begin{lemma}
\label{lem-ALSS-weak-star}
Let $\psi \in \L^\infty(G)$. Assume in addition that $\gamma \X_j = \X_j \gamma$ for any $j \in \N$ and any $\gamma \in \Gamma_j$ or that $\mu(\X_j \Delta \X_j^{-1}) = 0$ for any $j \in \N$. Then $\Xi_j(\psi) \xra[j]{} \psi$ for the weak* topology of $\L^\infty(G)$.
\end{lemma}

\begin{proof}
Let $g \in \L^1(G)$ be a testing element of weak* convergence. By density of $\mathrm{C}_c(G)$ in $\L^1(G)$ and the uniform estimate \eqref{equ-2-proof-lem-ALSS-weak-star}, we can assume in fact that $g \in \mathrm{C}_c(G)$. 

Then if $\chi \in \mathrm{C}_c(G)$ is a cut-off function with $\chi(s) = 1$ for all $s$ with\footnote{\thefootnote. Recall that $\exc(A,B)=\sup\{ \dist(a,B) :a \in A\}$.} $\dist(s,e) < R\overset{\mathrm{def}}=4 + \exc(\supp g,\{e\})$, (recall that the metric $\dist$ used previously is proper) we have $\psi\chi=1$ on $\supp(g)$. So $\langle \psi, g \rangle_{\L^\infty(G),\L^1(G)}=\langle \psi \chi, g \rangle_{\L^\infty(G),\L^1(G)}$. Moreover,  we have 
$$
\Xi_j(\psi) 
= \Xi_j(\psi \chi) + \Xi_j(\psi(1 - \chi)).
$$ 
Recall that $\psi(1-\chi)$ is zero if $\dist(s,e) <R$. Hence by applying Lemma \ref{Null-phij} with $\psi(1-\chi)$ instead of $\psi$, we deduce that the function $\Xi_j(\psi(1-\chi))$ is zero if $r \in B'(e,\exc(\supp g,\{e\}))$, in particular on $\supp g$. We conclude that $\langle \widetilde{\phi_j}, g \rangle = \langle \Xi_j(\psi \chi), g \rangle$. 

Now let $\psi_\epsi \in \mathrm{C}_c(G)$ be an $\epsi$-approximation in $\L^1(G)$ norm of $\psi \chi \in \L^1(G) \cap \L^\infty(G)$. Using \eqref{equ-1-proof-lem-ALSS-weak-star}, in the second equality, we obtain
\begin{align*}
\MoveEqLeft
\left| \big\langle \Xi_j(\psi) ,g \big\rangle_{\L^\infty(G),\L^1(G)} - \langle \psi, g \rangle_{\L^\infty(G),\L^1(G)} \right| 
=\left| \big\langle \Xi_j(\psi\chi), g \big\rangle - \langle \psi\chi, g \rangle \right|\\
&\leq \left| \big\langle (\Xi_j - \Id_{\L^1(G)})(\psi \chi - \psi_\epsi), g \big\rangle \right| + \left| \big\langle \Xi_j(\psi_\epsi)-\psi_\epsi, g \big\rangle \right| \\
&\leq \left( \frac1c + 1 \right) \norm{\psi \chi - \psi_\epsi}_{\L^1(G)} \norm{g}_{\L^\infty(G)} + \left| \big\langle \Xi_j(\psi_\epsi) -\psi_\epsi, g \big\rangle \right|\\ 
&\leq \left( \frac1c + 1 \right) \epsi \norm{g}_{\L^\infty(G)} + \left| \big\langle \Xi_j(\psi_\epsi) -\psi_\epsi, g \big\rangle \right|.
\end{align*}
Thus the first term becomes small uniformly in $j \geq j_0$. For the second term, we use the pointwise convergence $\Xi_j \psi_\epsi(s) \to \psi_\epsi(s)$ from \eqref{equ4-Fourier-complementation} together with the domination $|\Xi_j \psi_\epsi(s) g(s)| \leq \frac1c \norm{\psi_\epsi}_{\L^\infty(G)} |g(s)|$.
\end{proof}

If the assumptions of Lemma \ref{lem-ALSS-weak-star} are satisfied, we deduce by Lemma \ref{Lemma-Symbol-weakstar-convergence-weakLp-vrai} that $M_{\widetilde{\phi_j}} \to M_\psi$ in the weak operator topology of $\B(\L^p(\VN(G)))$ (point weak* topology if $p = \infty$). Moreover, this convergence also holds if $\psi$ is a continuous and bounded symbol. Indeed, according to \eqref{equ4-Fourier-complementation}, we have a pointwise convergence $\widetilde{\phi_j}(s) \to \psi(s)$, which together with the uniform bound $\|\widetilde{\phi_j}\|_{\L^\infty(G)} \leq \frac{1}{c} \|M_\psi\|_{\cb,\L^p(\VN(G)) \to \L^p(\VN(G))}$ of \eqref{inegalite-1-sur-c} also implies weak* convergence $\widetilde{\phi_j} \to \psi$, so that we can again appeal to Lemma \ref{Lemma-Symbol-weakstar-convergence-weakLp-vrai}. According to the description of the predual space \eqref{equ-predual-bracket}, we have for the convergent subnet $M_{\widetilde{\phi_{j(k)}}}$ of $M_{\widetilde{\phi_j}}$ that
$$
\Big\langle M_{\widetilde{\phi_{j(k)}}} f, g \Big\rangle_{\L^p(\VN(G),\L^{p^*}(\VN(G))}
\xra[k]{} \big\langle P_G^p(M_\psi) f, g \big\rangle_{\L^p(\VN(G),\L^{p^*}(\VN(G))}.
$$
for $f \in \L^p(\VN(G))$ and $g \in \L^{p^*}(\VN(G))$. Since a subnet of a convergent net converges to the same limit, we deduce $P_G^p(M_\psi) = M_\psi$.

Now, we turn to the case $p = 1$. We simply put
$$
P_G^1 \co \CB(\L^1(\VN(G))) \to \CB(\L^1(\VN(G))), \: T \mapsto P^\infty_G(T^*)_*.
$$
Note that $P^\infty_G(T^*)$ belongs to $\mathfrak{M}^{\infty,\cb}(G)$, so that it admits indeed a preadjoint $P^\infty_G(T^*)_*$ belonging to $\mathfrak{M}^{1,\cb}(G)$ by Lemma \ref{lemma-duality-Fourier-multipliers}. We check now the claimed properties of $P_G^1$. Linearity and boundedness are clear. If $T \co \L^1(\VN(G)) \to \L^1(\VN(G))$ is completely positive, then by Lemma \ref{Lemma-adjoint-cp}, $T^*$ is also completely positive and hence also $P_G^\infty(T^*)$. We conclude that $P_G^1(T) = P_G^\infty(T^*)_*$ is completely positive. If $M_\psi \in \mathfrak{M}^{1,\cb}(G)$, then we have $P_G^1(M_{\psi}) = P_G^\infty((M_{\psi})^*)_*=(P_G^\infty(M_{\check{\psi}}))_* =(M_{\check{\psi}})_*=M_{\psi}$.

It remains to check the claimed compatibility property. We need the following lemma.

\begin{lemma}
\label{lem-compatibilite-p=1}
For $j \in \N$ and any completely bounded map $T \co \L^1(\VN(G)) \to \L^1(\VN(G))$, we have $P_j^1(T)^* = P_j^\infty(T^*)$.
\end{lemma}

\begin{proof}
In this proof we denote by $\phi_j^T$ the symbol of $\frac1c P_{\Gamma_j}^p(\Psi_j^p T \Phi_j^p)$. Let $S \co \L^1(\VN(\Gamma_j)) \to \L^1(\VN(\Gamma_j))$ be a completely bounded map. We denote by $\psi_j^S$ the symbol of the Fourier multiplier $P_{\Gamma_j}^1(S)$ given by Corollary \ref{Prop-complementation-Fourier-subgroups} with $G=H=\Gamma_j$. The symbol $\psi_j^{(S^*)}$ of the Fourier multiplier $P_{\Gamma_j}^\infty(S^*)$ is given by (where $\gamma \in \Gamma_j$)
$$
\psi_j^{(S^*)}(\gamma) 
=\tau_{\Gamma_j} \big(S^*(\lambda_\gamma) \lambda_\gamma^{-1}\big) 
=\tau_{\Gamma_j} \big(\lambda_\gamma S(\lambda_\gamma^{-1})\big)
= \tau_{\Gamma_j} \big(S(\lambda_\gamma^{-1})\lambda_\gamma \big)
=\psi_j^S(\gamma^{-1})
=\check{\psi}_j^S(\gamma).
$$ 
Using Lemma \ref{lemma-duality-Fourier-multipliers} in the second equality, we obtain 
\begin{equation}
	\label{Equality-pratique}
	 P_{\Gamma_j}^\infty(S^*)
	=M_{\check{\psi}_j^S}
	=(M_{\psi_j^S})^*
	=\big(P_{\Gamma_j}^1(S)\big)^*.
\end{equation}
Note that $
\Psi^\infty_j T^* \Phi^\infty_j 
= (\Psi^1_j T \Phi^1_j)^*$. 
This implies 
$$
M_{\phi_j^{(T^*)}} 
= \frac{1}{c} P_{\Gamma_j}^\infty\big(\Psi^\infty_j T^* \Phi^\infty_j\big) 
= \frac{1}{c} P_{\Gamma_j}^\infty\big((\Psi^1_j T \Phi^1_j)^*\big) 
= \frac{1}{c} P_{\Gamma_j}^1\big(\Psi^1_j T \Phi^1_j\big)^*
=\big(M_{\phi_j^{T}}\big)^*
=M_{\check{\phi}_j^T}
$$ 
where we use \eqref{Equality-pratique} in the central equality. 
Now, using \eqref{equ1-Fourier-complementation}, $(1_{\X_j})\check{\phantom{i}} = 1_{\X_j^{-1}}$ and $\check{\mu}_{\Gamma_j} = \mu_{\Gamma_j}$, we deduce
\begin{align*}
\MoveEqLeft
\widetilde{\phi_j^{(T^*)}}
=\frac{1}{\mu(\X_j)} 1_{\X_j} \ast \big(\phi_j^{(T^*)}\mu_{\Gamma_j} \big) \ast 1_{\X_j^{-1}}
=\frac{1}{\mu(\X_j)} 1_{\X_j} \ast \big(\check{\phi}_j^T\mu_{\Gamma_j} \big) \ast 1_{\X_j^{-1}}\\
&= \frac{1}{\mu(\X_j)} \overbrace{1_{\X_j} \ast \big(\phi_j^T\mu_{\Gamma_j} \big) \ast 1_{\X_j^{-1}}}^{\check{}} 
=\check{\widetilde{\phi_j^T}},
\end{align*}
thus finishing the proof of the lemma since 
$
P_j^1(T)^*
=\big(M_{\widetilde{\phi_j^T}}\big)^*
=M_{\check{\widetilde{\phi_j^T}}}
=M_{\widetilde{\phi_j^{(T^*)}}}
=P_j^\infty(T^*)$. 
\end{proof}

Now suppose that $T$ belongs to both $\CB(\L^1(\VN(G)))$ and $\CB(\L^p(\VN(G)))$. Recall that the symbol $\widetilde{\phi_j^T}$ of $P_j^p(T)$ does not depend on $p$ if $T$ belongs to two different spaces $\CB(\L^p(\VN(G)))$ and $\CB(\L^q(\VN(G)))$. Consequently the symbols of $P_j^{p}(T)^*$ and $P_j^1(T)^*$ are identical and the symbols of $P_j^\infty(T^*)$ and $P_j^{p^*}(T^*)$ are also identical. Using the previous lemma, we conclude that
$$
P_j^{p}(T)^* = P_j^{p^*}(T^*).
$$ 
Passing to the limit when $j \to \infty$, we infer that $P_G^p(T)^* = P_G^{p^*}(T^*)$.  Therefore, for any $x \in \L^1(\VN(G)) \cap \L^p(\VN(G))$ and any $y \in \VN(G) \cap \L^{p^*}(\VN(G))$, using the compatibility of the $P_G^q$ already proven, we have
\begin{align*}
\big\langle P_G^1(T) x , y \big\rangle 
& = \big\langle P_G^\infty(T^*)_* x, y \big\rangle 
= \big\langle x, P_G^\infty(T^*)y \big\rangle 
= \big\langle x, P_G^{p^*}(T^*) y \big\rangle 
= \big\langle x, P_G^p(T)^* y \big\rangle \\
& = \big\langle P_G^p(T) x, y \big\rangle.
\end{align*}
This shows the compatibility on the $\L^1$ level.

For the last sentence, use Proposition \ref{Prop-recover-weak-star-continuity}.
\end{proof}


\begin{remark}\normalfont
\label{Rem-continuous-symbol}
We ignore if the condition \eqref{equ-Haar-measure-convergence} can be removed.
\end{remark}

Since the symbol of a completely bounded Fourier multiplier $\M_\phi \co \VN(G) \to \VN(G)$ is equal almost everywhere to a continuous function, see e.g. \cite[Corollary 3.3]{Haa3}, the previous theorem gives projections at the level $p=\infty$ and $p=1$. 

\begin{cor}
\label{cor-ADS-complementation-without-amenability}
Let $G$ be a second countable unimodular locally compact group satisfying $\mathrm{ALSS}$ such that \eqref{equ-Haar-measure-convergence} holds. Then there exist projections $P_G^\infty \co \CB_{\mathrm{w}^*}(\VN(G)) \to \CB_{\mathrm{w}^*}(\VN(G))$ and $P_G^1 \co \CB(\L^1(\VN(G))) \to \CB(\L^1(\VN(G)))$ which are compatible, onto $\mathfrak{M}^{\infty,\cb}(G)$ and $\mathfrak{M}^{1,\cb}(G)$ of norm at most $\frac1c$ preserving complete positivity.
\end{cor}

\subsection{Examples of computations of the density}
\label{subsec-computations-density}

In this section, we will describe some concrete non-abelian groups in which Theorem \ref{thm-complementation-Fourier-ADS-amenable} applies. Before that, we start by recalling some information on semidirect products.

\paragraph{Semidirect products.} Let $G_1$ and $G_2$ be topological groups and consider some group homomorphism $\eta \co G_2 \to \Aut(G_1)$ such that the map\footnote{\thefootnote. If $\Aut(G_1)$ is equipped with the well-known Braconnier topology, the continuity of the map $(s,t) \mapsto \eta_t(s)$ from $G_1 \times G_2$ onto $G_1$ is equivalent to the continuity of the homomorphism $\eta \co G_2 \to \Aut(G_1)$.} 
\begin{equation}
\label{Continuity-semidirect}
G_1 \times G_2 \to G_1, \ (s,t) \mapsto \eta_t(s)\text{ is continuous}.
\end{equation}
The semidirect product $G_1 \rtimes_\eta G_2$ \cite[page 183]{FeD1} is the topological group with the underlying set $G_1 \times G_2$ equipped with the product topology and with the group operations given by
\begin{equation}
\label{Operations-semidirect}
(s,t) \rtimes_\eta (s',t')
=\big(s\eta_t(s'),tt'\big)
\quad \text{and} \quad
(s,t)^{-1}
=\big(\eta_{t^{-1}}(s^{-1}), t^{-1}\big).
\end{equation}
The group $G_1$ identifies to a closed normal subgroup of $G_1 \rtimes_\eta G_2$ and $G_2$ as a closed subgroup \cite[page 183]{FeD1} and we have $(G_1 \rtimes_\eta G_2)/G_1=G_2$.

If $G_1$ and $G_2$ are locally compact groups then $G_1 \rtimes_\eta G_2$ is a locally compact group. If $G_1$ and $G_2$ are in addition equipped with some left Haar measures $\mu_{G_1}$ and $\mu_{G_2}$, by \cite[Proposition 9.5 Chapter III]{FeD1} (see also \cite[15.29]{HR}) a left Haar measure of $G$ is given by $\mu_{G}=\mu_{G_1} \ot (\delta \mu_{G_2})$ where $\delta \co G_2 \to (0,\infty)$ is defined by $\delta(t)=\textrm{mod}\, \eta_{t}$ where $t \in G_2$. By \cite[Chapter III, (9.6)]{FeD1}, a right Haar measure is given by $\Delta_{G_1}\mu_{G_1} \ot \Delta_{G_2} \mu_{G_2}$. It is folklore and easy to deduce from \cite[pages 119-120]{Loo1} that if $G_1$ and $G_2$ are unimodular and if each automorphism $\eta_t$ of $G_1$ is measure-preserving, i.e. if
\begin{equation*}
\label{Preserves-Haar-measure}
\int_{G_1} f(\eta_t(s)) \d\mu_{G_1}(s)
=\int_{G_1} f(s) \d\mu_{G_1}(s), \quad t \in G_2, f \in \mathrm{C}_c(G_1).
\end{equation*}
then the group $G_1 \rtimes_\eta G_2$ is unimodular. In this case, $\mu_G= \mu_{G_1} \ot \mu_{G_2}$ gives a Haar measure on $G$. 

We will use the following lemma.

\begin{lemma}
\label{Lemma-Fundamental-semidirect}
Let $G_1$ and $G_2$ be locally compact groups. Let $\Gamma_1$ and $\Gamma_2$ be lattices in $G_1$ and $G_2$. Suppose that $\eta \co G_2 \to \Aut(G_1)$ is a homomorphism satisfying \eqref{Continuity-semidirect}. If $\eta_t(\Gamma_1) \subset \Gamma_1$ for any $t \in G_2$ then $\Gamma=\Gamma_1 \rtimes_{\eta|\Gamma_2} \Gamma_2$ is a lattice of $G_1 \rtimes_\eta G_2$. If in addition $\X_1$ and $\X_2$ are associated fundamental domains, then $\X=\X_1 \times \X_2$ is a fundamental domain associated with $\Gamma$.
\end{lemma}

\begin{proof}
The first part is \cite[Exercise B.3.5]{BHV}. It remains to show that $\X$ is a fundamental domain of $\Gamma$. Indeed, this subset is clearly Borel measurable. Consider some arbitrary element $(s_1,s_2)$ of $G$. Since $\X_1$ is a fundamental domain of $\Gamma_1$, we can write $s_1= \omega_1 \gamma_1$ with $\omega_1 \in \X_1$ and $\gamma_1 \in \Gamma_1$ and similarly $s_2=\omega_2 \gamma_2$ with $\omega_2 \in \X_2$ and $\gamma_2 \in \Gamma_2$. Consequently, using \eqref{Operations-semidirect}, we have
$$
(s_1,s_2)
=(\omega_1\gamma_1,\omega_2 \gamma_2)
=\big(\omega_1\eta_{\omega_2}(\eta_{\omega_2^{-1}}(\gamma_1)),\omega_2 \gamma_2\big)
=(\omega_1,\omega_2) \rtimes_\eta \big(\eta_{\omega_2^{-1}}(\gamma_1),\gamma_2\big)
$$ 
where $(\omega_1,\omega_2) \in \X$ and $\big(\eta_{\omega_2^{-1}}(\gamma_1),\gamma_2\big) \in \Gamma$. So we obtain \eqref{DGamma=G}. 

Consider some $(\omega_1,\omega_2), (\omega_1',\omega_2') \in \X$ where $\omega_1,\omega_1' \in \X_1$ and $\omega_2,\omega_2' \in \X_2$ and some elements $(\gamma_1,\gamma_2)$ and $(\gamma_1',\gamma_2')$ of $\Gamma$. If $(\omega_1,\omega_2) \rtimes_\eta (\gamma_1,\gamma_2)=(\omega_1',\omega_2') \rtimes_\eta (\gamma_1',\gamma_2')$ then $
\big(\omega_1\eta_{\omega_2}(\gamma_1),\omega_2\gamma_2\big)
=\big(\omega_1'\eta_{\omega_2'}(\gamma_1'),\omega_2'\gamma_2'\big)$.
Therefore $\omega_2\gamma_2 = \omega_2'\gamma_2'$. Since $\X_2$ is a fundamental domain we deduce by \eqref{Dgamma=Dgamma2} that $\omega_2 = \omega_2'$ and $\gamma_2 = \gamma_2'$.
Inserting into the previous first variable, we get $\omega_1\eta_{\omega_2}(\gamma_1) = \omega_1'\eta_{\omega_2}(\gamma_1')$. Since $\X_1$ is a fundamental domain we have by \eqref{Dgamma=Dgamma2}, $\omega_1= \omega_1'$ and $\eta_{\omega_2}(\gamma_1) = \eta_{\omega_2}(\gamma_1')$. So $\gamma_1=\gamma_1'$. We conclude that $\X$ satisfies \eqref{Dgamma=Dgamma2}. 
\end{proof}

\paragraph{Groups acting on locally finite trees.} We give now some examples of compact non-discrete $\ALSS$ groups acting on locally finite trees for which Theorem \ref{thm-complementation-Fourier-ADS-amenable} yields a bounded map $P_G^p \co \CB(\L^p(\VN(G))) \to \mathfrak{M}^{p,\cb}(G)$ with sharp norm, i.e. with a norm equal to one. 

Let $(m_j)_{j \geq 1}$ be a sequence of integers with $m_j \geq 2$. Let $\ovl{Y}=(Y_j)_{j \geq 1}$ be a sequence of alphabets with $|Y_j|=m_j$ and $Y_j=\{y_{j,1},\ldots,y_{j,m_j}\}$. If $n \geq 0$, a word of length $n$ over $\ovl{Y}$ is a sequence of letters of the form $w=w_1w_2\dots w_n$ with $w_j \in Y_j$ for all $j$. The unique word of length 0, the empty word, is denoted by $\emptyset$. The set of words of length $n$ is called the $n$th level.

Now we introduce the prefix relation $\leq$ on the set of all words over $\ovl{Y}$. Namely, we let $w \leq z$ if $w$ is an initial segment of the sequence $z$, i.e. if $w=w_1\dots w_n$, $z=z_1\dots z_k$ with $n \leq k$ and $w_j=z_j$ for all $j \in \{1,\dots,n\}$. This relation is a partial order and the partially ordered set $\tree$ of words over $\ovl{Y}$ is called the spherically homogeneous tree over $\ovl{Y}$. We refer to \cite{BGS1} and \cite{Gri1} for more information.

Let us give now the graph-theoretical interpretation of $\tree$. Every word over $\ovl{Y}$ represents a vertex in a rooted tree. Namely, the empty word $\emptyset$ represents the root, the $m_1$ one-letter words $y_{1,1},\ldots,y_{1,m_1}$ represent the $m_1$ children of the root, the $m_2$ two-letter words $y_{1,1}y_{2,1},\ldots,y_{1,1}y_{2,m_2}$ represent the $m_2$ children of the vertex $y_{1,1}$, etc. 

\begin{center}
\includegraphics[scale=0.55]{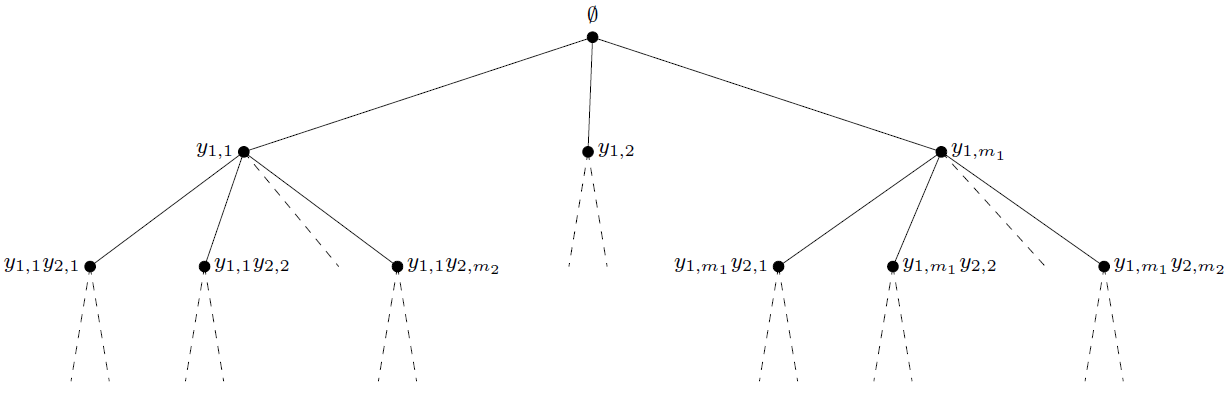}
\end{center}

 
An automorphism of $\tree$ is a bijection of $\tree$ which preserves the prefix relation. From the graph-theoretical point of view, an automorphism is a bijection which preserves edge incidence and the distinguished root vertex $\emptyset$. We denote by $\Aut(\tree)$ the group of automorphisms of $\tree$ and if $j \geq 0$ by $\Aut_{[j]}(\tree)$ the subgroup of automorphisms whose vertex permutations at level $j$ and below\footnote{\thefootnote. The action is trivial on the levels $j,j+1,j+2,\ldots$.} are trivial.

We equip $\tree$ with the discrete topology and $\Aut(\tree)$ with the topology of pointwise convergence. By \cite[page 133]{Gri1}, the sequence $(\Aut_{[j]}(\tree))_{j \geq 0}$ of finite groups and the canonical inclusions $\psi_{ij} \co \Aut_{[j]}(\tree) \to \Aut_{[i]}(\tree)$ where $j \geq i \geq 0$ define an inverse system and we have an isomorphism
\begin{equation}
\label{Aut-tree}
\Aut(\tree)
=\varprojlim \Aut_{[j]}(\tree).	
\end{equation} 
In particular, $\Aut(\tree)$ is a profinite group, hence compact and totally disconnected by \cite[Corollary 1.2.4]{Wil}.

If $j \geq 0$, we denote by $\St(j)$ the $j$th level stabilizer consisting of automorphisms of $\tree$ which fix all the vertices on the level $j$ (and of course on the levels $0,1,\ldots,j-1$). Then $\St(j)$ is a normal subgroup of $\Aut(\tree)$ which is open if $j \geq 1$. By \cite[page 20]{BGS1}, for any $j \geq 0$, we have an isomorphism
\begin{equation}
\label{Aut-tree-semidirect}
\Aut(\tree)
=\St(j) \rtimes \Aut_{[j]}(\tree).	
\end{equation}

\begin{prop}
\label{prop-tree-group-constant-c}
The compact group $\Aut(\tree)$ is second countable and $\ALSS$ with respect to the sequence $(\Aut_{[j]}(\tree))_{j \geq 1}$ of finite lattice subgroups and to the sequence $(\St(j))_{j \geq 1}$ of symmetric fundamental domains. Moreover, \eqref{equ-Haar-measure-convergence} holds with $c=1$. More precisely, for any integer $j \in \N$ and any $\gamma \in \Aut_{[j]}(\tree)$, we have 
\begin{equation}
\label{equ-tree}
\frac{1}{\mu(\St(j))} \int_{\Aut(\tree)} \frac{\mu(\St(j) \cap \gamma \St(j) s)^2}{\mu(\St(j))^2} \d\mu(s) 
= 1.
\end{equation}
Consequently, Theorem \ref{thm-complementation-Fourier-ADS-amenable} applies.
\end{prop}

\begin{proof}
Since the inverse system is indexed by $\N$, by \cite[Proposition 4.1.3]{Wil}, the group $\Aut(\tree)$ is second countable. By \eqref{Aut-tree-semidirect}, we have $\Aut(\tree)=\St(j)\Aut_{[j]}(\tree)$. Suppose that $\gamma_1, \gamma_2$ belong to $\Aut_{[j]}(\tree)$ and that $\omega_1, \omega_2 \in \St(j)$ satisfy $\omega_1\gamma_1=\omega_2\gamma_2$. Then $\omega_2^{-1}\omega_1=\gamma_2\gamma_1^{-1}$. Using again \eqref{Aut-tree-semidirect}, we infer that $\gamma_1=\gamma_2$. Moreover, $\St(j)$ is open hence Borel measurable, and a subgroup hence symmetric. We conclude that $\St(j)$ is a symmetric fundamental domain for $\Aut_{[j]}(\tree)$. 

Now, we have a homeomorphism $\Aut(\tree)/\Aut_{[j]}(\tree)=(\St(j) \rtimes \Aut_{[j]}(\tree))/\Aut_{[j]}(\tree)=\St(j)$. Note that the subgroup $\St(j)$ is open, hence closed in the compact group $\Aut(\tree)$ by \cite[Theorem 5.5]{HR} and finally compact. We conclude that $\Aut_{[j]}(\tree)$ is a cocompact lattice. Moreover, by \cite[page 133]{Gri1}, the sequence $(\St(j)_{j \geq 1}$ is an open neighborhood basis of $\Id_{\tree}$ in $\Aut(\tree)$. 


It remains to compute \eqref{equ-tree}. By normality of $\St(j)$, for any $\gamma \in \Aut_{[j]}(\tree)$, we have $\gamma \St(j) = \St(j) \gamma$. Using that $\mu$ is a left Haar measure of $\Aut(\tree)$ in the last equality, for any $\gamma \in \Aut_{[j]}(\tree)$, we deduce that
\begin{align*}
\MoveEqLeft
  \frac{1}{\mu(\St(j))} \int_{\Aut(\tree)} \frac{\mu(\St(j) \cap \gamma \St(j) s)^2}{\mu(\St(j))^2} \d\mu(s) 
=\frac{1}{\mu(\St(j))} \int_{\Aut(\tree)} \frac{\mu(\St(j) \cap  \St(j) \gamma s)^2}{\mu(\St(j))^2} \d\mu(s)\\
&=\frac{1}{\mu(\St(j))} \int_{\Aut(\tree)} \frac{\mu(\St(j) \cap  \St(j) s)^2}{\mu(\St(j))^2} \d\mu(s).          
\end{align*}     
For any $s \in \Aut(\tree)$, the sets $\St(j)$ and $\St(j) s$ are right cosets of the subgroup $\St(j)$ in $\Aut(\tree)$. Since two right cosets are either identical or disjoint, we deduce that
$$
\St(j) \cap \St(j) s 
= 
\begin{cases} 
\St(j) & \text{ if } s \in \St(j) \\ 
\emptyset &  \text{ if } s \not\in \St(j) 
\end{cases}.
$$
Now, we can conclude since 
$$
\frac{1}{\mu(\St(j))} \int_{\Aut(\tree)} \frac{\mu(\St(j) \cap \St(j) s)^2}{\mu(\St(j))^2} \d\mu(s) 
=\frac{1}{\mu(\St(j))} \int_{\St(j)} \frac{\mu(\St(j))^2}{\mu(\St(j))^2} \d\mu(s) 
=\frac{\mu(\St(j))}{\mu(\St(j))} 
=1.
$$
\end{proof}

\begin{remark} \normalfont

\label{Remark-iterated-wreath}
By \cite[page 134]{Gri1}, note that we have an isomorphism $\Aut(\tree)=\varprojlim (\mathrm{Sym}(Y_j) \wr \cdots \wr \mathrm{Sym}(Y_2) \wr \mathrm{Sym}(Y_1))$. If $(G_j,Y_j)_{j \geq 1}$ denotes a sequence of finite permutation groups (such that the actions are faithful), the same method gives a generalization for the inverse limit $G=\varprojlim (G_j \wr \cdots \wr G_2 \wr G_1)$ of iterated permutational wreath products. The verification is left to the reader.
\end{remark}

\paragraph{Stability under products.} 
The (good) behaviour of \eqref{equ-Haar-measure-convergence} under direct products is described in the following result.

\begin{prop}
\label{prop-product-complementation-ADS}
Let $G_1$ and $G_2$ be two second countable (unimodular) locally compact groups satisfying $\ALSS$ with respect to the sequences $(\Gamma_{1,j}), \: (\Gamma_{2,j})$ of lattices and to the sequences $(\X_{1,j})$, $(\X_{2,j})$ of associated fundamental domains. Suppose that \eqref{equ-Haar-measure-convergence} holds for both groups $G_1$ and $G_2$ with constants $c_1$ and $c_2$. Then $G = G_1 \times G_2$ is $\ALSS$ with respect to the lattices $(\Gamma_j) = (\Gamma_{1,j} \times \Gamma_{2,j})$ and associated fundamental domains $(\X_j)= (\X_{1,j} \times \X_{2,j})$ and it satisfies \eqref{equ-Haar-measure-convergence} with constant $c = c_1 \cdot c_2$.
Moreover, if $\X_{1,j}$ and $\X_{2,j}$ are symmetric (resp. $\gamma_k \X_{k,j} = \X_{k,j} \gamma_k$ for $k = 1,2$ and $\gamma_k \in \Gamma_{k,j}$) then $\X_j$ is symmetric (resp. $\gamma \X_j = \X_j \gamma$ for $\gamma \in \Gamma_j$).
Let $1 \leq p \leq \infty$ and suppose that $G_1$ and $G_2$ are amenable if $1 < p < \infty$. Then Theorem \ref{thm-complementation-Fourier-ADS-amenable} applies to $G=G_1 \times G_2$.
\end{prop}

\begin{proof}
If $G_1$ and $G_2$ are second countable then $G_1 \times G_2$ is also second countable. By Lemma \ref{Lemma-Fundamental-semidirect}, $\Gamma_j = \Gamma_{1,j} \times \Gamma_{2,j}$ is a lattice subgroup of $G_1 \times G_2$ and $\X_j = \X_{1,j} \times \X_{2,j}$ is an associated fundamental domain. If $\mu_1$ and $\mu_2$ are Haar measures on $G_1$ and $G_2$ then $\mu=\mu_1 \ot \mu_2$ is a Haar measure on $G$. We check that $G_1 \times G_2$ is $\ALSS$ with respect to $(\Gamma_j)$ and $(\X_j)$. Let $V$ be a neighborhood of $e \in G_1 \times G_2$. Then there exist neighborhoods $U_1$ of $e_1 \in G_1$ and $U_2$ of $e_2 \in G_2$ such that $U_1 \times U_2 \subset V$. Since $G_1$ and $G_2$ are $\ALSS$, there exists $j_0 \in \N$ such that $\X_{1,j} \subset U_1$ and $\X_{2,j} \subset U_2$ for any $j \geq j_0$. Consequently, $\X_j = \X_{1,j} \times \X_{2,j} \subset U_1 \times U_2 \subset V$. Consequently $G_1 \times G_2$ is $\ALSS$. Now for $\gamma_1 \in \Gamma_{1,j}$, we put
$$
I_1(\gamma_1) 
\ov{\mathrm{def}}{=} \frac{1}{\mu_1\big(\X_{1,j}\big)} \int_{G_1} \frac{\mu_1^2\big(\X_{1,j} \cap \gamma_1 \X_{1,j} s_1\big)}{\mu_1^2\big(\X_{1,j}\big)} \d\mu_1(s_1)
$$
and similarly, for given $\gamma_2 \in \Gamma_{2,j}$ resp. $\gamma \in \Gamma_j$, we define $I_2(\gamma_2)$ resp. $I(\gamma)$. We claim that $I((\gamma_1,\gamma_2)) = I_1(\gamma_1) I_2(\gamma_2)$. Indeed, using the elementary fact $(A \times B) \cap (C \times D) = (A \cap C) \times (B \cap D)$, we have
\begingroup
\allowdisplaybreaks
\begin{align*}
\MoveEqLeft
I((\gamma_1,\gamma_2)) 
=\frac{1}{\mu\big(\X_{1,j} \times \X_{2,j}\big)} \int_{G_1 \times G_2} \frac{\mu^2\big((\X_{1,j} \times \X_{2,j}) \cap (\gamma_1, \gamma_2) (\X_{1,j} \times \X_{2,j}) (s_1,s_2)\big)}{\mu^2\big(\X_{1,j} \times \X_{2,j}\big)} \d\mu(s_1,s_2) \\
&=\frac{1}{\mu_1\big(\X_{1,j}\big) \mu_2\big(\X_{2,j}\big)} \int_{G_1 \times G_2} \frac{\mu^2\big((\X_{1,j} \times \X_{2,j}) \cap (\gamma_1 \X_{1,j} s_1) \times (\gamma_2 \X_{2,j} s_2)\big)}{\mu_1^2\big(\X_{1,j}\big) \mu_2^2\big(\X_{2,j}\big)} \d\mu(s_1,s_2) \\
&=\frac{1}{\mu_1\big(\X_{1,j}\big) \mu_2\big(\X_{2,j}\big)} \int_{G_1 \times G_2} \frac{\mu^2\big((\X_{1,j} \cap \gamma_1 \X_{1,j} s_1) \times (\X_{2,j} \cap \gamma_2 \X_{2,j} s_2)\big)}{\mu_1^2\big(\X_{1,j}\big) \mu_2^2\big(\X_{2,j}\big)} \d\mu(s_1,s_2) \\
&=\frac{1}{\mu_1\big(\X_{1,j}\big) \mu_2\big(\X_{2,j}\big)} \int_{G_1} \frac{\mu_1^2\big(\X_{1,j} \cap \gamma_1 \X_{1,j} s_1\big)}{\mu_1^2\big(\X_{1,j}\big)} \d\mu_1(s_1) \int_{G_2} \frac{\mu_2^2\big(\X_{2,j} \cap \gamma_2 \X_{2,j} s_2\big)}{\mu_2^2\big(\X_{2,j}\big)} \d\mu_2(s_2) \\
&=I_1(\gamma_1) I_2(\gamma_2).
\end{align*}
\endgroup
Now let $K$ be a compact subset of $G_1 \times G_2$. We check \eqref{equ-Haar-measure-convergence}, that is
$$
\lim_{j \to \infty} \sup_{(\gamma_1,\gamma_2) \in \Gamma_j \cap K} \left| I((\gamma_1,\gamma_2)) - c_1 c_2 \right| 
= 0.
$$
Denoting $\pi_k \co G_1 \times G_2 \to G_k$ the canonical continuous projection, we have that $\pi_k(K) \subset G_k$ is compact $(k = 1,2)$. Then 
\begingroup
\allowdisplaybreaks
\begin{align*}
&\sup_{(\gamma_1,\gamma_2) \in \Gamma_j \cap K} \left| I((\gamma_1,\gamma_2)) - c_1 c_2 \right| 
\leq \sup_{(\gamma_1,\gamma_2) \in \Gamma_j \cap \pi_1(K) \times \pi_2(K)} \left| I((\gamma_1,\gamma_2)) - c_1 c_2 \right| \\
& = \sup_{\gamma_1 \in \Gamma_{1,j} \cap \pi_1(K)} \sup_{\gamma_2 \in \Gamma_{2,j} \cap \pi_2(K)} \left| I_1(\gamma_1) I_2(\gamma_2) - c_1 c_2 \right| \\
& \leq \sup_{\gamma_1 \in \Gamma_{1,j} \cap \pi_1(K)} \sup_{\gamma_2 \in \Gamma_{2,j} \cap \pi_2(K)} \left| I_1(\gamma_1) I_2(\gamma_2) - c_1 I_2(\gamma_2) \right| + \left| c_1 I_2(\gamma_2) - c_1 c_2 \right| \\
& \leq \sup_{\gamma_1 \in \Gamma_{1,j} \cap \pi_1(K)} |I_1(\gamma_1) - c_1| \sup_{\gamma_2 \in \Gamma_{2,j} \cap \pi_2(K)} |I_2(\gamma_2)| + c_1 \sup_{\gamma_2 \in \Gamma_{2,j} \cap \pi_2(K)} \left| I_2(\gamma_2) - c_2 \right| \\
& \xra[j \to +\infty]{} 0 \cdot c_2 + c_1 \cdot 0 
= 0.
\end{align*}
\endgroup
Thus, \eqref{equ-Haar-measure-convergence} follows for $G_1 \times G_2$ and constant $c_1 c_2$.
The statement about preservation of symmetric fundamental domains (resp. commutation $\gamma \X_j = \X_j \gamma$) is easy to check.
For the application of Theorem \ref{thm-complementation-Fourier-ADS-amenable}, we only note that $G_1 \times G_2$ is amenable once $G_1$ and $G_2$ are amenable.
\end{proof}

\begin{remark}\normalfont
\label{remark-discrete-group-c}
Let $G$ be a countable discrete group. The group $G$ is $\ALSS$ with respect to the constant sequences $(\Gamma_j)$ and $(\X_j)$ defined by $\Gamma_j=G$ and by $\X_j=\{e\}$ for any $j$. Moreover, for any $\gamma \in G$ and any $j$, it is obvious that
$$
\frac{1}{\mu_G(\X_j)} \int_G \frac{\mu_G^2(\X_j \cap \gamma \X_j s)}{\mu_G^2(\X_j)} \d\mu_G(s)
=1.
$$
\end{remark}

\paragraph{Semidirect products of abelian groups by discrete groups.} 

For semidirect products, the situation is not as good as direct products.

\begin{prop}
\label{Prop-complementation-Fourier-ADS-amenable}
Let $G_1$ be a second countable abelian locally compact group which is $\ALSS$ with respect to a sequence $(\Gamma_{1,j})$ of lattice subgroups associated to a sequence $(\X_{1,j})$ of fundamental domains such that \eqref{equ-Haar-measure-convergence} is satisfied. Let $G_2$ be a countable discrete group. Suppose that $\eta \co G_2 \to \Aut(G_1)$ is a homomorphism satisfying $\eta_t(\Gamma_{1,j}) \subset \Gamma_{1,j}$ for any $t \in G_2$ and any $j$. Then the semidirect product $G = G_1 \rtimes_\eta G_2$ is second countable and $\ALSS$ with respect to the sequences $(\Gamma_j)$ and $(\X_j)$ defined by $\Gamma_j = \Gamma_{1,j} \times G_2$ and $\X_j = \X_{1,j} \times \{e_{G_2}\}$. If in addition $\eta_{t}(\X_{1,j}) \subset \X_{1,j}$ for any $t \in G_2$ and any $j$ then \eqref{equ-Haar-measure-convergence} holds with some $c \in (0,1]$. If the $\X_{1,j}$ are symmetric (resp. $\gamma_1 \X_{1,j} = \X_{1,j} \gamma_1$ for any $\gamma_1 \in \Gamma_{1,j}$) then the $\X_j$ are symmetric (resp. $\gamma \X_j = \X_j \gamma$ for any $\gamma \in \Gamma_j$).
Consequently, Theorem \ref{thm-complementation-Fourier-ADS-amenable} applies in the case $p=1$ and $p=\infty$. If $G_2$ is in addition amenable, the result applies in the case $1<p<\infty$.
\end{prop}

\begin{proof}
It is obvious that $G$ is second countable. By Lemma \ref{Lemma-Fundamental-semidirect}, each $\Gamma_j$ is a lattice of $G$ and each $\X_j$ is an associated fundamental domain. We check that $G_1 \times G_2$ is $\ALSS$ with respect to $(\Gamma_j)$ and $(\X_j)$. Let $V$ be a neighborhood of the neutral element $e$ of $G_1 \times G_2$. Then there exist neighborhood $U_1$ of $e_1 \in G_1$ such that $U_1 \times \{e_2\} \subset V$. Since $G_1$ is $\ALSS$, there exists $j_0 \in \N$ such that $\X_{1,j} \subset U_1$ for any $j \geq j_0$. Consequently, $\X_j = \X_{1,j} \times \{e_2\} \subset U_1 \times \{e_2\} \subset V$. Thus $G_1 \times G_2$ is $\ALSS$. 

Using \cite[Proposition B.2.2 page 332]{BHV}, the existence of a lattice implies that $G$ is unimodular and $\mu_G= \mu_{G_1} \ot \mu_{G_2}$ gives a Haar measure on $G$. It remains to check \eqref{equ-Haar-measure-convergence}. To this end, consider $\gamma = (\gamma_1,\gamma_2) \in \Gamma_j$, $\omega = (\omega_1,e_{G_2}) \in \X_j$ and $s = (s_1,s_2) \in G$. Then using \eqref{Operations-semidirect} 
$$
\gamma \omega s 
=(\gamma_1,\gamma_2)\rtimes_\eta(\omega_1,e_{G_2})\rtimes_\eta(s_1,s_2)
=(\gamma_1,\gamma_2)\rtimes_\eta(\omega_1 + s_1,s_2) 
=\big(\gamma_1+\eta_{\gamma_2}(\omega_1 + s_1),\gamma_2 s_2\big).
$$ 
This element belongs to $\X_j = \X_{1,j} \times \{e_{G_2}\}$ if and only if $s_2 = \gamma_2^{-1}$ and $\gamma_1 + \eta_{\gamma_2}(\omega_1 + s_1) \in \X_{1,j}$. By the assumption $\eta_{\gamma_2}(\X_{1,j}) \subset \X_{1,j}$, the latter condition is equivalent with 
$$
\eta_{\gamma_2}^{-1}\big(\gamma_1 + \eta_{\gamma_2}(\omega_1 + s_1)\big) \in \X_{1,j},
$$
that is $\eta_{\gamma_2}^{-1}(\gamma_1) + \omega_1 + s_1 \in \X_{1,j}$. For any $\gamma = (\gamma_1,\gamma_2) \in \Gamma_j$ and $s = (s_1,s_2) \in G$, we infer that 
\begin{align*}
\MoveEqLeft
\mu_G(\X_j \cap \gamma \X_j s)
=(\mu_{G_1} \ot \mu_{G_2})\big((\X_{1,j} \times \{e_{G_2}\}) \cap \gamma \X_j s\big)\\
&=\mu_{G_1}\big(\{\omega_1 \in \X_{1,j}  : \eta_{\gamma_2}^{-1}(\gamma_1) + \omega_1 + s_1 \in \X_{1,j}\}\big).
\end{align*}
Moreover, we have $\mu_G(\X_j)=\mu_{G_1 \ot G_2}( \X_{1,j} \times \{e_{G_2}\})=\mu_{G_1}(\X_{1,j})\mu_{G_2}(\{e_{G_2}\})=\mu_{G_1}(\X_{1,j})$. Therefore, with a change of variable in the second equality and using the fact that $G_1$ satisfies \eqref{equ-Haar-measure-convergence} in the passage to the limit, we finally obtain
\begin{align*}
\MoveEqLeft
\int_G \frac{\mu_G(\X_j \cap \gamma \X_j s)^2}{\mu_G(\X_j)^3} \d\mu_G(s) 
= \int_{G_1} \frac{\mu_{G_1}(\{\omega_1 \in \X_{1,j} :\: \omega_1 + s_1 + \eta_{\gamma_2}^{-1}(\gamma_1) \in \X_{1,j} \})^2}{\mu_{G_1}(\X_{1,j})^3} \d\mu_{G_1}(s_1) \\
&=\int_{G_1} \frac{\mu_{G_1}(\{ \omega_1 \in \X_{1,j} :\: \omega_1 + s_1 \in \X_{1,j} \})^2}{\mu_{G_1}(\X_{1,j})^3} \d\mu_{G_1}(s_1) \\
& = \int_{G_1} \frac{\mu_{G_1}(\X_{1,j} \cap (\X_{1,j} + s_1))^2}{\mu_{G_1}(\X_{1,j})^3} \d\mu_{G_1}(s_1)\\
& \xra[j \to +\infty]{} c \in (0,1].
\end{align*}
The statement about the symmetry (resp. commutativity with elements of $\Gamma_j$) of the fundamental domain is easy to check with \eqref{Operations-semidirect}. If $G_2$ is amenable then $G$ is an amenable group by \cite[Proposition G.2.2 (ii)]{BHV}, being a group extension of an amenable group by an abelian (hence also amenable) group.
\end{proof}

For applying the previous result, we compute the density \eqref{equ-Haar-measure-convergence} for some abelian groups. By \cite[Corollary 4.2.6]{DeE}, the groups described in the following proposition are the compactly generated locally compact abelian groups of Lie type.

\begin{prop}
\label{Prop-constantes-Lie-type}
Suppose that $G=\Z^l \times \R^n \times \T^m \times F$ where $l,n,m \in \N$ and where $F$ is a finite abelian group. For any integer $j$, consider the lattice subgroup
\begin{equation*}
\label{Equa-a-corriger}
\Gamma_{j} 
\ov{\mathrm{def}}{=} \Z^l \times (2^{-j} \Z)^n \times \left\{ 2^{-j} r:\: r \in \{0,\ldots,2^j-1\} \right\}^m  \times F
\end{equation*}
and the associated symmetric fundamental domain
\begin{equation*}
\label{Fundamental-domains}
\X_{j} 
\ov{\mathrm{def}}{=} \{0\}^l \times [-2^{-j-1},2^{-j-1})^n \times [-2^{-j-1},2^{-j-1})^m  \times \{e_F\}.
\end{equation*}
Then the group $G$ is $\ALSS$ with respect to the sequences $(\Gamma_j)$ and $(\X_j)$. Moreover, for any $j$ and any $\gamma \in \Gamma_{j}$, we have
$$
\frac{1}{\mu_{G}(\X_{j})}\int_{G} \frac{\mu_{G}^2\left(\X_{j} \cap (\gamma + \X_{j} + s)\right)}{\mu_{G}^2(\X_{j})} \d\mu_{G}(s)
= \left(\frac23\right)^{n+m}.
$$
\end{prop}

\begin{proof}
Using Lemma \ref{Lemma-Fundamental-semidirect}, it is clear that the $\Gamma_{j}$'s are lattice subgroups and that the $\X_{j}$'s are associated fundamental domains. It is obvious that $G$ is $\ALSS$ with respect to these sequences. For any $j$, a simple computation gives 
\begin{align*}
\MoveEqLeft
\mu_{G}(\X_{j}) 
=(\mu_{\R^n} \ot \mu_{\T^m})\big(\big[-2^{-j-1},2^{-j-1}\big)^n \times \big[- 2^{-j-1}, 2^{-j-1}\big)^m\big)\\
&=\Big(\mu_{\R}\big(\big[-2^{-j-1},2^{-j-1}\big)\Big)^n \Big(\mu_{\T}\big(\big[-2^{-j-1},2^{-j-1}\big)\Big)^m
=2^{-j(n+m)}.
\end{align*}
Now, note that if $-2a \leq x \leq 2a$ then we have $\mu_\R([-a,a] \cap [-a+x,a+x])=2a-|x|$. Further, for any $j$ and any $\gamma \in \Gamma_{j}$, we have, writing $s = (x_1,\ldots, x_n,y_1,\ldots,y_m,z_1,\ldots,z_l,f)$
\begingroup
\allowdisplaybreaks
\begin{align*}
&\int_{G} \mu\big(\X_{j} \cap (\gamma + \X_{j} + s)\big)^2 \d\mu_{G}(s) 
=\int_{G} \mu_{G}\big(\X_{j} \cap (\X_{j} + s)\big)^2 \d\mu_{G}(s) \\
&=\int_{\R^n} \prod_{k=1}^n \mu_{\R} \left([-2^{-j-1},2^{-j-1}) \cap [-2^{-j-1} + x_k,2^{-j-1} + x_k)\right)^2 \d(x_1,\ldots,x_n) \times \\
& \times \int_{\T^m} \prod_{l=1}^m \mu_{\T}\left([-2^{-j-1},2^{-j-1}) \cap [-2^{-j-1} + y_l,2^{-j-1} + y_l)\right)^2 \d(y_1,\ldots,y_m) \\
& = \left(\int_{-2^{-j}}^{2^{-j}} (2^{-j} - |x|)^2 \d x \right)^n \left( \int_{-2^{-j}}^{2^{-j}} (2^{-j} - |y|)^2 \d y \right)^m 
= \left(2\int_{0}^{2^{-j}} (2^{-j}-x)^2 \d x \right)^{n+m}\\
&=\left(2\int_{0}^{2^{-j}} u^2 \d u \right)^{n+m}
=\left(\frac23\right)^{n+m} 2^{-3j(n+m)}.
\end{align*}
\endgroup
Thus
$$
\int_{G} \frac{\mu_{G}\left(\X_{j} \cap (\gamma + \X_{j} + s)\right)^2}{\mu_{G}(\X_{j})^3} \d\mu_{G}(s) 
=2^{3j(n+m)} \cdot \left(\frac23\right)^{n+m} 2^{-3j(n+m)} 
=\left(\frac23\right)^{n+m} \in (0,1].
$$
\end{proof}

\begin{remark}\normalfont
\label{remark-cor-complementation-Fourier-ADS-amenable}
The assumptions of Proposition \ref{Prop-complementation-Fourier-ADS-amenable} are satisfied in the following situation. Assume that $G_1= \Z^l \times \R^n \times \T^m  \times F$ where $l,n,m \in \N$ and where $F$ is a finite abelian group. Let $G_2$ be a subgroup of $\mathrm{Sym}(n) \times \mathrm{Sym}(m)$ where $\mathrm{Sym}(n)$ and $\mathrm{Sym}(m)$ are the permutation groups of $n$ and $m$ elements. For $(\sigma_1,\sigma_2) \in G_2$, let further 
\begin{align}
\MoveEqLeft
 \label{equ-remark-cor-complementation-Fourier-ADS-amenable}
\eta_{(\sigma_1,\sigma_2)}(z_1,\ldots,z_l,x_1,\ldots,x_n,y_1,\ldots,y_m,f) \\
&= \big(z_1,\ldots,z_l,x_{\sigma_1(1)},\ldots,x_{\sigma_1(n)},y_{\sigma_2(1)},\ldots,y_{\sigma_2(m)},f\big). \nonumber
\end{align}
For any integer $j$, consider the lattice $\Gamma_{1,j}=\Z^l \times  (2^{-j} \Z)^n \times \left\{ 2^{-j} r:\: r \in \{0,\ldots,2^j-1\} \right\}^m  \times F$ of $G_1$ and the symmetric fundamental domain $\X_{1,j}=\{0\}^l \times  [-2^{-j-1},2^{-j-1})^n \times [-2^{-j-1},2^{-j-1})^m  \times \{e_F\}$. It is easy to check that the transformation \eqref{equ-remark-cor-complementation-Fourier-ADS-amenable} preserves both $\Gamma_{1,j}$ and $\X_{1,j}$. Then $G_1$, $G_2$, $(\Gamma_{1,j})$, $(\X_{1,j})$ and $\eta$ satisfy all the assumptions of Proposition \ref{Prop-complementation-Fourier-ADS-amenable} and consequently Theorem \ref{thm-complementation-Fourier-ADS-amenable} applies to the group $G = G_1 \rtimes_\eta G_2$.

More generally, $G_2$ can be any countable discrete (amenable) group such that $\eta_t$ is given by a coordinate permutation as in \eqref{equ-remark-cor-complementation-Fourier-ADS-amenable} for any $s \in G_2$.
\end{remark}

Now, we give a natural semidirect product for which we can apply Proposition \ref{Prop-complementation-Fourier-ADS-amenable} and \ref{Prop-constantes-Lie-type}. Let $\mathbb{H}_n=\R^{2n+1}$ be the (continuous) Heisenberg group with group operations
\begin{equation}
\label{Operations-Heisenberg}
(a,b,t)\cdot(a',b',t') 
=(a+a',b+b',t+t'+a\cdot b')
\quad \text{and} \quad
(a,b,t)^{-1}=(-a,-b,-t+a\cdot b)	
\end{equation}
where $a,b,a',b' \in \R^n$ and $t,t' \in \R$ and where $\cdot$ denotes the canonical scalar product on $\R^{n}$. Recall that $\mathbb{H}_n$ is unimodular and the Haar measure on $\mathbb{H}_n$ is just usual Lebesgue measure on $\R^n$.
We can use our results with the semi-discrete Heisenberg group described in the following result, see \cite[page 1459]{Pal2} for more information on this group.

\begin{prop}
\label{prop-Heisenberg-semidiscrete}
Let $\mathrm{H}_n= \{(x,y,t) \in \mathbb{H}_n :\: x,y \in \Z^n, t \in \R\}$ be the (amenable) closed subgroup of the Heisenberg group $\mathbb{H}_n$. For any integer $j$, we consider the lattice subgroup $\Gamma_j = \Z^n \times \Z^n \times 2^{-j} \Z$ of $\mathrm{H}_n$ and the associated symmetric fundamental domain $\X_j = \{0\} \times \{0\} \times [-2^{-j-1},2^{-j-1})$. Then $\mathrm{H}_n$ is $\ALSS$ with respect to the increasing sequence $(\Gamma_j)$ and to the sequence $(\X_j)$. Moreover, for any $j$ and any $\gamma \in \Gamma_{j}$, we have
\begin{equation}
\label{equ-1-prop-Heisenberg-semidiscrete}
\frac{1}{\mu(\X_j)} \int_{\mathrm{H}_n} \frac{\mu^2(\X_j \cap \gamma \X_j s)}{\mu^2(\X_j)} \d\mu(s) 
= \frac23.
\end{equation}
In particular, Theorem \ref{thm-complementation-Fourier-ADS-amenable} applies.
\end{prop}

\begin{proof}
Using \eqref{Operations-Heisenberg}, it is easy to see that $\mathrm{H}_n$ is a closed subgroup of $\mathbb{H}_n$, so it is locally compact. If $G_1=\{(0,b,t): b \in \Z^n, t \in \R\}$ and $G_2=\{ (a,0,0): a \in \Z^n\}$, it is not difficult to check by using again \eqref{Operations-Heisenberg} that $G_1$ and $G_2$ are closed subgroups of $\mathrm{H}_n$, $\mathrm{H}_n=G_1G_2$, $G_1 \cap G_2=\{(0,0,0)\}$ and that $G_1$ is normal in $\mathrm{H}_n$. By \cite[Proposition page 184]{FeD1}, we deduce an isomorphism $\mathrm{H}_n=G_1 \rtimes_\eta G_2$ of topological groups where 
\begin{equation}
\label{action-eta}
\eta_{(a,0,0)}(0,b,t)
=(0,b,t+b \cdot a), \quad a,b \in \Z^n, t \in \R.	
\end{equation}
Note that $G_1$ is isomorphic to $\Z^n \times \R$ and that $G_2$ is isomorphic to $\Z^n$. For any $j$, we consider $\Gamma_{1,j}=\Z^n \times 2^{-j} \Z$ and $\X_{j,1}=\{0\}^n \times [-2^{-j-1},2^{-j-1})$. For any $(a,0,0) \in G_2$ and any integer $j$, using \eqref{action-eta}, we see that $\eta_{(a,0,0)}(\Gamma_{1,j}) \subset \Gamma_{1,j}$ and $\eta_{(a,0,0)}(\X_{1,j}) \subset \X_{1,j}$. By Proposition \ref{Prop-complementation-Fourier-ADS-amenable}, we deduce that $\Gamma_j$ is a lattice subgroup of $\mathrm{H}_n$, that $\X_j$ is an associated fundamental domain and that the group $\mathrm{H}_n$ is $\ALSS$ with respect to the sequences $(\Gamma_j)$ and $(\X_j)$. Finally the equality \eqref{equ-1-prop-Heisenberg-semidiscrete} is a consequence of Proposition \ref{Prop-constantes-Lie-type} and Proposition \ref{Prop-complementation-Fourier-ADS-amenable}.
\end{proof}

We finish by bringing to light a bad behaviour of \eqref{equ-Haar-measure-convergence} with respect to the Heisenberg group $\mathbb{H}_3$.

\begin{prop}
\label{prop-Heisenberg-constant-c}
For any integer $j$, we consider the lattice subgroup $\Gamma_j = 2^{-j} \Z \times 2^{-j} \Z \times 2^{-2j} \Z$ of the Heisenberg group $\mathbb{H}_3$ and the associated fundamental domain $\X_j=[-2^{-j-1},2^{-j-1}) \times [-2^{-j-1},2^{-j-1}) \times [-2^{-2j-1},2^{-2j-1})$. Then the Heisenberg group $\mathbb{H}_3$ is $\ALSS$ with respect to the increasing sequence $(\Gamma_j)$ and to the sequence $(\X_j)$. Moreover, for every fixed $\gamma = (\gamma_1,\gamma_2,\gamma_3) \in \Gamma_{j_0}$ for some $j_0 \in \N$ with $(\gamma_1,\gamma_2) \not= (0,0)$ and $\gamma_1 \cdot \gamma_2 = 0$ we have
\begin{equation}
\label{equ-1-prop-Heisenberg}
\lim_{j \to +\infty} \frac{1}{\mu(\X_j)} \int_{\mathbb{H}_3} \frac{\mu^2(\X_j \cap \gamma \X_j s)}{\mu^2(\X_j)} \d\mu(s) 
=0. 
\end{equation}
In particular, for this choice of group, and sequences of lattices and fundamental domains, Theorem \ref{thm-complementation-Fourier-ADS-amenable} is not applicable.
\end{prop}

\begin{proof}
Note that it is obvious that $\mathbb{H}_3$ is $\ALSS$ with respect to the sequences $(\Gamma_j)$ and $(\X_j)$. First observe that for any $s \in \mathbb{H}_3$ and any integer $j$ we have
$$
\mu(\X_j \cap  \gamma\X_j s)
=\int_{\mathbb{H}_3} 1_{\X_j \cap  \gamma\X_j s}(t) \d \mu(t)
=\int_{\mathbb{H}_3} 1_{\X_j}(t) 1_{\gamma\X_j s}(t) \d \mu(t)
=\int_{\mathbb{H}_3} 1_{\X_j}(t) 1_{\X_j}(\gamma^{-1} t s^{-1}) \d \mu(t).
$$
For any $s \in \mathbb{H}_3$, any $j$ and any $\gamma \in \Gamma_j$, we have using the invariance of the Haar measure in the third equality (to use $u=\gamma^{-1} ts^{-1}$)
\begin{align}
\MoveEqLeft
\label{equ2-Heisenberg} 
\frac{1}{\mu(\X_j)} \int_{\mathbb{H}_3} \frac{\mu(\X_j \cap \gamma \X_j s)}{\mu^2(\X_j)} \d\mu(s)
=\frac{1}{\mu(\X_j)^3} \int_{\mathbb{H}_3} \mu(\X_j \cap \gamma \X_j s)\mu(\X_j \cap \gamma \X_j s) \d\mu(s)\\
&=\frac{1}{\mu(\X_j)^3} \int_{\mathbb{H}_3} \int_{\mathbb{H}_3} \int_{\mathbb{H}_3} 1_{\X_j}(t)1_{\X_j}(r) 1_{\X_j}(\gamma^{-1} ts^{-1})1_{\X_j}(\gamma^{-1} rs^{-1})  \d\mu(r) \d\mu(t)\d\mu(s) \nonumber\\
&=\frac{1}{\mu(\X_j)^3} \int_{\mathbb{H}_3} \int_{\mathbb{H}_3} \int_{\mathbb{H}_3} 1_{\X_j}(t)1_{\X_j}(r) 1_{\X_j}(u)1_{\X_j}(\gamma^{-1} rt^{-1}\gamma u)  \d\mu(r) \d\mu(t)\d\mu(u)\nonumber\\
&=\frac{1}{\mu(\X_j)^3} \int_{\X_j} \int_{\X_j} \int_{\X_j} 1_{\X_j}(\gamma^{-1} rt^{-1}\gamma u) \d\mu(r)\d\mu(u)\d\mu(t)\nonumber \\
&=\frac{1}{\mu(\X_j)^{3}} \int_{\R^3} \int_{\R^3} \int_{\R^3} 1_{|r_1|,|u_1|,|t_1|, |r_2|,|u_2|,|t_2| 
\leq 2^{-j-1}} 1_{|r_3|,|u_3|,|t_3| 
\leq 2^{-2j-1}} 1_{\X_j}(\gamma^{-1} rt^{-1}\gamma u) \\
&\qquad \qquad \qquad \qquad\qquad \qquad \qquad \qquad\qquad \qquad \qquad \qquad\qquad \qquad \qquad \qquad\d r \d u \d t. \nonumber
\end{align} 
If $\gamma=(\gamma_1,\gamma_2,\gamma_3) \in \Gamma_j$ and if $r,u,t \in \X_j$, by \eqref{Operations-Heisenberg}, a tedious yet elementary calculation yields
\begin{equation}
\label{equ-3-prop-Heisenberg}
\gamma^{-1} rt^{-1}\gamma u
=(u_1 + r_1 - t_1,u_2+r_2-t_2,u_3 + r_3 - t_3 -\gamma_1 r_2 + t_1t_2-t_1\gamma_2-t_1u_2 +r_1 \gamma_2 +r_1 u_2 -r_1t_2 +\gamma_1t_2 ). 
\end{equation}
We estimate from above. The last indicator function in the previous triple integral can be majorized by $1_{|(\gamma^{-1} rt^{-1}\gamma u)_3| \leq 2^{-2j-1}}$. If $|(\gamma^{-1} rt^{-1}\gamma u)_3| \leq 2^{-2j-1}$ and $r,u,t \in \X_j$, then by triangle inequality and \eqref{equ-3-prop-Heisenberg}, we have
\begin{align*}
\MoveEqLeft
 |-\gamma_1r_2 - t_1 \gamma_2 + r_1 \gamma_2 +\gamma_1 t_2| 
    \leq \left|(\gamma^{-1} rt^{-1}\gamma u)_3\right| +|u_3 + r_3 - t_3 + t_1t_2-t_1u_2 +r_1 u_2 -r_1t_2|\\
		&\leq 2^{-2j-1}\bigg(1+1+1+1+2+2+2+2\bigg)
		=6 \cdot 2^{-2j}.
\end{align*}  
Using the equality $\frac{1}{\mu(\X_j)^{3}}=2^{12j}$, this says that \eqref{equ2-Heisenberg} is less than  
\begin{align*}
\MoveEqLeft
2^{12j} \int_{\R^3} \int_{\R^3} \int_{\R^3} 1_{|r_1|,|u_1|,|t_1|,|r_2|,|u_2|,|t_2| 
\leq 2^{-j-1}} 1_{|r_3|,|u_3|,|t_3| 
\leq 2^{-2j-1}} 
\\
&\qquad \qquad \qquad \qquad\qquad \qquad \qquad \qquad 1_{|-\gamma_1r_2 - t_1 \gamma_2 + r_1 \gamma_2 +\gamma_1 t_2| \leq 6 \cdot 2^{-2j}} 
\d r \d u \d t.
\end{align*}
We cheaply integrate over $u_1,u_2,r_3,u_3$ and $t_3$ and obtain
$$
= 2^{12 j} 2^{-8j} \int_{\R^4} 1_{|r_1|,|t_1|,|r_2|,|t_2| 
\leq 2^{-j-1} } 1_{|-\gamma_1r_2 - t_1 \gamma_2 + r_1 \gamma_2 +\gamma_1 t_2| 
\leq 6 \cdot 2^{-2j} } \d r_1 \d r_2 \d t_1 \d t_2.
$$
Now suppose first that $\gamma_2 = 0$ and $\gamma_1 \neq 0$. Then the last indicator function  can be simplified and we can cheaply integrate over $r_1$ and $t_1$ to estimate further
\begin{align*}
\MoveEqLeft
\leq 2^{4j} 2^{-2j} \int_{\R^2} 1_{|r_2|,|t_2| \leq 2^{-j-1} } 1_{|-r_2 + t_2| 
\leq \frac{1}{|\gamma_1|} 6 \cdot 2^{-2j}} \d r_2 \d t_2 \\
&=2^{2j} \int_{-2^{-j-1}}^{2^{-j-1}} \int_{-2^{-j-1}}^{2^{-j-1}}  1_{|t_2- r_2| \leq \frac{1}{|\gamma_1|}6 \cdot 2^{-2j}} \d r_2 \d t_2 
=\frac{1}{4}\int_{-1}^1 \int_{-1}^{1} 1_{|t_2' - r_2'| \leq \frac{1}{|\gamma_1|}12\cdot  2^{-j}} \d r_2' \d t_2'.
\end{align*}
where we have performed the change of variables $r_2' = 2^{j+1} r_2$, $t_2' = 2^{j+1} t_2$. Now the last double integral is easily seen to converge to $0$ as $j \to \infty$. The case $\gamma_1 = 0$ and $\gamma_2 \neq 0$ can be treated in the same way by symmetry.
\end{proof}

\subsection{Pro-discrete groups}
\label{subsec-pro-discrete-groups}

An inverse system of topological groups indexed by a directed set $I$ consists of a family $(G_j)_{j \in I}$ of topological groups and a family $(\psi_{ij} \co G_j \to G_i)_{i,j \in I, j \geq i}$ of continuous homomorphisms such that $\psi_{ii}=\Id_{G_i}$ and $\psi_{ij} \psi_{jk}=\psi_{ik}$ whenever $k \geq j \geq i$ \cite[Definition 1.1.1]{Wil}. An inverse system is called a surjective inverse system if each map $\psi_{ij}$ is surjective. Now let $(G_j, \psi_{ij})$ be an inverse system of topological groups and let $G$ be a topological group. We shall call a family of continuous homomorphisms $\psi_j \co G \to G_j$ compatible with the inverse system if $\psi_{ij}\psi_j=\psi_i$ whenever $j \geq i$. An inverse limit of an inverse system ($G_j$, $\psi_{ij}$) of topological groups is a topological group $G$ together with a compatible family $\psi_j \co G \to G_j$ of continuous homomorphisms with the following universal property: whenever $\psi_j' \co G' \to G_j$ is a compatible family of continuous homomorphisms from a topological group $G'$, there exists a unique continuous homomorphism $\varphi \co G' \to G$ such that $\psi_j\varphi=\psi_j'$ for each $j$. Each inverse system admits an inverse limit, given by the following construction \cite[Proposition 1.1.4]{Wil}:
\begin{equation}
\label{Def-profinite}
\varprojlim G_j
=\bigg\{s \in \prod_{j \in I} G_j: p_i(s)=\psi_{ij}(p_j(s)) \text{ for all } i \leq j\bigg\}
\end{equation}
with the subspace topology from the product topology and with projection maps $\psi_j$ given by the restrictions to $\varprojlim G_j$ of the projection maps $p_i \co \prod_{j \in I} G_j \to G_i$ from the product. 

We say that a topological group $G$ is pro-discrete if it is isomorphic to the inverse limit of an inverse system of discrete groups. We have the following characterization for locally compact groups which is a variation of \cite[Lemma 1.3]{Sau}. For the sake of completeness, we give a complete proof.

\begin{prop}
\label{prop-pro-discrete}
A locally compact group $G$ is pro-discrete if and only if it admits a basis $(\X_j)$ of neighborhoods of the identity $e_G$ consisting of open compact normal subgroups. In this case, we have $G=\varprojlim G_j$ where the inverse system is given by the groups $G_j=G/\X_j$ and by the homomorphisms $\psi_{ij} \co G_j \to G_i$, $s\X_j \mapsto s\X_i$ for $j \geq i$ and where the preorder is the opposite of inclusion\footnote{\thefootnote. We let $j \geq i$ if and only if $\X_j \subset \X_i$.} of the $\X_j$'s. Moreover, if $G$ is first countable then there exists a countable basis of open compact normal subgroups.
Finally, a pro-discrete locally compact group $G$ is always totally disconnected.
\end{prop}

\begin{proof}
Suppose that $G$ admits a family $(\X_j)$ of open compact normal subgroups forming a neighborhood basis of $e_G$. For any $j \in I$, we set $G_j \overset{\textrm{def}}= G/\X_j$, which is discrete by \cite[Theorem 5.21]{HR} since $\X_j$ is open. We use the preorder defined in the statement of the result. For $j \geq i$, i.e. $\X_j \subset \X_i$, we also consider the well-defined homomorphism $\psi_{ij} \co G_j \to G_i$, $s\X_j \mapsto s\X_i$. It is plain to check\footnote{\thefootnote. If $k \geq j \geq i$ we have $\psi_{ij} \psi_{jk}(s\X_k)=\psi_{ij}(s\X_j)=s\X_i=\psi_{ik}(s\X_k)$.} that the $(\psi_{ij})_{j \geq i}$ is an inverse system. We consider the construction \eqref{Def-profinite} of the inverse limit $\varprojlim G_j$. Note that the family of continuous homomorphisms $\psi_j' \co G \to G/\X_j$, $s \mapsto s \X_j$ is compatible\footnote{\thefootnote. If $j \geq i$ we have $\psi_{ij}\psi_j'(s)=\psi_{ij}(s\X_j)=s\X_i=\psi_i'(s)$.}. According to the universal property, there exists a continuous homomorphism $\varphi \co G \to \varprojlim G_j$ satisfying the compatibility $\psi_j' = \psi_j \varphi$. For any $s \in G$, this means that
$$
s \X_j 
=\psi_j'(s) 
=\psi_j (\varphi(s)) 
=p_j(\varphi(s)),
$$ 
so that $\varphi(s)$ is equal to the element $(s \X_j)_{j \in  I}$ of the product $\prod_{j \in I} G/\X_j$.

It remains to check that $\varphi$ is bijective. For the injectivity, suppose that $\varphi(s) = e$, so $s\X_j = \X_j$ for all $j$. Thus, $s \in \X_j$ for all $j$. Since $G$ is Hausdorff and since the $\X_j$'s form a basis of neighborhoods, we obtain $s = e_G$. For the surjectivity, let $t=(s_j \X_j)_{j \in I}$ be an element of $\varprojlim G_j$. Let $F$ be a finite subset of $I$. Consider some $i \in I$ such that $i \geq j$ for any $j \in F$. For $j \in F$, we have
$$
s_j \X_j 
=p_j(t)
=\psi_{ji}(p_i(t))
=\psi_{ji}(s_i \X_i)
=s_i \X_j
$$
so $s_i \in s_j \X_j$. Hence $s_i$ belongs to $\cap_{j \in F} s_j \X_j$. We infer that the collection of the compact subsets $s_j \X_j$ has the finite intersection property. We conclude that there exists $s \in \cap_{j \in I} s_j \X_j$. Consequently $\varphi(s)= (s \X_j)_{j \in I}=(s_j \X_j)_{j \in I}=t$. We conclude that $G \cong \varprojlim G_j$.

Assume now that $G$ is an inverse limit $\varprojlim G_j$ of discrete groups $G_j$. Again, we use the description \eqref{Def-profinite}. Since each $\psi_j$ is continuous, each kernel $\ker \psi_j = \psi_j^{-1}(\{e_j\})$ is the preimage of an open set, hence open in $G$. We also know that $\ker \psi_j$ is normal and closed as a kernel of a continuous homomorphism. It only remains to check that the $\ker \psi_j$'s form a neighborhood basis of the identity $e_G$. Indeed since $\ker \psi_j$ will fall within any given compact neighborhood of $e_G$ for big enough $j$, $\ker \psi_j$ will also be compact for such $j$.

Let $U$ be any neighborhood of $e_G$ in $G$. Then by trace topology, there exists a neighborhood $\tilde{U}$ of $e_G$ in $\prod_{j \in I} G_j$ with $U = \tilde{U} \cap G$. By the definition of the product topology, there exists some finite subset $F$ of $I$ such that the subset $\tilde{V} = \prod_{j \in I} A_j$ of $\prod_{j \in I} G_j$ satisfies $\tilde{V} \subset \tilde{U}$ with $A_j=\{e_j\}$ if $j \in F$ and $A_j=G_j$ if $j \not\in F$. Since $I$ is directed, we can choose $i \in I$ such that $i \geq j$ for any $j \in F$. Then for any $s \in \ker \psi_i$ and any $j \in F$, we have 
$$
p_j(s)
= \psi_{ji}(p_i(s))
=\psi_{ji}(\psi_i(s))
= \psi_{ji}(e_i) 
=e_j.
$$ 
Hence $\ker \psi_j \subset \tilde{V}$. Consequently, we have $\ker \psi_j \subset \tilde{V} \cap G \subset U$. We have shown that the $\ker \psi_j$'s form a neighborhood basis of the identity. 

If $G$ is first countable, there exists a countable neighborhood basis of $e_G$, so we can also extract a sequence of the $\ker \psi_j$ forming a neighborhood basis of $e_G$.

We turn to the last claim. Recall that the intersection of all open subgroups of a locally compact group is the connected component of the identity $e_G$ by \cite[Theorem 7.8]{HR}. Since $G$ is Hausdorff, the intersection of closed neighborhoods of $e_G$ is $\{e_G\}$. Since an open subgroup is always closed \cite[Theorem 5.5]{HR}, we infer that the component of the identity is equal to $\{e_G\}$. By \cite[Theorem 7.3]{HR}, we conclude that $G$ is totally disconnected.
\end{proof}

In particular, by \cite[Proposition 3 page 20]{Bou1}, a pro-discrete locally compact group $G$ is unimodular. 

\begin{remark} \normalfont
\label{Rem-total-disc}
Note that a locally compact group $G$ is totally disconnected if and only if  the compact open subgroups form a basis of neighborhoods of the identity $e_G$. The end of the proof of Proposition \ref{prop-pro-discrete} proves the more general implication $\Leftarrow$. The converse is \cite[Theorem 7.7]{HR}.
\end{remark}

There is the following variant of Theorem \ref{thm-complementation-Fourier-ADS-amenable}.

\begin{thm}
\label{thm-complementation-quotient}
Let $G=\varprojlim G_j$ be a second countable pro-discrete locally compact group with respect to an inverse system indexed by $\N$. Suppose $1 \leq p \leq \infty$. Assume that $G$ is amenable if $1 <p< \infty$. Then there exists a contractive map
$$
P_G^p \co \CB(\L^p(\VN(G))) \to \mathfrak{M}^{p,\cb}(G)
$$
with the properties:
\begin{enumerate}
\item If $T$ is completely positive, then $P_G^p(T)$ is also completely positive.
\item If $T = M_\psi$ is a Fourier multiplier on $\L^p(\VN(G))$ with bounded measurable symbol $\psi \co G \to \C$ then $P_G^p(M_\psi) = M_\psi$.
\end{enumerate}
Moreover, $P_G^p$ has the following compatibility: if $T \in \CB(\L^p(\VN(G))) \cap \CB(\L^q(\VN(G)))$ for some $1 \leq p,q \leq \infty$, then $P_G^p(T)$ being twice defined as an element of $\mathfrak{M}^{p,\cb}(G)$ and $\mathfrak{M}^{q,\cb}(G)$ coincides on $\L^p(\VN(G)) \cap \L^q(\VN(G))$. Note that in the case $p = \infty$, we can take $\CB_{\w^*}(\VN(G))$ as the domain space of $P_G^\infty$.
\end{thm}

\begin{proof}
Let $G=\varprojlim G_j$ be a second countable pro-discrete locally compact group. By Proposition \ref{prop-pro-discrete}, $G$ admits a (countable) basis $(\X_j)$ of neighborhoods of the identity $e_G$ consisting of open compact normal subgroups. By \eqref{Iso-averaging-projection}, we have an isomorphism from the group von Neumann algebra $\VN(G/\X_j)$ onto $p_{\X_j}\VN(G)$. Using Lemma \ref{Lemma-trace-preserving-averaging}, we obtain a completely positive and completely contractive map $\L^p(\VN(G/\X_j)) \to \L^p(p_{\X_j}\VN(G))= p_{\X_j}\L^p(\VN(G))$. By composing this map with the identification $\L^p(p_{\X_j}\VN(G)) \subset \L^p(\VN(G))$, we obtain a (normal if $p=\infty$) completely positive and completely contractive map
$$
\Phi_{j}^{p} \co \L^p(\VN(G/\X_j)) \to \L^p(\VN(G)), \lambda_{G/\X_j,s \X_j} \mapsto \mu_G(\X_j)^{\frac1p} p_{\X_j}\lambda_{G,s}.
$$ 
Furthermore, we consider the adjoint (preadjoint if $p=1)$ $\Psi_{j}^{p}= \big(\Phi_{j}^{p^*}\big)^* \co \L^p(\VN(G)) \to \L^p(\VN(G/\X_j))$ of $\Phi_{j}^{p^*}$ which is also (normal if $p=\infty$) completely contractive and completely positive by Lemma \ref{Lemma-adjoint-cp} for any $1 \leq p \leq \infty$.

Let $T \co \L^p(\VN(G)) \to \L^p(\VN(G))$ be some completely bounded map. Now, using Theorem \ref{Prop-complementation-Schur-Fourier} for the discrete group $G/\X_j$ (since $\X_j$ is open; note that if $p \not= \infty$, $G/\X_j$ is amenable by \cite[Proposition G.2.2]{BHV}), we define the completely bounded Fourier multiplier
$$
M_{\varphi_j}
=P_{G/\X_j}^p\big(\Psi_{j}^{p} T \Phi_{j}^{p}\big) 
\co \L^p(\VN(G/\X_j)) \to \L^p(\VN(G/\X_j))
$$
if $1 \leq p<\infty$ and $M_{\varphi_j}
=P_{G/\X_j}^\infty\big(\Psi_{j}^{\infty} P_{\w^*}(T) \Phi_{j}^{\infty}\big) 
\co \VN(G/\X_j) \to \VN(G/\X_j)$ if $p = \infty$, where the contractive map $P_{\w^*} \co \CB(\VN(G)) \to \CB(\VN(G))$ is described in Proposition \ref{Prop-recover-weak-star-continuity}. Note that $\varphi_j \co G/\X_j \to \C$ is defined by $\varphi_j(s/\X_j)=\tau_{G/\X_j}\big(\Psi_{j}^{p} T \Phi_{j}^{p} (\lambda_{s\X_j}) \lambda_{s^{-1}\X_j}\big)$ (if $T$ is normal in the case $p=\infty$). Then 
\begin{align*}
\MoveEqLeft
 \bnorm{M_{\varphi_j}}_{\cb,\L^p(\VN(G/\X_j)) \to \L^p(\VN(G/\X_j))}
=\bnorm{P_{G/\X_j}^p\big(\Psi_{j}^{p} T \Phi_{j}^{p}\big)}_{\cb,\L^p(\VN(G/\X_j)) \to \L^p(\VN(G/\X_j))}\\
&\leq \bnorm{\Psi_{j}^{p} T \Phi_{j}^{p}}_{\cb,\L^p(\VN(G/\X_j)) \to \L^p(\VN(G/\X_j))}
\leq \norm{T}_{\cb,\L^p(\VN(G)) \to \L^p(\VN(G))},
\end{align*}
in the case $1\leq p<\infty$ and similarly in the case $p=\infty$. Note that each function $\varphi_j$ is continuous since $G/\X_j$ is discrete. Now, we define the continuous complex function $\widetilde{\varphi}_j=\varphi_j \circ \pi_j \co G \to \C$ where $\pi_j \co G \to G/\X_j$ is the canonical surjective map. Since the homomorphism $\pi_j$ is continuous, according to Proposition \ref{prop-homomorphism-Fourier-multipliers-cb}, the symbol $\widetilde{\varphi}_j$ induces a completely bounded Fourier multiplier on $\L^p(\VN(G))$ and we have the estimate
\begin{align*}
\bnorm{M_{\widetilde{\varphi}_j}}_{\cb,\L^p(\VN(G)) \to \L^p(\VN(G))} 
&= \bnorm{M_{\varphi_j}}_{\cb,\L^p(\VN(G/\X_j)) \to \L^p(\VN(G/\X_j))}\\
&\leq \norm{T}_{\cb,\L^p(\VN(G)) \to \L^p(\VN(G))}. \nonumber  
\end{align*}

Now, we suppose that $T=M_{\psi}$ for a (bounded) measurable symbol $\psi \co G \to \C$ giving rise to a completely bounded $\L^p$ Fourier multiplier. We start by giving a description of the symbol $\widetilde{\varphi}_j$ as an average of $\psi$.

\begin{lemma}
\label{Lemma-equ-symbol-profinite}
For any $s \in G$, we have
\begin{equation}
\label{average-prodiscret}
\widetilde{\varphi}_j(s) 
=\int_{\X_j} \psi(st) \d\mu_{\X_j}(t).	
\end{equation}
\end{lemma}

\begin{proof}
The subgroup $\X_j$ is open, so $\mu_G|_{\X_j}$ is a left Haar measure on $\X_j$ and $\mu_{\X_j} = c_j \mu_G|_{\X_j}$ where $c_j=\frac{1}{\mu_G(\X_j)}$. Moreover, for any $s \in G$, the indicator function $1_{s\X_j}$ belongs to $\mathrm{C}_c(G)$ since $s\X_j$ is an open and compact subset of $G$. For any $s,t \in G$, note that 
\begin{align*}
\MoveEqLeft
(1_{s\X_j}*\check{1}_{\X_j})(t)
=\int_G 1_{s\X_j}(r)\check{1}_{X_j}(r^{-1}t) \d \mu_G(r)
=\int_{sX_j} 1_{X_j}(t^{-1}r) \d \mu_G(r)\\
&=\mu_G(s\X_j \cap t\X_j)
=\mu_G(sX_j)1_{s\X_j}(t).
\end{align*} 
We conclude that $1_{s\X_j} \in \mathrm{C}_c(G)*\mathrm{C}_c(G)$. Then, for any $s \in G$, using the definition of a Fourier multiplier and Lemma \ref{Lemma-trace-preserving-averaging}, we see that 
$$
M_\psi\big(\lambda_{G,s} p_{\X_j}\big) 
=M_\psi\big(\lambda_{G,s} c_j\lambda_G(1_{X_j})\big)
=c_j M_\psi\big(\lambda_G(1_{s\X_j}) \big)\\
=c_j \lambda_G(\psi 1_{s\X_j})
$$
and similarly
$$
\lambda_{G,s^{-1}} p_{\X_j}
=c_j\lambda_{G,s^{-1}}\lambda_G(1_{\X_j})
=c_j\lambda(1_{s^{-1}\X_j}).  
$$
For any $s \in G$, using the Plancherel Formula \eqref{Formule-Plancherel}, we obtain
\begin{align*}
\MoveEqLeft
\widetilde{\varphi}_j(s) 
=\varphi_j \circ \pi_j(s) 
=\tau_{G/\X_j}\left(\Psi_{j}^{p} M_\psi \Phi_{j}^{p}\big(\lambda_{G/\X_j,s \X_j}\big) \lambda_{G/\X_j,s^{-1} \X_j} \right) \\
&=\tau_{G} \left(M_\psi \Phi_{j}^{p}\big(\lambda_{G/\X_j,s \X_j}\big)\Phi_{j}^{p^*}(\lambda_{G/\X_j,s^{-1} \X_j}) \right) \\
&=\mu_G(\X_j)^{\frac1p} \mu_G(\X_j)^{1-\frac1p} \tau_{G}\left(M_\psi\big(\lambda_{G,s} p_{\X_j}\big) \lambda_{G,s^{-1}} p_{\X_j}\right) 
=c_j\tau_{G}\left(\lambda_{G}(\psi 1_{s\X_j}) \lambda_{G}\big(1_{s^{-1}\X_j}\big)\right).
\end{align*}
Now, using the normality of the subgroup $\X_j$, we see that
\begin{align*}
\MoveEqLeft
\widetilde{\varphi}_j(s) 
		=c_j\int_G \psi(r) 1_{s\X_j}(r)1_{s^{-1}\X_j}(r^{-1}) \d\mu_G(r)
		=c_j\int_{s\X_j} \psi(r) 1_{\X_j s}(r) \d\mu_G(r)\\
		&=c_j\int_{s\X_j} \psi(r) \d\mu_G(r)
		=c_j\int_{\X_j} \psi(st) \d\mu_{G}(t)
		=\int_{\X_j} \psi(st) \d\mu_{\X_j}(t).
\end{align*}
\end{proof} 

Let $\E_j^\infty \co \L^\infty(G) \to \L^\infty(G)$ be the normal conditional expectation associated with the $\sigma$-algebra generated by the left cosets of $\X_j$ in $G$ considered in \cite[page 182-183]{JeR} (see also \cite[page 69]{IoT1}) and $\E_j^1 \co \L^1(G) \to \L^1(G)$ the contractive associated map.  The previous lemma says that for any integer $j$ we have $\widetilde{\varphi}_j=\E_j^\infty(\psi)$. Now, we prove the following convergence result.

\begin{lemma}
\label{lem-prodiscrete-weak-star}
Let $G$ be a pro-discrete locally compact group and let $(\X_j)$ be a decreasing basis of neighborhoods of the identity $e_G$ consisting of open compact normal subgroups. Let $\E_j^\infty \co \L^\infty(G) \to \L^\infty(G)$ be the normal conditional expectation associated with the $\sigma$-algebra generated by the left cosets of $\X_j$ in $G$. For any $\psi \in \L^\infty(G)$, the net $(\E_j^\infty(\psi))$ converges to $\psi$ for the weak* topology of $\L^\infty(G)$.
\end{lemma}

\begin{proof}
By \cite[Proposition 2.6.32]{HvNVW}, for any $f \in \L^\infty(G)$ and any $g \in \L^1(G)$, we have $\big\langle \E_j^\infty(f), g \big\rangle_{\L^\infty(G),\L^1(G)}=\big\langle f, \E_j^1(g) \big\rangle_{\L^\infty(G),\L^1(G)}$. 
Consequently, the map $\E_j^\infty \co \L^\infty(G) \to \L^\infty(G)$  admits as preadjoint the contractive map $\E_j^1 \co \L^1(G) \to \L^1(G)$. So it suffices to show that the net $(\E_j^1)$ converges to the identity for the weak operator topology. Actually, we will show that the convergence is true\footnote{\thefootnote. This fact is proved in the second countable case in \cite[Theorem 3.3]{JeR} and seesms alluded without proof in the general case in \cite[page 184]{JeR} (see also \cite[page 71]{IoT1} for a proof). Here, we give an alternative argument. Finally, Bourbaki transformed this into an exercise \cite[Exercise 10 page 89]{Bou3}, as usual without giving any reference.} for the norm topology of $\L^1(G)$. Since the net $(\E_j^1)$ is uniformly bounded, by \cite[Proposition 5, Chapt. III, 17.4]{Bou5}, it suffices to show that $\E_j^1(g)$ converges to $g$ in $\L^1(G)$ for any $g$ belonging to some total subset of $\L^1(G)$. By \cite[Lemma 2 a), VII.15]{Bou3}, the subset of positive functions with compact support constant on the left cosets of some $\X_j$ is total. So let $g$ be such a function. If $i \geq j$, i.e. if $\X_i \subset \X_j$, each left coset of $\X_i$ in $G$ is a subset of a left coset of $\X_j$ in $G$. Then for almost all $s \in G$ we have
$$
\big(\E_i^1(g)\big)(s)
=\int_{\X_i} g(st) \d\mu_{\X_i}(t)
=\int_{\X_i} g(s) \d\mu_{\X_i}(t)
=g(s)
$$ 
So $\E_i^1(g)=g$. Hence, for this $g$, the assertion is true. The proof is complete.
\end{proof}

Using Lemma \ref{lem-prodiscrete-weak-star} together with Lemma \ref{Lemma-Symbol-weakstar-convergence-weakLp-vrai}, we deduce that the sequence $(M_{\widetilde{\varphi_j}})$ converges to $M_\psi$ in the weak operator topology of $\B(\L^p(\VN(G)))$ (in the point weak* topology if $p = \infty$).
Then we proceed as in the proof of Theorem \ref{thm-complementation-Fourier-ADS-amenable} to construct the contractive linear maps $P_G^p \co \CB(\L^p(\VN(G))) \to \mathfrak{M}^{p,\cb}(G)$ and to show that $P_G^p(M_\psi) = M_\psi$ whenever $M_\psi \in \mathfrak{M}^{p,\cb}(G)$.

Finally, we show that the map $P_G^p$ preserves the complete positivity. Suppose that $T$ is (normal if $p=\infty)$ completely positive. The operator $\Psi_{j}^{p} T \Phi_{j}^{p}$ is completely positive. Hence the multiplier $M_{\varphi_j}=P_{G/\X_j}^p\big(\Psi_{j}^{p} T \Phi_{j}^{p}\big)$ is also completely positive. By Theorem \ref{prop-homomorphism-Fourier-multipliers-cb}, we infer that $M_{\widetilde{\varphi_j}}=M_{\varphi_j \circ \pi_j}$ is completely positive. Using Lemma \ref{lem-completely-positive-weak-limit}, it is easy to deduce that $P_G^p(T)$ is completely positive.
\end{proof}


\begin{remark}\normalfont
According to \cite[Theorem 12.3.26]{Pal2}, a second countable nilpotent\footnote{\thefootnote. Recall that nilpotent implies unimodular by \cite{MS1}.} compactly generated totally disconnected locally compact group admits a sequence $(\X_j)$ satisfying the assumptions of the theorem. Moreover, any second countable compactly generated uniscalar\footnote{\thefootnote. Note that uniscalar implies unimodular, see \cite[Theorem 12.3.26]{Pal2}.} $p$-adic Lie group admits such a sequence $(\X_j)$ by \cite[Theorem 5.2]{GW}. Moreover, 
$p$-adic can be replaced by pro-$p$-adic \cite[Proposition 7.4]{GW}. Finally, there exists an example of a compactly generated totally disconnected uniscalar locally compact group which does not have an open compact normal subgroup, see \cite{BhM} and \cite{KeWi}.
\end{remark}

\begin{remark} \normalfont
\label{rem-arbre}
Note that the result applies to the profinite groups acting on locally finite trees described in Section \ref{subsec-computations-density}. 
\end{remark}



\subsection{Amenable groups and convolutors}
\label{sec:complementation-convolutor}

In this section, we observe that we can obtain compatible projections on spaces of Fourier multipliers associated to abelian locally compact groups and more generally on spaces of convolutors associated to amenable locally compact groups.

\paragraph{Convolution operators.}  Let $G$ be a locally compact group and $1 \leq  p \leq \infty$. Here we use the left translation $\lambda_s \co \L^p(G) \to \L^p(G)$ with a similar definition to the one of \eqref{Left-translation}.
A bounded linear operator $T \co \L^p(G) \to \L^p(G)$ (supposed to be weak* continuous in the case $p = \infty$\footnote{\thefootnote. If $G$ is not compact, note that there exist bounded operators $T \co \L^\infty(G) \to \L^\infty(G)$ which commute with left translations and which are not weak* continuous. We refer to \cite{Lau1} for more information.}) is said to be a $p$-convolution operator of $G$ \cite[page 8]{Der4} if for every $s \in  G$ we have $\lambda_sT=T\lambda_s$. The set of all convolution operators (or convolutors) of $G$ is denoted $\CV_p(G)$. If $G$ is abelian then $\CV_p(G)=\mathfrak{M}^{p}(\hat{G})$ isometrically, see \cite[Chapter 1]{Der4}.

If $X$ is a Banach space, the subset $\CV_p(G,X)$ of $\B(\L^p(G,X))$ is defined as the space of convolution operators $T$ such that $T \ot \Id_X$ extends to a bounded operator on $\L^p(G,X)$. The space $\CV_{p,\cb}(G)$ of completely bounded convolutors on $\L^p(G)$ coincides with $\CV_p(G,S^p)$.

\vspace{0.2cm}

Proposition \ref{prop-ma-complementation-Fourier-multipliers-abelian} is slight generalization of a particular case of the result \cite[Corollaire page 79]{Der3} (rediscovered in part in \cite[Theorem 1.1]{ArV}). We will thank Antoine Derighetti to communicate this reference. 

\begin{prop}
\label{prop-ma-complementation-Fourier-multipliers-abelian}
Let $G$ be an amenable locally compact group. Suppose $1 \leq p \leq \infty$. Then there exists a contractive projection $P_{G}^p \co \B(\L^p(G)) \to \B(\L^p(G))$ (in the  case $p = \infty$, we have $P_{G}^{\infty} \co \B_{\w^*}(\L^\infty(G)) \to \B_{\w^*}(\L^\infty(G))$) onto $\CV_{p}(G)$ such that if $T \co \L^p(G) \to \L^p(G)$ is positive\footnote{\thefootnote. Recall that the notions of ``positivity'' and ``complete positivity'' are identical on commutative $\L^p$-spaces by Proposition \ref{prop-positive-imply-cp-if-depart-commutative} and a completely positive map is completely bounded by Theorem \ref{thm-norm=normcb-hyperfinite}.} then $P_{G}^p(T)$ is positive. Moreover, all these mappings are compatible with each other. Moreover, if $1<p<\infty$, the restriction of $P_{G}^p$ to $\CB(\L^p(G))$ induces a well-defined contractive projection $P_{G}^{p,\cb} \co \CB(\L^p(G)) \to \CB(\L^p(G))$ onto $\CV_{p,\cb}(G)$.
\end{prop}

\begin{proof}
The case $1<p<\infty$ is \cite[Theorem 5]{Der3} and \cite[Theorem 1.1]{ArV}. The case $p=1$ of \cite[Theorem 1.1]{ArV}) gives a projection $P_{G}^{1} \co \B(\L^1(G)) \to \B(\L^1(G))$. Now for a weak* continuous operator $T \co \L^\infty(G) \to \L^\infty(G)$, we let $P_G^{\infty}(T)=P_G^{1}(T_*)^*$. We obtain the desired projection. The verifications are left to the reader. 

Suppose $1<p <\infty$. Let $T \co \L^p(G) \to \L^p(G)$ be a completely bounded operator. For any $f \in \L^p(G)$ and any $g \in \L^{p^*}(G)$, we consider the complex function $h_{T,f,g} \co G \to \C$, $s \mapsto \big\langle T(\lambda_{s}(f)),\lambda_{s}(g) \big\rangle_{\L^p(G),\L^{p^*}(G)}$ defined on $G$. The function $h_{T,f,g}$ is\footnote{\thefootnote. For any $s \in G$, we have
\begin{align*}
\MoveEqLeft
  \left|\big\langle T(\lambda_{s}(f)),\lambda_{s}(g) \big\rangle_{\L^p(G),\L^{p^*}(G)} \right| 
\leq \norm{T}_{\L^p(G) \to \L^p(G)} \norm{\lambda_{s}(f)}_{\L^p(G)} \norm{\lambda_{s}(g)}_{\L^{p^*}(G)}\\
&=\norm{T}_{\L^p(G) \to \L^p(G)} \norm{f}_{\L^p(G)} \norm{g}_{\L^{p^*}(G)}.
\end{align*}} bounded. 
By \cite[Theorem 20.4]{HR}, the maps $G \to \L^{p}(G)$, $s \mapsto T(\lambda_{s}(f))$ and $G \to \L^{p^*}(G)$, $s \mapsto \lambda_{s}(g)$ are continuous. Using the continuity of the duality bracket $\langle \cdot, \cdot\rangle_{L^{p}(G),\L^{p^*}(G)}$ \cite[Corollary 6.40]{AlB} on bounded subsets, we deduce that the map $h_{T,f,g}$ is continuous, hence measurable.

Since $G$ is amenable, by \cite[Proposition 4.23]{Pie}, there exists a right invariant mean\footnote{\thefootnote. That is a unital positive bounded linear form $\mathfrak{M} \co \L^\infty(G) \to \C$ such that $\mathfrak{M}(f_t)=\mathfrak{M}(f)$ for any $t \in G$ where $f_t(s)=f(st)$.} $\mathfrak{M} \co \L^\infty(G) \to \C$. Since $\L^\infty(G)$ is a unital commutative C*-algebra, the map $\mathfrak{M}$ is completely contractive by \cite[Lemma 5.1.1]{ER}. The map $\mathfrak{B} \co \L^p(G) \times \L^{p^*}(G) \to  \C$, $(f,g) \mapsto \mathfrak{M}\big(h_{T,f,g}\big)$ is clearly bilinear. Moreover, for any integer $n$, any $[f_{ij}] \in \M_{n}(\L^p(G))$ and any $[g_{kl}] \in \M_{n}(\L^{p^*}(G))$, we have
\begin{align*}
\MoveEqLeft
    \bnorm{\big[\mathfrak{B}(f_{ij},g_{k,l})\big]}_{\M_{n^2}}
		=\bnorm{\big[\mathfrak{M}(h_{T,f_{ij},g_{kl}})\big]}_{\M_{n^2}}
		 \leq \bnorm{\big[h_{T,f_{ij},g_{kl}}\big]}_{\M_{n^2}(\L^\infty(G))}\\
		&=\Bnorm{\Big[s \mapsto \big\langle T(\lambda_{s}(f_{ij})),\lambda_{s}(g_{kl})\big\rangle_{\L^p(G),\L^{p^*}(G)}\Big]}_{\M_{n^2}(\L^\infty(G))}\\
		&=\Bnorm{s \mapsto \Big[\big\langle T(\lambda_{s}(f_{ij})),\lambda_{s}(g_{kl}) \big\rangle_{\L^p(G),\L^{p^*}(G)}\Big]}_{\L^\infty(G,\M_{n^2})}\\
		&=\sup_{s \in G} \norm{\Big[\big\langle T(\lambda_{s}(f_{ij})),\lambda_{s}(g_{kl}) \big\rangle_{\L^p(G),\L^{p^*}(G)}\Big]}_{\M_{n^2}}.
\end{align*}
Now, using \cite[(3.2.3)]{ER} in the first inequality and the fact left to the reader (to use \cite[Proposition 2.1]{Pis4}) that each $\lambda_s \co \L^p(G) \to \L^p(G)$ is completely isometric in the last equality, we obtain for any $s \in G$
\begin{align*}
\MoveEqLeft
\Bnorm{\Big[\big\langle T(\lambda_{s}(f_{ij})),\lambda_{s}(g_{kl}) \big\rangle_{\L^p(G),\L^{p^*}(G)}\Big]}_{\M_{n^2}}  
		=\Bnorm{\big\langle \big\langle \big[T(\lambda_{s}(f_{ij}))\big],\big[\lambda_{s}(g_{kl})\big] \big\rangle \big\rangle}_{\M_{n^2}} \\
		&\leq \Bnorm{\big[T(\lambda_{s}(f_{ij}))\big]}_{\M_{n}(\L^p(G))} \bnorm{[\lambda_{s}(g_{kl})]}_{\M_{n}(\L^{p^*}(G))}\\
		&\leq \norm{T}_{\cb,\L^p(G) \to \L^p(G)} \bnorm{[\lambda_{s}(f_{ij})]}_{\M_{n}(\L^p(G))} \bnorm{[\lambda_{s}(g_{kl})]}_{\M_{n}(\L^{p^*}(G))}\\
		&= \norm{T}_{\cb,\L^p(G) \to \L^p(G)} \bnorm{[f_{ij}]}_{\M_{n}(\L^p(G))} \bnorm{[g_{kl}]}_{\M_{m}(\L^{p^*}(G))}.
\end{align*}
Taking the supremum, we infer that
$$
\bnorm{[\mathfrak{B}(f_{ij},g_{k,l})]}_{\M_{n^2}} 
\leq \norm{T}_{\cb,\L^p(G) \to \L^p(G)} \bnorm{[f_{ij}]}_{\M_{n}(\L^p(G))} \bnorm{[g_{kl}]}_{\M_{n}(\L^{p^*}(G))}.
$$
We conclude that $\mathfrak{B}$ is completely bounded in the sense of \cite[page 126]{ER} with $\norm{\mathfrak{B}}_{\cb} \leq \norm{T}_{\cb,\L^p(G) \to \L^p(G)}$. Hence, by \cite[Proposition 7.1.2]{ER} there exists a unique completely bounded operator $P_{G}^{p,\cb}(T) \co \L^p(G) \to \L^p(G)$ such that 
$$
\mathfrak{B}(f,g)
=\big\langle P_{G}^{p,\cb}(T)(f),g \big\rangle_{\L^p(G),\L^{p^*}(G)},\qquad f \in \L^p(G), g \in \L^{p^*}(G).
$$
Moreover, we have $\bnorm{P_{G}^{p,\cb}(T)}_{\cb,\L^p(G) \to \L^p(G)} =\norm{\mathfrak{B}}_{\cb} \leq \norm{T}_{\cb,\L^p(G) \to \L^p(G)}$. This operator coincides with the operator $P_{G}^p(T)$ provided by a slightly simplified\footnote{\thefootnote. We can replace the space of right uniformly continuous functions by $\L^\infty(G)$. Moreover, note that translations of \cite[Theorem 1.1]{ArV} differ from our notation.} proof of \cite[Theorem 1.1]{ArV}. The compatibility is left to the reader.
\end{proof}

\begin{remark} \normalfont
\label{Remark-convolutors-amenability}
Consider a locally compact group $G$. It would be interesting to know if the amenability of $G$ is characterized by the property of Proposition \ref{prop-ma-complementation-Fourier-multipliers-abelian}.
\end{remark}

\subsection{Description of the decomposable norm of multipliers}
\label{subsec-description-decomposable-norm-Fourier}


The following is a variant of Theorem \ref{prop-amenable-discrete-Fourier-multiplier-dec-infty}.

\begin{thm}
\label{prop-continuous-Fourier-multiplier-dec-infty}
Let $G$ be an amenable second countable unimodular locally compact group which is $\ALSS$ satisfying the assumption \eqref{equ-Haar-measure-convergence}. Suppose $1 \leq p \leq \infty$. Then a measurable function $\phi \co G \to \C$ induces a decomposable Fourier multiplier on $\L^p(\VN(G))$ if and only if it induces a (completely) bounded Fourier multiplier on $\VN(G)$. In this case, we have
\begin{equation}
\label{inequality-with-c}
c\norm{M_\phi}_{\VN(G) \to \VN(G)} 
\leq \norm{M_\phi}_{\dec,\L^p(\VN(G)) \to \L^p(\VN(G))} 
\leq \norm{M_\phi}_{\VN(G) \to \VN(G)}.	
\end{equation}
\end{thm}

\begin{proof}
$\Rightarrow$: We start with the case of a decomposable Fourier multiplier $M_\phi \co \L^p(\VN(G)) \to \L^p(\VN(G))$ with a \textit{continuous} symbol. By Proposition \ref{prop-decomposable-se-decompose-en-cp}, we can write 
$
M_\phi
=T_1-T_2+\mathrm{i}(T_3-T_4)
$ 
where each $T_j$ is a completely positive map on $\L^p(\VN(G))$. Using the map $P_G^p$ of Theorem \ref{thm-complementation-Fourier-ADS-amenable} (since $G$ is amenable) and the continuity of $\phi$, we obtain that
$$
M_\phi
=P_G^p(M_\phi)
=P_G^p\big(T_1-T_2+\i(T_3-T_4)\big)
=P_G^p(T_1)-P_G^p(T_2)+\i\big(P_G^p(T_3)-P_G^p(T_4)\big)
$$
where each $P_G^p(T_j)$ is a completely positive Fourier multiplier on $\L^p(\VN(G))$. Hence, by Proposition \ref{Th-cp-Fourier-multipliers}, it induces a completely positive Fourier multiplier on $\VN(G)$. We conclude that $\phi$ induces a decomposable Fourier multiplier on $\VN(G)$. If $\phi$ is only bounded and measurable, but the approximating fundamental domains $\X_j$ are symmetric (resp. $\gamma \X_j = \X_j \gamma$ for $\gamma \in \Gamma_j$), then according to Theorem \ref{thm-complementation-Fourier-ADS-amenable}, we can argue the same way.


Without the assumption of continuity (resp. symmetry or commutativity of the fundamental domains), we adapt the method of approximation of \cite[Remark 9.3]{CPPR} by completely bounded multipliers on $\VN(G)$. Let $M_\phi \co \L^p(\VN(G)) \to \L^p(\VN(G))$ be a decomposable Fourier multiplier. Since $G$ is amenable, by Leptin Theorem \cite[Theorem 10.4]{Pie}, there exists a contractive approximative unit $(\psi_i)$ of the Fourier algebra $\mathrm{A}(G)$ such that each $\psi_i$ has compact support. In addition, consider a contractive approximate unit $(\chi_j)$ of $\L^1(G)$ such that each $\chi_j$ is a function belonging to $\mathrm{C}_c(G)$ with $\norm{\chi_j}_{\L^1(G)}=1$ and $\chi_j \geq 0$ satisfying the properties of \cite[(14.11.1)]{Dieu2} (see \cite[Example 14.11.2]{Dieu2} for the existence).
For any $i,j$, we let $\phi_{i,j}=\chi_j \ast (\psi_i\phi)$. 

We claim that for any $i,j$, we have
\begin{equation}
\label{equ-1-prop-continuous-Fourier-multiplier-dec-infty}
\norm{M_{\phi_{i,j}}}_{\reg, \L^p(\VN(G)) \to \L^p(\VN(G))} 
\leq \norm{M_{\phi}}_{\reg, \L^p(\VN(G)) \to \L^p(\VN(G))}. 
\end{equation}
Indeed, since $G$ is amenable, the von Neumann algebra $\VN(G)$ is approximately finite-dimen\-sional by \cite[Corollary 6.9 (a)]{Con2}. Using Theorem \ref{thm-dec=reg-hyperfinite}, \cite[Definition 2.1]{Pis4}, the duality \cite[Theorem 4.7]{Pis5} and Plancherel formula \eqref{Formule-Plancherel}, we need to show that for any $N \in \N$, and any $f_{kl},g_{kl} \in \mathrm{C}_c(G)*\mathrm{C}_c(G)$ 
where $1 \leq k,l \leq N$ we have
\begin{align*}
& \left| \big\langle \big[ M_{\phi_{i,j}}(\lambda(f_{kl}))\big], \big[\lambda(g_{kl})\big] \big\rangle_{\L^p(\VN(G),\M_N),\L^{p^*}(\VN(G),S^1_N)} \right| 
= \left| \sum_{k,l=1}^N \int_G \phi_{i,j}(s) f_{kl}(s) \check{g}_{kl}(s) \d\mu_G(s) \right| \\
& \leq \norm{M_{\phi}}_{\reg,\L^p(\VN(G)) \to \L^p(\VN(G))} \bnorm{[\lambda(f_{kl})]}_{\L^p(\VN(G),\M_N)} \bnorm{[\lambda(g_{kl})]}_{\L^{p^*}(\VN(G),S^1_N)}.
\end{align*}
Note that
\begin{align*}
\MoveEqLeft
\left|\sum_{k,l=1}^N \int_G \psi_i(t)\phi(t)f_{kl}(t) \check{g}_{kl}(t) \d\mu_G(t) \right|
=\left| \sum_{k,l=1}^N \big\langle M_{\psi_i \phi}(\lambda(f_{kl})) , \lambda(g_{kl}) \big\rangle \right| \\
& \leq \norm{M_{\psi_i}}_{\reg,\L^p(\VN(G)) \to \L^p(\VN(G))} \norm{M_\phi}_{\reg,\L^p \to \L^p}\bnorm{[\lambda(f_{kl})]}_{}\bnorm{[\lambda(g_{kl})]}_{}.
\end{align*}
By the second and the last part of the proof, we have $\norm{M_{\psi_i}}_{\reg,\L^p \to \L^p} \leq \norm{M_{\psi_i}}_{\VN(G) \to \VN(G)} \leq \|\psi_i\|_{A(G)} \leq 1$. Using the fact that $\norm{[\lambda_{s^{-1}} \delta_{kl}]}_{\M_N(\VN(G))} = 1$, it is not difficult to prove that the regular norm is translation invariant, so that
$$
\left|\sum_{k,l=1}^N \int_G \psi_i(s^{-1}t) \phi(s^{-1}t) f_{kl}(t) \check{g}_{kl}(t)\d\mu_G(t) \right|
\leq \norm{M_\phi}_{\reg,\L^p \to \L^p}\bnorm{[\lambda(f_{kl})]}_{}\bnorm{[\lambda(g_{kl})]}_{}.
$$
Consequently, since $\norm{\chi_j}_{\L^1(G)} \leq 1$
\begin{align*}
\MoveEqLeft
 \left|\int_G \chi_j(s)\sum_{k,l=1}^N\bigg(\int_G \psi_i(s^{-1}t)\phi(s^{-1}t) f_{kl}(t) \check{g}_{kl}(t)\d\mu_G(t)\bigg)\d\mu_G(s)\right|\\
& \leq \int_G |\chi_j(s)| \left|\sum_{k,l}\int_G \psi_i(s^{-1}t)\phi(s^{-1}t) f_{kl}(t) \check{g}_{kl}(t)\d\mu_G(t)\right| \d\mu_G(s)        \\
		&\leq \norm{M_\phi}_{\reg,\L^p \to \L^p}\bnorm{[\lambda(f_{kl})]}_{}\bnorm{[\lambda(g_{kl})]}_{}.
\end{align*}
But by Fubini Theorem, we have
\begin{align*}
& \left|\int_G \chi_j(s)\sum_{k,l}\bigg(\int_G \psi_i(s^{-1}t)\phi(s^{-1}t) f_{kl}(t) \check{g}_{kl}(t)\d\mu_G(t)\bigg)\d\mu_G(s)\right| \\
& =\left|\sum_{k,l=1}^N \int_G \bigg(\int_G \chi_j(s)\psi_i(s^{-1}t)\phi(s^{-1}t)\d\mu_G(s)\bigg) f_{kl}(t) \check{g}_{kl}(t)\d\mu_G(t)\right|.
\end{align*}
We deduce that
$$
\left|\sum_{k,l=1}^N \int_G (\chi_j \ast (\psi_i \phi))(t) f_{kl}(t) \check{g}_{kl}(t)\d\mu_G(t)\right|
\leq \norm{M_\phi}_{\reg,\L^p \to \L^p}\bnorm{[\lambda(f_{kl})]}_{}\bnorm{[\lambda(g_{kl})]}_{},
$$
and finally, \eqref{equ-1-prop-continuous-Fourier-multiplier-dec-infty} follows.

Recall that $\psi_i \in \mathrm{C}_c(G)$ and $\phi \in \L^\infty(G)$, so $\psi_i \phi \in \L^\infty(G)$ with compact support, so $\psi_i \phi \in \L^2(G)$. Moreover, each function $\chi_j$ belongs to $\L^2(G)$. We conclude that $\phi_{i,j}=\chi_j \ast (\psi_i\phi)$ belongs to $\L^2(G) \ast \L^2(G)$, which equals $\mathrm{A}(G)$ \cite[Th\'eor\`eme page 218]{Eym}, so it is a continuous symbol. Then the first part of the proof and the last part show that each function $\phi_{i,j}$ induces a (completely) bounded multiplier on $\VN(G)$ with a uniform completely bounded norm. Thus, there exists a constant $C < \infty$ such that for any $i,j$, we have for $f,g \in \mathrm{C}_c(G)*\mathrm{C}_c(G)$ (to adapt if $p = \infty$ or $p = 1$)
$$
\left| \int_G \phi_{i,j}(t) f(t) \check{g}(t) \d\mu_G(t) \right| 
\leq C \norm{\lambda(f)}_{\VN(G)} \norm{\lambda(g)}_{\L^1(\VN(G))}.
$$
If $\phi_{i,j}$ converges to $\phi$ in the weak* topology of $\L^\infty(G)$, then this will yield
$$
\left| \int_G \phi(t) f(t) \check{g}(t) \d\mu_G(t) \right| 
\leq C \norm{\lambda(f)}_{\VN(G)} \norm{\lambda(g)}_{\L^1(\VN(G))}
$$
and consequently, that $\norm{M_{\phi}}_{\VN(G) \to \VN(G)} \leq C$. We show the claimed weak* convergence. For a given $h \in \L^1(G)$, we write $\langle \phi_{i,j} , h \rangle_{\L^\infty(G),\L^1(G)} 
=\langle \chi_j \ast (\psi_i \phi) - \psi_i \phi , h \rangle + \langle \psi_i \phi - \phi, h \rangle$. For the second summand, note that $\|\psi_i\|_\infty \leq \norm{\psi_i}_{\mathrm{A}(G)} \leq 1$, so that $\psi_i \phi - \phi$ is uniformly bounded in $\L^\infty(G)$. Moreover, $\psi_i(s) \to 1$ for any $s \in G$, since it is an approximate unit. By dominated convergence, we deduce $\langle \psi_i \phi - \phi, h \rangle \to 0$ as $i \to \infty$. Now for a fixed large $i$, we have that $\langle \chi_j \ast (\psi_i \phi) - \psi_i \phi, h \rangle \to 0$ according to \cite[(14.11.1)]{Dieu2}.

$\Leftarrow$: Let $M_\phi \co \VN(G) \to \VN(G)$ be a decomposable Fourier multiplier. Similarly, with Corollary \ref{cor-ADS-complementation-without-amenability}, we can write $M_\phi=M_{\phi_1}-M_{\phi_2}+\mathrm{i}(M_{\phi_3}-M_{\phi_4})$ where each $M_{\phi_j} \co \VN(G) \to \VN(G)$ is completely positive. By Proposition \ref{Th-cp-Fourier-multipliers}, each Fourier multiplier $\phi_j$ induces a completely positive multiplier on $\L^p(\VN(G))$. Using Proposition \ref{prop-decomposable-se-decompose-en-cp}, we conclude that $\phi$ induces a decomposable Fourier multiplier on $\L^p(\VN(G))$.

The proof of last part is similar to the proof to the one of Theorem \ref{prop-amenable-discrete-Fourier-multiplier-dec-infty} together with Theorem \ref{thm-dec=reg-hyperfinite} when one remembers that the von Neumann algebra $\VN(G)$ is approximately finite-dimensional.
\end{proof}

\begin{remark}\normalfont
\label{amenability-replaced-by-hyperfinite}
If we replace the amenability assumption by supposing that $\VN(G)$ is approximately finite-dimensional then the end of the proof shows that for any function $\varphi$ inducing a completely bounded Fourier multiplier on $\VN(G)$ we have the inequalities \eqref{inequality-with-c}.
\end{remark}

Similarly, we obtain the following result:

\begin{thm}
\label{prop-continuous-Fourier-multiplier-dec-infty2}
Let $G$ be a second countable amenable pro-discrete locally compact group. 
Suppose $1 \leq p \leq \infty$. Then a function $\phi \co G \to \C$ induces a decomposable Fourier multiplier $M_\phi \co \L^p(\VN(G)) \to \L^p(\VN(G))$ if and only if it induces a (completely) bounded Fourier multiplier on $M_\phi \co \VN(G) \to \VN(G)$. In this case, we have
$$
\norm{M_\phi}_{\dec,\L^p(\VN(G)) \to \L^p(\VN(G))} 
=\norm{M_\phi}_{\cb, \VN(G) \to \VN(G)}
=\norm{M_\phi}_{\VN(G) \to \VN(G)}.
$$
\end{thm}

\begin{remark}\normalfont
\label{Rem-amenability-approximation}\label{Rem-continuity-dec-multipliers-Lp}
In both situations, a function $\phi \co G \to \C$ which induces a decomposable Fourier multiplier $M_\phi \co \L^p(\VN(G)) \to \L^p(\VN(G))$ is equal to a continuous function almost everywhere, see e.g. \cite[Corollary 3.3]{Haa3}.
\end{remark}

The following observation was communicated\footnote{\thefootnote. In \cite{ArK1}, we will give another argument.}  to us by Sven Raum whom we thank for this. It shows that in the pro-discrete case, a similar remark to Remark \ref{amenability-replaced-by-hyperfinite} is useless.

\begin{prop}
\label{Prop-equivalence-}
A second countable pro-discrete locally compact group $G$ is amenable if and only if its von Neumann algebra $\VN(G)$ is approximately finite-dimensional.
\end{prop}

\begin{proof}
Consider a pro-discrete locally compact group $G$ such that $\VN(G)$ is approximately finite-dimensional. By Proposition \ref{prop-pro-discrete}, there exists an open compact normal subgroup $K$ of $G$. Using the central projection $p_K$ of Lemma \ref{Lemma-iso-GK-VN}, we have a $*$-isomorphism $\pi \co \VN(G/K) \to \VN(G)p_K$, $\lambda_{sK} \mapsto \lambda_s p_K$. It is well-known\footnote{\thefootnote. This observation relies on the equivalence between ``injective'' and ``approximately finite-dimensional''.} that this implies that $\VN(G)p_K$ is approximately finite-dimensional and thus that $\VN(G/K)$ is approximately finite-dimensional. Furthermore, since $K$ is open, the group $G/K$ is discrete by \cite[Theorem 5.26]{HR}. By \cite[Theorem 3.8.2]{SS}, we infer that $G/K$ amenable. Since $K$ is amenable, by \cite[Proposition G.2.2]{BHV}, we conclude that the group $G$ is amenable.

The converse is \cite[Corollary 6.9 (a)]{Con2}.
\end{proof}

Similarly, 
we obtain a proof of the next result. The first part is\footnote{\thefootnote. We warn the reader that the proof \cite[Proposition 3.3]{Are} is really problematic. The proof of the \textit{fundamental} point (the surjectivity of the map $\tau_p$) is lacking.} essentially stated in \cite[Proposition 3.3]{Are}.

\begin{thm}
\label{thm-convolutor-description-dec-infty}
Let $G$ be an amenable locally compact group. Suppose $1 < p < \infty$. Then a convolutor $T \co \L^p(G) \to \L^p(G)$ of $\CV_p(G)$ is regular if and only if it induces a bounded convolutor $T \co \L^\infty(G) \to \L^\infty(G)$. In this case, we have
$$
\norm{T}_{\reg,\L^p(G) \to \L^p(G)} 
=\norm{T}_{\L^\infty(G) \to \L^\infty(G)}
(=\norm{T}_{\cb,\L^\infty(G) \to \L^\infty(G)}).
$$
\end{thm}

This result applies to decomposable Fourier multipliers $M_\phi \co \L^p(\VN(G)) \to \L^p(\VN(G))$ on an abelian locally compact group $G$.

\begin{remark} \normalfont
\label{Remark-convolutors-amenability}
Consider a locally compact group $G$. It would be interesting to know if the amenability of $G$ is characterized by the property of Theorem \ref{thm-convolutor-description-dec-infty}.
\end{remark}

\section{Strongly and CB-strongly non decomposable operators}
\label{sec:Existence-of-strongly}

In this section, we construct completely bounded operators $T \co \L^p(M) \to \L^p(M)$ which cannot be approximated by decomposable operators. We particularly investigate different types of multipliers. We also give explicit examples of such operators on the noncommutative $\L^p$-spaces associated to the free groups (see Theorem \ref{Prop-concrete-free-groups} and Theorem \ref{prop-free-Hilbert-transform}).

\subsection{Definitions}
\label{subsec:Existence-of-strongly-Definitions}

The following definition is an extension of the one of \cite[Remark, page 163]{ArV} on classical $\L^p$-spaces to noncommutative $\L^p$-spaces since the regular norm and the decomposable norm are identical by Theorem \ref{thm-dec=reg-hyperfinite}.

\begin{defi}
\label{defi-strongly-non-decomposable}
We say that an operator $T \co \L^p(M) \to \L^p(M)$ is strongly non decomposable if $T$ does not belong to the closure $\ovl{\Dec(\L^p(M))}$ of the space $\Dec(\L^p(M))$ with respect to the operator norm $\norm{\cdot}_{\L^p(M) \to \L^p(M) }$.
\end{defi}

It means that $T$ cannot be approximated by decomposable operators. We also introduce the following variation of this definition.

\begin{defi}
\label{defi-CB-strongly-non-decomposable}
We say that a completely bounded operator $T \co \L^p(M) \to \L^p(M)$ is $\CB$-strongly non decomposable if $T$ does not belong to the closure $\ovl{\Dec(\L^p(M))}^{\CB}$ of the space $\Dec(\L^p(M))$ with respect to the completely bounded norm $\norm{\cdot}_{\cb,\L^p(M) \to \L^p(M)}$.
\end{defi}

If $M$ is approximately finite-dimensional, we also use the words \textit{strongly non regular} and $\CB$-\textit{strongly non regular}. 

\begin{remark}\label{rem-defi-CB-strongly-non-decomposable} \normalfont
These two notions are related. Indeed, let $T \co \L^p(M) \to \L^p(M)$ be a completely bounded operator in $\ovl{\Dec(\L^p(M))}^{\CB}$. There exists a sequence $(T_n)$ of decomposable operators acting on $\L^p(M)$ such that $\norm{T-T_n}_{\cb, \L^p(M) \to \L^p(M)}$ tends to zero when $n$ approaches $+\infty$. Hence, we have
$$
\norm{T-T_n}_{\L^p(M) \to \L^p(M)} 
\leq \norm{T-T_n}_{\cb, \L^p(M) \to \L^p(M)} \xra[n \to+\infty]{} 0.
$$
Hence $T$ belongs to the closure $\ovl{\Dec(\L^p(M))}$. We deduce that if $T$ is completely bounded and strongly non decomposable then $T$ is $\CB$-strongly non decomposable.  
\end{remark}

\subsection{Strongly non regular completely bounded Fourier multipliers on abelian groups}
\label{sec:Existence-of-strongly-non-regular-Fourier-multipliers-abelian}

Arendt and Voigt proved that the Hilbert transforms on the groups $\R$, $\Z$ and $\T$ are strongly non regular \cite[Example 3.3, 3.4, 3.9]{ArV}. In the case of an arbitrary abelian locally compact group $G$, a notion of Hilbert transform is not available in general. Nevertheless, we prove in this section that there exists a strongly non regular completely bounded Fourier multiplier acting on $\L^p(G)$.

\paragraph{Complements on convolution operators.} If $\mu\in \M(G)$ is a bounded Borel measure on $G$, then $\rho^p_G(\mu)$ denotes the element of $\CV_p(G)$, defined by $\rho^p_G(\mu)(f) =f*\Delta_G^{\frac{1}{p^*}}\check{\mu}$ for any continuous function $f \co G \to \C$ with compact support, \cite[page 8]{Der4}. Moreover, if $\mu \in \M(G)$ and if $H$ is a closed subgroup of $G$ note that 
\begin{equation}
	\label{Deri-Res-measures}
	1_H\mu=i(\mathrm{Res}_H\mu)
\end{equation}
where $i(\nu)$ denotes the image of the measure $\nu$ under the inclusion map $i$ of $H$ in $G$.

If $X$ is a Banach space, the subset $\CV_p(G,X)$ of $\B(\L^p(G,X))$ is defined as the space of convolution operators $T$ such that $T \ot \Id_X$ extends to a bounded operator on $\L^p(G,X)$. 

\paragraph{Positive convolution operators.} The following is \cite[Theorem 9.6]{Pie} (see also \cite[page 8]{Der4}, and \cite[page 280-281]{Are} for a good explanation). Let $G$ be an \textit{amenable} locally compact group and suppose $1 < p < \infty$. Let $T \co \L^p(G) \to \L^p(G)$ be a positive convolution operator. Then there exists a positive bounded measure $\mu \in \mathrm{M}(G)$ on $G$ such that $T(f)=f*\Delta_G^{\frac{1}{p^*}}\check{\mu}$ for any continuous function $f \co G \to \C$ with compact support\footnote{\thefootnote. If $s \in G$ we have by \cite[page 7]{Der4}
$
\big(f*\Delta_G^{\frac{1}{p^*}}\check{\mu}\big)(s)
=\int_G f(st) \Delta_G(t)^{\frac{1}{p}}\d \mu(t)$. }. Moreover, we have $\norm{T}_{\L^p(G) \to \L^p(G)}=\norm{\mu}$.

\paragraph{Canonical isometry from $\CV_p(H,X)$ into $\CV_p(G,X)$.} Let $G$ be a locally compact group, $H$ a closed subgroup of $G$, $X$ a Banach space and $1 < p < \infty$. There exists a canonical linear isometry 
\begin{equation}
	\label{Deri-injection}
i \co \CV_p(H,X) \to \CV_p(G,X).
\end{equation}
It is a vectorial extension of \cite[Theorem 2 page 113]{Der4}, (see also \cite[Theorem 2.6]{Arh5}) which can be proven with a similar proof. Note that the remark \cite[Remark page 106]{Der4} gives for any $\mu \in \M(H)$ the equality
\begin{equation}
\label{mesures-et-i}
	i\big(\rho^p_H(\mu)\big)=\rho^p_G(i(\mu)).
\end{equation}
where $i(\mu)$ denotes the image of the measure $\mu$ under the inclusion map $i$ of $H$ in $G$. Suppose in addition that $G$ is abelian. Using the isomorphism $\widehat{G}/ H^{\perp}=\widehat{H}$ given by $\ovl{\chi}\mapsto \chi|H$ we can reformulate \cite[Theorem 1 page 123]{Der4} under the equality
$$
i(M_\varphi)=M_{\varphi \circ \pi}
$$ 
where $\pi \co \hat{G} \to \hat{G}/H^\perp$ is the canonical map.

\paragraph{Isometry from $\CV_p(G/H)$ into $\CV_p(G)$.} Let $G$ be an amenable locally compact group and $H$ be a normal closed subgroup of $G$ such that $G/H$ is compact. By \cite[page 4 and 11]{Der5}, there exist an isometry $\Omega \co \CV_p(G/H) \to \CV_p(G)$ and a contraction $R \co \CV_p(G) \to \CV_p(G/H)$ satisfying $R\Omega=\Id_{\CV_p(G/H)}$ such that for any $\mu \in \M(G)$ 
$$
R(\rho_{G}^p(\mu))=\rho_{G/H}^p(T_H\mu)
$$
where the measure $T_H(\mu)$ is defined by (see \cite[8.2.12 page 233]{Rei})
$$
\int_{G/H} g \d \big(T_H(\mu)\big)
=\int_{G} g \circ \pi_H \d \mu_G,
$$ 
for all continuous functions $g \co G/H \to \C$ with compact support.

Let $G$ be a locally compact abelian group and $H$ be a compact subgroup of $G$. We denote by $\pi \co G \to G/H$ the canonical map. The mapping $\chi\mapsto \chi\circ \pi$ is an isomorphism of $\widehat{G/H}$ onto $H^\perp$. If $\varphi \co H^{\perp} \to \C$ is a complex function, we denote by $\widetilde{\varphi}\co \widehat{G} \to \C$ the extension of $\varphi$ on $\widehat{G}$ which is zero off $H^{\perp}$. Let $X$ be a Banach space. By \cite[Proposition 2.8]{Arh5}, the linear map
\begin{equation} 
\label{prop-extension-0-Fourier-multipliers-abelien}
    \CV_p(G/H,X) \to  \CV_p(G,X), \quad  M_{\varphi} \to M_{\tilde{\varphi}}
\end{equation} 
is an isometry.

\paragraph{Projection from $\B(\L^p(G))$ onto $\CV_p(G)$.} Let $G$ be an \textit{amenable} group and suppose $1 \leq p < \infty$. The result \cite[Theorem 1.1]{ArV} says that there exists a positive contractive projection 
\begin{equation}
	\label{Projection-Arendt-Voigt}
	P_G \co \B(\L^p(G)) \to \B(\L^p(G)).
\end{equation}
onto $\CV_p(G)$. 

\paragraph{Projection from $\CV_p(G)$ onto $\CV_p(H)$.} Let $G$ be a locally compact group and $H$ be an amenable closed subgroup. Suppose $1<p<\infty$. By \cite[Theorems 12 and 15]{Der2}, there exists a projection $\mathcal{P} \co \CV_p(G) \to \CV_p(G)$ onto $\{S \in \CV_p(G) \ : \ \supp S \subset H\}$ such that if $Q_H= i^{-1} \circ \mathcal{P} \co \CV_p(G) \to \CV_p(H)$ we have the following properties:
\begin{enumerate}
	\item $\mathcal{P}(\rho^p_G(\mu)) = \rho^p_G(1_H \mu)$ for every bounded measure $\mu \in \mathrm{M}(G)$,
	\item $\norm{Q_H(T)}_{\L^p(H) \to \L^p(H)} \leq \norm{T}_{\L^p(G) \to \L^p(G)}$,
	\item $Q_H(i(S)) = S$ for $S \in \CV_p(H)$.
\end{enumerate}

\paragraph{Restriction of multipliers.} Let $G$ be a locally compact \textit{abelian} group. Let $H$ be a closed subgroup of the dual group $\widehat{G}$. Suppose $1 \leq p \leq \infty$. Let $\varphi \co \widehat{G} \to \C$ be a continuous complex function which induces a bounded Fourier  multiplier (i.e. a convolutor) $M_\varphi \co \L^p(G) \to \L^p(G)$. Then, by \cite[Corollary 4.6]{Sae} (see also \cite[abstract and page 6]{Cow}), the restriction $\varphi_{\mid_H} \co H \to \C$ induces a bounded Fourier multiplier $M_{\varphi_{\mid_H}} \co \L^p(\hat{H}) \to \L^p(\hat{H})$ and we have 
\begin{equation}
\label{thm-Restriction-multipliers-abelian groups}
\bnorm{M_{\varphi_{\mid_H}}}_{\L^p(\hat{H}) \to \L^p(\hat{H})} \leq \bnorm{M_\varphi}_{\L^p(G) \to \L^p(G)}.
\end{equation}

\vspace{0.2cm}

We start with a useful observation.

\begin{lemma}
\label{Lemma-Der-est-pos}
Let $G$ be a unimodular amenable locally compact group and $H$ be a closed subgroup of $G$. Suppose $1<p <\infty$. The map $Q_H \co \CV_p(G) \to \CV_p(H)$ is positive.
\end{lemma}

\begin{proof}
Let $T \co \L^p(G) \to \L^p(G)$ be a positive convolution operator. There exists a positive measure $\nu \in \M(G)$ such that $T=\rho^p_G(\check{\nu})$. We consider $\mu=\check{\nu}$. We have $T=\rho^p_G(\mu)$. Using \eqref{mesures-et-i} and \eqref{Deri-Res-measures}, we see that
\begin{align*}
\MoveEqLeft
  \mathcal{P}\big(\rho_G^p(\mu)\big)  
		=\rho_G^p(1_H\mu) 
		=\rho_G^p\big(i(\mathrm{Res}_H\mu)\big) 
		=i\big(\rho^p_H(\mathrm{Res}_H\mu)\big).
\end{align*}
Using the definition $Q_H=i^{-1} \circ \mathcal{P}$ of $Q_H$, we obtain finally
$$
Q_{H}(T)
=Q_{H}\big(\rho_G^p(\mu)\big)
=i^{-1}\big(\mathcal{P}(\rho_G^p(\mu))\big)
=\rho_H^p(\mathrm{Res}_H\mu).
$$
Since $\mathrm{Res}_H\mu$ is a positive measure, we deduce that $Q_{H}(T)$ is a positive operator.
\end{proof}

Similarly, we can prove the two following results.

\begin{lemma}
\label{Lemma-R-est-pos}
Let $G$ be a unimodular amenable locally compact group and $H$ be a normal closed subgroup of $G$ such that $G/H$ is compact. Suppose $1<p <\infty$. The map $R \co \CV_p(G) \to \CV_p(G/H)$ is positive.
\end{lemma}

\begin{lemma}
\label{Lemma-i-est-pos}
Let $G$ be a unimodular amenable locally compact group and $H$ be a closed subgroup of $G$. Suppose $1<p <\infty$. The map $i \co \CV_p(H) \to \CV_p(G)$ is positive.
\end{lemma}

Now, we state our first transference result.

\begin{prop}
\label{prop-transfer-subgroups} 
Let $G$ be a unimodular amenable locally compact group and $H$ be a closed subgroup of $G$. Then a convolution operator $T \co \L^p(H) \to \L^p(H)$ is a strongly non regular Fourier multiplier if and only if the convolutor $i(T) \co \L^p(G) \to \L^p(G)$ is strongly non regular. 
\end{prop}


\begin{proof}
Note that $H$ is also amenable since it is a subgroup of the amenable group $G$.

$\Leftarrow$: Suppose that $T$ belongs to $\ovl{\Reg(\L^p(H))}^{\B(\L^p(H))}$. Let $\epsi > 0$. Then there exist some positive operators $R_1,R_2,R_3,R_4 \co \L^p(H) \to \L^p(H)$ and a bounded map $R \co \L^p(H) \to \L^p(H)$ of norm less than $\epsi$ such that $T=R_1-R_2+\mathrm{i}(R_3-R_4)+R$. Since $H$ is amenable, we can use the map \eqref{Projection-Arendt-Voigt} and suppose that $R_1,R_2,R_3,R_4$ and $R$ are convolution operators. Using the isometry $i \co \CV_p(H) \to \CV_p(G) $ we obtain
$$
i(T)
=i(R_1)-i(R_2)+\i(i(R_3)-i(R_4))+i(R).
$$
Using Lemma \ref{Lemma-i-est-pos}, we see that the operators $i(R_j)$ are positive. Moreover, note that we have $\norm{i(R)}_{\L^p(G) \to \L^p(G)}=\norm{R}_{\L^p(H) \to \L^p(H)} \leq \epsi$. It follows that the convolution operator $i(T)$ is $\epsi$-close to $\Reg(\L^p(G))$ in the Banach space $\B(\L^p(G))$. So letting $\epsi \to 0$ yields that $i(T) \in \overline{\Reg(\L^p(G))}^{\B(\L^p(G))}$. This is the desired contradiction.

$\Rightarrow$: Suppose that $i(T)$ belongs to $\overline{\Reg(\L^p(G))}^{\B(\L^p(G))}$. Let $\epsi > 0$. Then there exist some positive maps $R_1,R_2,R_3,R_4 \co \L^p(G) \to \L^p(G)$ and a bounded map $R \co \L^p(G) \to \L^p(G)$ of norm less than $\epsi$ such that $i(T)=R_1-R_2+\i(R_3-R_4)+R$. Since $G$ is amenable, using the map \eqref{Projection-Arendt-Voigt}, we can suppose that $R_1,R_2,R_3,R_4$ and $R$ are convolution operators.


Since $H$ is amenable, we can use the contraction $Q_{H} \co \CV_p(G) \to \CV_p(H)$. We obtain
\begin{align*}
\MoveEqLeft
 T=Q_{H}\big(i(T)\big)
    =Q_{H}\big(R_1-R_2+\i(R_3-R_4)+R\big)\\ 
		&=Q_{H}(R_1)-Q_{H}(R_2)+\i \big(Q_H(R_3)-Q_{H}(R_4)\big)+Q_{H}(R).
\end{align*}
Moreover, by the contractivity of $Q_H$, the convolution operator $Q_{H}(R) \co \L^p(H) \to \L^p(H)$ is bounded of norm less than $\epsi$. Furthermore, by Lemma \ref{Lemma-Der-est-pos}, each convolution operator $Q_{H}(R_k) \co \L^p(H) \to \L^p(H)$ is a positive operator. It follows that $T$ is $\epsi$-close to $\Reg(\L^p(H))$ in the Banach space $\B(\L^p(H))$. So letting $\epsi \to 0$ yields that $T \in \ovl{\Reg(\L^p(H))}^{\B(\L^p(H))}$. This is the desired contradiction.
\end{proof}

\begin{prop}
\label{prop-transfer-quotient} 
Let $G$ be a unimodular amenable locally compact group and $H$ be a normal closed subgroup of $G$ such that $G/H$ is compact. If the convolution operator $T \co \L^p(G/H) \to \L^p(G/H)$ is strongly non regular then the convolution operator $\Omega(T) \co \L^p(G) \to \L^p(G)$ is strongly non regular.
\end{prop}

\begin{proof}
Suppose that $\Omega(T)$ belongs to $\overline{\Reg(\L^p(G))}^{\B(\L^p(G))}$. Let $\epsi > 0$. Then there exist some positive maps $S_1,S_2,S_3,S_4 \co \L^p(G) \to \L^p(G)$ and a bounded map $S \co \L^p(G) \to \L^p(G)$ of norm less than $\epsi$ such that $\Omega(T)=S_1-S_2+\i(S_3-S_4)+S$. Since $G$ is amenable, using the map \eqref{Projection-Arendt-Voigt}, we can suppose that $S_1,S_2,S_3,S_4$ and $S$ are convolution operators. Using the contraction $R \co \mathrm{CV}_p(G) \to \mathrm{CV}_p(G/H)$, we obtain
$$
T
=R(\Omega(T))
=R(S_1)-R(S_2)+\i(R(S_3)-R(S_4))+R(S).
$$
Moreover, by the contractivity of $R$, the convolution operator $R(S) \co \L^p(G/H) \to \L^p(G/H)$ is bounded of norm less than $\epsi$. By Lemma \ref{Lemma-R-est-pos}, each convolution operator $R(S_k) \co \L^p(G/H) \to \L^p(G/H)$ is positive. It follows that $T$ is $\epsi$-close to $\Reg(\L^p(G/H))$ in the Banach space $\B(\L^p(G/H))$. So letting $\epsi \to 0$ yields that $T \in \ovl{\Reg(\L^p(G/H))}^{\B(\L^p(G/H))}$. This is the desired contradiction. 
\end{proof}

\begin{prop}
\label{prop-transfert-extension-0-abelian}
Let $G$ be a compact abelian group and let $H$ be a closed subgroup of $G$. If $\varphi \co H^{\perp} \to \C$ is a complex function, we denote by $\widetilde{\varphi} \co \widehat{G} \to \C$ the extension of $\varphi$ on $\widehat{G}$ which is zero off $H^{\perp}$. If the function $\varphi$ induces a strongly non regular Fourier multiplier $M_{\varphi} \co \L^p(G/H) \to \L^p(G/H)$ then the function $\widetilde{\varphi}$ induces a strongly non regular Fourier multiplier $M_{\widetilde{\varphi}} \co \L^p(G) \to \L^p(G)$.
\end{prop}

\begin{proof}
Suppose that $M_{\tilde{\varphi}}$ belongs to $\ovl{\Reg(\L^p(G))}^{\B(\L^p(G))}$. Let $\epsi > 0$. Then there exist some positive maps $R_1,R_2,R_3,R_4 \co \L^p(G) \to \L^p(G)$ and a bounded map $R \co \L^p(G) \to \L^p(G)$ of norm less than $\epsi$ such that $M_{\widetilde{\varphi}}=R_1-R_2+\i(R_3-R_4)+R$.

Since $G$ is amenable, the linear map \eqref{Projection-Arendt-Voigt} yields the existence of some complex functions $\phi_1,\phi_2,\phi_3,\phi_4$ and $\psi$ on $\widehat{G}$ such that $M_{\widetilde{\varphi}}= M_{\phi_1}-M_{\phi_2}+\i (M_{\phi_3}-M_{\phi_4})+M_\psi$ such that the Fourier multipliers $M_{\phi_k}$ are positive on $\L^p(G)$ and $M_\psi$ is again of norm less than $\epsi$. 

By
Proposition \ref{Th-cp-Fourier-multipliers}, each (continuous\footnote{\thefootnote. Note that the group $\hat{G}$ is discrete.}) function $\phi_k$ induces a positive linear operator $M_{\phi_k} \co \L^\infty(G) \to \L^\infty(G)$ and $\phi_k$ is positive definite. We infer that the restriction $\phi_k|H^\perp \co G \to \C$ is (continuous and) positive definite, and thus by \cite[Proposition 4.2]{DCH}, induces a positive operator $M_{\phi_k|H^\perp} \co \L^\infty(G/H) \to \L^\infty(G/H)$. Then by Proposition \ref{Th-cp-Fourier-multipliers}, it follows that the Fourier multiplier $M_{\phi_k|H^\perp} \co \L^p(G/H) \to \L^p(G/H)$ is positive.

Note that the group $H^\perp=\widehat{G/H}$ is discrete. By \eqref{thm-Restriction-multipliers-abelian groups}, since the function $\psi$ is continuous, the Fourier multiplier $M_{\psi|H^\perp} \co \L^p(G/H) \to \L^p(G/H)$ is bounded of norm less than $\epsi$. Since
$$
M_{\varphi} 
=M_{\varphi_1|H^\perp}-M_{\varphi_2|H^\perp}+\i \big(M_{\varphi|H^\perp}-M_{\varphi_4|H^\perp}\big)+M_{\psi|H^\perp}
$$
it follows that $M_{\varphi}$ is $\epsi$-close to $\Reg(\L^p(G/H))$ in the Banach space $\B(\L^p(G/H))$, so that letting $\epsi \to 0$ yields that $M_{\varphi} \in \overline{\Reg(\L^p(G/H))}^{\B(\L^p(G/H))}$. This is the desired contradiction.
\end{proof}

Let $(\epsi_k)_{k \geq 0}$ be a sequence of independent Rademacher variables on some probability space $\Omega_0$. Let $X$ be a Banach space and let $1<p<\infty$. We let $\Rad_p(X)\subset \L^p(\Omega_0,X)$ be the closure of 
$\mathrm{span}\bigl\{\epsi_k \ot x\ |\ k \geq 0,\ x\in X\bigr\}$ in the Bochner space $\L^p(\Omega_0,X)$. Thus, for any finite family $(x_k)_{0 \leq k \leq n}$ of elements of $X$, we have
\begin{equation*}
\Bgnorm{\sum_{k=0}^{n} \epsi_k \ot x_k}_{\Rad_p(X)} 
=
\Bigg(\int_{\Omega_0} \bgnorm{\sum_{k=0}^{n} \epsi_k(\omega) x_k}_{X}^{p} \d\omega \Bigg)^{\frac{1}{p}}.
\end{equation*}
We simply write $\Rad(X)=\Rad_2(X)$. By Kahane's inequalities (see e.g. \cite[Theorem 11.1]{DJT}), the Banach spaces $\Rad(X)$ and $\Rad_p(X)$ are canonically isomorphic. We will use the following result which is a variant of \cite[Theorem 4.1.9]{EdG}.

\begin{prop}
\label{prop-vector-valued-LP}
Let $X$ be a $\mathrm{UMD}$ Banach space. Suppose $1 < p < \infty$.
\begin{enumerate}
\item Let $G$ be a countably infinite discrete abelian group. Assume that there exists a sequence $(H_n)_{n \geq 0}$ of subgroups of the (compact) dual group $\hat{G}$ such that
\begin{itemize}
\item[(a)] each $H_n$ is open,
\item[(b)] $H_{n+1} \subsetneqq H_n$,
\item[(c)] $\bigcap_{n \geq 0} H_n = \{0\}$ and $H_0 = \hat{G}$.
\end{itemize}
For any integer $n \geq 0$, consider the subset $\Delta_n = H_n \backslash H_{n+1}$ of $\hat{G}$. Then for any $f \in \L^p(G,X)$, the series $\sum_{n=0}^\infty \epsi_n \ot (M_{1_{\Delta_n}} \ot \Id_X)(f)$ converges in $\Rad(\L^p(G,X))$ and we have the norm equivalence
\begin{equation}
	\label{Eq-LP1}
\norm{f}_{\L^p(G,X)} 
\approx \Bgnorm{\sum_{n=0}^\infty \epsi_n \ot (M_{1_{\Delta_n}} \ot \Id_X)(f)}_{\Rad(\L^p(G,X))}.
\end{equation}

\item Let $G$ be a compact abelian group. Assume that there exists a sequence $(Y_n)_{n \geq 0}$ of subgroups of the (discrete) dual group $\hat{G}$ such that
\begin{itemize}
\item[(a)] each $Y_n$ is finite
\item[(b)] $Y_n \subsetneqq Y_{n+1}$,
\item[(c)] $Y_0 = \{0\}$ and $\bigcup_{n \geq 0} Y_n = \hat{G}$.
\end{itemize}
Let $\Delta_0=Y_0$ and $\Delta_n = Y_{n} \backslash Y_{n-1}$ for $n \geq 1$. Then for any $f \in \L^p(G,X)$, the series $\sum_{n=0}^\infty \epsi_n \ot (M_{1_{\Delta_n}} \ot \Id_X)(f)$ converges in $\Rad(\L^p(G,X))$ and we have the norm equivalence
\begin{equation}
\label{Eq-LP2}
\norm{f}_{\L^p(G,X)} 
\approx \Bgnorm{\sum_{n=0}^\infty \epsi_n \ot (M_{1_{\Delta_n}} \ot \Id_X)(f)}_{\Rad(\L^p(G,X))}.
\end{equation}
\end{enumerate}
\end{prop}

\begin{proof}
1. Let $\mathcal{F} = \mathcal{P}(G)$ denote the full $\sigma$-algebra of subsets of $G$. For $n \geq 0$, consider the annihilator $G_n \overset{\mathrm{def}}= H_n^\perp$ in $G$. Since each $H_n$ is open and compact, each $G_n$ is compact and open by \cite[Remark 4.2.22]{Rei}, hence finite ($G$ is discrete). 

For any negative integer $k \leq 0$ consider the $\sigma$-algebra $\mathcal{F}_k$ generated by the cosets of $G_{-k}$ in $G$. Since $G$ is countably infinite, there are only countably many cosets of $G_{-k}$ in $G$. So by \cite[Exercice 4 (a) page 227]{AbA} the elements of $\Fc_{k}$ are the sets which are a union of cosets of $G_{-k}$ in $G$. Since $H_{-k+1} \subset H_{-k}$ for all $k \leq 0$, by \cite[Proposition 4.2.24]{Rei}, we have $G_{-k} \subset G_{-k+1}$. Then it is not difficult to see that $\mathcal{F}_{k-1} \subset \mathcal{F}_{k}$ if $k \leq 0$. We conclude that $(\mathcal{F}_{k})_{k \leq 0}$ is a filtration in $G$. It is elementary to check
\footnote{\thefootnote. \label{footnote65} Let $s \in G$ and let $I_n\overset{\mathrm{def}}=s(H_n)$ be the subgroup of $\T$ where we identify $s$ with $\eta(s)$ where $\eta \co G \to \hat{\hat{G}}$ is the canonical map. Since $H_n$ is compact, $I_n$ is a closed subgroup of $\T$. Any decreasing sequence of closed subgroups of $\T$ stabilizes (each closed subgroup is finite or equal to $\T$). So there exists $N \geq 0$ such that $I_n$ is the same for all $n \geq N$. Let $I$ be this common value. We have $I \subset I_n=s(H_n)$ for any $n \geq 0$. 

If $I=\{1\}$, then $s$ annihilates $H_n$ for $n \geq N$. Hence $s \in G_n$ for $n \geq N$.

Suppose that $I$ is not trivial. Let $i \in I \setminus \{1\}$ and let $C_n\overset{\mathrm{def}}=s^{-1}(\{i\}) \cap H_n$. Then the sets $C_n$ are nonempty for any $n \geq 0$ and form a decreasing sequence of compact subsets of $\hat{G}$. The intersection $C\overset{\mathrm{def}}=\bigcap_{n \geq 0} C_n$ is thus nonempty. But $C \subset \bigcap_{n \geq 0} H_n=\{0\}$, so this means $0 \in C$. Hence $0 \in s^{-1}(i)$. This is a contradiction, since $i \neq 1$ and $s(0)=1$.} that $\cup_{n \geq 0} G_n=G$.

Moreover, since $G$ is countable, the counting measure $\mu_G$ is $\sigma$-finite. 
Since the $G_{-k}$ are finite, so the restriction of $\mu_G$ to each $\Fc_k$ is also $\sigma$-finite. So, by \cite[Corollary 2.6.30]{HvNVW}, the conditional expectation $\E(\cdot|\Fc_k)$ with respect to $\Fc_k$ is well-defined and it is explicitly described in \cite[page 183]{JeR} (see also \cite[page 69]{IoT1}), since $G_{-k}$ is compact, by 
$$
\E(f|\Fc_k)
= T_{G_{-k}}(f) \circ \pi_k \quad \text{(almost everywhere)}.
$$
where $\pi_{k} \co G \to G/G_{-k}$ is the canonical map and where $T_{G_{-k}}$ is essentially defined in \cite[page 100]{Rei}. For any integer $k \leq 0$, since $H_{-k}$ is open, the Poisson formula \cite[5.5.4]{Rei} says that
$$
\big(T_{G_{-k}}(f) \circ \pi_k\big)(s)
=\int_{H_{-k}} \chi(s)\hat{f}(\chi) \d\mu_{H_{-k}}(\chi)
=\int_{\hat{G}} \chi(s)1_{H_{-k}}(\chi)\hat{f}(\chi) \d\mu_{\hat{G}}(\chi).
$$ 
We conclude that the conditional expectation $\E(\cdot|\Fc_k) \co \L^p(G) \to \L^p(G)$ is\footnote{\thefootnote. We can alternatively compute the conditional expectation with \cite[Exercice 4 (c) page 227]{AbA} instead of the Poisson formula.} a Fourier multiplier whose symbol is the indicator function $1_{H_{-k}}$. Hence for any $n \geq 0$
$$ 
M_{1_{\Delta_n}} 
=M_{1_{H_n \backslash H_{n+1}}}
=M_{1_{H_n}}-M_{1_{H_{n+1}}}
=\E(\cdot|\Fc_{-n}) - \E(\cdot|\Fc_{-n-1})
$$
as bounded operators on $\L^p(G)$. Note that the right hand side is regular on $\L^p(G)$.  Consequently, their tensor products with the identity $\Id_X$ also coincide. 



For any $f \in \L^p(G,X)$ and any integer $k \leq 0$, we let $f_{k} \overset{\mathrm{def}}= \big(\E(\cdot|\Fc_{k}) \ot \Id_X\big)(f)$. By \cite[Proposition 2.6.3 and Example 3.1.2]{HvNVW}, we obtain a martingale $(f_k)_{k \leq 0}$ with respect to the filtration $(\Fc_{k})_{k \leq 0}$. Note that since $G_0 = H_0^\perp=\hat{G}^\perp=\{0\}$ we have $\Fc_0 = \Fc$ and thus $f_0=\big(\E(\cdot|\Fc_{0}) \ot \Id_X\big)(f) = f$. Consequently, for any integer $N \geq 1$, we have $
\sum_{k=-N+1}^{0} d f_k
=\sum_{k=-N+1}^{0} (f_k-f_{k-1})
=f_0-f_{-N}
=f-f_{-N}$ and $d f_k=f_k-f_{k-1}=\big(\E(\cdot|\Fc_{k}) \ot \Id_X\big)(f)-\big(\E(\cdot|\Fc_{k-1}) \ot \Id_X\big)(f)$. By \cite[Proposition 4.2.3]{HvNVW} with the change of index $n=-k$, we infer that
$$ 
\norm{f-f_{-N}}_{\L^p(G,X)} 
\cong \Bgnorm{\sum_{n=0}^{N-1} \epsi_n \ot (M_{1_{\Delta_n}} \ot \Id_X)(f)}_{\Rad(\L^p(G,X))}.
$$


It is straightforward to check
\footnote{\thefootnote. \label{Dernierfootnote}
Let $A \in \bigcap_{k \leq 0} \Fc_k$. Suppose that $A \neq \emptyset$. Now, we construct a sequence $(s_k)$ of elements of $G$ by induction. There exists some $s_0 \in G$ such that $\{s_0\}=s_0 G_0 \subset A$. Suppose that $s_{-k} \in G$ for some $k \leq 0$ satisfy $s_{-k} G_{-k} \subset A$. Since we can write $A = \bigcup_{s \in I_{-k+1}} s G_{-k+1}$ for some index set $I_{-k+1}$ and since $G_{-k}$ is a subgroup of $G_{-k+1}$, we can choose $s_{-k+1} \in G$ such that $s_{-k} G_{-k} \subset s_{-k+1} G_{-k+1} \subset A$. Moreover, we have $s_{-k}G_{-k} = s_{-k-1}G_{-k}$. Indeed, since $s_{-k-1} G_{-k-1} \subset s_{-k} G_{-k}$, we have $s_{-k-1} \in s_{-k} G_{-k}$. Hence there exists $r_{-k} \in G_{-k}$ such that $s_{-k-1} = s_{-k} r_{-k}$. We deduce that $s_{-k} = s_{-k-1} r_{-k}^{-1}$ and consequently $s_{-k} G_{-k} = s_{-k-1} r_{-k}^{-1} G_{-k} = s_{-k-1} G_{-k}$. Finally, we obtain 
$$
s_0 \bigcup_{k \leq 0} G_{-k}
=\bigcup_{k \leq 0} s_0 G_{-k} 
\subset \bigcup_{k \leq 0} s_{-k} G_{-k}
\subset A.
$$ 
On the other hand, we have already observed that the first set equals $G$.
} 
that $\bigcap_{k \leq 0} \Fc_{k}=\{\emptyset,G\}$. We conclude that the restriction of the measure $\mu_G$ to $\bigcap_{k \leq 0} \Fc_{k}$ is purely infinite in the sense of \cite[Definition 1.2.27 (c)]{HvNVW} on the $\sigma$-algebra $\Fc_{-\infty} \overset{\textrm{def}}= \bigcap_{k \leq 0} \Fc_k$. According to \cite[Theorem 3.3.5 (3)]{HvNVW}, $f_{-N}$ converges to zero in $\L^p(G,X)$ when $N$ goes to $\infty$. Since $X$ is UMD, $X$ does not contain the Banach space $c_0$. Using Hoffmann-Jorgensen-Kwapien Theorem \cite{HoJ}, \cite{Kwa1}, it is not difficult to conclude that the series $\sum_{n=0}^\infty \epsi_n \ot (M_{1_{\Delta_n}} \ot \Id_X)(f)$ converges in $\Rad(\L^p(G,X))$ and to obtain the claimed norm equivalence of Littlewood-Paley type.

2. Let $\mathcal{F}$ denote the Borel $\sigma$-algebra generated by the open subsets of $G$. For $n \geq 0$, consider the annihilator $G_n \overset{\mathrm{def}}= Y_n^\perp$ in $G$ and the $\sigma$-algebra $\mathcal{F}_n$ generated by the cosets of $G_n$ in $G$. Since each $Y_n$ is open and compact, each $G_n$ is compact and open by \cite[Remark 4.2.22]{Rei}. Since $Y_{n} \subset Y_{n+1}$ for all $n \geq 0$, we have $G_{n+1} \subset G_{n}$ and finally $\mathcal{F}_{n} \subset \mathcal{F}_{n+1}$. We conclude that $(\mathcal{F}_{n})_{n \geq 0}$ is a filtration in $G$. Since $G$ is compact, the Haar measure $\mu_G$ is finite, so trivially $\sigma$-finite on each $\Fc_n$. So, by \cite[Corollary 2.6.30]{HvNVW}, the conditional expectation $\E(\cdot|\Fc_n)$ with respect to $\Fc_n$ is well-defined and it is explicitly described in \cite[page 69]{IoT1} (since $G_n$ is compact) by 
$$
\E(f|\Fc_n)
= T_{G_{n}}(f) \circ \pi_n \quad \text{(almost everywhere)}.
$$
where $\pi_{n} \co G \to G/G_{n}$ is the canonical map and where $T_{G_{n}}$ is essentially defined in \cite[page 100]{Rei}. For any integer $n \geq 0$, since $Y_{n}$ is open, the Poisson formula \cite[(5.5.4)]{Rei} says that
$$
\big(T_{G_{n}}(f) \circ \pi_n\big)(s)
=\int_{Y_{n}} \chi(s)\hat{f}(\chi) \d\mu_{Y_{n}}(\chi)
=\int_{\hat{G}} \chi(s)1_{Y_{n}}(\chi)\hat{f}(\chi) \d\mu_{\hat{G}}(\chi).
$$ 
We conclude that the conditional expectation $\E(\cdot|\Fc_n) \co \L^p(G) \to \L^p(G)$ is a Fourier multiplier whose symbol is the indicator function $1_{Y_{n}}$. Hence for any $n \geq 1$
$$ 
M_{1_{\Delta_n}} 
=M_{1_{Y_{n} \backslash Y_{n-1}}}
=M_{1_{Y_n}}-M_{1_{Y_{n-1}}}
=\E(\cdot|\Fc_{n}) - \E(\cdot|\Fc_{n-1})
$$
as bounded operators on $\L^p(G)$. Note that the right hand side is regular on $\L^p(G)$.  Consequently, their tensor products with the identity $\Id_X$ also coincide. Similarly, we have $M_{1_{\Delta_0}} \ot \Id_X=M_{1_{Y_{0}}} \ot \Id_X=\E(\cdot|\Fc_{0}) \ot \Id_X$.

For any $f \in \L^p(G,X)$ and any integer $n \geq 0$, we let $f_{n} \overset{\mathrm{def}}= \big(\E(\cdot|\Fc_{n}) \ot \Id_X\big)(f)$. By \cite[Proposition 2.6.3 and Example 3.1.2]{HvNVW}, we obtain a martingale $(f_n)_{n \geq 0}$ with respect to the filtration $(\Fc_{n})_{n \geq 0}$. For any integer $N \geq 1$, we have $
\sum_{n=1}^{N} d f_n
=\sum_{n=1}^{N} (f_n-f_{n-1})
=f_N-f_{0}$ and $d f_n=f_n-f_{n-1}=\big(\E(\cdot|\Fc_{n}) \ot \Id_X\big)(f)-\big(\E(\cdot|\Fc_{n-1}) \ot \Id_X\big)(f)$ if $n \geq 1$ and $df_0=f_0=\big(\E(\cdot|\Fc_{0}) \ot \Id_X\big)(f)$. 

Note that $\bigcap_{n \geq 0} G_n = \{0\}$. Indeed, if $t \in G_n$, then for any $\chi \in Y_n=G_n^\perp$ we have $\chi(t)=1$. So if $t \in \bigcap_{n \geq 0} G_n$, then $\chi(t) = 1$ for all $\xi \in \bigcup_{n \geq 0} Y_n = \hat{G}$. Thus $t=0$ and the claim is proved. Then it is not difficult to check\footnote{\thefootnote. If $U$ is an open subset of $G$ containing $0$, consider the decreasing sequence of compact subsets $(G-U) \cap G_n$ and conclude that $G_n \subset U$ if $n$ is large enough.} that $(G_n)_{n \geq 0}$ is a neighborhood system at $0$. Now by \cite[(4.21)]{HR} (see also \cite[Example page 223]{Bou4}), the family of subsets of the form $sG_n$ where $n \geq 0$ and where $s$ runs through $G$ is an open basis for $G$. So the limit $\sigma$-algebra $\Fc_\infty = \sigma\left(\bigcup_{n \geq 0} \Fc_n \right)$ equals $\Fc$. 

According to \cite[Theorem 3.3.2 (2)]{HvNVW}, $f_{N}$ converges to $\big(\E(\cdot|\Fc_{\infty}) \ot \Id_X\big)(f)=f$ in $\L^p(G,X)$ when $N$ goes to $\infty$. Similarly to the case 1, we obtain the convergence of the series $\sum_{n=0}^\infty \epsi_n \ot (M_{1_{\Delta_n}} \ot \Id_X)(f)$ and the equivalence
$$ 
\norm{f-f_0}_{\L^p(G,X)} 
\cong \Bgnorm{ \sum_{n = 1}^\infty \epsi_n \ot ( M_{1_{\Delta_n}} \ot \Id_X)(f)}_{\Rad(\L^p(G,X))}.
$$
One easily incorporates $\norm{f_0}_{\L^p(G,X)} = \norm{(M_{1_{\Delta_0}} \ot \Id_X)(f)}_{\L^p(G,X)}$ on both sides with \cite[page 5]{HvNVW2} to deduce the claimed Littlewood-Paley norm equivalence.
\end{proof}

Note that in the case $X=\C$, using the Maurey-Khintchine inequalities \cite[16.11]{DJT} the equivalences \eqref{Eq-LP1} and \eqref{Eq-LP2} become
\begin{equation}
\label{equ-Littlewood-Paley-equivalence}
\norm{f}_{\L^p(G)} 
\approx \left\| \left( \sum_{n=0}^\infty |M_{1_{\Delta_n}} f|^2 \right)^{\frac12} \right\|_{\L^p(G)}.
\end{equation}

We need the following characterization \cite{Ram} of the closure $\ovl{\B(\hat{G})}$ of the Fourier-Stieltjes algebra $\B(\hat{G})\overset{\mathrm{def}}= \{ \hat{\mu} :\: \mu\in \M(G)\}$ of the dual of a locally compact abelian group $G$ in the space $\mathrm{C}_b(\hat{G})$ of bounded continuous complex-valued functions on $\hat{G}$ equipped with the norm $\norm{\cdot}_{\infty}$. If $f \co \hat{G} \to \C$ is a bounded continuous function then $f$ belongs to $\ovl{\B(\hat{G})}$ if and only if for any sequence $(\mu_n)$ of bounded Borel measures on $\hat{G}$ the conditions $\sup_{n \geq 1} \norm{\mu_n}<\infty$ and $\widehat{\mu_n}(x) \xra[n \to +\infty]{} 0$ for all $x \in G$ imply that $\int_{\hat{G}} f \d \mu_n \xra[n \to +\infty]{} 0$.


\begin{prop}
\label{prop-existence-strongly-non-regular-compact-group}
Let $G$ be an infinite compact abelian group. Suppose $1 < p < \infty$. Then there exists a strongly non regular Fourier completely bounded Fourier multiplier on $\L^p(G)$.
\end{prop}

\begin{proof}
Since $G$ is compact, its dual $\hat{G}$ is discrete. Suppose first that $\hat{G}$ contains an element of infinite order, thus a (necessarily closed) subgroup isomorphic with $\Z$. Consider the closed subgroup $H=\Z^\perp$ of $G$. Then we have an isomorphism $\widehat{G/H}=H^\perp=\Z$. Hence $G/H$ is isomorphic to $\T$. According to \cite[Example 3.9]{ArV}, the Hilbert transform on $\T$ defines a strongly non regular Fourier multiplier on $\L^p(G/H)$. 

Since $S^p$ is $\mathrm{UMD}$, the Hilbert transform induces a bounded operator on $\L^p(\R,S^p)$ by \cite[Theorem 5.1.1]{HvNVW}. According to the estimate \cite[Proposition 5.2.5]{HvNVW}, the Hilbert transform on $\T$ induces a bounded operator $\L^p(G/H,S^p) \to \L^p(G/H,S^p)$. Then by the canonical isometry $\L^p(G/H,S^p) = S^p(\L^p(G/H))$ of \cite[(3.6)]{Pis5} and Proposition \ref{lemma-charact-Lp-cb}, we deduce that the Hilbert transform is completely bounded on $\L^p(G/H)$. Thus, by Proposition \ref{prop-transfert-extension-0-abelian} and using the isometry \eqref{prop-extension-0-Fourier-multipliers-abelien}, we deduce that there exists a strongly non regular Fourier multiplier on $\L^p(G)$.

Now suppose that no element in $\hat{G}$ has infinite order, i.e. $\hat{G}$ is an infinite abelian torsion group. Then it contains a countably infinite abelian torsion group (consider some countably infinite collection of elements in $\hat{G}$ and take the subgroup spanned by this collection, which is again countably infinite). Arguing as before with Proposition \ref{prop-transfert-extension-0-abelian} and the isometry \eqref{prop-extension-0-Fourier-multipliers-abelien}, it suffices to find a strongly non regular Fourier multiplier on a group having as dual this countable group, so we assume now that $\hat{G}$ is a countably infinite abelian torsion discrete group.

It is (really) elementary to see there exists a sequence $(Y_n)_{n \geq 0}$ of subgroups of $\hat{G}$ with the properties:
\begin{enumerate}
\item each $Y_n$ is finite,
\item $Y_n \subsetneqq Y_{n+1}$,
\item $Y_0=\{0\}$ and $\bigcup_{n=0}^\infty Y_n = \hat{G}$.
\end{enumerate}

Consider now $\Delta_0 \ov{\mathrm{def}}{=} Y_0, \: \Delta_n \ov{\mathrm{def}}{=} Y_{n} \backslash Y_{n-1}$ for $n \geq 1$. According to Proposition \ref{prop-vector-valued-LP}, the Littlewood-Paley equivalence \eqref{equ-Littlewood-Paley-equivalence} holds. This in turn is equivalent \cite[1.2.5 pages~8 and 14]{EdG} to the property that any $\psi \in \L^\infty(\hat{G})$ which is constant on any $\Delta_n$, $n=0,1,2,\ldots$ and vanishes on all but finitely many $\Delta_n$ induces a bounded Fourier multiplier $M_\psi$ on $\L^p(G)$ with $\norm{M_\psi}_{\L^p(G) \to \L^p(G)} \leq C_p \norm{\psi}_{\L^\infty(\hat{G})}$. For any integer $n$, consider the function $\phi_n \ov{\mathrm{def}}{=} \sum_{k=0}^n 1_{\Delta_{2k+1}}$ defined on $\hat{G}$. Since $\norm{\phi_n}_{\L^\infty(\hat{G})} \leq 1$, we have $\norm{M_{\phi_n}}_{\L^p(G) \to \L^p(G)} \leq C_p$. Consider the function $\phi \ov{\mathrm{def}}{=} \sum_{n=0}^{\infty} (1_{Y_{2n+1}} - 1_{Y_{2n}})$ of $\L^\infty(\hat{G})$. Since $\phi_n(x) \to \phi(x)$ as $n \to \infty$ for any $x \in \hat{G}$, we conclude using Proposition \ref{Prop-Convergence-of-multipliers} that the Fourier multiplier $M_\phi$ is bounded on $\L^p(G)$, $1 < p  < \infty$. 

Now, we prove that $M_\phi$ is strongly non regular. According to \cite[Theorem 3.1]{ArV}, it suffices to show that $\phi$ does not belong to the closure of the Fourier-Stieltjes algebra $\B(\hat{G})$ in $\L^\infty(\hat{G})$-norm. For this in turn, it suffices to find a sequence of measures $\mu_n$ on $\hat{G}$ with the properties
\begin{enumerate}
\item $\norm{\mu_n}_{\M(\hat{G})} \leq 2$,
\item $\widehat{\mu_n}(s) \xra[n \to +\infty]{} 0$ for any $s \in G$,
\item $\dsp \int_{\hat{G}} \phi \d\mu_n \centernot{\xra[n \to +\infty]{}} 0$.
\end{enumerate}
We choose the sequence $(\mu_n)$ defined by 
$$
\mu_n
\overset{\mathrm{def}}=\frac{1}{|Y_{n+1}|} \sum_{x \in Y_{n+1}} \delta_{x}-\frac{1}{|Y_n|} \sum_{x \in Y_n} \delta_{x}.
$$
Then property 1 is clearly satisfied, since the Haar measure on $\hat{G}$ is the counting measure.

For property 2, we have $\widehat{\mu_n} = 1_{G_{n+1}}-1_{G_n}$, where $G_n$ is the annihilator of $Y_n$ in $G$, i.e. $G_n = Y_n^\bot = \{ s \in G: \: \xi(s) = 1 \text{ for all } \xi \in Y_n \}$. 
\begin{itemize}
	\item If $s = 0$ then $\widehat{\mu_n}(s) = 1_{G_{n+1}}(0) - 1_{G_n}(0) = 0$ for all $n \in \N$. 

\item Consider now the case $s \in G \backslash \{0\}$. Recall that we have seen in the proof of Proposition \ref{prop-vector-valued-LP} that $\bigcap_{n \geq 0} G_n = \{0\}$. Hence, by the fact that the $Y_n$ increase and thus the $G_n$ decrease, there exists an index $n_0$ such that $s \not\in G_n$ for any $n \geq n_0$. Therefore, $\widehat{\mu_n}(s) = 0$ for any $n \geq n_0$.
\end{itemize}
It remains to show property 3. We have
\begin{align*}
\MoveEqLeft
 \int_{\hat{G}} \phi \d\mu_{2n} 
=   \frac{1}{|Y_{2n+1}|} \sum_{x \in Y_{2n+1}} \phi(x)-\frac{1}{|Y_{2n}|} \sum_{x \in Y_{2n}} \phi(x)\\
		&=\frac{1}{|Y_{2n+1}|} \sum_{x \in Y_{2n+1}} \bigg(\sum_{k=0}^\infty (1_{Y_{2k+1}} - 1_{Y_{2k}})\bigg)(x)-\frac{1}{|Y_{2n}|} \sum_{x \in Y_{2n}} \bigg(\sum_{k=0}^\infty (1_{Y_{2k+1}} - 1_{Y_{2k}})\bigg)(x)\\
		&=\frac{1}{|Y_{2n+1}|} \sum_{x \in Y_{2n+1}} \bigg(\sum_{k=0}^\infty 1_{Y_{2k+1}}(x) - 1_{Y_{2k}}(x)\bigg)-\frac{1}{|Y_{2n}|} \sum_{x \in Y_{2n}} \bigg(\sum_{k=0}^\infty 1_{Y_{2k+1}}(x) - 1_{Y_{2k}}(x)\bigg)\\
		&=\frac{1}{|Y_{2n+1}|} \sum_{x \in Y_{2n+1}} \sum_{k=0}^{n} \big(1_{Y_{2k+1}}(x) - 1_{Y_{2k}}(x)\big)-\frac{1}{|Y_{2n}|} \sum_{x \in Y_{2n}} \sum_{k=0}^{n-1} \big(1_{Y_{2k+1}}(x) - 1_{Y_{2k}}(x)\big)\\
	 &=\frac{1}{|Y_{2n+1}|} \bigg(\sum_{k=0}^{n} \sum_{x \in Y_{2n+1}} 1_{Y_{2k+1}}(x) - \sum_{x \in Y_{2n+1}}1_{Y_{2k}}(x)\bigg)\\
	&-\frac{1}{|Y_{2n}|}  \bigg(\sum_{k=0}^{n-1} \sum_{x \in Y_{2n}}1_{Y_{2k+1}}(x) - \sum_{x \in Y_{2n}}1_{Y_{2k}}(x)\bigg)\\
&=\frac{1}{|Y_{2n+1}|} \sum_{k=0}^n \big( |Y_{2k+1}| - |Y_{2k}| \big) - \frac{1}{|Y_{2n}|} \sum_{k=0}^{n-1} \big( |Y_{2k+1}| - |Y_{2k}| \big) \\
& = \left( \frac{1}{|Y_{2n+1}|} - \frac{1}{|Y_{2n}|} \right) \sum_{k=0}^{n-1} \big( |Y_{2k+1}| - |Y_{2k}| \big) + \frac{1}{|Y_{2n+1}|} \big( |Y_{2n+1}| - |Y_{2n}| \big) \\
& = 1 - \frac{|Y_{2n}|}{|Y_{2n+1}|} + \left( \frac{1}{|Y_{2n+1}|} - \frac{1}{|Y_{2n}|} \right) \sum_{k=0}^{n-1} \big( |Y_{2k+1}| - |Y_{2k}| \big) \\
& \geq 1 - \frac{|Y_{2n}|}{|Y_{2n+1}|} - \frac{1}{|Y_{2n}|} \sum_{k=0}^{n-1} \big( |Y_{2k+1}| - |Y_{2k}| \big).
\end{align*}
The second term in the last line is smaller than $1/2$ in modulus by the fact that the $Y_n$ increase strictly and the fact that the order of a subgroup divides the order of the whole group. For the third term, we note that by skipping several indices $n$ we can assume recursively that $|Y_{2n}|$ is so large that $\frac{1}{|Y_{2n}|} \sum_{k=0}^{n-1} \left( |Y_{2k+1}| - |Y_{2k}| \right) < \frac14$. Thus the whole expression in the last line is bigger than $1 - \frac12 - \frac14 = \frac14$ and hence does not converge to $0$.

According to Proposition \ref{lemma-charact-Lp-cb}, it suffices now to show that $M_\phi \otimes \Id_{S^p}$ extends to a bounded operator on the Bochner space $\L^p(G,S^p)$. Using both inequalities of Proposition \ref{prop-vector-valued-LP}, the fact that $S^p$ has UMD and Kahane's contraction principle \cite[Proposition 2.5]{KW} for the scalars $\delta_{n \text{ even}}$, we get
\begingroup
\allowdisplaybreaks
\begin{align*}
\MoveEqLeft
\bnorm{(M_\phi \ot \Id_{S^p}) f}_{\L^p(G,S^p)} 
 \lesssim \E \Bgnorm{ \sum_{n = 0}^\infty \epsi_n \ot (M_{1_{\Delta_n}} \ot \Id_{S^p}) (M_\phi \ot \Id_{S^p}) (f)}_{\L^p(G,S^p)} \\
 &= \E \Bgnorm{ \sum_{n = 0}^\infty \epsi_n \ot \delta_{n \text{ even}} (M_{1_{\Delta_n}} \ot \Id_{S^p})(f)}_{\L^p(G,S^p)}\\ 
 &\leq \E \Bgnorm{ \sum_{n = 0}^\infty \epsi_n \ot (M_{1_{\Delta_n}} \ot \Id_{S^p})(f)}_{\L^p(G,S^p)} 
 \lesssim \norm{f}_{\L^p(G,S^p)}.
\end{align*}
\endgroup
The proof is complete.
\end{proof}

Recall that a topological space $X$ is 0-dimensional if $X$ is a non-empty $\mathrm{T}_1$-space and if the family of all sets that are both open and closed is a basis for the topology \cite[page 11]{HR} \cite[page 360]{Eng}. By \cite[Theorem 6.2.1]{Eng}, every 0-dimensional space is totally disconnected, i.e. $X$ does not contain any connected subsets of cardinality larger than one.

\begin{prop}
\label{prop-existence-strongly-non-regular-discrete-group}
Let $G$ be an infinite discrete abelian group. Suppose $1<p<\infty$. Then there exists a strongly non regular completely bounded Fourier multiplier on $\L^p(G)$.
\end{prop}

\begin{proof}
Suppose first that $G$ contains an element of infinite order, so a (closed) subgroup $H$ isomorphic with $\Z$. Then by \cite[Example 3.4]{ArV}, the Hilbert transform induces a strongly non regular Fourier multiplier on $\L^p(H)$. Since $S^p$ is UMD and according to \cite[Theorem 2.8]{BGM}, the Hilbert transform is bounded on $\L^p(H,S^p)$ so completely bounded on $\L^p(H)$ by Proposition \ref{lemma-charact-Lp-cb}. Now, using Proposition \ref{prop-transfer-subgroups} and the isometry \eqref{Deri-injection}, the composed Fourier multiplier $M_{\phi \circ \pi}$ on $\L^p(G)$, where $\pi \co \hat{G} \to \hat{G} / H^\bot$ is the canonical map, is a strongly non regular completely bounded Fourier multiplier.

Now suppose that every element of $G$ is of finite order, so $G$ is a torsion group.
We can assume that $G$ is countably infinite. Indeed, otherwise choose a countably infinite number of elements in $G,$ and let $H$ be the subgroup of $G$ generated by these elements. Then $H$ is again countably infinite. If there is a strongly non regular completely bounded Fourier multiplier on $\L^p(H)$ then Proposition \ref{prop-transfer-subgroups} and the isometry \eqref{Deri-injection} yield a strongly non regular Fourier multiplier on $\L^p(G)$.

Note that since $G$ is countably infinite, by \cite[Theorem 24.15]{HR}, its dual $\hat{G}$ is metrizable. The fact that $G$ is torsion implies by \cite[Theorem 24.21]{HR} that $\hat{G}$ is $0$-dimensional. This in turn implies that $\hat{G}$ is totally disconnected.

So $\hat{G}$ is an infinite compact abelian metrizable totally disconnected group. By the second part of \cite[Remark page~68]{EdG}, there exists a sequence $(H_n)_{n \geq 0}$ of closed subgroups of $\hat{G}$ such that
\begin{enumerate}
\item each $H_n$ is open,
\item $H_{n+1} \subsetneqq H_n$,
\item $\bigcap_{n=0}^\infty H_n = \{0\}$, $H_0 = \hat{G}$.
\end{enumerate}
Then the sets $\Delta_n = H_n \backslash H_{n+1}$ enjoy the Littlewood-Paley equivalence \eqref{equ-Littlewood-Paley-equivalence} according to Proposition \ref{prop-vector-valued-LP}. With $\phi = \sum_{n=1}^\infty (1_{H_{2n-1}} - 1_{H_{2n}})$, as in the proof of Proposition \ref{prop-existence-strongly-non-regular-compact-group}, we see that $M_\phi$ is a bounded Fourier multiplier on $\L^p(G)$, $1 < p  < \infty$.

It remains to show that $M_\phi$ is strongly non regular. Invoking \cite[Theorem 3.1 and Remark 3.2]{ArV}, it suffices to show that $\phi$ is not equal almost everywhere to a continuous function.

So assume that $\psi \co \hat{G} \to \C$ is a continuous function with $\psi = \phi$ almost everywhere. We will show a contradiction, which will end the proof. Since the $H_n$ are closed and open by the point 1, $H_{n-1} \backslash H_{n}$ is open. As it is also non-empty by the point 2, it must be of positive Haar measure. Therefore, there exists $x_n \in H_{n-1} \backslash H_n$ with 
$$
\psi(x_n) 
=\phi(x_n) 
= 
\begin{cases} 
0 & n \text{ even} \\ 
1 & n \text{ odd}
\end{cases}.
$$
Consider now the sequence $y_n = x_{2n-1}$. By compactness, there exists a subsequence of $y_n$ which converges against some $\xi \in \hat{G}$. Since $y_n$ belongs to $H_{2n-1},$ by the point 2, $y_m$ belongs to $H_{2n-1}$ for all $m \geq n$. As $H_{2n-1}$ is closed, $\xi$ belongs to $H_{2n-1}$, so to $\bigcap_{n=1}^\infty H_{2n-1} = \bigcap_{n=1}^\infty H_n = \{0\}$. Therefore, a subsequence of $y_n$ converges to $0$.

In the same manner, one shows that a subsequence of $x_{2n}$ converges to $0$. However, $\psi$ applied to these two subsequences is constant to 1 and to 0 respectively, so does not converge. Hence $\psi$ cannot be continuous.

Now use Proposition \ref{prop-vector-valued-LP} in a similar fashion to the compact case to deduce that $M_\phi$ is completely bounded on $\L^p(G)$. The proof is complete.
\end{proof}

Recall the following structure theorem for locally compact abelian groups, see e.g. \cite[Theorem 24.30]{HR} and \cite[Theorem 4.2.31]{Rei}.

\begin{thm}
\label{Th structure} Any locally compact abelian group is isomorphic to a product $\R^n\times G_0$ where $n\geq 0$ is an integer and $G_0$ is a locally compact abelian group containing a compact subgroup $K$ such that $G_0/K$ is discrete.
\end{thm}

With the help of the previous theorem, we can now prove the following.

\begin{thm}
\label{thm-existence-strongly-non-decomposable-abelian-groups}
\label{thm-existence-CB-strongly-non-decomposable-abelian-groups}
Let $G$ be an infinite locally compact abelian group. Suppose $1 < p < \infty$. Then there exists a strongly non regular Fourier multiplier on $\L^p(G)$ which is completely bounded and $\CB$-strongly non decomposable.
\end{thm}

\begin{proof}
We use the previous structure Theorem \ref{Th structure} to decompose $G$ and we distinguish three cases. 

If $n \geq 1$ then $G$ has a closed subgroup $H$ isomorphic to $\R$ and we consider the Hilbert transform on $\L^p(H)$ which is strongly non regular by \cite[Example 3.3]{ArV}. Since the Schatten class $S^p$ has UMD, the Hilbert transform is bounded on $\L^p(H,S^p)$ and hence completely bounded on $\L^p(H)$ according to Proposition \ref{lemma-charact-Lp-cb}. Now appeal to the isometry \eqref{Deri-injection} and Proposition \ref{prop-transfer-subgroups} to extend the Hilbert transform to a strongly non regular and completely bounded Fourier multiplier on $\L^p(G)$.

If $n = 0$ then $G = G_0$. Suppose first that the compact subgroup $K$ is infinite.
Using Proposition \ref{prop-existence-strongly-non-regular-compact-group}, there exists a completely bounded Fourier multiplier which is strongly non regular. Again, using the isometry \eqref{Deri-injection} and Proposition \ref{prop-transfer-subgroups}, we obtain a strongly non regular and completely bounded Fourier multiplier on $\L^p(G)$. 

If $n = 0$ and if the compact subgroup $K$ is finite, then it is itself discrete (since it is Hausdorff) and thus $G = G_0$ is discrete and infinite. Now, use Proposition \ref{prop-existence-strongly-non-regular-discrete-group} to find a strongly non regular completely bounded Fourier multiplier on $\L^p(G)$.

The last assertion is a consequence of Section \ref{subsec:Existence-of-strongly-Definitions}.
\end{proof}

\subsection{Strongly non regular completely bounded convolutors on non-abelian groups}
\label{sec:Existence-of-strongly-non-regular-convolutors}

\begin{thm}
\label{Th-strongly-non-regular-convolutor--un-amenable-groups}
Let $G$ be a unimodular amenable locally compact group which contains an infinite abelian subgroup. Suppose $1<p<\infty$. There exists a strongly non regular completely bounded convolution operator $T \co \L^p(G) \to \L^p(G)$.
\end{thm}

\begin{proof}
Suppose that $G$ contains an infinite abelian group $H$. Note that the closure $\ovl{H}$ of $H$ is a closed abelian infinite subgroup of $G$. By Theorem \ref{thm-existence-CB-strongly-non-decomposable-abelian-groups}, there exists a strongly non regular completely bounded Fourier multiplier on $\L^p(\ovl{H})$. Since $G$ is amenable and unimodular, we conclude by using Proposition \ref{prop-transfer-subgroups}.
\end{proof}

\begin{cor}
\label{Cor-strongly-non-regular-convolutor-compact-groups}
Let $G$ be an infinite compact group. Suppose $1<p<\infty$. There exists a strongly non regular completely bounded convolution operator $T \co \L^p(G) \to \L^p(G)$.
\end{cor}

\begin{proof}
Note that $G$ is amenable \cite[Proposition 12.1]{Pie} and unimodular \cite[VII.12]{Bou3}. By \cite[Theorem 2]{Zel}, the infinite compact group $G$ contains an infinite abelian subgroup. Hence, we can use Theorem \ref{Th-strongly-non-regular-convolutor--un-amenable-groups}.
\end{proof}

A group $G$ is locally finite if each finitely generated subgroup is finite, see \cite[page 422]{Rob}. A locally compact group $G$ is called topologically locally finite if
the closure of every finitely generated subgroup of $G$ is compact \cite[Section 2]{Cap1}.

\begin{cor}
\label{Th-strongly-non-regular-convolutor-locally-finite-groups}
Let $G$ be an infinite unimodular locally finite locally compact group. Suppose $1<p<\infty$. There exists a strongly non regular completely bounded convolution operator $T \co \L^p(G) \to \L^p(G)$. 
\end{cor}

\begin{proof}
Observe that a locally finite locally compact group is topologically locally finite, hence amenable by \cite[Corollary 2.4]{Cap1}. By \cite[Theorem 14.3.7]{Rob}, such a group has an infinite abelian subgroup.
We conclude with Theorem \ref{Th-strongly-non-regular-convolutor--un-amenable-groups}.
\end{proof}

\begin{cor}
\label{Cor-strongly-non-regular-convolutor-nilpotent-groups}
Let $G$ be an infinite nilpotent locally compact group. Suppose $1<p<\infty$. There exists a strongly non regular completely bounded convolution operator $T \co \L^p(G) \to \L^p(G)$.
\end{cor}

\begin{proof}
Such a group is unimodular \cite{MS1} (see also \cite[page 53]{Fol1} in the connected case) and amenable \cite[Corollary 13.5]{Pie} since it is solvable. Now, if $G$ is locally finite we can use Corollary \ref{Th-strongly-non-regular-convolutor-locally-finite-groups}. Otherwise, $G$ contains an infinite finitely generated subgroup which is nilpotent as a subgroup of a nilpotent group. By \cite[Lemma 8.2.2]{ECHLPT}, this group has an element of infinite order, so also contains an infinite abelian subgroup.
\end{proof}

Finally, since a discrete group is unimodular \cite[VII.12]{Bou3}, we obtain the following result.

\begin{cor}
\label{Cor-strongly-non-regular-convolutor-discrete-groups}
Let $G$ be an amenable discrete group which contains an infinite abelian subgroup. Suppose $1<p<\infty$. There exists a strongly non regular completely bounded convolution operator $T \co \L^p(G) \to \L^p(G)$.
\end{cor}

\subsection{CB-strongly non decomposable Schur multipliers}
\label{sec:Existence-of-strongly-non-Schur-multipliers}

We start with a result which gives a manageable condition which is necessary for that a completely bounded Schur multiplier belong to the closure of the space of decomposable operators.

\begin{prop}
\label{prop-closure-dec-closure-mult}
Suppose $1<p<\infty$. If the Schur multiplier $M_\phi \co S^p_I \to S^p_I$ is completely bounded and belongs to the closure $\ovl{\Dec(S^p_I)}^{\CB(S^p_I)}$ of the space $\Dec(S^p_I)$ with respect to the completely bounded norm then $M_\phi$ belongs to the closure $\ovl{\mathfrak{M}_I^{\infty}}^{\ell^\infty_{I \times I}}$ of the space $\mathfrak{M}_I^{\infty}$ in the Banach space $\ell^\infty_{I \times I}$.
\end{prop}

\begin{proof}
Let $R \co S^p_I \to S^p_I$ be a decomposable operator. By Proposition \ref{prop-decomposable-se-decompose-en-cp}, we can write $R=R_1-R_2+\mathrm{i}(R_3-R_4)$ where each $R_j$ is a completely positive map on $S^p_I$. Using the projection $P_I \co \CB(S^p_I) \to \mathfrak{M}_{I}^{p,\cb}$ of Corollary \ref{Prop-complementation-Schur}, we obtain
$$
P_I(R)
=P_I\big(R_1-R_2+\i(R_3-R_4)\big)
=P_I(R_1)-P_I(R_2)+\i\big(P_I(R_3)-P_I(R_4)\big).
$$
By Proposition \ref{prop-linearcp-imply-decomposable}, we conclude that the Schur multiplier $P_I(R)$ is decomposable. By Proposition \ref{prop-continuous-Schur-multiplier-dec-infty}, we infer that $P_I(R)$ is bounded on $S^\infty_I$, i.e. belongs to $\mathfrak{M}_I^{\infty}$. According to Proposition \ref{prop-continuous-Schur-multiplier-dec-infty}, it also belongs to $\mathfrak{M}^{\infty,\cb}_I$ with same norm. Now, using the contractivity of $P_I$, we have 
\begin{align*}
\MoveEqLeft
\bnorm{M_\phi-R}_{\cb, S^p_I \to S^p_I} 
\geq \bnorm{P_I(M_\phi-R)}_{\cb,S^p_I \to S^p_I}
=\bnorm{P_I(M_\phi)-P_I(R)}_{\cb,S^p_I \to S^p_I}\\
		&=\bnorm{M_\phi-P_I(R)}_{\cb,S^p_I \to S^p_I}
		\geq \bnorm{M_\phi-P_I(R)}_{S^2_I \to S^2_I}
		\geq \dist_{\ell^\infty_{I \times I}}(M_\phi,\mathfrak{M}_I^{\infty}).
\end{align*}
Hence, we deduce that
$$
\dist_{\CB(S^p_I)}\big(M_\phi,\Dec(S^p_I)\big)
\geq \dist_{\ell^\infty_{I \times I}}(M_\phi,\mathfrak{M}_I^{\infty}).
$$
\end{proof}


It is folklore that if $M_A \co \B(\ell^2) \to \B(\ell^2)$ is a bounded Schur multiplier and the limits
$$
\lim_{i \to \infty} \lim_{j \to \infty} a_{ij}=s
\text{ and } 
\lim_{j \to \infty} \lim_{i\to \infty} a_{ij}=t
$$
exist then $s=t$, see \cite[Ex 8.15 page 118]{Pau}. This property turns out to be also true for Schur multipliers belonging to the \textit{closure} $\ovl{\mathfrak{M}^\infty}^{\ell^\infty_{\N \times \N}}$.

\begin{prop}
\label{Prop-Schur-limits-variant}
Let $M_A \in \ovl{\mathfrak{M}^\infty}^{\ell^\infty_{\N \times \N}}$. If the limits
$$
\lim_{i \to \infty} \lim_{j \to \infty} a_{ij}=s
\text{ and } 
\lim_{j \to \infty} \lim_{i \to \infty} a_{ij}=t
$$
exist then $s=t$.
\end{prop}

\begin{proof}
Let $\epsi > 0$ and let $[b_{ij}]$ be a matrix corresponding to a bounded Schur multiplier $M_B \co \B(\ell^2) \to \B(\ell^2)$, such that $|b_{ij} - a_{ij}| \leq \epsi$ for any $i,j \in \N$. By the description \cite[Corollary 8.8]{Pau} of bounded Schur multipliers $\B(\ell^2) \to \B(\ell^2)$, there exist a Hilbert space $H$, some bounded sequences $(x_i)_{}$ and $(y_j)_{}$ of elements of $H$ such that $b_{ij} = \langle x_i , y_j \rangle$ for any $i,j \in \N$. By the weak compactness of closed bounded subsets of $H,$ there exist subsequences $i_k$ and $j_l$ and $x,y \in H$ such that weak-$\lim_k x_{i_k} = x$ and weak-$\lim_l y_{j_l} = y$. Thus, we have 
$$
\lim_{k \to +\infty} b_{i_k j_l} = \lim_{k \to +\infty }\langle x_{i_k} , y_{j_l} \rangle=\langle x , y_{j_l} \rangle
$$ 
and finally
$$
\lim_{l \to +\infty} \lim_{k \to +\infty} b_{i_k j_l} 
=\lim_{l \to +\infty}\langle x , y_{j_l} \rangle
=\langle x , y \rangle.
$$ 
By the same reasoning, we also have $\dsp \lim_{k \to +\infty} \lim_{l \to +\infty} b_{i_k j_l} 
= \langle x , y \rangle$. Now, we infer that 
$$
\left|\lim_{k \to +\infty} b_{i_k j_l} - \lim_{k \to +\infty} a_{i_k j_l}\right| 
\leq \epsi 
\quad \text{ and thus }\quad 
\left|\lim_{l \to +\infty} \lim_{k \to +\infty} b_{i_k j_l} - t\right| 
\leq \epsi.
$$ 
Similarly, we have 
$$
\left|\lim_{k \to +\infty} \lim_{l \to +\infty}b_{i_k j_l} - s\right| \leq \epsi.
$$ 
We infer that $|s-t| \leq 2 \epsi$. Letting $\epsi$ go to zero yields the proposition.
\end{proof}

Recall \cite[Section 6]{NR} that the triangular truncation $\mathcal{T} \co S^p \to S^p$ and the discrete noncommutative Hilbert transform $\mathcal{H} \co S^p \to S^p$ are completely bounded Schur multipliers defined by $\mathcal{T}([a_{ij}])=[\delta_{i \leq j} a_{ij}]$ and that $\mathcal{H}([a_{ij}])=[-\mathrm{i}\delta_{i < j}  a_{ij}+\mathrm{i}\delta_{i > j}a_{ij}]$ for any $[a_{ij}] \in S^p$ where $\i^2=-1$.
The fact that $\mathcal{T}$ and $\mathcal{H}$ are completely bounded on $S^p$ can be found in \cite[Section 6]{NR}.

From the last two propositions, we deduce the following result.

\begin{cor}
\label{Cor-strongly-non-reg-truncation}
The triangular truncation $\mathcal{T} \co S^p \to S^p$ and the discrete noncommutative Hilbert transform $\mathcal{H} \co S^p \to S^p$ are $\CB$-strongly non decomposable.
\end{cor}

\subsection{CB-strongly non decomposable Fourier multipliers}
\label{sec:Existence-of-strongly-non-Fourier-multipliers}


We start with a transference result.


\begin{prop}
\label{prop-transfer-between-discrete}
Let $G$ and $H$ be two discrete groups such that $H$ is a subgroup of $G$. If $\varphi \co H \to \C$ is a complex function, we denote by $\widetilde{\varphi} \co G \to \C$ the extension of $\varphi$ on $G$ which is zero off $H$. Suppose $1<p<\infty$ and that $\VN(G)$ has $\QWEP$. If $\varphi$ induces a $\CB$-strongly non decomposable Fourier multiplier $M_{\varphi} \co \L^p(\VN(H)) \to \L^p(\VN(H))$ then $\widetilde{\varphi}$ induces a $\CB$-strongly non decomposable Fourier multiplier $M_{\widetilde{\varphi}} \co \L^p(\VN(G)) \to \L^p(\VN(G))$.
\end{prop}

\begin{proof}
Let $\E$ be the trace preserving conditional expectation from $\VN(G)$ onto $\VN(H)$ and $J$ be the canonical inclusion of $\VN(H)$ into $\VN(G)$. The map $J M_{\varphi}\E$ is completely bounded on $\L^p(\VN(G))$ and is clearly equal to the Fourier multiplier $M_{\widetilde{\varphi}}$ induced by $\widetilde{\varphi}$. Suppose that $M_{\widetilde{\varphi}}$ belongs to $\ovl{\Dec(\L^p(\VN(G)))}^{\CB(\L^p(\VN(G)))}$. Let $\epsi > 0$. Then there exist some completely positive maps $R_1,R_2,R_3,R_4 \co \L^p(\VN(G)) \to \L^p(\VN(G))$ and a completely bounded map $R \co \L^p(\VN(G)) \to \L^p(\VN(G))$ of completely bounded norm less than $\epsi$ such that $M_{\widetilde{\varphi}}=R_1-R_2+\i(R_3-R_4)+R$. For any $h \in H$, we have
$$
\tau_G\big(M_{\widetilde{\varphi}}(\lambda_{h})(\lambda_{h})^*\big)=\widetilde{\varphi}(h)\tau_G\big(\lambda_{h}(\lambda_{h})^*\big)=\varphi(h).
$$
Hence, using the map $P_H^p$ given by Corollary \ref{Prop-complementation-Fourier-subgroups} since $\VN(G)$ is QWEP, we obtain
\begin{align*}
\MoveEqLeft
 M_{\varphi}
    =P_H^p\big(M_{\widetilde{\varphi}}\big)
    =P_H^p\big(R_1-R_2+\i(R_3-R_4)+R\big)\\ 
		&=P_H^p(R_1)-P_H^p(R_2)+\i\big( P_H^p(R_3)-P_H^p(R_4)\big)+P_H^p(R).
\end{align*}
Moreover, by the contractivity of $P_H$, the Fourier multiplier $P_H^p(R) \co \L^p(\VN(H)) \to \L^p(\VN(H))$ is completely bounded of completely bounded norm less than $\epsi$. Furthermore, each Fourier multiplier $P_H^p(R_i) \co \L^p(\VN(H)) \to \L^p(\VN(H))$ is completely positive. It follows that $M_{\varphi}$ is $\epsi$-close to $\Dec(\L^p(\VN(H)))$ in the Banach space $\CB(\L^p(\VN(H)))$. So letting $\epsi \to 0$ yields that $M_{\varphi}$ belongs to $\ovl{\Dec(\L^p(\VN(H)))}^{\CB(\L^p(\VN(H)))}$. This is the desired contradiction.
\end{proof}

\begin{cor}
\label{Coro-Discrete-strongly-non-dec}
Let $G$ be a discrete group which contains an infinite abelian subgroup such that $\VN(G)$ is $\QWEP$. Suppose $1<p<\infty$. There exists a $\CB$-strongly non decomposable Fourier multiplier on $\L^p(\VN(G))$.
\end{cor}

\begin{proof}
It suffices to use Proposition \ref{prop-transfer-between-discrete}, Theorem \ref{thm-existence-CB-strongly-non-decomposable-abelian-groups} and Remark \ref{rem-defi-CB-strongly-non-decomposable}.
\end{proof}

For example, consider $1 < p < \infty$, $n \in \N$ and the free group $G = \F_n$ of $n$ generators. Since $\VN(\F_n)$ is $\QWEP$, there exists a $\CB$-strongly non decomposable Fourier multiplier on $\L^p(\VN(\F_n))$. The next criterion allows us to give concrete examples in Proposition \ref{Prop-concrete-free-groups} and Proposition \ref{prop-free-Hilbert-transform}.

\begin{prop}
\label{Prop-strongly-non-dec-Fourier-mult}
Let $G$ be a unimodular locally compact group. Suppose $1 \leq p \leq \infty$.
\begin{enumerate}
\item Let $\varphi \co G \to \C$ be a complex function inducing a completely bounded Fourier multiplier on $\L^p(\VN(G))$. Suppose that there exists a bounded, complete positivity preserving mapping $P_G^p \co \CB(\L^p(\VN(G))) \to \mathfrak{M}^{p,\cb}(G)$, such that $P_G^p(M_\varphi) = M_\varphi$. If $M_\varphi \in \ovl{\Dec(\L^p(\VN(G)))}^{\CB(\L^p(\VN(G)))}$ then the multiplier $M_\varphi$ belongs to $\ovl{\mathfrak{M}^{\infty,\cb}(G)}^{\L^\infty(G)}$.

\item Assume that the limits $\dsp\lim_{n \to +\infty} \varphi(s^n)$ and $\dsp\lim_{n \to +\infty} \varphi(s^{-n})$ exist for some $s \in G$ and that $M_\varphi$ belongs to the closure $\ovl{\mathfrak{M}^{\infty,\cb}(G)}^{\L^\infty(G)}$ for some measurable $\varphi \co G \to \C$. Then 
$$
\lim_{n \to +\infty} \varphi(s^n) 
=\lim_{n \to +\infty} \varphi(s^{-n}).
$$
\end{enumerate}
\end{prop}

\begin{proof}
1. Let $R \co \L^p(\VN(G)) \to \L^p(\VN(G))$ be a decomposable operator. By Proposition \ref{prop-decomposable-se-decompose-en-cp}, we can write
$$
R
=R_1-R_2+\i(R_3-R_4)
$$
where each $R_j$ is a completely positive map on $\L^p(\VN(G))$. Using the mapping $P_G^p$ from the statement of the proposition, we obtain
$$
P_G^p(R)
=P_G^p\big(R_1-R_2+\i(R_3-R_4)\big)
=P_G^p(R_1)-P_G^p(R_2)+\i\big(P_G^p(R_3)-P_G^p(R_4)\big).
$$
Using Proposition \ref{Th-cp-Fourier-multipliers}, we see that the Fourier multiplier $P_G^p(R)$ is decomposable on $\VN(G)$ and in particular completely bounded by Proposition \ref{Prop-cb-leq-dec}. Now, using the boundedness of $P_G^p$ and Lemma \ref{prop-M2-Fourier-multipliers}, we obtain
\begin{align*}
\MoveEqLeft
\bnorm{P_G^p} \, \bnorm{M_\varphi-R}_{\cb,\L^p(\VN(G)) \to \L^p(\VN(G))} 
\geq \bnorm{P_G^p(M_\varphi-R)}_{\cb,\L^p(\VN(G)) \to \L^p(\VN(G))}\\
		&=\bnorm{P_G^p(M_\varphi)-P_G^p(R)}_{\cb,\L^p(\VN(G)) \to \L^p(\VN(G))}
		=\bnorm{M_\varphi-P_G^p(R)}_{\cb, \L^p(\VN(G))\to \L^p(\VN(G))}\\
		&\geq \bnorm{M_\varphi-P_G^p(R)}_{\L^2(\VN(G)) \to \L^2(\VN(G))}
		\geq \dist_{\L^\infty(G)}\big(M_\varphi,\mathfrak{M}_G^{\infty,\cb}\big).
\end{align*}
Hence, we deduce that
$$
\bnorm{P_G^p} \, \dist_{\CB(\L^p(\VN(G)))}\big(M_\varphi,\Dec(\L^p(\VN(G)))\big)
\geq \dist_{\L^\infty(G)}\big(M_\phi,\mathfrak{M}_G^{\infty,\cb}\big).
$$

2. Suppose that $M_\varphi$ belongs to $\ovl{\mathfrak{M}^{\infty,\cb}(G)}^{\L^\infty(G)}$. Let $\epsi > 0$ and $M_\psi \in \mathfrak{M}^{\infty,\cb}(G)$ with $\norm{\varphi-\psi}_\infty \leq \epsi$. According to \cite[page 2]{Ste}, there exist a Hilbert space $H$ and two maps $P,Q \co G \to H$ with $\norm{P}_\infty = \sup_{r \in G} \|P(r)\|_H,\: \norm{Q}_\infty = \sup_{t \in G}\norm{Q(t)}_H < \infty$ such that $\psi(rt^{-1}) = \big\langle P(r),Q(t) \big\rangle_H$ for any $r,t \in G$. The sequences $(P(s^i))_{i \geq 0}$ and $(Q(s^{j}))_{j \geq 0}$ are bounded in $H$ and thus admit weak* convergent subsequences $(P(s^{i_k}))$ and $(Q(s^{j_l}))$ to some elements $h_1$ and $h_2$ of $H$. Thus, for any $l$, we have 
$$
\lim_{k \to +\infty} \psi\big(s^{i_k-j_l}\big)  
= \lim_{k \to +\infty} \big\langle P(s^{i_k}) , Q(s^{j_l}) \big\rangle \\
= \big\langle h_1 , Q(s^{j_l}) \big\rangle,
$$ 
which implies
 \begin{align*}
\MoveEqLeft
\lim_{l \to +\infty} \lim_{k \to +\infty} \psi(s^{i_k-j_l})  
=\lim_{l \to +\infty}\big\langle h_1 , Q(s^{j_l}) \big\rangle 
=\langle h_1, h_2 \rangle.
\end{align*}
We obtain similarly that $\lim_{k \to +\infty} \lim_{l \to +\infty} \psi(s^{i_k-j_l})
= \langle h_1 , h_2 \rangle$. But by $\norm{\varphi - \psi}_\infty \leq \epsi$, we deduce that
$$
\left|\lim_{k \to +\infty} \varphi\big(s^{i_k-j_l}\big) - \lim_{k \to +\infty} \psi\big(s^{i_k-j_l}\big)\right| 
\leq \epsi \ \ \text{and thus}\ \
\left|\lim_{n \to +\infty} \varphi(s^n) - \lim_{l \to +\infty} \lim_{k \to +\infty}\psi\big(s^{i_k-j_l}\big)\right| 
\leq \epsi.
$$ 
Similarly, we have 
$$
\left|\lim_{n \to +\infty} \varphi(s^{-n}) - \lim_{k \to +\infty} \lim_{l \to +\infty} \psi\big(s^{i_k-j_l}\big)\right| 
\leq \epsi.
$$ 
Hence the limit $\lim_{n \to +\infty} \varphi(s^n)$ is $2\epsi$-close to $\lim_{n \to +\infty} \varphi(s^{-n})$. We deduce 2. by letting $\epsi \to 0$.
\end{proof}

\begin{thm}
\label{thm-strongly-non-decomposable-Fourier-multipliers-quotient-compact-subgroup} Let $G$ be a second countable amenable locally compact group and $H$ be a normal open (and then also closed) subgroup of $G$ (so $G/H$ is discrete). Let $\pi \co G \to G/H$ be the canonical map and $\varphi \co G/H \to \C$ be a continuous bounded complex function. Suppose $1<p<\infty$. If the complex function $\varphi \circ \pi \co G \to \C$ induces a $\CB$-strongly non decomposable Fourier multiplier $M_{\varphi \circ \pi} \co \L^p(\VN(G)) \to \L^p(\VN(G))$ then $\varphi$ induces a $\CB$-strongly non decomposable Fourier multiplier $M_{\varphi} \co \L^p(\VN(G/H)) \to \L^p(\VN(G/H))$.
\end{thm}

\begin{proof}
Note that the Fourier multiplier $M_\varphi$ is completely bounded by Theorem \ref{prop-homomorphism-Fourier-multipliers-cb}. Suppose that $M_{\varphi}$ belongs to $\overline{\Dec(\L^p(\VN(G/H)))}^{\CB(\L^p(\VN(G/H)))}$. Let $\epsi > 0$. Then, by Proposition \ref{prop-decomposable-se-decompose-en-cp}, there exist some completely positive maps $R_1,R_2,R_3,R_4 \co \L^p(\VN(G/H)) \to \L^p(\VN(G/H))$ and a completely bounded map $R \co \L^p(\VN(G/H)) \to \L^p(\VN(G/H))$ of completely bounded norm less than $\epsi$ such that $M_{\varphi}=R_1-R_2+\i(R_3-R_4)+R$.

Corollary \ref{Prop-complementation-Fourier-subgroups} yields the existence of some complex functions $\varphi_1,\varphi_2,\varphi_3,\varphi_4$ and $\psi$ such that $M_{\varphi} =M_{\varphi_1}-M_{\varphi_2}+\i(M_{\varphi_3}-M_{\varphi_4})+M_\psi$ such that the Fourier multipliers $M_{\varphi_k}$ are completely positive on $\L^p(\VN(G/H))$ and $M_\psi$ is again of completely bounded norm less than $\epsi$. Since $G/H$ is discrete, the functions $\psi,\varphi_1,\varphi_2,\varphi_3,\varphi_4$ are continuous.  Then by Theorem \ref{prop-homomorphism-Fourier-multipliers-cb} it follows that $M_{\varphi_k \circ \pi} \co \L^p(\VN(G)) \to \L^p(\VN(G))$ is completely positive and the Fourier multiplier $M_{\psi \circ \pi} \co \L^p(\VN(G)) \to \L^p(\VN(G))$ is completely bounded of completely bounded norm less than $\epsi$. Since
$$
M_{\varphi \circ \pi} 
=M_{\varphi_1 \circ \pi}-M_{\varphi_2 \circ \pi}+\i\big(M_{\varphi_3 \circ \pi}-M_{\varphi_4 \circ \pi}\big)+M_{\psi \circ \pi}
$$
it follows that $M_{\varphi \circ \pi}$ is $\epsi$-close to $\Dec(\L^p(\VN(G)))$ in the Banach space $\CB(\L^p(\VN(G)))$, so that letting $\epsi \to 0$ yields that $M_{\varphi \circ \pi} \in \overline{\Dec(\L^p(\VN(G)))}^{\CB(\L^p(\VN(G)))}$. This is the desired contradiction.
\end{proof}

\paragraph{Riesz transforms.} An affine representation $(\mathcal{H},\alpha,b)$ of a discrete group $G$ is an orthogonal representation $\alpha \co G \to \mathrm{O}(\mathcal{H})$ over a real Hilbert space $\mathcal{H}$ together with a mapping $b \co G \to \mathcal{H}$ satisfying the cocycle condition $b(st) = \alpha_s(b(t)) + b(s)$ for any $s,t \in G$, see \cite[Definition 10.6]{PaR} and \cite{BHV}. In this situation, by \cite[Theorem 10.10]{PaR} the function $s \mapsto \norm{b(s)}_{\mathcal{H}}^2$ is conditionally of negative type, vanishes at the identity $e$ and is symmetric. We also refer to \cite{ArK2} for related information.
By \cite[page~532]{JMP2}, for any normalized vector $h \in \mathcal{H}$, we can consider the Riesz transform $R_h = M_\phi$ whose symbol $\phi \co G \to \R$ is defined by
\begin{equation}
\label{equ-symbol-of-Riesz-transform}
\phi(s) 
\ov{\mathrm{def}}{=} 
\begin{cases}
\dsp \frac{\langle b(s), h\rangle_{\mathcal{H}}}{\norm{b(s)}_{\mathcal{H}}} & \text{ if } b(s) \not= 0\\
0& \text{ if } b(s) = 0
\end{cases}.
\end{equation}
We will use the subgroup $G_0 \ov{\mathrm{def}}{=} \{s \in G \ :\ b(s)=0 \}$ of $G$. 

\begin{lemma}
\label{lem-Riesz-transform-cb}
Let $G$ be a discrete group equipped with an affine representation $(\mathcal{H},\alpha,b)$. Suppose $1 < p < \infty$. The symbol $\phi$ from \eqref{equ-symbol-of-Riesz-transform} induces a completely bounded operator $R_h = M_\phi \co \L^p(\VN(G)) \to \L^p(\VN(G))$.
\end{lemma}

\begin{proof}
It is essentially shown in \cite{JMP2} that $R_h$ is completely bounded on the subspace $\L^p_0(\VN(G)) \overset{\textrm{def}}= \textrm{Ran}(\Id_{\L^p(\VN(G))}-M_{1_{G_0}})$ of $\L^p(\VN(G))$. Indeed, consider some orthonormal basis $(e_j)$ of $\mathcal{H}$ with $e_1 = h$ and some independent Rademacher variables $\epsi_1,\epsi_2,\ldots$ on some probability space $\Omega_0$. For any $x \in S^p_m(\L^p_0(\VN(G)))$, using the inequalities \cite[Theorem A1 and Remark 1.8]{JMP2} for $p \in [2,\infty)$, we have 
\begin{align*}
\MoveEqLeft
   \norm{(\Id_{S^p_m} \ot R_h)(x)}_{S^p_m(\L^p(\VN(G)))} 
		\leq \norm{\sum_{i} \epsi_i \ot (\Id_{S^p_m} \ot R_{e_i})(x)}_{\L^p(\Omega_0,S^p_m(\L^p(\VN(G))))}\\
		&\approx \norm{\big((\Id_{S^p_m} \ot R_{e_i}) x\big)}_{\mathrm{RC}_p(S^p_m(\L^p(\VN(G))))}
		\lesssim \norm{x}_{S^p_m(\L^p(\VN(G)))}.
\end{align*}
Thus $R_h$ is completely bounded on $\L^p_0(\VN(G))$ for $p \in [2,\infty)$.
 
Since $G$ is discrete, the indicator function $1_{G_0}$ is continuous. Let $G/G_0$ denote the discrete space of left cosets of $G_0$ and consider the quasi-left regular representation $\pi_{G_0} \co G \to \B\big(\ell^2_{G/G_0}\big)$  given by $\pi_{G_0}(s)\delta_{tG_0} = \delta_{stG_0}$. For any $s \in G$, we can write $1_{G_0}(s) = \langle \pi_{G_0}(s)\delta_{G_0},\delta_{G_0} \rangle$. Consequently the indicator function $1_{G_0}$ is a continuous positive definite function. According to Proposition \ref{Th-cp-Fourier-multipliers}, this function induces a completely positive Fourier multiplier on $\L^p(\VN(G))$. We deduce that  $\Id_{\L^p(\VN(G))} - M_{1_{G_0}}$ is completely bounded on $\L^p(\VN(G))$. If $s \in G$ satisfies $\phi(s) \not= 0$, then $s$ does not belong to $G_0$, so $\phi = \phi \cdot (1-1_{G_0})$. Hence we can write 
$$
R_h 
= R_h \big(\Id_{\L^p(\VN(G))} - M_{1_N{G_0}}\big) 
= M_{\phi \cdot (1-1_{G_0})}.
$$ 
We conclude  that $R_h$ is completely bounded on $\L^p(\VN(G))$ for $p \in [2,\infty)$, and by duality and selfadjointness (note that $\phi$ is real-valued) also for $p \in (1,2]$.
\end{proof}

Let $\mathcal{H}$ be a real Hilbert space and fix some non-zero vectors $h_1,\ldots,h_n$ in $\mathcal{H}$ (or a sequence if $n=\infty$). We introduce the affine representation $(\mathcal{H},\alpha,b)$ of the free group $\F_n$ defined by $\alpha_s = \Id_\mathcal{H}$ for all $s \in G$ and 
$$
b\big(g_{i_1}^{j_1}\cdots g_{i_N}^{j_N}\big) 
\ov{\mathrm{def}}{=} j_1 h_{i_1} + \cdots + j_N h_{i_N},\quad j_1,\ldots,j_N \in \Z
$$ 
where $g_1,\ldots,g_n$ stand for the generators of $\F_n$. 


\begin{prop}
\label{Prop-concrete-free-groups}
Let $G = \F_n$ the free group on $n$ generators. Suppose $1 < p < \infty$. The previous Riesz transform $R_{h}$, associated with a family $(h_i)$ where $h=h_1$ is normalized, is a $\CB$-strongly non decomposable selfadjoint Fourier multiplier on $\L^p(\VN(\F_n))$.
\end{prop}

\begin{proof}
We have shown in Lemma \ref{lem-Riesz-transform-cb} that $R_h$ is completely bounded on $\L^p(\VN(\F_n))$. On the other hand, for any $m \in \Z \backslash \{ 0 \}$, we have 
$$
\phi(g_1^m) 
=\frac{\langle b(g_1^m), h_1 \rangle_{\mathcal{H}}}{\norm{b(g_1^m)}_{\mathcal{H}}} 
=\frac{\langle m h_1, h_1 \rangle_{\mathcal{H}}}{\norm{m h_1}_{\mathcal{H}}} 
=\sign(m) \frac{\langle h_1, h_1 \rangle_{\mathcal{H}}}{\norm{h_1}_{\mathcal{H}}} 
= \sign(m) \norm{h_1}_{\mathcal{H}}.
$$ 
So we have $\lim_{m \to +\infty} \phi(g_1^m) =  \norm{h_1}_{\mathcal{H}} \neq -  \norm{h_1}_{\mathcal{H}} = \lim_{m \to +\infty} \phi(g_1^{-m})$. Using Proposition \ref{Prop-strongly-non-dec-Fourier-mult} (since $G = \F_\infty$ is discrete and that $\VN(\F_\infty)$ is $\QWEP$), we conclude that $R_h$ is $\CB$-strongly non decomposable.
\end{proof}

\paragraph{Free Hilbert transform.} A different class of linear operators which are $\CB$-strongly non decomposable on $\L^p(\VN(\F_\infty))$ is given in \cite{MR}. Namely, let $G = \F_\infty$ be the free group with a countable sequence of generators $g_1,g_2, \ldots$. For $n \in \N$, let $L_n^\pm \co \L^2(\VN(\F_\infty)) \to \L^2(\VN(\F_\infty))$ be the orthogonal projection such that
$$
L_n^\pm(\lambda_s) 
= \begin{cases}
\lambda_s & s\text{ starts with the letter }g_n^{\pm 1} \\
0 & \text{otherwise}
\end{cases}.
$$
Let further $\epsi_n^+, \epsi_n^- \in \{-1,1\}$ for any $n \in \N$. Following \cite{MR}, we define the free Hilbert transform associated with $\epsi = (\epsi_n^\pm)$ as $H_\epsi = \sum_{n \in \N} \epsi_n^+ L_n^+ + \epsi_n^- L_n^-$. Clearly, since the ranges of the $L_n^\pm$ are mutually orthogonal, $H_\epsi$ is bounded on $\L^2(\VN(\F_\infty))$. The far reaching generalization in \cite[Section 4]{MR} is that $H_\epsi$ induces a completely bounded map on $\L^p(\VN(\F_\infty))$ for any $1 < p < \infty$.

\begin{prop}
\label{prop-free-Hilbert-transform}
Let $1 < p < \infty$ and $\epsi$ as previously. If $\epsi$ is not identically constant $1$ or $-1$, then $H_\epsi$ is $\CB$-strongly non decomposable on $\L^p(\VN(\F_\infty))$.
\end{prop}

\begin{proof}
Clearly, $H_\epsi = M_{\phi_\epsi}$ is a Fourier multiplier with symbol $\phi_\epsi(s)$ depending only on the first letter of $s$. This implies that $\phi_\epsi(s^n) = \phi_\epsi(s)$ for $n \in \N$. According to Proposition \ref{Prop-strongly-non-dec-Fourier-mult}, it suffices now to find some $s \in \F_\infty$ such that $\phi_\epsi(s) \not= \phi_\epsi(s^{-1})$. Take $n,m \in \N$ and $a,b \in \{\pm\}$ such that $\epsi^a_n \not= \epsi^b_m,$ whose existence is guaranteed by $H_\epsi \not= \pm \Id_{\L^p(\VN(\F_\infty))}$. Take further $s = g_n^a g_k g_m^{-b}$ for some $k \in \N \backslash \{n,m \}$. Then $\phi_\epsi(s) = \epsi_n^a \not= \epsi_m^b = \phi_\epsi(s^{-1})$.
\end{proof}

\subsection{CB-strongly non decomposable operators on approximately finite-dimen. algebras}
\label{sec:CB-strongly-non-decomposable-operators-on-afd-algebras}

We start with a transference result.

\begin{prop}
\label{prop-transfer-N-in-M-strongly}
Let $M$ be a von Neumann algebra and $N$ be a sub von Neumann algebra equipped with a faithful normal semifinite trace such that the inclusion $N \subset M$ is trace preserving. Suppose $1<p<\infty$. We denote by $\mathbb{E} \co \L^p(M)\to \L^p(N)$ the canonical conditional expectation and $J \co \L^p(N) \to \L^p(M)$ the canonical embedding map. Then
\begin{enumerate}
	\item The map
	\begin{equation*}
\begin{array}{cccc}
   \mathcal{I}   \co &  \CB(\L^p(N))   &  \longrightarrow   &  \CB(\L^p(M))  \\
           &   T  &  \longmapsto       &  JT\mathbb{E}  \\
\end{array}
\end{equation*}
is an isometry and the map 
\begin{equation*}
\begin{array}{cccc}
    \mathcal{Q}  \co &  \CB(\L^p(M))   &  \longrightarrow   & \CB(\L^p(N))   \\
           &   S  &  \longmapsto       &  \mathbb{E}SJ  \\
\end{array}
\end{equation*}
is a contraction. Both maps preserve the complete positivity and satisfy the equality $\mathcal{Q}\mathcal{I}=\Id_{\CB(\L^p(N))}$. 

\item We have $\mathcal{Q}(\Dec(\L^p(M)))=\Dec(\L^p(N))$ and $\mathcal{I}(\Dec(\L^p(N))) \subset \Dec(\L^p(M))$. Moreover, the previous maps induce an isometry $\mathcal{I} \co \Dec(\L^p(N)) \to \Dec(\L^p(M))$ and a contraction $\mathcal{Q} \co \Dec(\L^p(M)) \to \Dec(\L^p(N))$. 

\item For any completely bounded operator $T \co \L^p(N) \to \L^p(N)$, we have
$$
\dist_{\CB(\L^p(N))}\big(T,\Dec(\L^p(N))\big) 
= \dist_{\CB(\L^p(M))}\big(\mathcal{I}(T),\Dec(\L^p(M))\big).
$$
In particular, $T$ is $\CB$-strongly non decomposable if and only if $\mathcal{I}(T)$ is $\CB$-strongly non decomposable.
\end{enumerate}
\end{prop}

\begin{proof}
1. Recall that $\mathbb{E}J=\Id_{\L^p(N)}$. We have $\mathcal{Q}\mathcal{I}(T)=\mathcal{Q}(JT\mathbb{E})=\mathbb{E}JT\mathbb{E}J=T$. Now, it is obvious that $\mathcal{Q}$ is a contraction and that $\mathcal{I}$ is an isometry. Since $\mathbb{E}$ and $J$ are completely positive, the maps $\mathcal{Q}$ and $\mathcal{I}$ preserve the complete positivity.

2. Let $T \co \L^p(N) \to \L^p(N)$ be a decomposable operator. Since $\mathbb{E}$ and $J$ are contractively decomposable, we deduce by composition that $\mathcal{I}(T)$ is decomposable. Hence we have $\mathcal{I}(\Dec(\L^p(N))) \subset \Dec(\L^p(M))$. Similarly, we have the inclusion $\mathcal{Q}(\Dec(\L^p(M))) \subset \Dec(\L^p(N))$. Moreover, we have
$$
\Dec(\L^p(N))
=\mathcal{Q}\mathcal{I}\big(\Dec(\L^p(N))\big) 
\subset \mathcal{Q}\big(\Dec(\L^p(M))\big).
$$
We conclude that $\mathcal{Q}(\Dec(\L^p(M)))=\Dec(\L^p(N))$. Other statements are obvious.

3. Let $T \co \L^p(N) \to \L^p(N)$ be a completely bounded operator. Using the isometric map $\mathcal{I}$ and the inclusion $\mathcal{I}(\Dec(\L^p(N))) \subset \Dec(\L^p(M))$ we see that
\begin{align*}
\MoveEqLeft
  \dist_{\CB(\L^p(N))}\big(T,\Dec(\L^p(N))\big)
	=\dist_{\CB(\L^p(M))}\big(\mathcal{I}(T),\mathcal{I}(\Dec(\L^p(N))\big)\\
&\geq \dist_{\CB(\L^p(M))}\big(\mathcal{I}(T),\Dec(\L^p(M))\big).
\end{align*}
Now, consider a sequence $(T_n)$ of decomposable operators acting on $\L^p(M)$ with
$$
\norm{\mathcal{I}(T)-T_n}_{\cb, \L^p(M) \to \L^p(M)} 
\xra[n \to+\infty]{} 
\dist_{\CB(\L^p(M))}\big(\mathcal{I}(T),\Dec(\L^p(M))\big).
$$
By part 2, the operator $\mathcal{Q}(T_n) \co \L^p(N) \to \L^p(N)$ is decomposable. Moreover, we have
\begin{align*}
\MoveEqLeft
 \dist_{\CB(\L^p(N))}(T,\Dec(\L^p(N)))\leq \norm{T-\mathcal{Q}(T_n)}_{\cb,\L^p(N) \to \L^p(N)} \\
&=\bnorm{\mathcal{Q}(\mathcal{I}(T)-T_n)}_{\cb,\L^p(N) \to \L^p(N)} 
\leq \norm{\mathcal{I}(T)-T_n}_{\cb,\L^p(M) \to \L^p(M)}.
\end{align*}
Letting $n$ go to infinity, we obtain that 
$$
\dist_{\CB(\L^p(N))}\big(T,\Dec(\L^p(N))\big) 
\leq \dist_{\CB(\L^p(M))}\big(\mathcal{I}(T),\Dec(\L^p(M))\big).
$$
\end{proof}

We will use the following elementary lemma.

\begin{lemma}
\label{prop-comparison-norms-Schur-multiplier-p-infty}
Suppose $1<p<\infty$. For any matrix $A \in \M_n$, we have
$$
\norm{M_A}_{S^\infty_n \to S^\infty_n} \leq n^{\frac{1}{p}} \norm{M_A}_{S^p_n \to S^p_n}
$$
\end{lemma}

\begin{proof}
Let $B \in S^\infty_n$. We denote by $s_1(B),\ldots,s_n(B)$ the singular values of $B$. We have
$$
\norm{B}_{S^p_n} 
=\left(\sum_{i=1}^n s_i(B)^p\right)^{\frac{1}{p}} 
\leq \left(n \cdot\sup_{1 \leq i \leq n} s_i(B)^p\right)^{\frac{1}{p}} = n^{\frac{1}{p}} \cdot \sup_{1 \leq i \leq n} s_i(B) 
= n^{\frac{1}{p}} \cdot \norm{B}_{S^\infty_n}.
$$
We deduce that
\begin{align*}
\MoveEqLeft
 \norm{M_A(B)}_{S^\infty_n} 
		\leq  \norm{M_A(B)}_{S^p_n} 
		\leq \norm{M_A}_{S^p_n \to S^p_n} \norm{B}_{S^p_n}
		\leq n^{\frac{1}{p}} \norm{M_A}_{S^p_n \to S^p_n} \norm{B}_{S^\infty_n}.
\end{align*}
\end{proof}

\begin{prop}
\label{prop-hyperfinite-factor-von-Neumann-strongly-non-decomposable}
Let $\mathcal{R}$ be the hyperfinite factor of type $\mathrm{II}_1$ equipped with a normal finite faithful trace. Let $1 < p<\infty$. There exists a $\CB$-strongly non decomposable operator $T \co \L^p(\mathcal{R}) \to \L^p(\mathcal{R})$.
\end{prop}

\begin{proof}
Let $G$ be the discrete group of permutations of the integers that leave fixed all but a finite set of integers (the set may vary with the permutation). By \cite[page 902]{KaR}, the von Neumann algebra $\VN(G)$ is $*$-isomorphic to the hyperfinite factor $\mathcal{R}$ of type $\mathrm{II}_1$. Moreover, by \cite[page 902]{KaR}, the group $G$ is locally finite. 
By \cite[Theorem 14.3.7]{Rob}, it has an infinite abelian subgroup. Now, it suffices to use  Corollary \ref{Coro-Discrete-strongly-non-dec}.
\end{proof} 

We introduce the sub von Neumann algebra $K_\infty=\oplus_{n \geq 1} \M_n$ of $\B(\ell^2 \ot_2 \ell^2)$ equipped with its canonical trace and its noncommutative $\L^p$-space $K^p=\oplus_{n \geq 1}^p S^p_n$. We denote by $J \co K_\infty \to \B(\ell^2 \ot_2 \ell^2)$ the canonical inclusion and $\E \co \B(\ell^2 \ot_2 \ell^2) \to K_\infty$ the canonical trace preserving faithful normal conditional expectation.

\begin{prop}
\label{Prop-Kp-strongly-non-dec}
Let $1 < p<\infty$, $p \not=2$. There exists a $\CB$-strongly non decomposable operator $T \co K^p \to K^p$.
\end{prop}

\begin{proof}
If $n=2^m$, by \cite[page 53]{Dou}, there exists a positive constant $C$ and matrices $D_n \in \M_{n}$ such that $Cn^{\frac{1}{2}} \leq \norm{M_{D_n}}_{S^\infty_n \to S^\infty_n}$ and $\norm{M_{D_n}}_{S^p_n \to S^p_n} \leq n^{\left|\frac{1}{2}-\frac{1}{p}\right|}$ for $n$ large enough. Since the argument of \cite{Dou} of the latter inequality is based on interpolation and duality, we have the better estimate 
$$
\norm{M_{D_n}}_{\cb, S^p_n \to S^p_n} \leq n^{\left|\frac{1}{2}-\frac{1}{p}\right|}.
$$ 
Still working with $n=2^m$, we consider the matrix 
$$
A_n=\frac{1}{n^{\left|\frac{1}{2}-\frac{1}{p}\right|}} D_n.
$$
By Proposition \ref{Prop-Duality-dec}, we can suppose $p>2$. For $n$ large enough, we have 
$$
\norm{M_{A_n}}_{S^\infty_n \to S^\infty_n}
=\bgnorm{\frac{1}{n^{\left|\frac{1}{2}-\frac{1}{p}\right|}} M_{D_n}}_{S^\infty_n \to S^\infty_n}
=\frac{1}{n^{\left|\frac{1}{2}-\frac{1}{p}\right|}} \norm{M_{D_n}}_{S^\infty_n \to S^\infty_n} 
\geq cn^{\frac{1}{2}}\frac{1}{n^{\left|\frac{1}{2}-\frac{1}{p}\right|}}
=cn^{\frac{1}{p}}.
$$
Moreover, we have the estimate
$$
\norm{M_{A_n}}_{\cb,S^p_n \to S^p_n}
=\bgnorm{\frac{1}{n^{\left|\frac{1}{2}-\frac{1}{p}\right|}} M_{D_n}}_{\cb,S^p_n \to S^p_n} 
\leq 1.
$$
Now, we introduce the well-defined completely bounded linear operator
\begin{equation*}
\begin{array}{cccc}
   \Phi   \co &  K^p   &  \longrightarrow   &  K^p  \\
           &  (B_n)   &  \longmapsto       & (0,M_{A_2}(B_2),0,M_{A_4}(B_4),0,0,0,M_{A_8}(B_8),0,\ldots)   \\
\end{array}
\end{equation*}
Using the map $\mathcal{I}$ of Proposition \ref{prop-transfer-N-in-M-strongly}, we note that the map $\mathcal{I}(\Phi)=J\Phi\E \co S^p(\ell^2 \ot_2 \ell^2) \to S^p(\ell^2 \ot_2 \ell^2)$ is a completely bounded Schur multiplier $M_A$ on $ S^p(\ell^2 \ot_2 \ell^2)$. Now, we will use the following lemma.

\begin{lemma}
\label{conj-proche-de-tronc-is-non-dec}
There exists $\epsi>0$ small enough such that if a completely bounded Schur multiplier $M_B \co S^p(\ell^2 \ot_2 \ell^2) \to S^p(\ell^2 \ot_2 \ell^2)$ satisfies $\norm{M_B-M_A}_{\cb,S^p(\ell^2 \ot_2 \ell^2) \to S^p(\ell^2 \ot_2 \ell^2)} \leq \epsi$ then $M_B$ is not decomposable.
\end{lemma}

\begin{proof}
If $n=2^m$, let $B_n$ the $n \times n$-submatrix of the matrix $B$ occupying the same place as $A_n$ in $A$. The triangular inequality and Lemma \ref{prop-comparison-norms-Schur-multiplier-p-infty} give
\begin{align*}
\MoveEqLeft
    \norm{M_{B_n}}_{S^\infty_n \to S^\infty_n}
		\geq \norm{M_{A_n}}_{S^\infty_n \to S^\infty_n}-\norm{M_{B_n}-M_{A_n}}_{S^\infty_n \to S^\infty_n}\\
		&\geq \norm{M_{A_n}}_{S^\infty_n \to S^\infty_n}-n^{\frac{1}{p}}\norm{M_{B_n}-M_{A_n}}_{S^p_n \to S^p_n}\\
		&\geq\norm{M_{A_n}}_{S^\infty_n \to S^\infty_n}-n^{\frac{1}{p}}\norm{M_{B_n}-M_{A_n}}_{\cb,S^p_n \to S^p_n}.
\end{align*}
We take $0<\epsi<c$. Suppose $\norm{M_B-M_A}_{\cb,S^p(\ell^2 \ot_2 \ell^2) \to S^p(\ell^2 \ot_2 \ell^2) } \leq \epsi$. In particular, for any integer $n$, we have $\norm{M_{B_n}-M_{A_n}}_{\cb, S^p_n \to S^p_n} \leq \epsi$. If $n$ is large enough we obtain
$$
\norm{M_{B_n}}_{S^\infty_n \to S^\infty_n} 
\geq cn^{\frac{1}{p}}-\epsi n^{\frac{1}{p}}=(c-\epsi) n^{\frac{1}{p}}
\xra[n \to+\infty]{} +\infty.
$$
Hence the matrix $B$ does not induces a bounded Schur multiplier $M_B$ on $\B(\ell^2 \ot_2 \ell^2)$. By Theorem \ref{prop-continuous-Schur-multiplier-dec-infty}, we conclude that $M_B$ is not decomposable.
\end{proof} 

Now, suppose that there exists a decomposable operator $T \co S^p(\ell^2 \ot_2 \ell^2) \to S^p(\ell^2 \ot_2 \ell^2)$ such that $\norm{T-M_A}_{\cb,S^p(\ell^2 \ot_2 \ell^2) \to S^p(\ell^2 \ot_2 \ell^2)} \leq \epsi$. We can write
$$
T 
=T_1-T_2+\i(T_3-T_4)
$$
where each $T_j$ is a completely positive map acting on $S^p(\ell^2 \ot_2 \ell^2)$. Using the projection $P$ of Theorem \ref{Prop-complementation-Schur-Fourier}, we obtain $P(T)=P(T_1) - P(T_2) + \i(P(T_3) - P(T_4))$. Since each $P(T_j)$ is completely positive, we conclude that the Schur multiplier $P(T) \co S^p(\ell^2 \ot_2 \ell^2) \to S^p(\ell^2 \ot_2 \ell^2)$ is decomposable. Note also that 
\begin{align*}
\MoveEqLeft
   \norm{P(T)-M_A}_{\cb,S^p(\ell^2 \ot_2 \ell^2) \to S^p(\ell^2 \ot_2 \ell^2)} 
=\norm{P(T-M_A)}_{\cb,S^p(\ell^2 \ot_2 \ell^2) \to S^p(\ell^2 \ot_2 \ell^2)}\\
&\leq \norm{T-M_A}_{\cb,S^p(\ell^2 \ot_2 \ell^2) \to S^p(\ell^2 \ot_2 \ell^2)} 
\leq \epsi.
\end{align*}
This is impossible by Lemma \ref{conj-proche-de-tronc-is-non-dec}. Hence the map $M_A=\mathcal{I}(\Phi)$ is $\CB$-strongly non decomposable. By the point 3 of Proposition \ref{prop-transfer-N-in-M-strongly}, we conclude that $\Phi$ is $\CB$-strongly non decomposable.
\end{proof}

\begin{thm}
\label{prop-hyperfinite-von-Neumann-strongly-non-decomposable}
Let $M$ be an infinite dimensional approximately finite-dimensional von Neumann algebra equipped with a faithful normal semifinite trace. Let $1 < p <\infty$, $p \neq 2$. There exists a $\CB$-strongly non decomposable operator $T \co \L^p(M) \to \L^p(M)$.
\end{thm}

\begin{proof}
By the classification given by \cite[Theorem 5.1]{HRS} (see also \cite[Theorem 10.1]{PS} and \cite{Suk1}), the operator space $\L^p(M)$ is completely isomorphic to precisely one of the following thirteen operator spaces:
$$
\ell^p,\quad \L^p([0,1]),\quad S^p,\quad K^p,\quad K^p \oplus \L^p([0,1]),\quad S^p \oplus \L^p([0,1]),\quad
\L^p([0,1],K^p),  
$$
$$
S^p \oplus \L^p([0,1],K^p),\quad \L^p([0,1],S^p),\quad \L^p(\mathcal{R}),\quad S^p \oplus \L^p(\mathcal{R}),\quad \L^p([0,1],S^p) \oplus \L^p(\mathcal{R}),\quad \L^p(\mathcal{R},S^p).
$$
A careful examination of the proofs of \cite[pages 59-60]{HRS} and \cite[pages 143-145]{Suk1} shows that we can replace ``completely isomorphic'' by ``completely order and completely isomorphic''. 

By \cite[Examples 3.4 and 3.9]{ArV}, the Hilbert transforms $\ell^p_\Z \to \ell^p_\Z$ and $\L^p(\T) \to \L^p(\T)$ are strongly non regular. Since the Schatten space $S^p$ is UMD, by Proposition \ref{lemma-charact-Lp-cb}, these operators are also completely bounded (use \cite[Theorem 2.8]{BGM} for the discrete case). Using Proposition \ref{prop-transfer-N-in-M-strongly}, Proposition \ref{Prop-Kp-strongly-non-dec}, Proposition \ref{prop-hyperfinite-factor-von-Neumann-strongly-non-decomposable} and Corollary \ref{Cor-strongly-non-reg-truncation}, it is not difficult to conclude using a reasoning by cases.
\end{proof}


\begin{cor}
\label{quest-Equivalence-existence}
Suppose $1 \leq p <\infty$, $p \not=2$. Let $M$ be an infinite dimensional approximately finite-dimensional von Neumann algebra equipped with a faithful normal semifinite trace. The following properties are equivalent
\begin{enumerate}
	\item $p=1$.
	\item $\CB(\L^p(M))=\Dec(\L^p(M))$.
	\item $\CB(\L^p(M))=\ovl{\Dec(\L^p(M))}^{\CB(\L^p(M))}$.
\end{enumerate}
\end{cor}

\begin{proof}
Implications $1. \Rightarrow 2. \Rightarrow 3.$ are obvious. Theorem \ref{prop-hyperfinite-von-Neumann-strongly-non-decomposable} says that the contraposition of 3. $\Rightarrow$ 1. is true.
\end{proof}

For the case $p=\infty$, the situation is well-known for every von Neumann algebra. Indeed, by \cite[page 171]{Haa}, if $M$ is a von Neumann algebra then we have the equality $\CB(M)=\Dec(M)$ if and only if $M$ is approximately finite-dimensional. Moreover, Haagerup showed that the following properties are equivalent.
\begin{enumerate}
	\item $M$ is approximately finite-dimensional.
	\item For every C*-algebra $A$ and every completely bounded map $T \co A \to M$ we have\\ $\norm{T}_{\dec}=\norm{T}_{\cb}$.
	\item For every integer $n \geq 1$ and for every linear map $T \co \ell^\infty_n \to M$ we have $\norm{T}_{\dec}=\norm{T}_{\cb}$.
	\item There exists a positive constant $C \geq 1$, such that for every integer $n \geq 1$ and every linear map $T \co \ell^\infty_n \to M$ we have $\norm{T}_{\dec} \leq C\norm{T}_{\cb}$.
\end{enumerate}
Now, we show that these equivalences do not admit extensions to the case $1 < p<\infty$. It suffices to use the following proposition and the completely positive and completely isometric inclusion $\ell^p_n \subset \ell^p$.

\begin{prop}
\label{prop-ellpn-dec-cb}
Suppose $1<p<\infty$. There exists an integer $n$ large enough and a (completely bounded) linear map $T \co \ell^p_n \to \ell^p_n$ such that we have $\norm{T}_{\cb,\ell^p_n \to \ell^p_n}<\norm{T}_{\dec,\ell^p_n \to \ell^p_n}$. More precisely, there does not exist a positive constant $C \geq 1$ satisfying for every integer $n \geq 1$ and every linear map $T \co \ell^p_n \to \ell^p_n$ the inequality $\norm{T}_{\dec,\ell^p_n \to \ell^p_n} \leq C\norm{T}_{\cb,\ell^p_n \to \ell^p_n}$.
\end{prop}

\begin{proof}
By Theorem \ref{thm-existence-CB-strongly-non-decomposable-abelian-groups}, there exists a strongly non regular Fourier multiplier $M_{\varphi} \co \L^p(\T) \to \L^p(\T)$ which is completely bounded. We can suppose $\norm{M_{\varphi}}_{\cb} \leq 1$. 
Now, we approximate $M_{\varphi}$ using the method of the proof \cite[Proposition 3.8]{Arh5} (and \cite[proof of Theorem 3.5]{Arh1}). We deduce the existence of Fourier multipliers $C_{a_n}$ on $\L^p(\Z/n\Z)=\ell^p_n$ with $\norm{C_{a_n}}_{\cb} \leq 1$ and arbitrary large $\norm{C_{a_n}}_{\reg}$ when $n$ goes to the infinity. We can apply this method since $\norm{T}_{\dec}=\norm{T}_{\reg}=\sup_X \norm{T \ot \Id_X}_{\L^p(\Omega,X) \to \L^p(\Omega,X)}$.
\end{proof}

\section{Property $(\mathcal{P})$ and decomposable Fourier multipliers}
\label{sec:Property-P}

In this section, we give a proof of Proposition \ref{prop-prop-P-amenable-groups-multipliers} which is our characterization of selfadjoint contractively decomposable Fourier multipliers. Section \ref{Application-Matsaev} describes new Fourier multipliers which satisfy the noncommutative Matsaev inequality, relying on Theorem \ref{Th-Factorizable-matricial mutipliers} which gives the new result of factorizability.

\subsection{A characterization of selfadjoint contractively decomposable multipliers}
\label{A-characterization-of-selfadjoint-contractively-decomposable-multipliers}

Let $M$ be a von Neumann algebra equipped with a normal semifinite faithful trace and $T \co M \to M$ be a weak* continuous operator. Recall the following definition from \cite[Definition 3]{Kri}. We say that $T$ satisfies $(\mathcal{P})$ if there exist linear maps $v_1$, $v_2 \co M \to M$ such that the linear map $
\begin{bmatrix} 
   v_1  &  T \\
   T^\circ  &  v_2  \\
\end{bmatrix}
\co \M_2(M) \to \M_2(M)
$ 
is completely positive, completely contractive, weak* continuous and selfadjoint\footnote{\thefootnote.  The assumption selfadjoint is equivalent to the selfadjointness of $v_1$, $v_2$ and $T$.}. In this case, $v_1$ and $v_2$ are completely positive, weak* continuous and selfadjoint. 
An operator $T$ satisfying $(\mathcal{P})$ is necessarily contractively decomposable, weak* continuous and selfadjoint. The converse statement is false by \cite[Example 2]{Kri} in general. 

We start to show that the converse is true for Fourier multipliers on discrete groups. If $M_{\phi} \co \VN(G,\sigma)\to \VN(G,\sigma)$ is a bounded Fourier multiplier on a discrete group $G$ equipped with a $\T$-valued $2$-cocycle $\sigma$, it is not difficult to check that $(M_{\phi})^\circ=M_{\check{\ovl{\phi}}}$ and that $M_\phi$ is selfadjoint in the sense of Section \ref{sec:Markov-maps} if and only if its symbol $\phi \co G \to \C$ is a real-valued function. Finally, it is straightforward to prove that the preadjoint $(M_\phi)_* \co \L^1(\VN(G,\sigma)) \to \L^1(\VN(G,\sigma))$ of $M_\phi$ identifies to $M_{\check{\phi}}$.

\begin{lemma}
\label{Lemma-extension-L1}
Let $G$ be a discrete group equipped with a $\T$-valued $2$-cocycle $\sigma$. Suppose that $\psi_1,\psi_2,\psi_3,\psi_4 \co G \to \C$ are some complex-valued functions inducing some bounded Fourier multipliers $M_{\psi_1},M_{\psi_2},M_{\psi_3}$ and $M_{\psi_4}$ on the von Neumann algebra $\VN(G,\sigma)$. If the operator 
$$
T=\begin{bmatrix} 
    M_{\psi_1} & M_{\psi_2} \\
    M_{\psi_3} & M_{\psi_4} \\
\end{bmatrix} 
\co \M_2(\VN(G,\sigma)) \to \M_2(\VN(G,\sigma))
$$ 
is completely contractive then it induces a completely contractive operator $T_1$ on the space $S_2^1(\L^1(\VN(G,\sigma)))$. Finally the Banach adjoint $(T_1)^* \co \M_2(\VN(G,\sigma)) \to \M_2(\VN(G,\sigma))$ identifies to $
\begin{bmatrix} 
M_{\check{\psi_1}} & M_{\check{\psi_2}} \\ 
M_{\check{\psi_3}} & M_{\check{\psi_4}} 
\end{bmatrix}$.
\end{lemma}

\begin{proof}
According to Theorem \ref{Prop-complementation-Schur-Fourier}, we have 
$$
\norm{T}_{\cb,\M_2(\VN(G,\sigma)) \to \M_2(\VN(G,\sigma))} 
=\norm{T}_{\cb, \M_2(\VN(G)) \to \M_2(\VN(G))}
$$
and similarly, $\norm{T}_{\cb,S^1_2(\L^1(\VN(G,\sigma))) \to S^1_2(\L^1(\VN(G,\sigma)))} = \norm{T}_{\cb, S^1_2(\L^1(\VN(G))) \to S^1_2((\VN(G)))}$ provided that one of these terms is finite. So if we prove the first statement of the lemma for the trivial cocycle $\sigma = 1$, then it follows for a general $\T$-valued $2$-cocycle $\sigma$. We thus suppose now that $\sigma = 1$ is trivial. Consider the $*$-anti-automorphism $\kappa \co \VN(G) \to \VN(G)$, $\lambda_s \mapsto \lambda_{s^{-1}}$. An easy computation gives $(\Id_{\M_2} \ot \kappa)
\begin{bmatrix} 
    M_{\psi_1} & M_{\psi_2} \\
    M_{\psi_3} & M_{\psi_4} \\
\end{bmatrix} (\Id_{\M_2} \ot \kappa)=\begin{bmatrix} 
M_{\check{\psi_1}} & M_{\check{\psi_2}} \\ 
M_{\check{\psi_3}} & M_{\check{\psi_4}} 
\end{bmatrix}$ where $\check{\psi_i}(s)=\psi_i(s^{-1})$. Since the map $\kappa \co \VN(G) \to \VN(G)^{\op}$ is
a complete isometry, we conclude that the linear map $\begin{bmatrix} 
M_{\check{\psi_1}} & M_{\check{\psi_2}} \\ 
M_{\check{\psi_3}} & M_{\check{\psi_4}} 
\end{bmatrix} \co \M_2(\VN(G)) \to \M_2(\VN(G))$ is completely contractive. Moreover, by Lemma \ref{lemma-duality-Fourier-multipliers}, each symbol $\psi_i$ induces a bounded Fourier multiplier $M_{\psi_i} \co \L^1(\VN(G)) \to \L^1(\VN(G))$. Consequently, $\begin{bmatrix} 
    M_{\psi_1} & M_{\psi_2} \\
    M_{\psi_3} & M_{\psi_4} \\
\end{bmatrix}$ induces a bounded operator on $S_2^1(\L^1(\VN(G)))$. 
Furthermore, by Proposition \ref{Lemma-dualite-calcul-adjoint} and Lemma \ref{lemma-duality-Fourier-multipliers}, we see that the Banach adjoint of the operator $
\begin{bmatrix} 
M_{\psi_1} & M_{\psi_2} \\ 
M_{\psi_3} & M_{\psi_4} 
\end{bmatrix}
\co S_2^1(\L^1(\VN(G))) \to S_2^1(\L^1(\VN(G)))$ identifies to the complete contraction
$$
\begin{bmatrix} 
(M_{\psi_1})^* & (M_{\psi_2})^* \\ 
(M_{\psi_3})^* & (M_{\psi_4})^* 
\end{bmatrix}
=\begin{bmatrix} 
M_{\check{\psi_1}} & M_{\check{\psi_2}} \\ 
M_{\check{\psi_3}} & M_{\check{\psi_4}} 
\end{bmatrix} \co \M_2(\VN(G)) \to \M_2(\VN(G)).
$$ 
We conclude that the operator $\begin{bmatrix} 
    M_{\psi_1} & M_{\psi_2} \\
    M_{\psi_3} & M_{\psi_4} \\
\end{bmatrix} \co S_2^1(\L^1(\VN(G))) \to S_2^1(\L^1(\VN(G)))$ is completely contractive.
Finally, the last statement of the lemma for a general $\T$-valued $2$-cocycle $\sigma$ follows from
\begin{align*}
\tau_{G,\sigma}\left( (M_\psi)^* (\lambda_{\sigma,s}) \lambda_{\sigma,t}\right) & = \tau_{G,\sigma}\left(\lambda_{\sigma,s} M_\psi (\lambda_{\sigma,t}) \right) = \psi(t) \tau_{G,\sigma}\left(\lambda_{\sigma,s}\lambda_{\sigma,t}\right) \\
& = \psi(t) \sigma(s,t) \delta_{s,t^{-1}} = \psi(s^{-1}) \sigma(s,s^{-1}) \delta_{s,t^{-1}}
\end{align*}
and
\[ \tau_{G,\sigma} \left( M_{\check\psi}(\lambda_{\sigma,s}) \lambda_{\sigma,t} \right) = \check{\psi}(s) \tau_{G,\sigma}\left( \lambda_{\sigma,s} \lambda_{\sigma,t} \right) = \psi(s^{-1}) \sigma(s,s^{-1}) \delta_{s,t^{-1}}. \]
\end{proof}

\begin{prop}
\label{prop-prop-P-amenable-groups-multipliers}
Let $G$ be a discrete group equipped with a $\T$-valued $2$-cocycle $\sigma$. Let $\phi \co G \to \C$ be a complex-valued function. The following assertions are equivalent.
\begin{enumerate}
	\item The complex function $\phi$ induces a selfadjoint contractively decomposable  Fourier multiplier $M_{\phi} \co \VN(G,\sigma) \to \VN(G,\sigma)$ on the twisted group von Neumann algebra $\VN(G,\sigma)$.
	
	\item The function $\phi$ induces a Fourier multiplier $M_{\phi} \co \VN(G,\sigma) \to \VN(G,\sigma)$ with $(\mathcal{P})$.
	
	\item There exist some real-valued functions $\varphi_1,\varphi_2 \co G \to \R$ such that
\[
\begin{bmatrix}
   M_{\varphi_1}  &  M_{\phi} \\
   M_{\phi}^\circ  &  M_{\varphi_2}  \\
\end{bmatrix}
\co \M_2(\VN(G,\sigma)) \to \M_2(\VN(G,\sigma))
\]
is unital, completely positive, weak* continuous and selfadjoint.
\end{enumerate}
\end{prop}

\begin{proof}
The statements 3. $\Rightarrow$ 2. and 2. $\Rightarrow$ 1. are obvious. Now, we show the lat implication 1. $\Rightarrow$ 3. The multiplier $M_\phi$ is selfadjoint thus we have $\ovl{\phi}=\phi$ and finally $(M_{\phi})^\circ=M_{\check{\ovl{\phi}}}=M_{\check{\phi}}$. Since the operator $M_\phi$ is contractively decomposable there exist linear maps $v_1,v_2 \co \VN(G,\sigma) \to \VN(G,\sigma)$ such that the map 
$ 
\begin{bmatrix} 
v_1 & M_\phi \\ 
M_{\check{\phi}} & v_2 
\end{bmatrix} 
\co \M_2(\VN(G,\sigma)) \to \M_2(\VN(G,\sigma))
$ 
is completely positive and completely contractive. By using the same reasoning as the one in the proof of Proposition \ref{Prop-Duality-dec}, we can suppose that this map is in addition weak* continuous. Since $G$ is discrete, we can use projection $P_{\{1,2\},G,\sigma}^\infty \co \CB_{\w^*}(\M_2(\VN(G,\sigma))) \to \mathfrak{M}^{\infty,\cb}_{\{1,2\}}(G,\sigma)$ from Theorem \ref{Prop-complementation-Schur-Fourier}. We obtain that
\begin{align*}
\MoveEqLeft
P_{\{1,2\},G,\sigma}^\infty\left(
\begin{bmatrix} 
v_1 & M_\phi \\ 
M_{\check{\phi}} & v_2 
\end{bmatrix} 
\right)
=
\begin{bmatrix} 
   P_G^\infty(v_1)  & P_G^\infty(M_\phi)  \\
  P_G^\infty(M_{\check{\phi}})  & P_G^\infty(v_2)  \\
  \end{bmatrix} 
		=
\begin{bmatrix} 
   P_G^\infty(v_1)  & M_\phi  \\
   M_{\check{\phi}}  & P_G^\infty(v_2)  \\
\end{bmatrix} .
\end{align*}
We deduce that there exist some complex functions $\psi_1,\psi_2 \co G \to \C$ such that the map 
$
T
\overset{\textrm{def}}= 
\begin{bmatrix} 
M_{\psi_1} & M_\phi \\ 
M_{\check{\phi}} & M_{\psi_2} 
\end{bmatrix} 
\co \M_2(\VN(G,\sigma)) \to \M_2(\VN(G,\sigma))
$ 
is completely positive, completely contractive and weak* continuous. By Lemma \ref{Lemma-extension-L1}, the operator $T$ induces a completely positive and completely contractive operator $T_1 \co S^1_2(\L^1(\VN(G,\sigma))) \to S^1_2(\L^1(\VN(G,\sigma)))$. 
The operator $(T_1)^* \co \M_2(\VN(G,\sigma)) \to \M_2(\VN(G,\sigma))$ is also completely contractive and completely positive by Lemma \ref{Lemma-adjoint-cp}. Again by Lemma \ref{Lemma-extension-L1}, we have 
$
(T_1)^* 
= \begin{bmatrix} 
M_{\check{\psi_1}} & M_{\check{\phi}} \\ 
M_{\phi} & M_{\check{\psi_2}} 
\end{bmatrix}
=
\begin{bmatrix} 
M_{\ovl{\psi_1}} & M_{\check{\phi}} \\ 
M_{\phi} & M_{\ovl{\psi_2}} 
\end{bmatrix}
$ 
where we used \cite[Proposition C.4.2]{BHV} and the fact that $\psi_1$ and $\psi_1$ are definite positive since $M_{\psi_1} \co \VN(G,\sigma) \to \VN(G,\sigma)$ and $M_{\psi_2} \co \VN(G,\sigma) \to \VN(G,\sigma)$ are completely positive. 

Consider the transpose map\footnote{\thefootnote. Here $\M_2^\op$ identifies to the algebra $\M_2$ with the multiplication reversed.} $\eta \co \M_2 \to \M_2^\op$, $A \mapsto \trans{A}$, which is an algebra isomorphism, hence a complete isometry and a completely positive map (see also Lemma \ref{equ-1-proof-op-mappings}). An easy computation gives $(\eta \ot \Id_{\VN(G,\sigma)})
\begin{bmatrix} 
M_{\ovl{\psi_1}} & M_{\check{\phi}} \\ 
M_{\phi} & M_{\ovl{\psi_2}} 
\end{bmatrix} (\eta \ot \Id_{\VN(G,\sigma)})
=\begin{bmatrix} 
M_{\ovl{\psi_1}} &  M_{\phi}\\ 
M_{\check{\phi}} & M_{\ovl{\psi_2}} 
\end{bmatrix}$. We conclude that the linear map $R\overset{\textrm{def}}=
\begin{bmatrix} 
M_{\ovl{\psi_1}} &  M_{\phi}\\ 
M_{\check{\phi}} & M_{\ovl{\psi_2}} 
\end{bmatrix} \co \M_2(\VN(G,\sigma)) \to \M_2(\VN(G,\sigma))$ is completely contractive and completely positive.

Now, $\frac{1}{2}(T+R) \co \M_2(\VN(G,\sigma)) \to \M_2(\VN(G,\sigma))$ is a matrix block multiplier $\begin{bmatrix} 
M_{\psi_3} & M_\phi \\ 
M_{\check{\phi}} & M_{\psi_4} 
\end{bmatrix} $ which is completely positive, completely contractive and selfadjoint with $M_\phi$ in the corner. Note that $M_{\psi_3}$ and $M_{\psi_4}$ are completely positive. So $\psi_3(e) =M_{\psi_3}(1)=\norm{M_{\psi_3}} \leq 1$ and similarly for $\psi_4$. Hence the linear maps $w_1=M_{\psi_3}+\tau_{G,\sigma}(\cdot)(1-\psi_3(e))1_{\VN(G,\sigma)}$ and $w_2=M_{\psi_4}+\tau_{G,\sigma}(\cdot)(1-\psi_4(e))1_{\VN(G,\sigma)}$ are completely positive, selfadjoint and weak* continuous. We have
\begin{align*}
\MoveEqLeft
 w_1(\lambda_{\sigma,s})   
		=\big(M_{\psi_3}+\tau_{G,\sigma}(\cdot)(1-\psi_3(e))1_{\VN(\sigma,\sigma)}\big)(\lambda_{\sigma,s}) \\
		&=M_{\psi_3}(\lambda_{\sigma,s})+\tau_{G,\sigma}(\lambda_{\sigma,s})(1-\psi_3(e))1_{\VN(G,\sigma)}\\
		&=\psi_3(s)\lambda_{\sigma,s}+\delta_{s,e}(1-\psi_3(e))1_{\VN(G,\sigma)}
		=\lambda_{\sigma,s}\begin{cases}
		\psi_3(s)&\text{if } s\not=e\\
		1&\text{if } s=e
		\end{cases}
\end{align*}
and similarly for $w_2$. We deduce that these maps are selfadjoint unital Fourier multipliers $M_{\varphi_1}$ and $M_{\varphi_2}$. Now, the map 
$
\Phi=\begin{bmatrix}
   M_{\varphi_1}  &  M_{\phi} \\
   M_{\phi}^\circ  &  M_{\varphi_2}  \\
\end{bmatrix}
\co \M_2(\VN(G,\sigma)) \to \M_2(\VN(G,\sigma))
$ 
is obviously unital, selfadjoint and weak* continuous. Moreover
$$
\Phi=\begin{bmatrix}
   M_{\varphi_1}  &  M_{\phi} \\
   M_{\phi}^\circ  &  M_{\varphi_2}  \\
\end{bmatrix}
=\begin{bmatrix}
   M_{\psi_3}  &  M_{\phi} \\
   M_{\phi}^\circ  &   M_{\psi_4} \\
\end{bmatrix}
+
\begin{bmatrix}
    \tau_{G,\sigma}(\cdot)(1-\psi_3(e))1_{\VN(G,\sigma)} &  0 \\
       0  &  \tau_{G,\sigma}(\cdot)(1-\psi_4(e))1_{\VN(G,\sigma)} \\
\end{bmatrix}.
$$
It is easy to conclude that $\Phi$ is completely positive.
\end{proof}

\begin{remark}\normalfont
Let $G$ be an amenable discrete group. By \cite[Corollary 1.8]{DCH}, a contractive Fourier multiplier $M_\varphi \co \VN(G) \to \VN(G)$ is completely contractive and finally contractively decomposable by \cite[Theorem 2.1]{Haa} since $\VN(G)$ is approximately finite-dimensional.
\end{remark}

\subsection{Factorizability of some matrix block multipliers}
\label{sec:factorizability-of-some-matrix-block-multipliers}


\paragraph{Second quantization.} We denote by $\mathrm{Sym}(n)$ the symmetric group of order $n$. If $\sigma$ is a permutation of $\mathrm{Sym}(n)$ we denote by $|\sigma|$ the number $\card\big\{ (i,j)\ :\ 1\leq i<j \leq n,\, \sigma(i)>\sigma(j)\big\}$ of inversions of $\sigma$.

Let $\mathcal{H}$ be a complex Hilbert space. The antisymmetric (or fermionic) Fock space over $\mathcal{H}$ is 
$
\mathcal{F}_{-1}(\mathcal{H})
=\C \Omega \oplus (\bigoplus_{n \geq 1} \mathcal{H}^{\ot_{n}})
$ 
where $\Omega$ is a unit vector called the vacuum and where the
scalar product on $\mathcal{H}^{\ot_{n}}$ is given, after dividing out the null space, by
\begin{equation*}
\langle h_1 \ot \dots \ot h_n ,k_1 \ot \cdots \ot k_n
\rangle_{-1}=\sum_{\sigma \in \mathrm{Sym}(n)} (-1)^{|\sigma|}\langle
h_1,k_{\sigma(1)}\rangle_{\mathcal{H}}\cdots\langle
h_n,k_{\sigma(n)}\rangle_{\mathcal{H}}.
\end{equation*}
The creation operator $c(e)$ for $e \in \mathcal{H}$ is given by $c(e) \co \mathcal{F}_{-1}(\mathcal{H}) \to \mathcal{F}_{-1}(\mathcal{H})$, $h_1 \ot \dots \ot h_n \mapsto  e \ot h_1 \ot \dots \ot h_n$. We have $c(e)^2=0$. Moreover, they satisfy the $q$-commutation relation
\begin{equation}
\label{q-commutation-relations}
c(f)^*c(e)+c(e)c(f)^* 
= \langle f,e\rangle_{\mathcal{H}} \Id_{\mathcal{F}_{-1}(\mathcal{H})}.
\end{equation}
We denote by $\omega(e) \co \mathcal{F}_{-1}(\mathcal{H}) \to \mathcal{F}_{-1}(\mathcal{H})$ the selfadjoint operator $c(e)+c(e)^*$. If $e \in \mathcal{H}$ has norm 1, then \eqref{q-commutation-relations} says that the operator $\omega(e)$ satisfies 
\begin{equation}
\label{Equa-omega2=I}
\omega(e)^2
=\Id_{\mathcal{F}_{-1}(\mathcal{H})}.
\end{equation}
Let $H$ be a real Hilbert space with complexification $H_{\C}$. We let $\mathcal{H}=H_{\C}$. The fermion von Neumann algebra $\Gamma_{-1}(H)$ is the von Neumann algebra generated by the operators $\omega(e)$ where $e \in H$. It is a finite von Neumann algebra with the trace $\tau$ defined by $\tau(x)=\langle\Omega,x\Omega\rangle_{\mathcal{F}_{-1}(\mathcal{H})}$ where $x \in \Gamma_{-1}(H)$.

Let $H$ and $K$ be real Hilbert spaces and $T \co H \to K$ be a contraction with complexification $T_{\C} \co \mathcal{H}=H_{\C} \to K_{\C}=\mathcal{K}$. We define the following linear map
\begin{equation*}
\begin{array}{cccc}
 \mathcal{F}_{-1}(T) \co   &  \mathcal{F}_{-1}(\mathcal{H})   &  \longrightarrow   & \mathcal{F}_{-1}(\mathcal{K})  \\
    &  h_1 \ot \dots \ot h_n   &  \longmapsto       &    T_{\C}h_1 \ot \dots \ot T_{\C}h_n. \\
\end{array}
\end{equation*}
Then there exists a unique map $\Gamma_{-1}(T) \co \Gamma_{-1}(H) \to
\Gamma_{-1}(K)$ such that for every $x \in \Gamma_{-1}(H)$ we have $(\Gamma_{-1}(T)(x))\Omega=\mathcal{F}_{-1}(T)(x\Omega)$. This map is normal, unital, completely positive and trace preserving. If $T\co H \to K $ is a surjective  isometry, $\Gamma_{-1}(T)$ is a $*$-isomorphism from $\Gamma_{-1}(H)$ onto $\Gamma_{-1}(K)$. 

Finally for any $e,f \in H$, we have the covariance formula
\begin{equation}
\label{petite formule de Wick} 
\tau\big(\omega(e)\omega(f)\big)
=\langle e,f \rangle_{H}.
\end{equation}

\paragraph{Kernels of positive type.} Let $X$ be a topological space. A (real) kernel of positive type on $X$  \cite[Definition C.1.1]{BHV} is a continuous function $\Phi \co X \times X \to \C$ (into $\R$) such that, for any integer $n \in \N$, any elements $x_1,\ldots,x_n \in X$ and any (real) complex numbers $c_1,\ldots,c_n$, the following inequality holds:
$$
\sum_{k,l=1}^{n} c_k\ovl{c_l}\Phi(x_k,x_l) 
\geq 0.
$$
In this case, we have $\Phi(x,y)=\ovl{\Phi(y,x)}$ for any $x,y \in X$ by \cite[Proposition C.1.2]{BHV}. If $\Phi$ is such a kernel, by \cite[page 82]{BCR} and \cite[Theorem C.1.4]{BHV}, then there exists a (real) Hilbert space $H$ and a continuous mapping $e \co X \to H$ with the following properties:
\begin{enumerate}
	\item $\Phi(x,y)=\big\langle e_x, e_y \big\rangle_{H}$ for any $x, y \in X$,
	\item the linear span of $\{e_x : x \in X\}$ is dense in $H$.
\end{enumerate}


\paragraph{Factorizable maps.} 

Let $M$ be a von Neumann equipped with a faithful normal finite trace $\tau_M$. A $\tau_M$-Markov map $T \co M \to M$ is called factorizable\footnote{\thefootnote. The definition given here is slightly different but equivalent by \cite[Remark 1.4 (a)]{HaM}.} \cite{AnD}, \cite{HaM}, \cite{JMX}, \cite{Ric2} if there exists a von Neumann algebra $N$ equipped with a faithful normal finite trace $\tau_N$, and $*$-monomorphisms $J_0 \co M \to N$ and $J_1 \co M \to N$ such that $J_0$ is $(\tau_M,\tau_N)$-Markov and $J_1$ is $(\tau_M,\tau_N)$-Markov, satisfying moreover $T=J_0^*\circ J_1$. We say that $T \co M \to M$ is $\QWEP$-factorizable \cite{Arh3} if $N$ has additionally $\QWEP$.

\paragraph{Twisted crossed products.} 
In order to prove our results we need the notion of crossed product. 
Let $H$ be a Hilbert space and $M$ be a sub-von Neumann algbra of $\B(H)$. We consider a discrete group $G$ equipped with a $\T$-valued 2-cocycle $\sigma$.  Let $\alpha \co G \to \mathrm{Aut}(M)$ be a representation of $G$ on $M$. The twisted crossed product von Neumann algebra $M \rtimes_{\sigma,\alpha} G$ \cite[Definition 2.1]{Sut} (see also \cite{ZM} for a unitary transform of this definition) is generated by the operators $\pi_\sigma(x)$ and $\lambda_{\sigma,s}$ acting on $\ell^2_G(H)$ where $x \in M$ and $s \in G$, defined by
$$
(\pi_{\sigma}(x)\xi)(s)
=\alpha_{s^{-1}}(x)\xi(s),\qquad x \in M,\ \xi \in \ell^2_G(H),\ s \in G.
$$
$$
(\lambda_{\sigma,s}\xi)(t)
=\sigma(t^{-1},s)\xi(s^{-1}t), \qquad \xi \in \ell^2_G(H),\ s,t \in G.
$$
We have the following relations of commutation \cite[Proposition 2.2]{Sut}:
\begin{equation}
\label{equa-relations-de-commutation}
\pi_{\sigma}\big(\alpha_s(x)\big)\lambda_{\sigma,s}
=\lambda_{\sigma,s}\pi_{\sigma}(x),
\quad \text{and} \quad
\lambda_{\sigma,s}\lambda_{\sigma,t}
=\sigma(s,t)\lambda_{\sigma,st} \qquad x \in M,\ s,t \in G.
\end{equation}
We can identify $M$ and $\VN(G,\sigma)$ as subalgebras of $M \rtimes_{\sigma,\alpha} G$. 

Suppose that $\tau$ is a $G$-invariant normal semi-finite faithful trace on $M$. If $\mathbb{E}$ is the normal conditional expectation from $M \rtimes_{\sigma,\alpha} G$ onto $M$ then $\tau_{\rtimes} \ov{\mathrm{def}}{=} \tau \circ \mathbb{E}$ defines a normal semifinite faithful trace on $M \rtimes_{\sigma,\alpha} G$, see \cite[Proposition 8.16]{ZM}. For any $x \in M$ and any $s \in G$, we have
\begin{equation}
\label{formule-trace-crossed}
\tau_{\rtimes}\big(x\lambda_{\sigma,s}\big)
=\delta_{s,e_G}\tau(x).
\end{equation}
Moreover, $\tau_{\rtimes}$ is finite if and only if $\tau$ is finite. Finally we will use the notation $M \rtimes_{\alpha} G=M \rtimes_{1,\alpha} G$.


The following proposition generalizes a part of \cite[Proposition 4.2]{DCH}. It probably admits a groupoid generalization (see also \cite{ArK1}).

\begin{prop}
\label{Prop-caracterisation-cp}
Suppose that $I$ is a finite set. Let $G$ be a discrete group equipped with a $\T$-valued $2$-cocycle $\sigma$. Let $(\varphi_{ij})_{i,j \in I}$ be a family of complex functions on $G$. Let $\Psi \co \M_I(\VN(G,\sigma)) \to \M_I(\VN(G,\sigma))$ be a normal completely positive map such that $\Psi([\lambda_{\sigma,s_{ij}}])=\big[\varphi_{ij}(s_{ij})\lambda_{\sigma,s_{ij}}\big]$ for any family $(s_{ij})_{i,j \in I}$ of elements of $G$. Then the map $ \Phi \co I \times G\times I  \times G    \to   \C$, $(i,s,j,s') \mapsto \varphi_{ij}(s^{-1}s')$ is a kernel of positive type, that is: for any integer $n \in \N$, any elements $i_1,\ldots, i_n \in I$, any $s_1,\ldots,s_n\in G$ and any complex numbers $c_1,\ldots,c_n$, the following inequality holds:
$$
\sum_{k,l=1}^{n} c_k \ovl{c_{l}}\varphi_{i_{k}i_{l}}\big(s_{k}^{-1}s_{l}\big) \geq 0.
$$
\end{prop}

\begin{proof} 
Consider $i_1,\ldots, i_n \in I$ and $s_1,\ldots,s_n \in G$ and some complex numbers $c_1,\ldots,c_n \in \C$. Let $\xi$ be a unit vector of $\L^2(\VN(G,\sigma))$. For any integer $1 \leq k \leq n$, we let $\xi_k\overset{\mathrm{def}}=\ovl{c_k} \lambda_{\sigma,s_k}^{-1}\xi$. Then using \eqref{product-adjoint-twisted} several times, we have
\begin{align*}
\MoveEqLeft
 \sum_{k,l=1}^{n}c_k \ovl{c_{l}}\varphi_{i_{k}i_{l}}\big(s_{k}^{-1}s_{l}\big)
=\sum_{k,l=1}^{n} \varphi_{i_{k}i_{l}}\big(s_{k}^{-1}s_{l}\big)c_k \ovl{c_{l}}\langle\xi,\xi\rangle
=\sum_{k,l=1}^{n} \varphi_{i_{k}i_{l}}\big(s_{k}^{-1}s_{l}\big) \big\langle \ovl{c_{k}}\xi,\ovl{c_l}\xi\big\rangle\\
&=\sum_{k,l=1}^{n}\varphi_{i_{k}i_{l}}\big(s_{k}^{-1}s_{l}\big) \big\langle \lambda_{\sigma,s_k} \xi_k,\lambda_{\sigma,s_l}\xi_l\big\rangle
=\sum_{k,l=1}^{n}\varphi_{i_{k}i_{l}}\big(s_{k}^{-1}s_{l}\big) \big\langle \xi_k,(\lambda_{\sigma,s_k})^*\lambda_{\sigma,s_l} \xi_l\big\rangle\\
		&=\sum_{k,l=1}^{n}\varphi_{i_{k}i_{l}}\big(s_{k}^{-1}s_{l}\big) \ovl{\sigma(s_k,s_k^{-1})}\big\langle \xi_k,\lambda_{\sigma,s_k^{-1}}\lambda_{\sigma,s_l} \xi_l\big\rangle\\
		&=\sum_{k,l=1}^{n}\varphi_{i_{k}i_{l}}\big(s_{k}^{-1}s_{l}\big) \ovl{\sigma(s_k,s_k^{-1})}\sigma(s_k^{-1},s_l)\big\langle \xi_k,\lambda_{\sigma,s_k^{-1}s_l} \xi_l\big\rangle\\
		&=\sum_{k,l=1}^{n} \ovl{\sigma(s_k,s_k^{-1})} \sigma(s_k^{-1},s_l)\big\langle \xi_k, M_{\varphi_{i_{k}i_{l}}} \big(\lambda_{\sigma,s_k^{-1}s_l}\big) \xi_l\big\rangle\\
		&=\sum_{k,l=1}^{n} \ovl{\sigma(s_k,s_k^{-1})} \big\langle \xi_k, M_{\varphi_{i_{k}i_{l}}} \big(\sigma(s_k^{-1},s_l)\lambda_{\sigma,s_k^{-1}s_l}\big) \xi_l\big\rangle\\
		&=\sum_{k,l=1}^{n} \ovl{\sigma(s_k,s_k^{-1})}\big\langle \xi_k,M_{\varphi_{i_{k}i_{l}}} \big(\lambda_{\sigma,s_k^{-1}}\lambda_{\sigma,s_l}\big) \xi_l\big\rangle
		=\sum_{k,l=1}^{n} \big\langle \xi_k,M_{\varphi_{i_{k}i_{l}}} \big((\lambda_{\sigma,s_k})^*\lambda_{\sigma,s_l}\big) \xi_l\big\rangle
\end{align*}	
where the brackets denote scalar products in the Hilbert space $\L^2(\VN(G,\sigma))$. Now, we consider the vector $\eta=(\eta_{l,t})_{l \in [\![1,n]\!], t \in I} \in \ell^2_n(\ell^2_I(\L^2(\VN(G,\sigma))))$, where each $\eta_{l,t}$ belongs to $\L^2(\VN(G,\sigma)$), defined by
$$
\eta_{l,t}
=\delta_{t,i_l} \xi_{l}.
$$
We consider $\Id_{\M_n} \ot \Psi=\big[M_{\varphi_{rt}}\big]_{k,l \in [\![1,n]\!], r, t \in I} \co \M_{[\![1,n]\!] \times I}(\VN(G,\sigma)) \to \M_{[\![1,n]\!] \times I}(\VN(G,\sigma))$ and the matrix
$$
C
=\big[(\lambda_{\sigma,s_k})^*\lambda_{\sigma,s_l}\big]_{k,l \in [\![1,n]\!], r, t \in I} \in \M_{[\![1,n]\!] \times I}(\VN(G,\sigma)).
$$
Note that $C$ is positive (a matrix $[a_i^* a_j]_{ij}$ of $\M_n(A)$ is positive \cite[page 34]{Pau} and we use \cite[Lemma 1.3.6]{Bha}) and that
$$
(\Id_{\M_n} \ot \Psi)(C)
=\Big[M_{\varphi_{rt}}\big((\lambda_{\sigma,s_k})^*\lambda_{\sigma,s_l}\big)\Big]_{k,l \in [\![1,n]\!], r, t \in I}
\overset{\mathrm{def}}= [b_{k,l,r,t}]_{k,l \in [\![1,n]\!], r, t \in I}.
$$ 
We have
\begin{align*}
\MoveEqLeft
0 \leq \big\langle \eta,(\Id_{\M_n} \ot \Psi)(C)\eta \big\rangle_{\ell^2_n(\ell^2_I(\ell^2_G))} \\
& =\big\langle (\eta_{k,r})_{k \in [\![1,n]\!],r \in I},[b_{k,l,r,t}]_{k,l \in [\![1,n]\!], r, t \in I}(\eta_{l,t})_{l \in [\![1,n]\!], t \in I} \big\rangle_{\ell^2_n(\ell^2_I(\L^2(\VN(G,\sigma))))}\\
		&=\Bigg\langle (\eta_{k,r})_{k \in [\![1,n]\!],r \in I},\bigg(\sum_{l=1}^{n} \sum_{t \in I}^{} b_{k,l,r,t}\eta_{l,t}\bigg)_{k \in [\![1,n]\!], r \in I} \Bigg\rangle_{\ell^2_n(\ell^2_I(\L^2(\VN(G,\sigma))))}\\
		&=\sum_{k,l=1}^{n}  \sum_{r,t \in I} \big\langle \eta_{k,r},  b_{k,l,r,t} \eta_{l,t}\big\rangle 
		=\sum_{k,l=1}^{n} \sum_{r,t \in I} \big\langle \eta_{k,r},M_{\varphi_{rt}}\big((\lambda_{\sigma,s_k})^*\lambda_{\sigma,s_l}\big)  \eta_{l,t}\big\rangle \\
		&=\sum_{k,l=1}^{n} \sum_{r,t \in I} \big\langle \delta_{r,i_k} \xi_{k}, M_{\varphi_{rt}}\big((\lambda_{\sigma,s_k})^*\lambda_{\sigma,s_l}\big) \delta_{t,i_l} \xi_{l} \big\rangle \\
		& =\sum_{k,l=1}^{n} \sum_{r,t \in I} \delta_{r,i_k} \delta_{t,i_l}\big\langle \xi_k, M_{\varphi_{rt}}\big((\lambda_{\sigma,s_k})^*\lambda_{\sigma,s_l}\big) \xi_{l} \big\rangle \\
		&=\sum_{k,l=1}^{n} \big\langle \xi_k, M_{\varphi_{i_ki_l}}\big((\lambda_{\sigma,s_k})^*\lambda_{\sigma,s_l}\big)  \xi_{l}\big\rangle.
\end{align*}
where the brackets denote scalar products in the Hilbert space $\L^2(\VN(G,\sigma))$. 
\end{proof}

The following result generalizes the results of \cite{Ric2}. 

\begin{thm}
\label{Th-Factorizable-matricial mutipliers}
Let $G$ be a discrete group equipped with a $\T$-valued $2$-cocycle $\sigma$ on $G$ and $I$ be a finite set. Let $(\varphi_{ij})_{i,j \in I}$ be a family of real-valued functions on $G$ such that $\varphi_{ii}(e)=1$ for any $i\in I$. If the (selfadjoint unital trace preserving\footnote{\thefootnote. Hence $(\tr \ot \tau_{G,\sigma})$-Markovian.}) map $[M_{\varphi_{ij}}] \co \M_I(\VN(G,\sigma)) \to \M_I(\VN(G,\sigma))$ is completely positive then $[M_{\varphi_{ij}}]$ is factorizable on a von Neumann algebra of the form $\M_I\big(\Gamma_{-1}(H) \rtimes_{\sigma,\alpha} G\big)$ where $\alpha$ is an action of $G$ on the von Neumann algebra $\Gamma_{-1}(H)$ for some Hilbert space $H$.
\end{thm}

\begin{proof}
By Proposition \ref{Prop-caracterisation-cp}, the map $\Phi \co I \times G\times I  \times G  \to \R$, $(i,s,j,s') \mapsto \varphi_{ij}(s^{-1}s')$ is a real  kernel of positive type. Hence for any $i,j \in I$ and any $s,s' \in G$ we have $\varphi_{ij}(s^{-1}s')=\varphi_{ji}(s'^{-1}s)$ in particular
\begin{equation}
	\label{symetrie-kernel}
\varphi_{ij}(s)=\varphi_{ji}(s^{-1}). 
\end{equation}
Moreover, there exists a real Hilbert space $H$ and a map $e  \co I \times G \to H$, $(i,s) \mapsto e_{i,s}$ such that the linear span of $\{e_{i,s} : i \in I, s \in G\}$ is dense in $H$ and such that for any $i,j \in I$ and any $s,s' \in G$
$$
\Phi(i,s,j,s')=\big\langle e_{i,s}, e_{j,s'} \big\rangle_{H}, \quad \text{i.e.} \quad
\varphi_{ij}(s^{-1}s')
= \big\langle e_{i,s}, e_{j,s'} \big\rangle_{H}.
$$
In particular, we have
\begin{equation}
	\label{Equa-eig-de-norme-1}
\varphi_{ij}(s)
=\big\langle e_{i,e}, e_{j,s} \big\rangle_{H}
\quad \text{and} \quad
	\bnorm{e_{i,s}}_{H}^2
	=\big\langle e_{i,s}, e_{i,s} \big\rangle_{H}
	=\varphi_{ii}\big(s^{-1}s\big)
	=\varphi_{ii}(e)
	=1.
\end{equation}
Note that for any family of real numbers $(a_{i,t})_{i \in I, t \in G}$ with only finitely many non-zero terms, we have
\begin{align*}
\MoveEqLeft
\Bgnorm{\sum_{i \in I,t \in G} a_{i,t}e_{i,st}}_{H}^2
=\sum_{i,j \in I} \sum_{t,t' \in G} a_{i,t}\ovl{a_{j,t'}}\big\langle e_{i,st},e_{j,st'}\big\rangle_{H}
		=\sum_{i,j \in I}\sum_{t,t' \in G} a_{i,t}\ovl{a_{j,t'}}\varphi_{ij}(t^{-1}t')\\
		&=\sum_{i,j \in I}\sum_{t,t' \in G} a_{i,t}\ovl{a_{j,t'}}\big\langle e_{i,t},e_{j,t'}\big\rangle_{H}
		=\Bgnorm{\sum_{i \in I, t \in G} a_{i,t}e_{i,t}}_{H}^2.
\end{align*}
Hence, we can define the following surjective isometric operator $\theta_s \co H \to  H$, $e_{i,t} \mapsto e_{i,st}$. Consequently, we obtain a group action $\theta$ of $G$ on the Hilbert space $H$. In order to simplify the notations in the sequel of the proof, in the von Neumann algebra $\Gamma_{-1}(H)$, we use the notation $\omega_{i,s}$ instead of $\omega(e_{i,s})$. For any $s \in G$, we define the trace preserving $*$-automorphism
\begin{equation*}
\alpha(s)=\Gamma_{-1}(\theta_s)    \co\begin{cases}  \Gamma_{-1}(H)   &  \longrightarrow  \Gamma_{-1}(H)  \\
     \omega_{i,t}  &  \longmapsto   \omega_{i,st}  \end{cases}.
\end{equation*}
The group homomorphism $\alpha \co G \to \Aut(\Gamma_{-1}(H))$ allows us to define the twisted crossed product von Neumann algebra $\Gamma_{-1}(H) \rtimes_{\sigma,\alpha} G$. We identify $\Gamma_{-1}(H)$ and $\VN(G,\sigma)$ as subalgebras of $\Gamma_{-1}(H) \rtimes_{\sigma,\alpha} G$. We can write the first relations of commutation \ref{equa-relations-de-commutation} as
\begin{equation}
\label{Relations-commutation}
\lambda_{\sigma,s} \omega_{i,t}
=\omega_{i,st} \lambda_{\sigma,s}
\end{equation}
We denote by $\tau$ the faithful finite normal trace on $\Gamma_{-1}(H)$. Recall that, for any $s \in G$, the map $\alpha(s)$ is trace preserving. Thus, the trace $\tau$ is $\alpha$-invariant. We equip $\Gamma_{-1}(H) \rtimes_{\sigma,\alpha} G$ with the induced canonical finite trace $\tau_\rtimes$. Now, we introduce the von Neumann algebra
\begin{equation}
\label{Von-Neumann-algebra-QWEP}
M=\M_I\big(\Gamma_{-1}(H) \rtimes_{\sigma,\alpha} G\big).
\end{equation}
equipped with its canonical trace $\tr \ot \tau_{\rtimes}$ and we consider the  element $d=\sum_{i \in I} e_{ii} \ot \omega_{i,e}$ 
of $M$. By \ref{Equa-eig-de-norme-1} and \eqref{Equa-omega2=I}, it is easy to see\footnote{\thefootnote. We have
$$
d^2
=\sum_{i,j \in I} (e_{ii} \ot \omega_{i,e})(e_{jj} \ot \omega_{j,e})
=\sum_{i\in I} \big(e_{ii} \ot \omega_{i,e}^2\big)
=\sum_{i\in I} \big(e_{ii} \ot 1_{\Gamma_{-1}(H) \rtimes_{\sigma,\alpha} G}\big)
=1_M.
$$} that $d^2=1_M$. We let $J_1 \co \M_I(\VN(G,\sigma)) \to M$ the canonical unital $*$-monomorphism and we define the unital $*$-monomorphism 
\begin{equation*}
\begin{array}{cccc}
    J_0  \co &  \M_I(\VN(G,\sigma))   &  \longrightarrow   &  M  \\
           &  e_{kl} \ot \lambda_{\sigma,t}   &  \longmapsto       &  d(e_{kl} \ot \lambda_{\sigma,t})d=e_{kl} \ot \omega_{k,e}\lambda_{\sigma,t}\omega_{l,e}  \\
\end{array}.
\end{equation*}
It is not difficult to check that the maps $J_0$ and $J_1$ are trace preserving, hence markovian. Now, for any $i,j,k,l \in I$ and any $s,t \in G$ we have
\begingroup
\allowdisplaybreaks
\begin{align*}
\MoveEqLeft
(\tr \ot \tau_{\rtimes})\big(J_{1}(e_{ij} \ot \lambda_{\sigma,s})J_{0}(e_{kl} \ot \lambda_{\sigma,t})\big)
		=(\tr \ot \tau_{\rtimes})\big((e_{ij} \ot \lambda_{\sigma,s})(e_{kl} \ot \omega_{k,e}\lambda_{\sigma,t}\omega_{l,e})\big)\\
		&=(\tr \ot \tau_{\rtimes})\big(e_{ij}e_{kl} \ot \lambda_{\sigma,s}\omega_{k,e}\lambda_{\sigma,t}\omega_{l,e}\big)
		=\tr(e_{ij}e_{kl})\tau_{\rtimes}(\lambda_{\sigma,s}\omega_{k,e}\lambda_{\sigma,t}\omega_{l,e})\\
		&=\delta_{jk}\delta_{il}\ \tau_{\rtimes}(\omega_{k,s}\lambda_{\sigma,s}\omega_{l,t}\lambda_{\sigma,t})\quad \text{by \eqref{Relations-commutation}}\\
		&=\delta_{jk}\delta_{il}\ \tau_{\rtimes}(\omega_{k,s}\omega_{l,st}\lambda_{\sigma,s}\lambda_{\sigma,t})\quad \text{by \eqref{Relations-commutation}}\\
		&=\delta_{jk}\delta_{il}\sigma(s,t)\ \tau_{\rtimes}(\omega_{k,s}\omega_{l,st}\lambda_{\sigma,st})
		=\delta_{jk}\delta_{il}\delta_{e,st}\sigma(s,t)\ \tau_{}(\omega_{k,s}\omega_{l,st})\quad \text{by \eqref{formule-trace-crossed}}\\
		&=\delta_{jk}\delta_{il}\delta_{e,st}\sigma(s,t)\big\langle e_{k,s}, e_{l,st}\big\rangle \qquad \text{by \eqref{petite formule de Wick}}\\
		&=\delta_{jk}\delta_{il}\delta_{e,st}\sigma(s,t)\varphi_{kl}(t)
		=\delta_{jk}\delta_{il}\delta_{s,t^{-1}}\sigma(s,t)\ \varphi_{ji}(s^{-1})\\
		&=\delta_{jk}\delta_{il}\delta_{s,t^{-1}}\sigma(s,t)\ \varphi_{ij}(s)\quad \text{by \eqref{symetrie-kernel}}\\
	&=\varphi_{ij}(s)\tr\big(e_{ij}e_{kl}\big) \tau_{G,\sigma}\big(\lambda_{\sigma,s}\lambda_{\sigma,t}\big)	
	=\varphi_{ij}(s)(\tr \ot \tau_{G,\sigma})\big(e_{ij}e_{kl} \ot \lambda_{\sigma,s}\lambda_{\sigma,t}\big)\\
		&=\varphi_{ij}(s)(\tr \ot \tau_{G,\sigma})\big((e_{ij} \ot \lambda_{\sigma,s})(e_{kl} \ot \lambda_{\sigma,t})\big) \\
		&=(\tr \ot \tau_{G,\sigma})\big(\big([M_{\varphi_{ij}}](e_{ij} \ot \lambda_{\sigma,s})\big)(e_{kl} \ot \lambda_{\sigma,t})\big).
\end{align*}
\endgroup
Hence, for any $x,y \in \M_2(\VN(G,\sigma))$, we deduce that
$$
(\tr \ot \tau_{G,\sigma})\big(\big([M_{\varphi_{ij}}](x)\big)y\big)
=(\tr \ot \tau_{\rtimes})\big(J_1(x)J_{0}(y)\big)
=(\tr \ot \tau_{G,\sigma})\big(J_{0}^*J_{1}(x)y\big).
$$
We conclude that $[M_{\varphi_{ij}}]=J_0^*\circ J_1$, i.e. that the map $[M_{\varphi_{ij}}]$ is factorizable. 
\end{proof}

\subsection{Application to the noncommutative Matsaev inequality}
\label{Application-Matsaev}

In this section, we give an application of Theorem \ref{Th-Factorizable-matricial mutipliers}. Other applications will be given in subsequent publications. If $1 \leq p \leq \infty$ we denote by $S \co \ell^p \to \ell^p$ the right shift operator defined by $S(a_0,a_1,a_2,\ldots) \ov{\mathrm{def}}{=} (0,a_0,a_1,a_2,\ldots)$. If $1<p<\infty$, $p\not=2$, the validity of the following inequality
\begin{equation}
\label{NCMatsaev} 
\bnorm{P(T)}_{\L^p(M) \to \L^p(M)} 
\leq \bnorm{P(S)}_{\cb ,\ell^p \to \ell^p}
\end{equation}
is open within the class of all contractions $T \co \L^p(M) \to \L^p(M)$ on a noncommutative $\L^p$-space $\L^p(M)$ and all complex polynomials $P$. We refer to the papers \cite{Arh1}, \cite{ALM} and \cite{Pel2} for more information on this problem. The following result allows us to generalize \cite[Corollary 4.5 and Corollary 4.7]{Arh1}.

\begin{thm}
\label{Application-NC-Matsaev}
Let $G$ be a discrete group and $\sigma$ be a $\T$-valued $2$-cocycle on $G$ such that for any real Hilbert space $H$, any action $\alpha$ from $G$ onto $\Gamma_{-1}(H)$ the crossed product $\Gamma_{-1}(H) \rtimes_{\alpha} G$ has $\QWEP$. Let $\varphi \co G \to \R$ be a real function which induces a (selfadjoint) contractively decomposable Fourier multiplier $M_\varphi \co \VN(G,\sigma) \to \VN(G,\sigma)$. Suppose $1 \leq p \leq \infty$. Then, the induced completely contractive Fourier multiplier $M_\varphi \co \L^p\big(\VN(G,\sigma)\big) \to \L^p\big(\VN(G,\sigma)\big)$ satisfies the noncommutative Matsaev inequality \eqref{NCMatsaev}. More precisely, for any complex polynomial $P$, we have
$$
\bnorm{P(M_\varphi)}_{\cb,\L^p(\VN(G,\sigma)) \to \L^p(\VN(G,\sigma))} 
\leq \bnorm{P(S)}_{\cb ,\ell^p \to \ell^p}.
$$
\end{thm}

\begin{proof}
Using \eqref{Normes-cb-cocycl-et-sans}, we can suppose that $\sigma=1$. Using Proposition \ref{prop-prop-P-amenable-groups-multipliers}, we see that there exist Fourier multipliers $M_{\psi_1}, M_{\psi_2} \co \VN(G) \to \VN(G)$ such that the map
$$
\begin{bmatrix}
   M_{\psi_1}  &  M_\varphi \\
   M_\varphi^\circ  &  M_{\psi_2}\\
\end{bmatrix}
\co \M_2(\VN(G)) \to \M_2(\VN(G))
$$
is unital, completely positive, selfadjoint and weak* continuous. Note that by Lemma \ref{Lemma-extension-L1} and interpolation, the previous map induces a (completely contractive) well-defined map on $S^p_2(\L^p(\VN(G)))$. For any complex polynomial $P$, we obtain
\begin{align*}
\MoveEqLeft
  \bnorm{P(M_\varphi)}_{\cb,\L^p(\VN(G)) \to \L^p(\VN(G))}  
		\leq \left|\left|\begin{bmatrix}
    P(M_{\psi_1}) & P(M_\varphi) \\
    P(M_\varphi^\circ) & P(M_{\psi_2}) \\
\end{bmatrix} \right|\right|_{\cb,S^p_2(\L^p(\VN(G))) \to S^p_2(\L^p(\VN(G)))}\\
		&=\left|\left|P\left(\begin{bmatrix}
    M_{\psi_1} & M_\varphi \\
    M_\varphi^\circ & M_{\psi_2} \\
\end{bmatrix}\right) \right|\right|_{\cb,S^p_2(\L^p(\VN(G))) \to S^p_2(\L^p(\VN(G)))}.
\end{align*}
By Theorem \ref{Th-Factorizable-matricial mutipliers}, the operator $\begin{bmatrix}
    M_{\psi_1} & M_\varphi \\
    M_\varphi^\circ & M_{\psi_2} \\
\end{bmatrix} \co \M_2(\VN(G)) \to \M_2(\VN(G))$ is $\QWEP$-factorizable. Using \cite[Theorem 4.4]{HaM}, we deduce that this operator is dilatable on a von Neumann algebra and it is left to the reader to check that this von Neumann algebra is QWEP. Finally, it is not difficult to deduce that the operator $\Id_{\B(\ell^2)} \ot \begin{bmatrix}
    M_{\psi_1} & M_\varphi \\
    M_\varphi^\circ & M_{\psi_2} \\
\end{bmatrix} \co \B(\ell^2) \otvn \M_2(\VN(G)) \to \B(\ell^2) \otvn \M_2(\VN(G))$ is also dilatable on a QWEP von Neumann algebra. We conclude by using \cite[Corollary 2.6 and (1.5)]{Arh1} that 
\begin{align*}
\MoveEqLeft
\left|\left|P\left(\begin{bmatrix}
    M_{\psi_1} & M_\varphi \\
    M_\varphi^\circ & M_{\psi_2} \\
\end{bmatrix}\right) \right|\right|_{\cb,S^p_2(\L^p(\VN(G))) \to S^p_2(\L^p(\VN(G)))}\\
&=\left|\left|P\left(\Id_{S^p} \ot
	\begin{bmatrix}
    M_{\psi_1} & M_\varphi \\
    M_\varphi^\circ & M_{\psi_2} \\
\end{bmatrix}\right) \right|\right|_{S^p(S^p_2(\L^p(\VN(G))))\to S^p(S^p_2(\L^p(\VN(G))))} 
\leq \bnorm{P(S)}_{\cb ,\ell^p \to \ell^p}. 
\end{align*}
The proof is complete.
\end{proof}




\vspace{0.4cm}

\textbf{Acknowledgment}. 
The authors would like to thank Gilles Pisier, Hatem Hamrouni, Colin Reid, Anne Thomas, Pierre de la Harpe, Christian Le Merdy, \'Eric Ricard, Antoine Derighetti, Sven Raum, Adam Skalski, Ami Viselter, Mikael de la Salle, Catherine Aaron, Dominique Manchon, Yves Stalder and Bachir Bekka for some discussions. We will thank Erwin Neuhardt for a useful remark.

Finally, we especially thank the referee of the present paper for a detailed report containing some corrections and some improvements.


Both authors were supported by the grant ANR-18-CE40-0021 (project HASCON).
The second author was partly supported by the grant ANR-17-CE40-0021 of the French National Research Agency ANR (project Front).


\vspace{0.2cm}
\footnotesize{
\noindent C\'edric Arhancet\\ 
\noindent6 rue Didier Daurat, 81000 Albi, France\\
URL: \href{http://sites.google.com/site/cedricarhancet}{https://sites.google.com/site/cedricarhancet}\\
cedric.arhancet@protonmail.com\\

\noindent Christoph Kriegler\\
Universit\'e Clermont Auvergne\\
CNRS\\
LMBP\\
F-63 000 Clermont-Ferrand, France \\
URL: \href{http://lmbp.uca.fr/~kriegler/indexenglish.html}{http://lmbp.uca.fr/{\raise.17ex\hbox{$\scriptstyle\sim$}}\hspace{-0.1cm} kriegler/indexenglish.html}\\
christoph.kriegler@uca.fr

\end{document}